\documentclass{article}

\usepackage[margin=1in]{geometry}
\usepackage[numbers]{natbib}

\usepackage[utf8]{inputenc} 
\usepackage[T1]{fontenc}    
\usepackage[colorlinks,linkcolor=blue,filecolor=blue,citecolor=black,urlcolor=blue]{hyperref}       
\usepackage{url}    

\usepackage{booktabs}      
\usepackage{amsfonts}       
\usepackage{nicefrac}       
\usepackage{microtype}      
\usepackage[table]{xcolor}         
\usepackage{tocvsec2}
\usepackage{titletoc}

\usepackage{xspace}

\usepackage{amsmath}
\usepackage{amsthm}
\usepackage{amssymb}

\usepackage{multicol}
\usepackage{multirow}
\usepackage{makecell}
\usepackage{enumitem}
\usepackage{wrapfig}
\usepackage{dsfont}

\usepackage{graphicx,epstopdf}
\usepackage{subfigure} 

\usepackage{algorithm}
\usepackage{algorithmic}

\usepackage{threeparttable, booktabs}
\usepackage{mathtools}
\mathtoolsset{showonlyrefs=true}
\usepackage{float}

\usepackage{hyperref}[pdfusetitle]
\hypersetup{colorlinks,urlcolor=blue}
\usepackage{tabularray}

\usepackage[skins,listings]{tcolorbox}

\newtcblisting{tikzexample}{
    sidebyside,
    center lower,
    bicolor,
    colbacklower=white,
    sharp corners,
    frame engine=empty
}

\usepackage{quiver}

\restylefloat{algorithm} 
\newtheorem{assumption}{Assumption}
\newtheorem{theorem}{Theorem}
\newtheorem{remark}{Remark}
\newtheorem{lemma}{Lemma}

\newtheorem{corollary}{Corollary}

\usepackage{diagbox}

\newcommand{\RR}{\mathbb{R}}

\newcommand{\ie}{{\emph i.e.}\xspace}

\newcommand{\DefinedAs }[0]{\mathrel{\mathop:}=}
\DeclareMathOperator*{\argmin}{argmin}

\newcommand{\ours}{{\sc SPARKLE}\xspace}
\newcommand{\oursgt}{{\sc SPARKLE-GT}\xspace}
\newcommand{\oursed}{{\sc SPARKLE-ED}\xspace}
\newcommand{\oursextra}{{\sc SPARKLE-EXTRA}\xspace}

\allowdisplaybreaks

\title{{ SPARKLE}: A Unified Single-Loop Primal-Dual  Framework for Decentralized Bilevel Optimization}

\author{
  Shuchen Zhu\thanks{Equal Contribution} \\
  Peking University\\
  \texttt{shuchenzhu@stu.pku.edu.cn} \\
   \and
  Boao Kong$^*$ \\
  Peking University \\
  \texttt{kongboao@stu.pku.edu.cn} \\
  \and
  Songtao Lu \\
  IBM Research \\
  \texttt{songtao@ibm.com} \\
  \and
  Xinmeng Huang \\
  University of Pennsylvania \\
  \texttt{
xinmengh@sas.upenn.edu} \\
  \and
  Kun Yuan\thanks{Corresponding Author: Kun Yuan. Kun Yuan is also affiliated with National Engineering Labratory for Big Data Analytics and Applications, and AI for Science Institute, Beijing, China.} \\
  Peking University \\
  \texttt{
kunyuan@pku.edu.cn} \\
}

\begin{document}

\maketitle

\begin{abstract}
This paper studies decentralized bilevel optimization, in which multiple agents collaborate to solve problems involving nested optimization structures with neighborhood communications. Most existing literature primarily utilizes gradient tracking to mitigate the influence of data heterogeneity, without exploring other well-known heterogeneity-correction techniques such as EXTRA or Exact Diffusion. Additionally, these studies often employ identical decentralized strategies for both upper- and lower-level problems, neglecting to leverage distinct mechanisms across different levels. To address these limitations, this paper proposes \textbf{SPARKLE}, a unified \underline{\bf S}ingle-loop \underline{\bf P}rimal-dual \underline{\bf A}lgo\underline{\bf R}ithm framewor\underline{\bf K} for decentra\underline{\bf L}ized bil\underline{\bf E}vel
 optimization. \ours offers the flexibility to incorporate various heterogeneity-correction strategies into the algorithm. Moreover, \ours allows for different strategies to solve upper- and lower-level problems. We present a unified convergence analysis for \ours, applicable to all its variants, with state-of-the-art convergence rates compared to existing decentralized bilevel algorithms. Our results further reveal that EXTRA and Exact Diffusion are more suitable for decentralized bilevel optimization, and using mixed strategies in bilevel algorithms brings more benefits than relying solely on gradient tracking.

\end{abstract}

\settocdepth{part}

\section{Introduction}
Numerous modern machine learning tasks, such as reinforcement learning~\cite{hong2023two}, meta-learning~\cite{bertinetto2019meta}, adversarial learning~\cite{madry2018towards}, hyper-parameter optimization~\cite{franceschi2018bilevel}, and imitation learning~\cite{arora2020provable}, 
entail nested optimization formulations that extend beyond the traditional single-level paradigm. For instance, hyper-parameter optimization aims to identify the optimal hyper-parameters for a specific learning task in the upper level by minimizing the validation loss, achieved through training models in the lower-level process. This nested optimization structure has spurred significant attention towards Stochastic Bilevel Optimization (SBO). 
Since the size of data samples involved in bilevel problems has become increasingly large, this paper investigates decentralized algorithms over a network of $n$ agents (nodes) that collaborate to solve the following distributed bilevel optimization problem:
\begin{subequations}
\label{DSBO_problem}
\begin{align}
\hspace{-1mm}\min_{x\in\mathbb{R}^{p}}\quad&\Phi(x)=f(x,y^{\star}(x))\DefinedAs \frac{1}{n}\sum_{i=1}^nf_i(x,y^\star(x)), \label{eq:prob-upper} &\mbox{(upper-level)} \\
\hspace{-1mm}\mathrm{s.t.}\quad&y^\star(x)=\argmin_{y\in\mathbb{R}^{q}}\Big\{ g(x,y)\DefinedAs  \frac{1}{n}\sum_{i=1}^ng_i(x,y)\Big\}.  &\mbox{(lower-level)}\label{eq:prob-lower} 
\end{align}
\end{subequations}
In this formulation, each agent $i$ holds a private
upper-level objective function $f_i: \mathbb{R}^p \times \mathbb{R}^q \to \mathbb{R}$ and a strongly convex lower-level objective function $g_i: \mathbb{R}^p \times \mathbb{R}^q \to \mathbb{R}$ defined as:
\begin{align}\label{eq:fi-gi}
 f_i(x, y)=\mathbb{E}_{\phi \sim \mathcal{D}_{f_i}}[F_i(x, y ; \phi)], \quad \quad  g_i(x, y)=\mathbb{E}_{\xi \sim \mathcal{D}_{g_i}}[G_i(x, y ; \xi)],
\end{align}
where $\mathcal{D}_{f_i}$ and $\mathcal{D}_{g_i}$ represent the local data distributions at agent $i$. This paper does not make any assumptions about these data distributions, implying there might be data heterogeneity across agents.

\paragraph{Linear speedup and transient iteration complexity.}
A decentralized stochastic algorithm achieves {\em linear speedup} if its iteration complexity decreases linearly with the network size $n$. Additionally, the {\em transient iteration complexity} refers to the number of transient iterations a decentralized algorithm must undergo to achieve the asymptotic linear speedup stage. The fewer the transient iterations, the faster the algorithm can achieve linear speedup. This paper aims to develop decentralized stochastic bilevel algorithms that can achieve linear speedup with as few transient iterations as possible.

\paragraph{Limitations in previous works.} A significant challenge in decentralized bilevel optimization lies in accurately estimating the hyper-gradient $\nabla \Phi(x)$ through neighborhood communications. Several studies have emerged to effectively address this challenge, such as those by~\cite{chen2022decentralized, lu2022stochastic, yang2022decentralized, dong2023single, gao2023convergence, niu2023distributed, zhang2023communication,kong2024decentralized}. However, existing works suffer from several critical limitations:

\begin{itemize}[leftmargin=1em]
\item \textbf{Stringent assumptions and inadequate convergence analysis.} Many existing studies rely on stringent assumptions to ensure convergence. For instance, references~\cite{chen2022decentralized,chen2023decentralized,lu2022stochastic,gao2023convergence,yang2022decentralized} assume bounded gradients, while reference~\cite{kong2024decentralized} assumes bounded data heterogeneity (also known as bounded gradient dissimilarity). These restrictive assumptions do not arise in centralized bilevel optimization, implying their potential unnecessity. Moreover, some of these works suffer from inadequate convergence analysis, unable to clarify the transient iteration complexity~\cite{chen2022decentralized,lu2022stochastic,chen2023decentralized} or provide a sharp estimation of the influence of network topologies~\cite{yang2022decentralized,gao2023convergence}.

\item \textbf{Limited exploration of various heterogeneity-correction techniques.} Several concurrent studies~\cite{dong2023single,zhang2023communication,niu2023distributed} have utilized Gradient Tracking {\sc (GT)} \cite{xu2015augmented,di2016next,nedic2017achieving} to remove the  assumption of bounded data heterogeneity. However, it remains uncertain whether {\sc GT} is the most suitable mechanism for decentralized bilevel optimization. Many other techniques are also useful for addressing data heterogeneity in single-level decentralized optimization, such as {EXTRA}~\cite{shi2015extra} and Exact-Diffusion (ED)~\cite{yuan2018exact1,li2017decentralized,yuan2020influence} (which is also known as D$^2$~\cite{tang2018d}). Even within GT, there are variants including Adapt-Then-Combine GT (ATC-GT)~\cite{xu2015augmented}, non-ATC-GT~\cite{nedic2017achieving}, and semi-ATC-GT~\cite{di2016next}. It remains unexplored whether these techniques for mitigating data heterogeneity converge and even outperform 
GT when employed in decentralized bilevel algorithms.

\item \textbf{Unknown effects of employing different upper- and lower-level update strategies.} 
In bilevel optimization, the challenges in solving the upper- and lower-level problems differ substantially. For instance, the upper-level problem \eqref{eq:prob-upper} is non-convex, whereas the lower-level problem \eqref{eq:prob-lower} is strongly convex. Moreover, estimating the gradient at the lower level is considerably simpler compared to estimating the hyper-gradient at the upper level. Understanding the roles of updates at each level is crucial to develop more efficient algorithms. However, most existing algorithms employ the same decentralized methods to solve both the upper- and lower-level problems. For example, references~\cite{chen2022decentralized,yang2022decentralized,kong2024decentralized} utilize decentralized gradient descent (DGD) for updates at both levels while~\cite{zhang2023communication,niu2023distributed,dong2023single} leverage GT, overlooking the potential advantages of mixed strategies.
\end{itemize}

To address these limitations, several critical questions naturally arise: Should each heterogeneity-correction mechanism listed in~\cite{xu2015augmented,di2016next,nedic2017achieving,yuan2018exact1,li2017decentralized,yuan2020influence,tang2018d} be explored one-by-one? Should we consider combining any two of these techniques to update the upper and lower-level problems, respectively? It is evident that examining each individual heterogeneity-correction technique, and even exploring their combinations, would involve an unbearable amount of effort.

\begin{table*}[t!]
\Large
\renewcommand\arraystretch{1.3}
\vspace{-10pt}
\caption{\small Comparison between different decentralized stochastic bilevel algorithms. $K$ denotes the number of (upper-level) iterations; $1-\rho$ denotes the spectral gap of the mixing matrix (see Assumption~\ref{net}); $b^2$ bounds the gradient dissimilarity; $\varepsilon$ is the target stationarity such that $\sum_{k=0}^{K-1}\mathbb{E}[\|\nabla\Phi(\bar {x}^k)\|^2]/K<\varepsilon$; $p$ and $q$ are the dimensions of the upper- and lower-level variables, reflecting per-round communication costs.
Assumptions of bounded gradient,  Lipschitz continuity, and bounded gradient dissimilarity are abbreviated as  BG, LC, and BGD, respectively. We also list the best-known results of single-level GT, EXTRA, and ED at the bottom. }
\vspace{2mm}
\centering
\resizebox{\textwidth}{!}{
\begin{threeparttable}
\begin{tabular}{lcccccl}
\toprule
\textbf{\LARGE Algorithms} & \textbf{\LARGE Assumption}$^\triangleright$ & \textbf{\LARGE A. Rate.}  $^\Diamond$ & \textbf{\LARGE A. Comp.} $^\dagger$& \textbf{\LARGE A. Comm.} $^\ddag$ & \hspace{-0.45cm}\textbf{\LARGE Tran. Iter.}$^\triangleleft$ &    \textbf{\LARGE Loopless}                                            \\ \midrule
\LARGE DSBO~\cite{chen2022decentralized}                &    LC                                    & \LARGE$\frac{1}{\sqrt{K}}$                                                &\LARGE $\frac{1}{\varepsilon^3}$& \LARGE $\left({pq}\log(\frac{1}{\varepsilon})+\frac{q}{\varepsilon}\right)\frac{1}{\varepsilon^2}$                          & \hspace{-0.7cm} N.A.                                                                     &   \hspace{5mm}No                                                             \vspace{1.5mm}      \\
\LARGE MA-DSBO~\cite{chen2023decentralized}                 &   LC                                     & \LARGE$\frac{1}{\sqrt{K}}$                                                &\LARGE $\frac{1}{\varepsilon^2}\log(\frac{1}{\varepsilon})$            &\LARGE $\frac{q}{\varepsilon^2}\log(\frac{1}{\varepsilon})+\frac{p}{\varepsilon^2}$              & \hspace{-0.7cm} N.A.                                                                     & \hspace{4mm} No                                                                 \vspace{1.5mm}    \\
\LARGE SLAM~\cite{lu2022stochastic}                 &        LC                                   &\LARGE $\frac{1}{\sqrt{nK}}$                                              &\LARGE $\frac{1}{n \varepsilon^2}\log(\frac{1}{\varepsilon})$             &\LARGE $\frac{p+q}{n \varepsilon^2}$            & \hspace{-0.7cm} N.A.                                                                     & \hspace{4mm}  No                                                                 \vspace{1.5mm}   \\
\LARGE MDBO~\cite{gao2023convergence}                &  BG                                        & \LARGE $ \frac{1}{\sqrt{nK}}$                                              &\LARGE $\frac{1}{n\varepsilon^2}\log(\frac{1}{\varepsilon})$               &\LARGE $\frac{p+q}{n\varepsilon^2}$             &\hspace{-0.7cm}\LARGE$\frac{n^3}{(1-\rho)^8}$                                                 & \hspace{4mm} No \vspace{2mm}   \\
\LARGE Gossip DSBO~\cite{yang2022decentralized}               &   BG                                   & \LARGE $\frac{1}{\sqrt{nK}}$                                              &\LARGE $\frac{1}{n \varepsilon^2}\log(\frac{1}{\varepsilon})$              &\LARGE $\frac{q^2}{n \varepsilon^2}\log(\frac{1}{\varepsilon})+\frac{pq}{n \varepsilon^2}$          &\hspace{-0.7cm} \LARGE $\frac{n^3 }{(1-\rho)^4}$                                                 &     \hspace{4mm}  No                                                                \vspace{1.5mm}   \\

\LARGE LoPA~\cite{niu2023distributed}$^*$            &     BGD                                &\LARGE $\frac{1}{\sqrt{K}}$                                                      &\LARGE $\frac{1}{\varepsilon^2}$          &\LARGE $\frac{p+q}{\varepsilon^2}$                                        & \hspace{-0.7cm}N.A.         & \hspace{4mm}Yes       \vspace{1.5mm}    \\ 
\LARGE D-SOBA~\cite{kong2024decentralized}             &    BGD                                  &\LARGE $\frac{1}{\sqrt{nK}}$                                                      &\LARGE $\frac{1}{n\varepsilon^2}$          &\LARGE $\frac{p+q}{n\varepsilon^2}$                                        &\LARGE $\hspace{-0.7cm} \max\left\{ \frac{ n^3}{(1-\rho)^2}, \frac{ n^3b^2}{(1-\rho)^4}\right\}$        &  \hspace{3mm} Yes     \vspace{1.5mm}    \\
\rowcolor{blue!15}{\LARGE \textbf{SPARKLE-GT (ours)}} 
&       \textbf{None}                            &\LARGE $\boldsymbol{\frac{1}{\sqrt{nK}}}$                                                      &\LARGE $\boldsymbol{\frac{1}{n\varepsilon^2}}$   &\LARGE $\boldsymbol{\frac{ap+q}{n\varepsilon^2}}^{\sharp}$                                                & \LARGE$\hspace{-0.7cm} \boldsymbol{\max\left\{ \frac{ n^3}{(1-\rho)^2}, \frac{ n}{(1-\rho)^{8/3}}\right\}}$     &  \hspace{3mm}  \textbf{Yes}      \vspace{1.5mm} \\
\rowcolor{blue!15}{\LARGE \textbf{SPARKLE-EXTRA (ours)}}               &     \textbf{None}                              & \LARGE$\boldsymbol{\frac{1}{\sqrt{nK}}}$                                                      & \LARGE$\boldsymbol{\frac{1}{n\varepsilon^2}}$  & \LARGE$\boldsymbol{\frac{ap+q}{n\varepsilon^2}}^{\sharp}$                                                 &\LARGE $\hspace{-0.7cm} \boldsymbol{ \frac{ n^3}{(1-\rho)^2}}$        &      \hspace{3mm}  \textbf{Yes }     \vspace{1mm}  \\  
\rowcolor{blue!15}{\LARGE \textbf{SPARKLE-ED (ours)}}               &   \textbf{None}                                     &\LARGE $\boldsymbol{\frac{1}{\sqrt{nK}}}$                                                      &\LARGE $\boldsymbol{\frac{1}{n\varepsilon^2}}$  &\LARGE $\boldsymbol{\frac{ap+q}{n\varepsilon^2}}^{\sharp}$                                                &\LARGE $\hspace{-0.7cm} \boldsymbol{ \frac{ n^3}{(1-\rho)^2}}$        & \hspace{3mm} \textbf{Yes }  \vspace{0.5mm}\\ 
\bottomrule
{\LARGE \text{Single-level GT}~\cite{alghunaim2021unified,koloskova2021improved}} 
&       None                          &\LARGE $\frac{1}{\sqrt{nK}}$                                                      &\LARGE $\frac{1}{n\varepsilon^2}$   &\LARGE $\frac{p}{n\varepsilon^2}$                                                &\LARGE $\hspace{-0.7cm} \max\left\{ \frac{ n^3}{(1-\rho)^2}, \frac{ n}{(1-\rho)^{8/3}}\right\}$        &  \hspace{3mm}  Yes      \vspace{1.5mm} \\
{\LARGE \text{Single-level EXTRA}~\cite{alghunaim2021unified}}               &     None                              &\LARGE $\frac{1}{\sqrt{nK}}$                                                      &\LARGE $\frac{1}{n\varepsilon^2}$  &\LARGE $\frac{p}{n\varepsilon^2}$                                                 &\LARGE $\hspace{-0.7cm} \frac{ n^3}{(1-\rho)^2}$        &      \hspace{3mm}  Yes     \vspace{1mm}  \\  
{\LARGE \text{Single-level ED}~\cite{alghunaim2021unified}}               &   None                                  &\LARGE $\frac{1}{\sqrt{nK}}$                                                      &\LARGE $\frac{1}{n\varepsilon^2}$  &\LARGE $\frac{p}{n\varepsilon^2}$                                                &\LARGE $\hspace{-0.7cm}  \frac{ n^3}{(1-\rho)^2}$        & \hspace{3mm} Yes  \vspace{1.5mm}\\ 
\bottomrule
\end{tabular}
\begin{tablenotes}
    \LARGE
    \item[$\Diamond$] The convergence rate when $K\to \infty$ (smaller is better). 
    \item[$\dagger$] The number of gradient/Jacobian/Hessian evaluations per agent to achieve $\varepsilon$-accuracy when $\epsilon \to 0$ (smaller is better).
     \item[$\ddag$] The communication  costs per agent to achieve $\varepsilon$-stationarity when $\epsilon \to 0$ (smaller is better).
    
     \item[$\triangleleft$] The transient iteration complexity to achieve linear speedup (smaller is better). ``N.A.'' means that the algorithm cannot achieve linear speedup or the transient time cannot be accessed from existing convergence analysis.
    \item[$\triangleright$]Additional assumptions
    beyond Assumption~\ref{smooth}. 
    \item[$*$]LoPA solves the personalized problem, where the lower-level objectives are local to agents.
    \item[$\sharp$]$a>0$ measures the relative sparsity of the mixing weights $\mathbf{W}_x,\mathbf{W}_y,\mathbf{W}_z$, which can be very small in certain cases. 
    Here $1-\rho$ in \textbf{Tran. Iter.} denotes the smallest spectral gap of $\mathbf{W}_x,\mathbf{W}_y,\mathbf{W}_z$.
     See more discussions in Appendix \ref{Theoretical Gap}. 
    
\end{tablenotes}
\end{threeparttable}}
\label{table:comparison}
\vspace{-10pt}
\end{table*}

\vspace{1mm}
\noindent \textbf{Main results and contributions.} This paper addresses all the aforementioned limitations without exhaustively exploring all heterogeneity-correction techniques. Our main results are as follows.

\begin{itemize}[leftmargin=1em]
\item \textbf{A unified decentralized bilevel framework.} To avoid  examining each single heterogeneity-correction technique, we propose \textbf{SPARKLE}, a unified \underline{\bf S}ingle-loop \underline{\bf P}rimal-dual \underline{\bf A}lgo\underline{\bf R}ithm framewor\underline{\bf K} for decentra\underline{\bf L}ized bil\underline{\bf E}vel optimization. By specifying certain hyper-parameters, {\ours} can be tailored to \ours-EXTRA, \ours-ED, and \ours-GT, which employ {EXTRA}~\cite{shi2015extra}, {ED}~\cite{yuan2018exact1,li2017decentralized}, or multiple {GT} variants~\cite{xu2015augmented,di2016next,nedic2017achieving}, respectively, to facilitate the upper and lower-level problems. Additionally, \ours is the \textit{first} algorithm enabling distinct updating strategies across different levels; for example, one can utilize GT in the upper-level but ED in the lower-level, resulting in a brand new \ours-GT-ED algorithm.

\item \textbf{A unified and sharp analysis of various heterogeneity-correction schemes.} We provide a unified convergence analysis for \ours, which immediately applies to all \ours variants with distinct heterogeneity-correction techniques. The analysis does not require restrictive assumptions such as gradient boundedness used in~\cite{chen2022decentralized,chen2023decentralized,lu2022stochastic,gao2023convergence,yang2022decentralized} or data-heterogeneity bounded used in~\cite{kong2024decentralized}. Moreover, our analysis demonstrates the provable superiority of \ours compared to existing algorithms, as evidenced by the convergence rates listed in Table~\ref{table:comparison}. Most importantly, our analysis shows that both \ours-EXTRA and \ours-ED outperform \ours-GT (see Table~\ref{table:comparison}), implying that {\em GT is not the best} 
scheme for decentralized bilevel optimization.

\item \textbf{Mixing strategies outperform employing GT alone.} We demonstrate how optimization at different levels affects convergence rates. Our theoretical analysis suggests that the updating strategy at {\em the lower level is crucial in determining the overall performance} in decentralized bilevel algorithms. Building upon this insight, we establish that incorporating the ED or EXTRA strategy in the lower-level update phase leads to better transient iteration complexity than relying solely on the GT mechanism in both levels as proposed in~\cite{chen2023decentralized,dong2023single,niu2023distributed}, see Table~\ref{table:mixed} for more details. 

\item \textbf{Comparable performance with single-level algorithms.} We elucidate the comparison between bilevel and single-level stochastic decentralized optimization. On one hand, we demonstrate that 
the convergence performance of all our proposed algorithms is not inferior to their single-level counterparts (see the bottom part in Table~\ref{table:comparison}). On the other hand, by considering specific lower-level loss functions, our bilevel results directly yield the non-asymptotic convergence of corresponding single-level algorithms. This is the {\em first} result demonstrating bilevel optimization essentially subsumes the convergence of the single-level optimization. 
\end{itemize}

Our main results are listed in Table~\ref{table:comparison}. All \ours variants achieve the state-of-the-art asymptotic rate, asymptotic gradient complexity, asymptotic communication cost, and transient iteration complexity under more relaxed assumptions compared to existing methods.

\begin{figure}
\centering\includegraphics[width=0.8\linewidth]{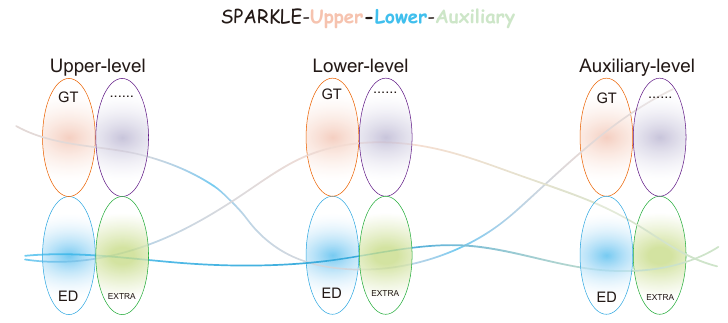}
    \caption{SPARKLE algorithms,  represented by ribbons, employ  mixed decentralized mechanisms at the upper-level, lower-level, and auxiliary-level. Distinct colors denote the various decentralized mechanisms.}
    \label{fig0}
\end{figure}
\begin{table}[t] 
\small   
    \caption{\small The transient iteration complexity of \ours with mixed updating strategies at various levels. The smaller the transient iteration complexity is, the faster the algorithm will achieve its linear speedup stage. The first row and column respectively indicate the updating strategy for the upper- and lower-level problems. 
Please refer to Appendix~\ref{Implementation details} for more implementation details and Appendix~\ref{deviation-tran} for proofs.
    }
    \label{extrasient}
    \centering
    \begin{tabular}{|c|c|c|c|}
        \hline 
        \diagbox[dir=NW]{\textbf{lower}}{\textbf{upper}} & \textbf{ED} & \textbf{EXTRA} & \textbf{GT}  \\
        \hline 
        \textbf{ED} & $\frac{n^3}{(1-{\rho})^2}$ & $\frac{n^3}{(1-{\rho})^2}$ & $\frac{n^3}{(1-{\rho})^2}$  \\
        \hline
         \textbf{EXTRA} & $\frac{n^3}{(1-{\rho})^2}$ & $\frac{n^3}{(1-{\rho})^2}$& $\frac{n^3}{(1-{\rho})^2}$   \\
        \hline
        \textbf{GT} & $\max\left\{  \frac{n^3}{(
1-\rho)^2} ,   \frac{n}{(1-\rho)^{{8}/{3}}}\right\}$ & $\max\left\{  \frac{n^3}{(
1-\rho)^2} ,   \frac{n}{(1-\rho)^{{8}/{3}}}\right\}$& $\max\left\{  \frac{n^3}{(
1-\rho)^2} ,   \frac{n}{(1-\rho)^{{8}/{3}}}\right\}$   \\
        \hline
    \end{tabular}
    \label{table:mixed}
    \vspace{-2mm}
\end{table}

\paragraph{Related works.} A significant challenge in decentralized bilevel optimization is accurately estimating the hyper-gradient $\nabla \Phi(x)$, necessitating solving global lower-level problems and estimating Hessian inversion. To this end, various decentralized techniques have been applied in bilevel optimization, including Neumann series in~\cite{yang2022decentralized}, JHIP oracle in~\cite{chen2022decentralized}, HIGP oracle in~\cite{chen2023decentralized}, and augmented Lagrangian-based communication in~\cite{lu2022stochastic}. Additionally, reference~\cite{kong2024decentralized} proposes a single-loop algorithm utilizing decentralized SOBA. To enhance algorithmic robustness against data heterogeneity, recent studies have employed Gradient Tracking (GT) in both lower- and upper-level optimization. However, existing works built upon GT suffer from several limitations. Results of~\cite{dong2023single,chen2022decentralized} concentrate solely on deterministic cases, while reference~\cite{niu2023distributed} addresses personalized problems in the lower-level, which do not require achieving global consensus in the lower-level problem. Moreover, ~\cite{chen2022decentralized,chen2023decentralized} introduce computationally expensive inner loops for GT steps. None of these works can establish smaller transient iteration complexity than {\sc D-SOBA} for decentralized SBO, even though the latter algorithm employs no heterogeneity-correction technique.

The unified framework for single-level decentralized optimization has been extensively studied in the literature. References \cite{alghunaim2020decentralized,xu2021distributed,jakovetic2018unification} propose frameworks for decentralized composite optimization in deterministic settings, while \cite{alghunaim2021unified} investigates a framework under stochastic settings. However, none of these works can be directly applied to decentralized bilevel algorithms. Several studies \cite{gao2023convergence,zhang2023communication} utilize variance reduction techniques to accelerate the convergence of stochastic decentralized bilevel algorithms. Our proposed \ours framework is orthogonal to variance reduction; it can also incorporate variance-reduced gradient estimation to achieve improved convergence rates.
More relevant works on decentralized optimization and bilevel optimization are discussed in Appendix~\ref{app-related-work}.

\paragraph{Notations.} 
We use lowercase letters to represent vectors and uppercase letters to represent matrices. 
We introduce $\text{col}\{x_1,...,x_n\}:=[x_1^{\top},...,x_n^{\top}]^{\top} \in \mathbb{R}^{pn}$ for brevity. 
Variables with overbar denote the average over all agents. For example, $\bar{x}^k=\sum_{i=1}^n x_i^k/n$. 
We denote $\overline{A}=A-\frac{1}{n}\mathbf{1}_n\mathbf{1}_n^{\top}$ for matrix  $A\in\mathbb{R}^{n\times n}$, where $\mathbf{1}_n\in \mathbb{R}^{n}$ denotes the $n$-dimensional vector with all entries being one.  
For a function $f(x,y):\mathbb{R}^p\times\mathbb{R}^q\to \mathbb{R}$, we use $\nabla_1f(x,y)\in\RR^p$, $\nabla_2 f(x,y)\in\RR^q$ to represent its partial gradients with respect
to $x$ and $y$, respectively. Similarly, $\nabla_{12}f(x,y)\in\mathbb{R}^{p\times q}$, $\nabla_{22} f(x,y)\in\mathbb{R}^{q\times q}$ represent the corresponding Jacobian and Hessian matrix. We use the notation $\lesssim$ to denote inequalities that hold up to constants related to the initialization of algorithms and smoothness constants. 

\section{\ours: A unified framework for decentralized bilevel optimization}
This section develops \ours, a unified framework for decentralized bilevel optimization, and discusses its numerous variants by specifying certain hyper-parameters.

\subsection{Three pillar subproblems in decentralized bilevel optimization.}
\vspace{-1mm}

When solving the upper-level problem \eqref{eq:prob-upper}, it is critical to obtain the hyper-gradient $\nabla\Phi(x)$, which can be expressed as~\cite{ghadimi2018approximation}
\begin{equation}
\label{eq: grad phix}
 \nabla \Phi(x)=\nabla_1 f(x,y^\star(x)) -  \nabla^2_{12} g(x,y^\star(x)) 
  \left[\nabla^2_{22} g(x,y^\star(x))\right]^{-1} \nabla_2 f(x,y^\star(x)).
\end{equation}
Evaluating this hyper-gradient is computationally expensive due to the inversion of the Hessian matrix. This evaluation becomes even more challenging over a decentralized network of collaborative agents. First, the inverse of the Hessian matrix cannot be obtained by simply averaging the local Hessian inverses due to $[\frac{1}{n}\sum_{i=1}^n\nabla^2_{22} g_i(x,y^\star(x))]^{-1}\neq \frac{1}{n}\sum_{i=1}^n [\nabla^2_{22} g_i(x,y^\star(x))]^{-1}$. Second, the global averaging operation cannot be realized through decentralized communication. To overcome these challenges, one can introduce an auxiliary variable $z^{\star}(x)\hspace{-2pt}\DefinedAs\hspace{-1pt} [\nabla^2_{22} g(x,y^\star(x))]^{-1}\nabla_2 f(x,y^{\star}(x))$~\cite{dagreou2022framework}, which is the solution to a quadratic problem
\begin{align}
\label{eq:soba}
z^{\star}(x)=\argmin_{z\in \RR^q}  
\left\{\frac{1}{2} z^\top\nabla_{22}^2 g\left(x, y^\star(x)\right)z-z^{\top}\nabla_{2}f\left(x, y^\star(x)\right)\right\}. 
\end{align}
Once $z^\star(x)$ is derived by solving \eqref{eq:soba}, we can substitute it into \eqref{eq: grad phix} to achieve $\nabla \Phi(x)$. 

Following this idea, solving the distributed bilevel optimization problem \eqref{DSBO_problem} essentially involves solving three subproblems, where $h_i(x,y^\star(x),z):=\frac{1}{2}z^\top\nabla_{22}^2 \hspace{-0.2mm}g_i(x, y^\star(x))z-z^{\top}\nabla_{2}f_i(x, y^\star(x))$, 
\vspace{-1mm}
\begin{subequations}
\label{eq:three-problems}
\begin{align}
\vspace{-1mm}
x^\star &= \argmin_{x\in\mathbb{R}^{p}}  \frac{1}{n}\sum_{i=1}^n f_i(x,y^\star(x)),  & \mbox{(upper-level)}  \label{eq:ul} \\
\vspace{-10mm}
y^\star(x)&=\argmin_{y\in\mathbb{R}^{q}}  \frac{1}{n}\sum_{i=1}^ng_i(x,y),  & \mbox{(lower-level)\vspace{-10mm}} \label{eq:ll}\\
\vspace{-1mm}
z^\star(x)&=\argmin_{z\in\mathbb{R}^{q}}  \frac{1}{n}\sum_{i=1}^n h_i(x,y^\star(x),z). &\mbox{(auxiliary-level)} \label{eq:al}
\end{align}
\end{subequations}
Given the variable \(x\), one can achieve \(y^\star(x)\) by solving the lower-level problem in \eqref{eq:ll}. With \(y^\star(x)\) determined, \(z^\star(x)\) can be obtained by solving the auxiliary-level problem in \eqref{eq:al}. Subsequently, with \(z^\star(x)\) available, one can directly compute the hyper-gradient and solve the upper-level problem in \eqref{eq:ul} using gradient descent. This constitutes the primary methodology to solve problem \eqref{DSBO_problem}.

A bilevel algorithm essentially solves three subproblems listed in \eqref{eq:three-problems}, each formulated as a single-level decentralized optimization problem.
Nevertheless, primary approaches may suffer from nested loops in algorithmic development. A few recent studies \cite{dagreou2022framework,chen2023optimal,zhang2023communication,kong2024decentralized} propose to solve each problem in \eqref{eq:ul}-\eqref{eq:al} approximately with {\em one single} iteration, leading to practical single-loop bilevel algorithms. 
For example, applying a D-SGD step~\cite{sayed2014adaptive} to each of \eqref{eq:ul}-\eqref{eq:al} yields the D-SOBA method~\citep{kong2024decentralized}, while further leveraging the GT technique leads to decentralized bilevel methods in~\cite{chen2022decentralized,dong2023single,zhang2023communication,gao2023convergence}. 

However, it is less explored whether numerous other heterogeneity-correction techniques~\cite{xu2015augmented,di2016next,nedic2017achieving,yuan2018exact1,li2017decentralized,yuan2020influence,tang2018d} beyond GT can be incorporated into algorithmic design to achieve even better performance in bilevel optimization. To avoid exploring each case individually, we next introduce a general framework that unifies all these techniques for solving single-level problems.

\subsection{A unified framework for decentralized single-level optimization.}
\vspace{-1mm}
In this subsection, we consider solving the single-level problem $\min_{x\in \mathbb{R}^p} \frac{1}{n}\sum_{i=1}^n f_i(x)$ over a network of $n$ nodes. For each $k$-th ($k\geq 0$) iteration, we let $x_i^k$ denote the local $x$-variable maintained by the $i$-th agent. Furthermore, we associate the topology with a weight matrix $W = [w_{ij}]_{i,j=1}^n \in \RR^{n\times n}$ in which $w_{ij} \in (0, 1)$ if node $j$ is connected to node $i$ otherwise $w_{ij} = 0$. We use bold symbols to denote stacked vectors or matrices across agents. For example, $\mathbf{x}^k= \text{col}\{x_1^k,...,x_n^k\} \in \RR^{pn}$ and $\mathbf{W}= W\otimes I_{p}$, where $\otimes$ denotes the Kronecker product operator.  

\paragraph{A unified framework with moving average.} Building on the formulation in~\cite{alghunaim2020decentralized,alghunaim2021unified}, we develop a unified primal-dual framework with moving average for decentralized optimization: 
\begin{equation}\label{spdupdate}
\begin{aligned}
\mathbf{r}^{k+1}=(1-\theta)\mathbf{r}^k+\theta\mathbf{g}^k, \quad 
 \mathbf{x}^{k+1}=\mathbf{C}\mathbf{x}^k-\alpha \mathbf{A}\mathbf{r}^{k+1} - \mathbf{B}\mathbf{d}^k,
\quad\mathbf{d}^{k+1}=\mathbf{d}^k+\mathbf{B}\mathbf{x}^{k+1}.
\end{aligned}
\end{equation}
Here $\mathbf{x}^k$ denotes the primal variable, $\mathbf{d}^k$ denotes the dual variable introduced to mitigate the influence of data-heterogeneity, $\mathbf{g}^k$ stacks all (stochastic) gradients evaluated at $x_i^{k}$ for $1\leq i\leq n$, $\mathbf{r}^{k}$ denotes the momentum introduced to boost training with coefficient  $\theta \in [0,1]$, and $\alpha > 0$ is the learning rate. Matrices $\mathbf{A}, \mathbf{B}, \mathbf{C}\in\mathbb{R}^{pn\times pn}$ are adapted from the mixing matrix $\mathbf{W}$, which determine how agents communicate with each other. See Appendix~\ref{primal-dual deviation} for more detailed motivations.

Framework \eqref{spdupdate} unifies various decentralized techniques in the literature. For instance, by letting $\theta = 1$ and specifying $\mathbf{A}, \mathbf{B}, \mathbf{C}$ delicately, framework \eqref{spdupdate} reduces to ED, EXTRA, and numerous GT variants, see Table~\ref{3extable} and Appendix~\ref{primal-dual deviation} for more details. Framework \eqref{spdupdate} is closely related to the unified decentralized method developed in \cite{alghunaim2020decentralized,alghunaim2021unified}. The primary difference lies in the incorporation of the momentum variable $\mathbf{r}^k$, which can help improve the transient iteration complexity of the framework \eqref{spdupdate} and relax the smoothness condition for bilevel algorithms \cite{chen2023optimal}. A detailed comparison between framework \eqref{spdupdate} and that proposed in \cite{alghunaim2020decentralized,alghunaim2021unified} is provided in Appendix \ref{Specific instances}.

\begin{algorithm}[t]
  \caption{\ours:
A unified framework for decentralized stochastic bilevel optimization}
  \label{D-SOBA-SUDA}
  \begin{algorithmic}
  \REQUIRE{Initialize $\mathbf{x}^0=\mathbf{y}^0=\mathbf{z}^0=\mathbf{r}^0=\mathbf{0}$, $\mathbf{d}_x^0=\mathbf{d}_y^0=\mathbf{d}_z^0=\mathbf{0}$, learning rate $\alpha_k,\beta_k,\gamma_k,\theta_k$}.
  \FOR{$k=0,1,\cdots,K-1$}        
  \STATE $\mathbf{y}^{k+1}=\mathbf{C}_y\mathbf{y}^k-\beta_k \mathbf{A}_y\mathbf{v}^k-\mathbf{B}_y\mathbf{d}_y^k$, \quad \hspace{1.9mm}$ \mathbf{d}_y^{k+1}=\mathbf{d}_y^k+\mathbf{B}_y\mathbf{y}^{k+1}; \quad \quad \triangleright \ \mbox{\footnotesize{lower-level update}}$
     \STATE $\mathbf{z}^{k+1}\hspace{0.5mm}=\mathbf{C}_z\mathbf{z}^k-\gamma_k \mathbf{A}_z\mathbf{p}^k-\mathbf{B}_z\mathbf{d}_z^k$,\quad \hspace{2.5mm} $ \mathbf{d}_z^{k+1}=\mathbf{d}_z^k+\mathbf{B}_z\mathbf{z}^{k+1};\quad \quad \hspace{0.5mm}\triangleright \ \mbox{\footnotesize{auxiliary-level update}}$
  \STATE$\mathbf{r}^{k+1} \hspace{0.3mm} =(1-\theta_k)\mathbf{r}^k+\theta_k\mathbf{u}^k; \hspace{57.8mm}\triangleright \ \mbox{\footnotesize{momentum update}}$ 
  \STATE $\mathbf{x}^{k+1}\hspace{0.2mm}=\mathbf{C}_x\mathbf{x}^k-\alpha_k\mathbf{A}_x \mathbf{r}^{k+1}-\mathbf{B}_x\mathbf{d}_x^k$,\hspace{1.7mm} $ \mathbf{d}_x^{k+1}=\mathbf{d}_x^k+\mathbf{B}_x\mathbf{x}^{k+1} ;\quad \quad \triangleright \ \mbox{\footnotesize{upper-level update}}$
  \ENDFOR
  \end{algorithmic}
\end{algorithm}

\subsection{A unified framework for decentralized bilevel optimization.} 
\vspace{-1mm}
\label{Design}
By utilizing the unified framework \eqref{spdupdate} to approximately solve each subproblem in \eqref{eq:three-problems} with only {\em one iteration}, we achieve SPARKLE, a unified single-loop framework for decentralized bilevel optimization. In particular, we independently sample data $\xi^k_i\sim\mathcal{D}_{f_i}$, $\zeta^k_i\sim\mathcal{D}_{g_i}$ within each node at iteration $k$, and evaluate stochastic gradients/Jacobians/Hessians as follows
\begin{subequations}
\label{alg-variables}
\begin{align}
l_{i}^k &=\nabla_1F_i(x_{i}^k,y_{i}^k;\xi^k_i), \quad \quad 
b_{i}^k=\nabla_2F_i(x_{i}^k,y_{i}^k;\xi^k_i),\quad \quad 
v_{i}^k = \nabla_{2}G_i(x_{i}^k,y_{i}^k;\zeta^k_i), \\
J_{i}^k &=\nabla^2_{12}G_i(x_{i}^k,y_{i}^k;\zeta^k_i),\quad 
H_{i}^k =\nabla^2_{22}G_i(x_{i}^k,y_{i}^k;\zeta^k_i). 
\end{align} 
\end{subequations}
Next we stack the descent directions for variables of each level as follows
\begin{equation}
\begin{aligned}
\mbox{lower-level stochstic gradient:}\quad\quad &\mathbf{v}^k=\text{col}\{v_1^k,...,v_n^k\}, \\
\mbox{auxilliary-level stochstic gradient:}\quad\quad  & \mathbf{p}^k=\text{col}\{H_1^kz_1^k-b_1^k,...,H_n^kz_n^k-b_n^k\}, \\
\mbox{upper-level stochstic gradient:}\quad\quad 
&\mathbf{u}^k=\text{col}\{l_{1}^k -J_{1}^kz_{1}^{k+1},...,l_{n}^k -J_{n}^kz_{n}^{k+1}\}.
\end{aligned} 
\end{equation}
The SPARKLE algorithm is detailed in Algorithm \ref{D-SOBA-SUDA}. In this algorithm, we utilize different dual variables \(\mathbf{d}_s\) and communication matrices \(\mathbf{A}_s, \mathbf{B}_s, \mathbf{C}_s\) for each variable \(s \in \{x, y, z\}\) to optimize their respective objective functions. We use momentum $\mathbf{r}^k$ only for updating the upper-level variable, which is sufficient to enhance convergence of bilevel algorithms and relax the smoothness condition.

\paragraph{Versatility in decentralized strategies.} SPARKLE is highly versatile, supporting various decentralized strategies by allowing the specification of different communication matrices \(\mathbf{A}_s\), \(\mathbf{B}_s\), and \(\mathbf{C}_s\). For example, by setting \(\mathbf{A}_s = \mathbf{I}\), \(\mathbf{B}_s = (\mathbf{I} - \mathbf{W})^{{1}/{2}}\), and \(\mathbf{C}_s = \mathbf{W}\) for any \(s \in \{x, y, z\}\), SPARKLE will utilize EXTRA to update variables \(x\), \(y\), and \(z\), resulting in the SPARKLE-EXTRA variant. Other variants can be achieved by setting \(\mathbf{A}_s\), \(\mathbf{B}_s\), and \(\mathbf{C}_s\) according to Table~\ref{3extable}. These variants can be implemented more efficiently than listed in Algorithm~\ref{D-SOBA-SUDA}, see Appendix~\ref{Implementation details}.

\paragraph{Flexibility across optimization levels.} SPARKLE supports different optimization and communication mechanisms for each level of \eqref{eq:three-problems}, 
which can be directly achieved by choosing different \(\mathbf{A}_s\), \(\mathbf{B}_s\), and \(\mathbf{C}_s\) matrices for each level \(s \in \{x, y, z\}\). For example, SPARKLE can utilize GT to update the upper-level variable \(x\) while employing ED to update the auxiliary- and lower-level variables \(y\) and \(z\). Throughout this paper, we denote \ours using the decentralized mechanism \textbf{L} for the lower-level and auxiliary variables, and \textbf{U} for the upper-level in Algorithm \ref{D-SOBA-SUDA}, by SPARKLE-\textbf{L}-\textbf{U}, or simply SPARKLE-\textbf{L} if \textbf{L} = \textbf{U}. In addition, SPARKLE even supports utilizing different mixing matrices $\mathbf{W}_x, \mathbf{W}_y, \mathbf{W}_z$ across levels.  

\section{Convergence analysis} 
In this section,  we establish the convergence properties of the \ours framework and examine the influence of different decentralized techniques utilized across optimization levels.

\subsection{Assumptions}  
\vspace{-1mm}
Before presenting the theoretical guarantees, we first introduce the assumptions used throughout this paper.
\begin{assumption} 
\label{smooth}
There exist constants $\mu_g, L_{f, 0},L_{f, 1}, L_{g, 1}, L_{g, 2}$ such that for any $1\leq i\leq n$,
\begin{enumerate}
\vspace{-2mm}
\item$ \nabla f_i, \nabla g_i, \nabla^2 g_i$ are $ L_{f, 1}, L_{g, 1}, L_{g, 2}$ Lipschitz continuous, respectively;
\vspace{-1mm}
\item $\|\nabla_2 f_i\left(x, y^\star(x)\right)\| \leq L_{f,0}$  for any  $x \in \mathbb{R}^p$;\footnote{This is more relaxed than Lipschitz continuous $f_i$, or bounded $\nabla_2 f_i$ in~\cite{gao2023convergence,zhang2023communication,lu2022stochastic, chen2023optimal}.}
\vspace{-1mm}
\item $g_i(x,y)$ is $\mu_g$-strongly convex with respect to $y$ for any fixed $x \in \mathbb{R}^p$. 
\vspace{-2mm}
\end{enumerate}
Moreover, we define $L:=\max\{L_{f, 0},L_{f, 1}, L_{g, 1}, L_{g, 2}\}$ and $\kappa:= {L}/{\mu_g}$.
\end{assumption}

\begin{assumption} \label{net}
For each $s\in\{x,y,z\}$, the corresponding mixing matrix $W_s\in \mathbb{R}^{n \times n}$ is non-negative, symmetric and doubly stochastic, \ie, 
$$
W_s=W_s^{\top}, \quad W_s \mathbf{1}_{n}=\mathbf{1}_{n}, \quad (W_{s})_{ij}  \geq 0, \quad \forall \,1\leq i, j\leq n,
$$
and the corresponding communication graph is strongly-connected, \ie, its eigenvalues satisfy $1=\lambda_{1}(W_s)>\lambda_{2}(W_s) \geq \ldots \geq \lambda_{n}(W_s)$ and $\rho(W_s)\DefinedAs \max \left\{\left|\lambda_{2}(W_s)\right|,\left|\lambda_{n}(W_s)\right|\right\}<1$.   
\end{assumption}
The value $1-\rho(W_s)$ is referred to as the spectral gap in the literature \cite{lu2021optimal, yuan2023removing,lian2017can} of $W_s$, which measures the connectivity of the communication graph. It would approach $0$ for sparse networks. For example, it holds that $1-\rho(W_s) = \Theta({1}/{n^2})$ for the matrix $W_s$ induced by a ring graph. 

\begin{assumption}\label{L_s}
For any $s\in\{x,y,z\}$, we assume the communication matrices $A_s,B_s,C_s$ used in \ours are polynomial functions of $W_s$. Furthermore, we assume $A_s,C_s$ are doubly stochastic, and $\mathrm{Null}( B_s)=\mathrm{Span}\{\mathbf{1}_n\}$. In addition, we assume all eigenvalues of the augmented matrix \[L_s:=\left[\begin{array}{cc}
\overline{C_s}-B_s^2 & B_s \\
-B_s & \overline{I_n}
\end{array}\right]\] are strictly less
than one in magnitude, where $\overline{C_s}\triangleq{C_s}-\frac{1}{n}\mathbf{1}_n\mathbf{1}_n^\top$ and $\overline{I_n}\triangleq{I_n}-\frac{1}{n}\mathbf{1}_n\mathbf{1}_n^\top$. 
\end{assumption}
We remark that Assumption \ref{L_s} is mild 
and is satisfied by all choices listed in Table~\ref{3extable}. See more discussions in Appendix \ref{Essential matrix norms}.

\begin{assumption}\label{var}
We assume $\nabla F_i(x, y ; \xi), \nabla G_i(x, y ; \xi)$, and $\nabla^2 G_i(x, y ; \xi)$ to be unbiased estimates of  $\nabla f_i(x, y )$, $\nabla g_i(x, y)$, and $\nabla^2 g_i(x, y)$ with bounded variances $\sigma_{f,1}^2, \sigma_{g, 1}^2, \sigma_{g, 2}^2$, respectively. 
\end{assumption}

\begin{table}[t]
  \caption{\small \ours facilitates different decentralized techniques by specifying $\mathbf{A}_s, \mathbf{B}_s, \mathbf{C}_s$ for $s \hspace{-2pt}\in \hspace{-2pt}\{x,y,z\}$. We denote the stacked local variables and the associate gradients estimates by $\mathbf{s} \hspace{-2pt}\in \hspace{-3pt}\{\mathbf{x},\mathbf{y},\mathbf{z}\}$ and $\mathbf{g}(\mathbf{s})$, respectively. The update rule refers to the specific algorithmic recursion for each level. See derivations in Appendix~\ref{Specific instances}.}
  \renewcommand{\arraystretch}{2.0}
  \label{3extable}
  \centering
   \resizebox*{\textwidth}{!}{
  \begin{threeparttable}
  \begin{tabular}{lcccl}  
  \toprule 
   Algorithms     & $\mathbf{A}_s$     & $\mathbf{B}_s$  & $\mathbf{C}_s$     & The specific update rule at the 
$k$-th iteration.      \\   
    \midrule
       ED & $\mathbf{W}_s$ & $(\mathbf{I}-\mathbf{W}_s)^{\frac{1}{2}}$ & $\mathbf{W}_s$  &$\mathbf{s}^{k+2}=\mathbf{W}_s\left(2 \mathbf{s}^{k+1}-\mathbf{s}^k-\alpha\left(\mathbf{g}(\mathbf{s}^{k+1})-\mathbf{g}(\mathbf{s}^k)\right)\right)$\\
         EXTRA & $\mathbf{I}$ & $(\mathbf{I}-\mathbf{W}_s)^{\frac{1}{2}}$ & $\mathbf{W}_s$  &$\mathbf{s}^{k+2}=\mathbf{W}_s\left(2 \mathbf{s}^{k+1}-\mathbf{s}^k\right)-\alpha\left( \mathbf{g}(\mathbf{s}^{k+1})-\mathbf{g}(\mathbf{s}^k)\right)$\\
       ATC-GT & $\mathbf{W}_s^2$ & $\mathbf{I}-\mathbf{W}_s$ & $\mathbf{W}_s^2$  & $\mathbf{s}^{k+1}=\mathbf{W}_s\left(\mathbf{s}^k-\alpha \mathbf{h}_s^k\right)$, $\mathbf{h}_s^{k+1}=\mathbf{W}_s\left(\mathbf{h}_s^k+ \mathbf{g}(\mathbf{s}^{k+1})-\mathbf{g}(\mathbf{s}^k)\right)$
       \\
     Semi-ATC-GT & $\mathbf{W}_s$ & $\mathbf{I}-\mathbf{W}_s$ & $\mathbf{W}_s^2$  & $\mathbf{s}^{k+1}=\mathbf{W}_s\mathbf{s}^k-\alpha \mathbf{h}_s^k$, $\mathbf{h}_s^{k+1}=\mathbf{W}_s\left(\mathbf{h}_s^k+ \mathbf{g}(\mathbf{s}^{k+1})-\mathbf{g}(\mathbf{s}^k)\right)$
       \\
     Non-ATC-GT & $\mathbf{I}$ & $\mathbf{I}-\mathbf{W}_s$ & $\mathbf{W}_s^2$  & $\mathbf{s}^{k+1}=\mathbf{W}_s\mathbf{s}^k-\alpha \mathbf{h}_s^k$, $\mathbf{h}_s^{k+1}=\mathbf{W}_s\mathbf{h}_s^k+ \mathbf{g}(\mathbf{s}^{k+1})-\mathbf{g}(\mathbf{s}^k)$
       \\
     \bottomrule
  \end{tabular}
  \vspace{-5mm}
  \end{threeparttable}}
\end{table}

\subsection{Convergence theorem}
\vspace{-1mm}
Under the above assumptions, we establish the convergence properties as follows. Proof details can be found in Appendix \ref{Proof}.

\begin{theorem}\label{thm1}
    Under Assumptions~\ref{smooth} -- \ref{var}, there exist proper constant step-sizes $\alpha,\,\beta,\,\gamma$ and momentum coefficient $\theta$, such that the \ours framework listed in Algorithm~\ref{D-SOBA-SUDA} will converge as follow:
\begin{equation}
  \begin{aligned}
&\frac{1}{K+1}\sum_{k=0}^K\mathbb{E}[\|\nabla\Phi(\bar{x}^k)\|^2]
\lesssim\frac{\kappa^5\sigma}{\sqrt{{nK}}}+\kappa^{\frac{16}{3}}\left(\delta_{y,1}+\delta_{z,1}\right)\frac{\sigma^{\frac{2}{3}}}{K^{\frac{2}{3}}}
+\kappa^{\frac{7}{2}}\delta_{x,1}\frac{\sigma^{\frac{1}{2}}}{K^{\frac{3}{4}}} \label{eq:SPARKLE-rate}
\\& \quad \quad +\left(\kappa^{\frac{26}{5}}\delta_{y,2}+\kappa^{6}\delta_{z,2}\right)\frac{\sigma^{\frac{2}{5}}}{K^{\frac{4}{5}}}
+\left(\kappa^{\frac{16}{3}}\delta_{y,3}+\kappa^{\frac{14}{3}}\delta_{z,3}+\kappa^{\frac{8}{3}}\delta_{x,3}\right)\frac{1}{K}
+\left(\kappa C_\alpha+\kappa^4C_\theta\right)\frac{1}{K},
  \end{aligned}
\end{equation}
where $\sigma\triangleq\max\{\sigma_{f,1},\sigma_{g,1},\sigma_{g,2}\}$, $\{\delta_{s,i}\}_{i=1}^3$ are constants depending only on $\mathbf{W}_s,\mathbf{A}_s,\mathbf{B}_s,\mathbf{C}_s$ for $s\in\{x,y,z\}$, and $C_\alpha, C_\theta$ are constants independent of $K$. See Lemma \ref{Stochasticrate} for their detailed values.
\end{theorem}

In the deterministic scenario with $\sigma = 0$, \ours converges at the rate $\mathcal{O}(1/K)$, see the formal theorem and derivation in Appendix~\ref{deterministic-proof}. This 
recovers the rate in \cite{dong2020eflops} under even milder assumptions. Unlike reference~\cite{dong2020eflops}, which only considers GT in the deterministic setting, \ours is a unified bilevel framework for the more general stochastic setting.

\paragraph{Linear speedup.} According to Theorem~\ref{thm1}, \ours achieves an asymptotic linear speedup as $K$ approaches infinity, which applies to all \ours variants regardless of the 
decentralized strategies employed and whether they are utilized at different optimization levels. 
Furthermore, the asymptotically dominant term $\kappa^5\sigma/(\sqrt{nK})$ matches exactly with the single-node bilevel algorithm SOBA \cite{dagreou2022framework} when $n=1$, implying the tightness of Theorem~\ref{thm1} in terms of the asymptotic rate. 
\begin{remark}
    We establish an upper bound for the consensus error $\frac{1}{K}\sum\limits_{k=0}^K\mathbb{E}\left[\frac{\Vert \mathbf{x}^k-\bar{\mathbf{x}}^k\Vert ^2}{n}+\frac{\Vert \mathbf{y}^k-\bar{\mathbf{y}}^k\Vert ^2}{n}\right]$. Please refer to Lemma \ref{Consensus-Error-lemma} in Appendix \ref{Consensus-Error} for more details.
\end{remark}
\subsection{Transient iteration complexity}
\vspace{-1mm}
With the non-asymptotic rate established in Theorem~\ref{thm1}, we can derive the transient iteration complexity of SPARKLE as follows. The proof is in Lemma~\ref{formaltransienttime}. 
\begin{corollary}\label{col1} 
Under the same assumptions as in Theorem~\ref{thm1}, the transient iteration complexity of \ours---with the influence of $\kappa$ and $\sigma^2$ omitted for brevity---is on the order of
     \begin{equation}\label{maintrant}
    \begin{aligned}
\max&\left\{n^2\delta_x,n^3\delta_y,n^3\delta_z,n\hat{\delta}_x,n\hat{\delta}_y,n\hat{\delta}_z\right\},
    \end{aligned}
    \end{equation}
where $\delta_s,\hat{\delta}_s$ only depend $\mathbf{W}_s,\mathbf{A}_s,\mathbf{B}_s,\mathbf{C}_s$ for $s\in\{x,y,z\}$. Their values are in Lemma~\ref{formaltransienttime}. 
\end{corollary}
We obtain the transient iteration complexity of each variant of \ours by applying Corollary~\ref{col1}. 
\begin{corollary}\label{corol-ED} 
For \oursed and \oursextra, if we choose$ \mathbf{W}_y = \mathbf{W}_z$, it holds that 
\begin{equation}
\begin{aligned}\label{ed-del-1}
&\delta_x=\mathcal{O}\left((1-\rho(\mathbf{W}_x))^{-2}\right), \hspace{1cm} \delta_y=\delta_z=\mathcal{O}\left((1-\rho(\mathbf{W}_y))^{-2}\right), \\&\hat{\delta}_x=\mathcal{O}\left((1-\rho(\mathbf{W}_x))^{-\frac{3}{2}}\right),  \hspace{0.85cm} \hat{\delta}_y=\hat{\delta}_z=\mathcal{O}\left((1-\rho(\mathbf{W}_y))^{-2}\right).
\end{aligned} 
\end{equation}
Furthermore, if we choose $\mathbf{W}_x = \mathbf{W}_y = \mathbf{W}_z$ and denote $\rho\triangleq \rho(\mathbf{W}_x)$, the transient iteration complexity derived in \eqref{maintrant} can be simplified as $n^3/(1-\rho)^2$.
\end{corollary}
\begin{corollary}\label{corol-GT} 
For \oursgt and its variants with semi/non-ATC-GT, if we let $\mathbf{W}_y = \mathbf{W}_z$, 
\begin{equation}\label{gt-del-1}
\begin{aligned}
    &\delta_x=\mathcal{O}\left((1-\rho(\mathbf{W}_x))^{-2}\right),  \hspace{1cm} \delta_y=\delta_z=\mathcal{O}\left((1-\rho(\mathbf{W}_y))^{-2}\right), \\&\hat{\delta}_x=\mathcal{O}\left((1-\rho(\mathbf{W}_x))^{-2}\right), \hspace{1cm} 
 \hat{\delta}_y=\hat{\delta}_z=\mathcal{O}\left((1-\rho(\mathbf{W}_y))^{-\frac{8}{3}}\right).
    \end{aligned}
\end{equation}
Furthermore, if we let $\mathbf{W}_x = \mathbf{W}_y = \mathbf{W}_z$ and denote $\rho\triangleq \rho(\mathbf{W}_x)$, the transient iteration complexity derived in \eqref{maintrant} can be simplified as $\max\{n^3/(1-\rho)^2, n/(1-\rho)^{8/3}\}$.
\end{corollary}
\begin{remark}[\em SOTA transient iterations] {\em Comparing with algorithms listed in Table~\ref{table:comparison}, all SPARKLE variants achieve smaller transient iteration complexity, implying that they can achieve linear speedup much faster than the other algorithms, especially over sparse network topologies with $1-\rho \to 0$. }
\end{remark}

\begin{remark}[\em GT is not the best technique for decentralized SBO]
\label{remark-GT-is-not-best}
{\em While GT is widely adopted in the literature \cite{dong2023single,gao2023convergence,zhang2023communication} to facilitate decentralized SBO, a comparison of Corollary~\ref{corol-ED} and \ref{corol-GT} reveals that both \oursextra and \oursed outperform \oursgt in terms of transient iteration complexity. This implies that EXTRA and ED are better than GT for decentralized SBO.}
\end{remark}

\subsection{Different strategies across optimization levels} 
\vspace{-1mm}
Corollary~\ref{col1} clarifies how different update strategies for $x$, $y$, and $z$ impact the transient iterations through constants $\{\delta_s, \hat{\delta}_s\}$ for $s\in\{x,y,z\}$. Since $\delta_y = \delta_z$ and $\hat{\delta}_y = \hat{\delta}_z$ when $\mathbf{W}_y = \mathbf{W}_z$ (Lemma~\ref{formaltransienttime}), we naturally employ the same strategy to update $y$ and $z$. The following corollary studies the utilization of both ED and GT in \ours. See the transient iterations of other mixed strategies in Appendix~\ref{deviation-tran} and Table~\ref{table:mixed}.
\begin{corollary}\label{corol-GT-ED} 
For {\rm SPARKLE-ED-GT} which uses ED to update $y$ and $z$ and GT to update $x$, if $\mathbf{W}_x = \mathbf{W}_y = \mathbf{W}_z$ and we denote $\rho =\rho(\mathbf{W}_x)$, it then holds that 
\begin{equation}\label{eq:corr-gt-ed}
\begin{aligned}
    &\delta_x=\delta_y=\delta_z=\mathcal{O}\left((1-\rho)^{-2}\right), \quad \quad 
 \hat{\delta}_x = \hat{\delta}_y=\hat{\delta}_z=\mathcal{O}\left((1-\rho)^{-2}\right),
    \end{aligned}
\end{equation}
which implies that the transient iteration complexity in \eqref{maintrant} can be simplified as $n^3/(1-\rho)^2$.
\end{corollary}
\begin{remark}[\em Mixed strategies outperform employing GT only]\label{mixed-improve-gt}
{\em Comparing Corollary \ref{corol-GT} and \ref{corol-GT-ED}, we find that using ED to update $y$ and $z$ will lead to smaller $\hat{\delta}_y$ and $\hat{\delta}_z$, which improves the transient iteration complexity compared to employing GT only in all optimization levels (see Corollary \ref{corol-GT})}.  
\end{remark}

\subsection{Different topologies across optimization levels}\label{topo}
\vspace{-1mm}
In SPARKLE, we can utilize different topologies across levels. Theorem \ref{thm1} and Corollary \ref{col1} have clarified the influence of using different topologies across levels through the constants $\{\delta_s, \hat{\delta}_s\}$ for $s\in \{x,y,z\}$. For instance, when substituting $\{\delta_s, \hat{\delta}_s\}$ established in \eqref{ed-del-1} into \eqref{maintrant}, SPARKLE-ED has the following transient iteration complexity:
\begin{align}
    \max\{ n^2(1-\rho(\mathbf{W}_x))^{-2}, n^3(1-\rho(\mathbf{W}_y))^{-2}\}
\end{align}
where $\mathbf{W}_x$ is the mixing matrix for updating $x$, while $\mathbf{W}_y$ is for updating $y$ and $z$. As long as $(1-\rho(\mathbf{W}_x))^{-1}\lesssim \sqrt{n}(1-\rho(\mathbf{W}_y))^{-1}$ holds, SPARKLE-ED retains the  transient iteration complexity of \( n^3(1 - \rho(\mathbf{W}_y))^{-2} \), which allows for the utilization of a sparser network topology when updating \( x \), thereby reducing communication overheads. Consequently, the ratio \( a \) of the communication volume per round for the variables \( x \) and \( y \) can be significantly less than one. 
See Appendix~\ref{Theoretical Gap} for discussion on how to use different topologies across levels in other SPARKLE variants. 

\subsection{Recovering single-level decentralized optimization} 
\vspace{-1mm}
Previous works typically study single-level and bilevel optimization separately. By taking $G_i(x,y,\xi)\equiv {|y|^2}/{2}$ and $F_i(x,y,\phi)= F_i(x,\phi)$ into \eqref{eq:fi-gi}, the decentralized SBO problem \eqref{DSBO_problem} reduces to stochastic single-level optimization. By setting $\mathbf{z}^k\equiv0$, $\mathbf{y}^k\equiv0$, $u_i^k=\nabla_1 f_i(x_i^k,\xi_i^k)$, SPARKLE reduces to the single-level framework \eqref{spdupdate}, whose convergence can be naturally guaranteed by Theorem~\ref{thm1}. Please refer to Appendix \ref{Degenerating} for the detailed proof and results. This is the {\em first} result demonstrating that bilevel optimization essentially subsumes the convergence of single-level optimization.

\section{Numerical experiments}\label{section:experiment}
In this section, we present experiments to validate our theoretical findings. We first explore how update strategies and network structures influence the convergence of \ours. Then we compare \ours to the existing decentralized SBO algorithms. Additional experiments about a decentralized SBO problem with synthetic data are in Appendix \ref{app:toyexperiment}.

\paragraph{Hyper-cleaning on FashionMNIST dataset.} 
We consider a data hyper-cleaning problem~\cite{shaban2019truncated} on a corrupted FashionMNIST dataset~\cite{xiao2017fashion}. Problem formulations and experimental setups can be found in Appendix \ref{app:cleaning}. Firstly, we equip \ours with different decentralized strategies in different optimization levels and then compare them with D-SOBA \cite{kong2024decentralized}, MA-DSBO-GT \cite{chen2023decentralized}, and MDBO \cite{gao2023convergence} using the corruption rate $p=0.1,0.2,0.3$, respectively. As is shown in Figure~\ref{fig: FashionMINST_p_1}, all the \ours-based algorithms generally achieve higher test accuracy than D-SOBA, while ED and EXTRA especially outperform GT. Meanwhile, using mixed strategies (\ie, \ours-ED-GT and \ours-EXTRA-GT) achieves similar test accuracy with \ours-ED and \ours-EXTRA and outperform \ours-GT, respectively. These observations match with the theoretical results in Corollary \ref{corol-ED}-\ref{corol-GT-ED} and Remark \ref{remark-GT-is-not-best}, \ref{mixed-improve-gt}.

Next, we test \ours-EXTRA with two communication strategies including \emph{fixed topology for updating $x$ and varying topology for $y,z$}, and \emph{fixed topology for updating $y,z$ and varying topology for $x$}. As illustrated in Figure~\ref{fig: FashionMINST_topo1}, maintaining a fixed topology for $x$ while reducing the connectivity of the topology for $y$ and $z$ will deteriorate the algorithmic performance. Conversely, preserving the topology for $y$ and $z$ while decreasing the connectivity for $x$ has little impact on the performance. This suggests that the influence of the network topology for $y$ and $z$ on the algorithm dominates over the topology for $x$, which is consistent with our discussion in Section \ref{topo}. We also numerically examine the influence of moving average on convergence, see discussions in Appendix \ref{app:cleaning}.

\begin{figure}[t!]
\hspace{-20pt}
\centering
	\subfigure{
        \includegraphics[width=0.34\textwidth]{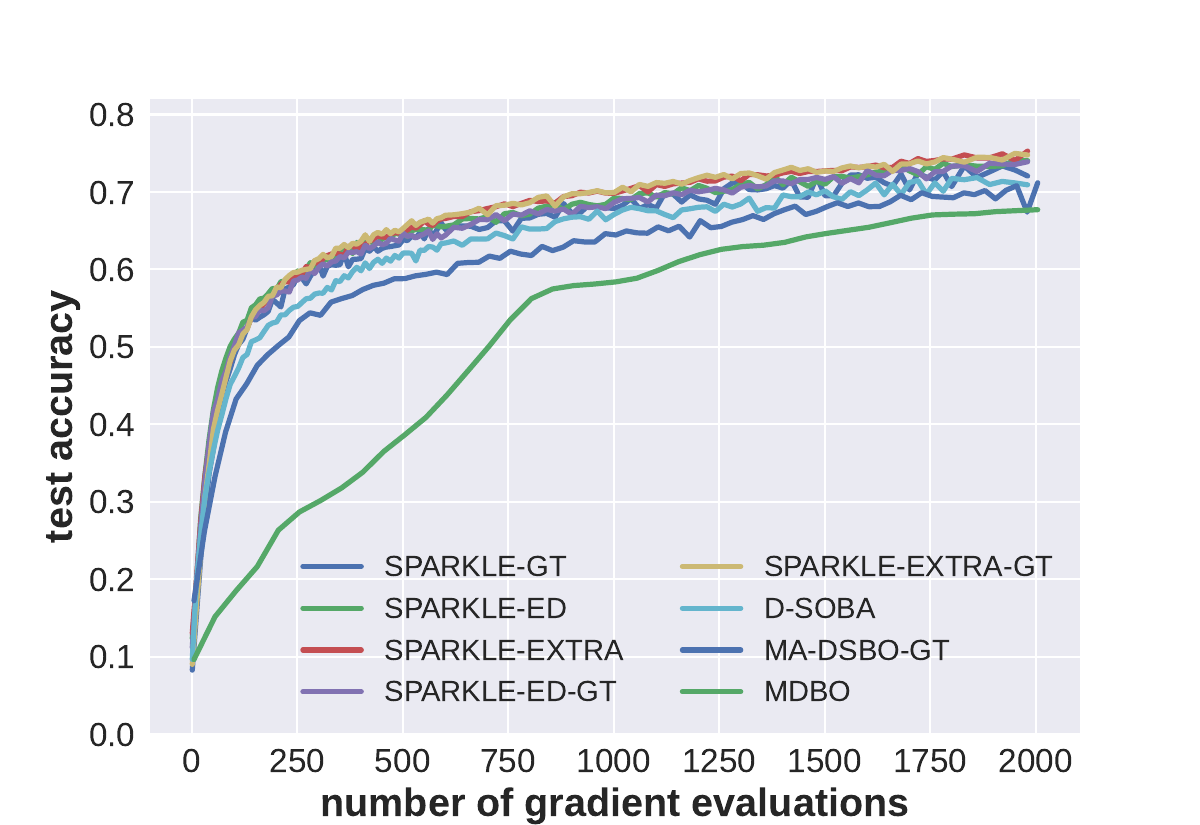}}
  \hspace{-20pt}
	\subfigure{
		\includegraphics[width=0.34\textwidth]{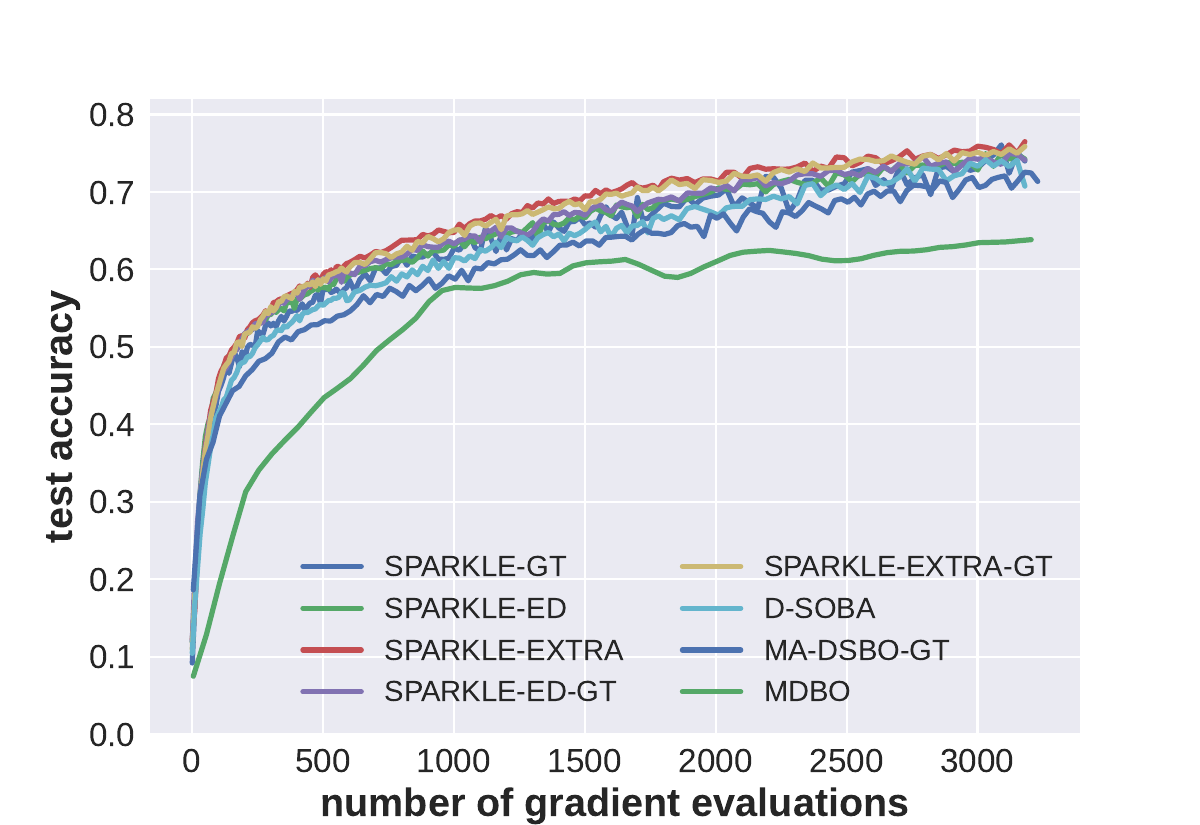}}
  \hspace{-20pt}
	\subfigure{
		\includegraphics[width=0.34\textwidth]{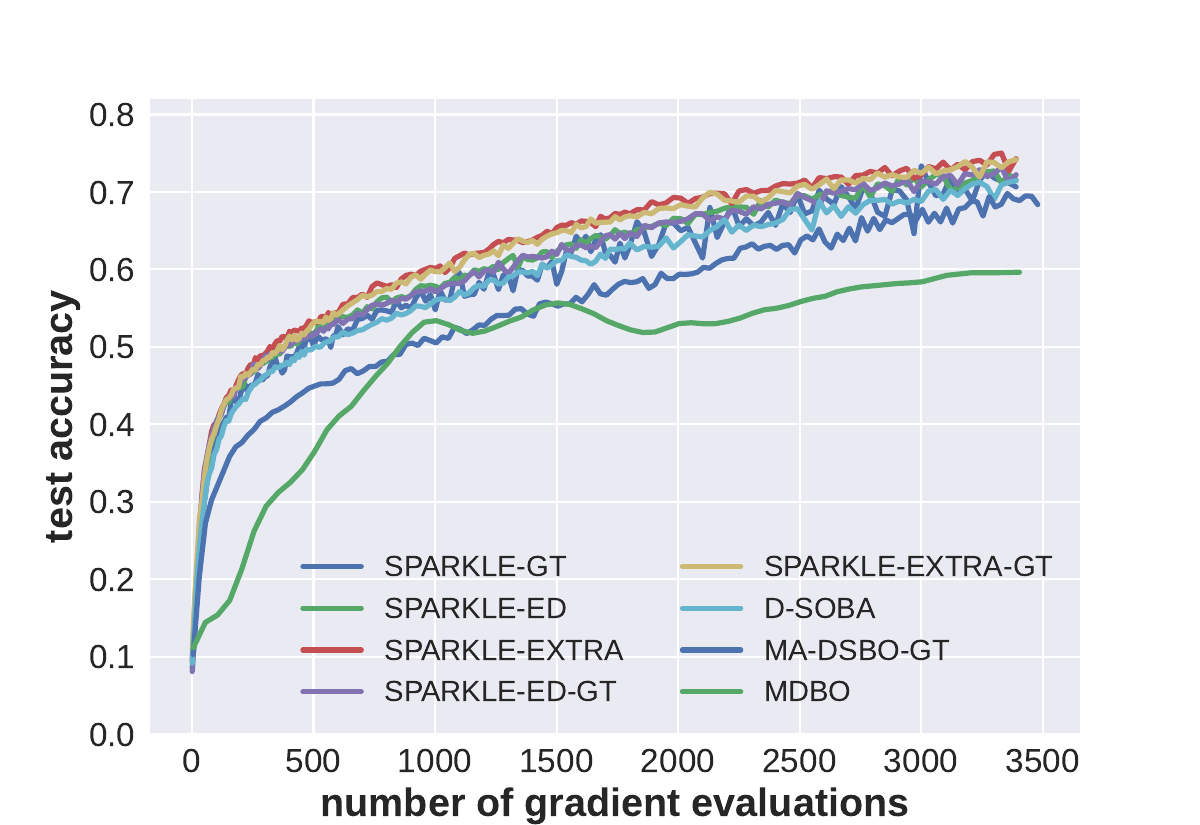}}
\hspace{-20pt}
\vspace{-5pt}
\caption{\small The test accuracy on hyper-cleaning with various \ours-based algorithms using different corruption rates $p$. (Left: $p=0.1$, Middle: $p=0.2$, Right: $p=0.3$.)}
\label{fig: FashionMINST_p_1}
\end{figure}

\begin{figure}
\hspace{-15pt}
	\centering
	\begin{minipage}[c]{0.5\textwidth}
    \hspace{-17pt}
    \centering
    	\subfigure{
    		\includegraphics[width=0.5\textwidth]{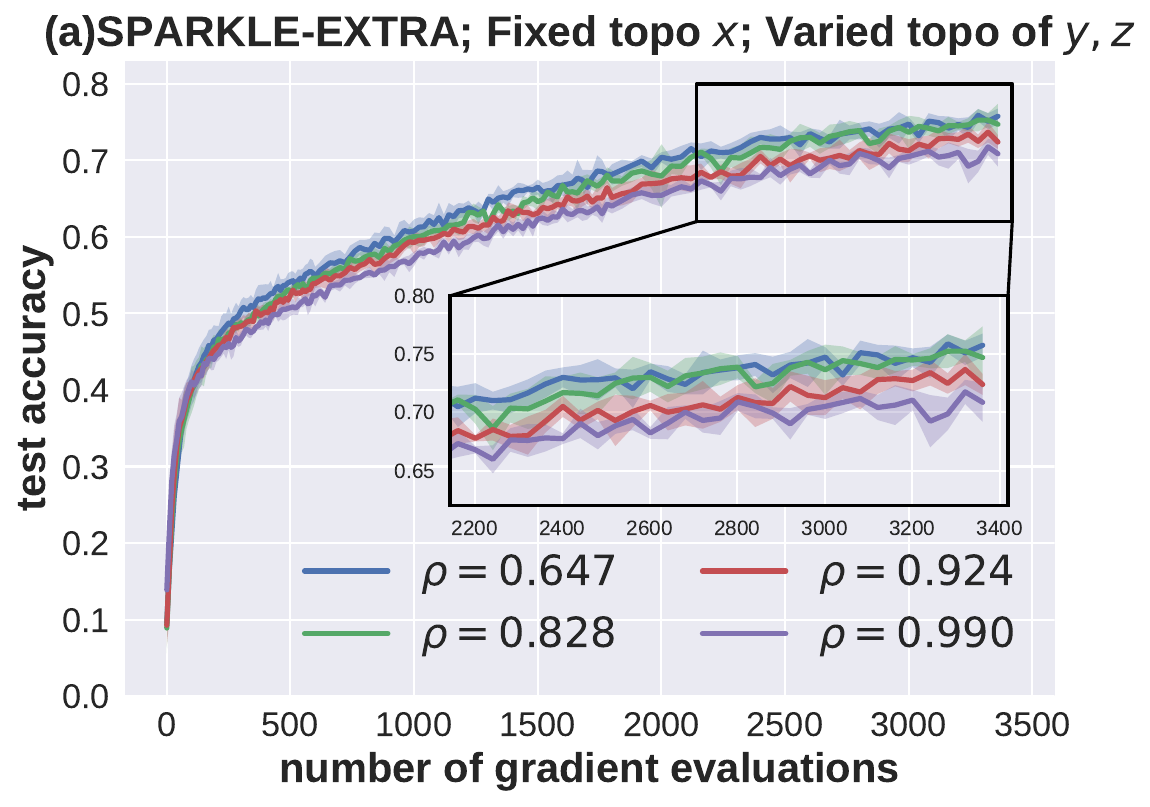}}
      \hspace{-8.5pt}
    	\subfigure{
    		\includegraphics[width=0.5\textwidth]{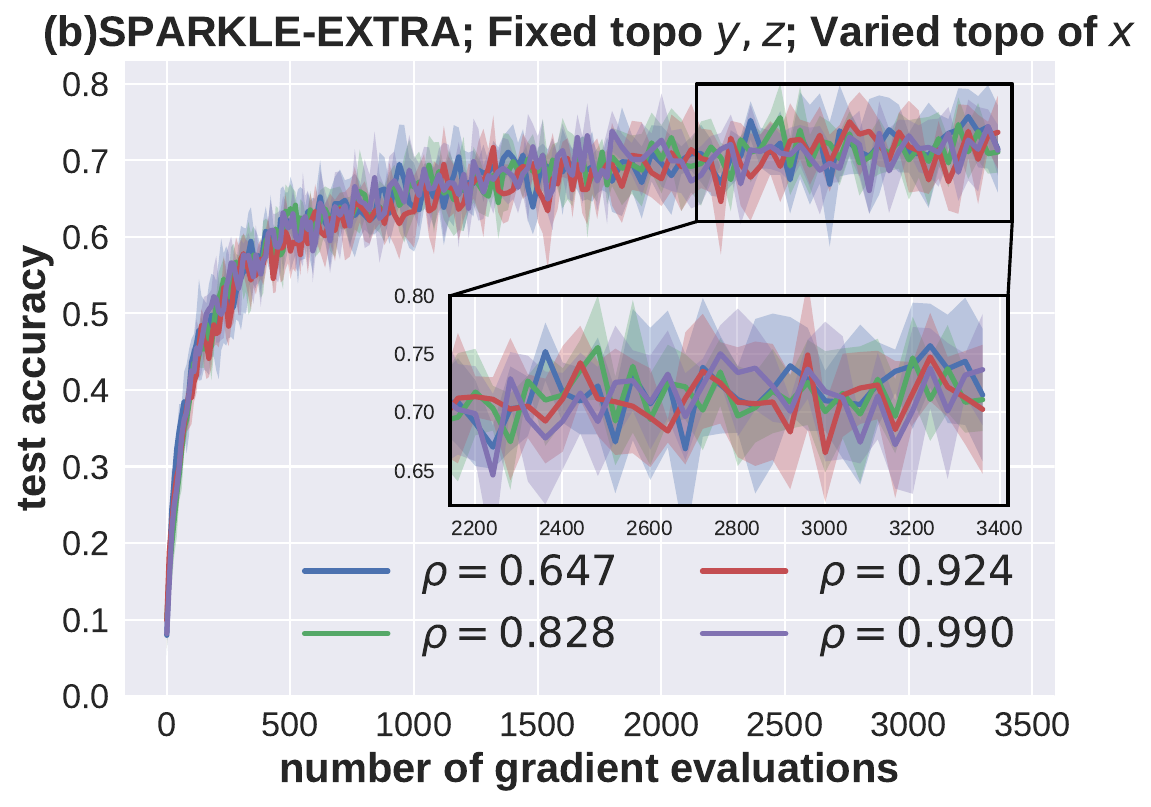}}
      \vspace{-10pt}
    \hspace{-16pt}
    \caption{\small Test accuracy of SPARKLE-EXTRA on hyper-cleaning. (Left: fixed graph for $x$ and varying graph for $y,z$; Right: fixed for $y,z$ and varying for $x$)}
    \label{fig: FashionMINST_topo1}
	\end{minipage} 
    \hspace{4pt}
	\begin{minipage}[c]{0.5\textwidth}
    \centering
      \hspace{-20pt}
    	\subfigure{
    		\includegraphics[width=0.5\textwidth]{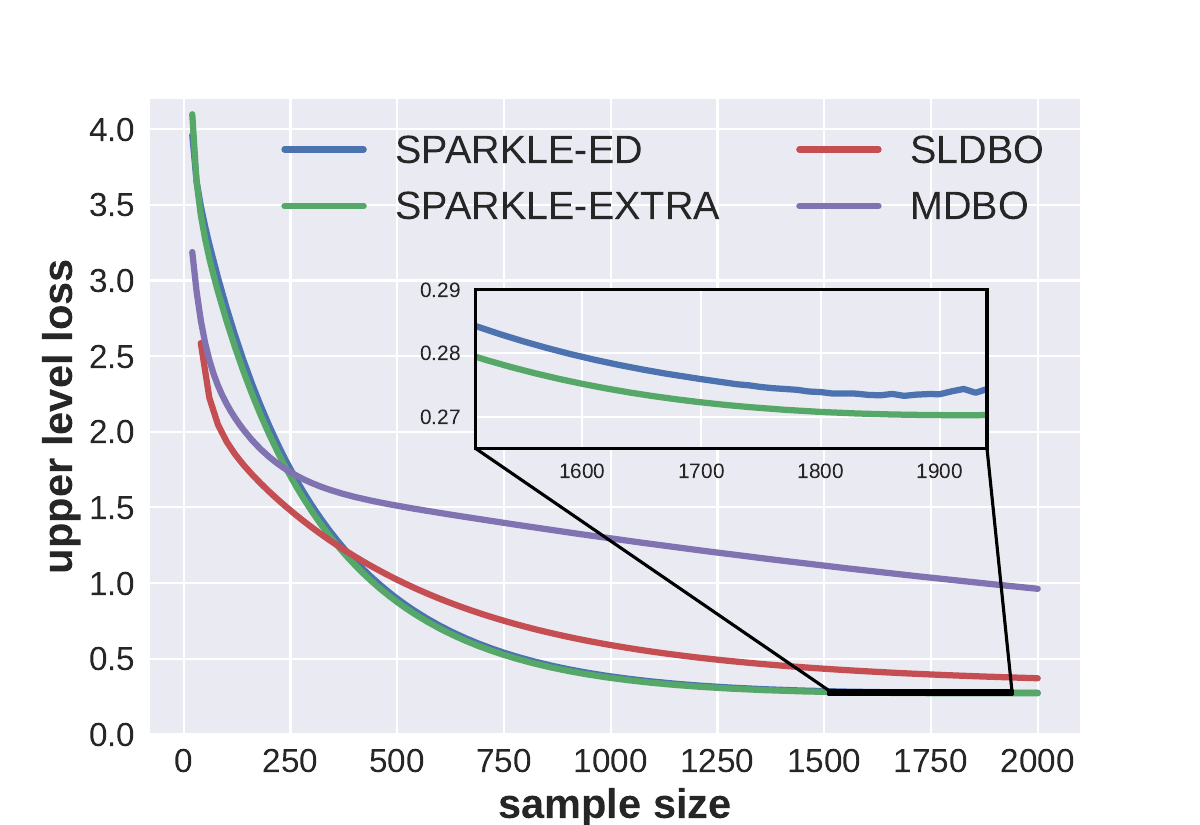}}
      \hspace{-14pt}
    	\subfigure{
    		\includegraphics[width=0.5\textwidth]{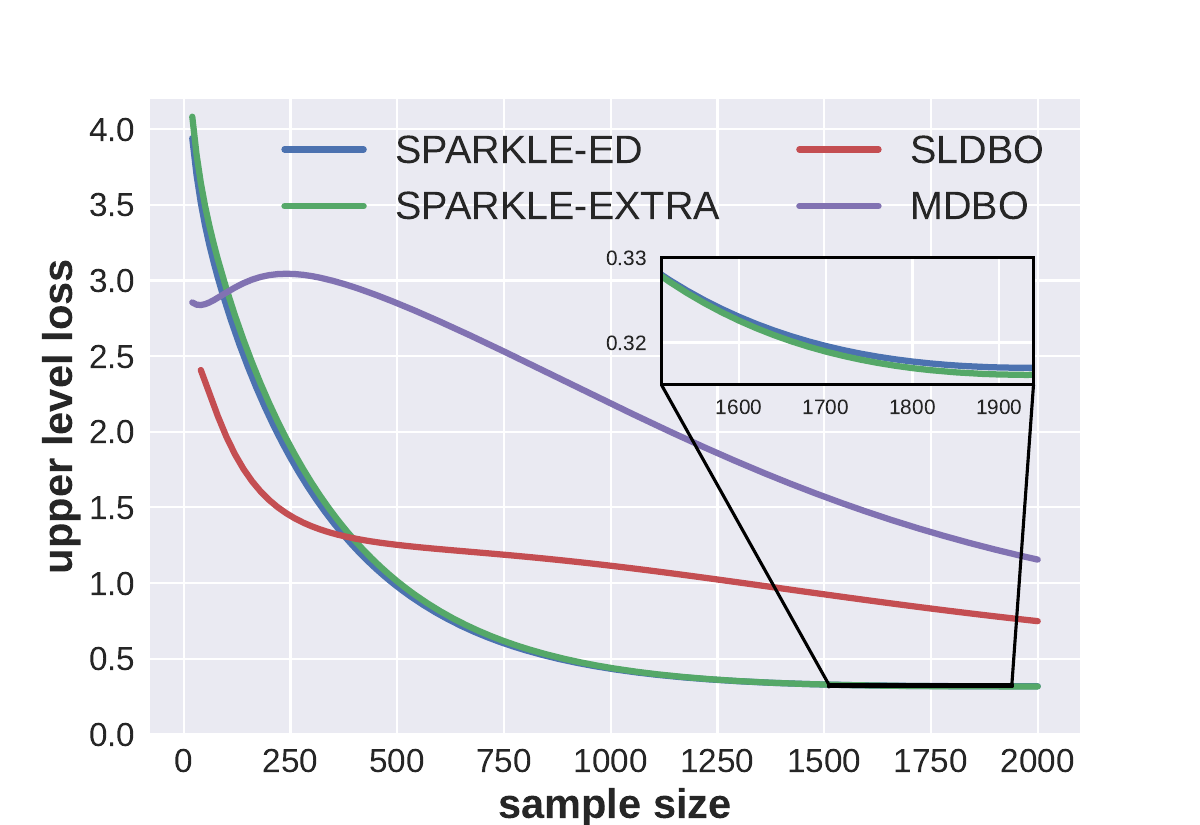}
        \label{fig: FashionMINST_algs_test_noniid}}
      \vspace{-10pt}
    \hspace{-20pt}
    \caption{\small The upper-level loss against samples generated by one agent of different algorithms in the policy evaluation. (Left: $n=10$, Right: $n=20$.)}
    \label{fig: rein_algs}
	\end{minipage} 
\hspace{-15pt}
\end{figure}

\paragraph{Distributed policy evaluation in reinforcement learning.} We consider a multi-agent MDP problem in reinforcement learning on a distributed setting with $n\in\{10,20\}$ agents respectively, which can be formulated as a  decentralized SBO problem~\cite {yang2022decentralized}. We compare \ours with existing decentralized SBO approaches including MDBO~\cite{gao2023convergence} and the stochastic extension of SLDBO~\cite{dong2023single} over a Ring graph. 
Figure~\ref{fig: rein_algs} illustrates that \ours converges faster and achieves a lower sample complexity than the other baselines, especially when $n=20$, which shows the empirical benefits of \ours in decentralized SBO algorithms with a large number of agents and sparse communication modes. More experimental details are in Appendix \ref{app:rein}. 

\paragraph{Decentralized meta-learning.} We investigate decentralized meta-learning on miniImageNet \citep{vinyals2016matching} with multiple tasks \citep{finn2017model}, formulating it as a decentralized bilevel optimization problem. This approach minimizes the validation loss with respect to shared parameters as the upper-level loss, while the training loss is managed by task-specific parameters at the lower level. Additional details about the experiment can be found in Appendix \ref{meta_learning}. Our method, \ours, is benchmarked against D-SOBA \citep{kong2024decentralized} and  MAML \citep{finn2017model}, demonstrating a significant improvement in training accuracy.

\section{Conclusions and limitations}

This paper proposes SPARKLE, a unified single-loop primal-dual framework for decentralized stochastic bilevel optimization. Being highly versatile, SPARKLE can support different decentralized mechanisms and topologies across optimization levels. Moreover, all SPARKLE variants have been demonstrated to achieve state-of-the-art convergence rate compared to existing algorithms. However, SPARKLE currently supports only strongly-convex problems in the lower-level optimization. Its compatibility with generally-convex lower-level problems remains unknown. Additionally, the condition number of the lower-level problem significantly impacts the performance, as is the case with existing bilevel algorithms. We aim to address these limitations in future work.

\newpage 
\section{Acknowledgment}

The work of Shuchen Zhu, Boao Kong, and Kun Yuan is supported by Natural Science Foundation of China under Grants 92370121, 12301392, and W2441021. This work is also supported by Open Project of Key Laboratory of Mathematics and Information Networks, Ministry of Education, China. No. KF202302.

{
\small

\bibliography{reference}

\begin{thebibliography}{10}

\bibitem{alghunaim2020decentralized}
S.~A. Alghunaim, E.~K. Ryu, K.~Yuan, and A.~H. Sayed.
\newblock Decentralized proximal gradient algorithms with linear convergence rates.
\newblock {\em IEEE Transactions on Automatic Control}, 66(6):2787--2794, 2020.

\bibitem{alghunaim2021unified}
S.~A. Alghunaim and K.~Yuan.
\newblock A unified and refined convergence analysis for non-convex decentralized learning.
\newblock {\em IEEE Transactions on Signal Processing}, 2022.

\bibitem{arora2020provable}
S.~Arora, S.~Du, S.~Kakade, Y.~Luo, and N.~Saunshi.
\newblock Provable representation learning for imitation learning via bi-level optimization.
\newblock In {\em International Conference on Machine Learning}, pages 367--376. PMLR, 2020.

\bibitem{bertinetto2019meta}
L.~Bertinetto, J.~Henriques, P.~Torr, and A.~Vedaldi.
\newblock Meta-learning with differentiable closed-form solvers.
\newblock In {\em International Conference on Learning Representations (ICLR), 2019}. International Conference on Learning Representations, 2019.

\bibitem{chang2014multi}
T.-H. Chang, M.~Hong, and X.~Wang.
\newblock Multi-agent distributed optimization via inexact consensus admm.
\newblock {\em IEEE Transactions on Signal Processing}, 63(2):482--497, 2014.

\bibitem{chen2012diffusion}
J.~Chen and A.~H. Sayed.
\newblock Diffusion adaptation strategies for distributed optimization and learning over networks.
\newblock {\em IEEE Transactions on Signal Processing}, 60(8):4289--4305, 2012.

\bibitem{chen2022single}
T.~Chen, Y.~Sun, Q.~Xiao, and W.~Yin.
\newblock A single-timescale method for stochastic bilevel optimization.
\newblock In {\em International Conference on Artificial Intelligence and Statistics}, pages 2466--2488. PMLR, 2022.

\bibitem{chen2021closing}
T.~Chen, Y.~Sun, and W.~Yin.
\newblock Closing the gap: Tighter analysis of alternating stochastic gradient methods for bilevel problems.
\newblock {\em Advances in Neural Information Processing Systems}, 34:25294--25307, 2021.

\bibitem{chen2022decentralized}
X.~Chen, M.~Huang, and S.~Ma.
\newblock Decentralized bilevel optimization.
\newblock {\em arXiv preprint arXiv:2206.05670}, 2022.

\bibitem{chen2023decentralized}
X.~Chen, M.~Huang, S.~Ma, and K.~Balasubramanian.
\newblock Decentralized stochastic bilevel optimization with improved per-iteration complexity.
\newblock In {\em International Conference on Machine Learning}, pages 4641--4671. PMLR, 2023.

\bibitem{chen2023optimal}
X.~Chen, T.~Xiao, and K.~Balasubramanian.
\newblock Optimal algorithms for stochastic bilevel optimization under relaxed smoothness conditions.
\newblock {\em arXiv preprint arXiv:2306.12067}, 2023.

\bibitem{dagreou2022framework}
M.~Dagr{\'e}ou, P.~Ablin, S.~Vaiter, and T.~Moreau.
\newblock A framework for bilevel optimization that enables stochastic and global variance reduction algorithms.
\newblock {\em Advances in Neural Information Processing Systems}, 35:26698--26710, 2022.

\bibitem{di2016next}
P.~Di~Lorenzo and G.~Scutari.
\newblock Next: In-network nonconvex optimization.
\newblock {\em IEEE Transactions on Signal and Information Processing over Networks}, 2(2):120--136, 2016.

\bibitem{domke2012generic}
J.~Domke.
\newblock Generic methods for optimization-based modeling.
\newblock In {\em Artificial Intelligence and Statistics}, pages 318--326. PMLR, 2012.

\bibitem{dong2020eflops}
J.~Dong, Z.~Cao, T.~Zhang, J.~Ye, S.~Wang, F.~Feng, L.~Zhao, et~al.
\newblock Eflops: Algorithm and system co-design for a high performance distributed training platform.
\newblock In {\em 2020 IEEE International Symposium on High Performance Computer Architecture (HPCA)}, pages 610--622, 2020.

\bibitem{dong2023single}
Y.~Dong, S.~Ma, J.~Yang, and C.~Yin.
\newblock A single-loop algorithm for decentralized bilevel optimization.
\newblock {\em arXiv preprint arXiv:2311.08945}, 2023.

\bibitem{duchi2011dual}
J.~C. Duchi, A.~Agarwal, and M.~J. Wainwright.
\newblock Dual averaging for distributed optimization: Convergence analysis and network scaling.
\newblock {\em IEEE Transactions on Automatic control}, 57(3):592--606, 2011.

\bibitem{finn2017model}
C.~Finn, P.~Abbeel, and S.~Levine.
\newblock Model-agnostic meta-learning for fast adaptation of deep networks.
\newblock {\em International Conference on Machine Learning}, pages 1126--1135, 2017.

\bibitem{franceschi2018bilevel}
L.~Franceschi, P.~Frasconi, S.~Salzo, R.~Grazzi, and M.~Pontil.
\newblock Bilevel programming for hyperparameter optimization and meta-learning.
\newblock In {\em International Conference on Machine Learning}, pages 1568--1577. PMLR, 2018.

\bibitem{gao2024lancbio}
B.~Gao, Y.~Yang, and Y.~xiang Yuan.
\newblock Lancbio: dynamic lanczos-aided bilevel optimization via krylov subspace.
\newblock {\em arXiv preprint arXiv:2404.03331}, 2024.

\bibitem{gao2023convergence}
H.~Gao, B.~Gu, and M.~T. Thai.
\newblock On the convergence of distributed stochastic bilevel optimization algorithms over a network.
\newblock In {\em International Conference on Artificial Intelligence and Statistics}, pages 9238--9281. PMLR, 2023.

\bibitem{ghadimi2018approximation}
S.~Ghadimi and M.~Wang.
\newblock Approximation methods for bilevel programming.
\newblock {\em arXiv preprint arXiv:1802.02246}, 2018.

\bibitem{grazzi2020iteration}
R.~Grazzi, L.~Franceschi, M.~Pontil, and S.~Salzo.
\newblock On the iteration complexity of hypergradient computation.
\newblock In {\em International Conference on Machine Learning}, pages 3748--3758. PMLR, 2020.

\bibitem{guo2021randomized}
Z.~Guo, Q.~Hu, L.~Zhang, and T.~Yang.
\newblock Randomized stochastic variance-reduced methods for multi-task stochastic bilevel optimization.
\newblock {\em arXiv preprint arXiv:2105.02266}, 2021.

\bibitem{hong2023two}
M.~Hong, H.-T. Wai, Z.~Wang, and Z.~Yang.
\newblock A two-timescale stochastic algorithm framework for bilevel optimization: Complexity analysis and application to actor-critic.
\newblock {\em SIAM Journal on Optimization}, 33(1):147--180, 2023.

\bibitem{jakovetic2018unification}
D.~Jakoveti{\'c}.
\newblock A unification and generalization of exact distributed first-order methods.
\newblock {\em IEEE Transactions on Signal and Information Processing over Networks}, 5(1):31--46, 2018.

\bibitem{ji2021bilevel}
K.~Ji, J.~Yang, and Y.~Liang.
\newblock Bilevel optimization: Convergence analysis and enhanced design.
\newblock In {\em International Conference on Machine Learning}, pages 4882--4892. PMLR, 2021.

\bibitem{koloskova2021improved}
A.~Koloskova, T.~Lin, and S.~U. Stich.
\newblock An improved analysis of gradient tracking for decentralized machine learning.
\newblock {\em Advances in Neural Information Processing Systems}, 34:11422--11435, 2021.

\bibitem{kong2024decentralized}
B.~Kong, S.~Zhu, S.~Lu, X.~Huang, and K.~Yuan.
\newblock Decentralized bilevel optimization over graphs: Loopless algorithmic update and transient iteration complexity.
\newblock {\em arXiv preprint arXiv:2402.03167}, 2024.

\bibitem{li2017decentralized}
Z.~Li, W.~Shi, and M.~Yan.
\newblock A decentralized proximal-gradient method with network independent step-sizes and separated convergence rates.
\newblock {\em IEEE Transactions on Signal Processing}, July 2019.
\newblock early acces. Also available on arXiv:1704.07807.

\bibitem{lian2017can}
X.~Lian, C.~Zhang, H.~Zhang, C.-J. Hsieh, W.~Zhang, and J.~Liu.
\newblock Can decentralized algorithms outperform centralized algorithms? {A} case study for decentralized parallel stochastic gradient descent.
\newblock In {\em Advances in Neural Information Processing Systems}, pages 5330--5340, 2017.

\bibitem{lin2021quasi}
T.~Lin, S.~P. Karimireddy, S.~U. Stich, and M.~Jaggi.
\newblock Quasi-global momentum: Accelerating decentralized deep learning on heterogeneous data.
\newblock In {\em International Conference on Machine Learning}, 2021.

\bibitem{lu2022stochastic}
S.~Lu, S.~Zeng, X.~Cui, M.~Squillante, L.~Horesh, B.~Kingsbury, J.~Liu, and M.~Hong.
\newblock A stochastic linearized augmented lagrangian method for decentralized bilevel optimization.
\newblock {\em Advances in Neural Information Processing Systems}, 35:30638--30650, 2022.

\bibitem{lu2021optimal}
Y.~Lu and C.~De~Sa.
\newblock Optimal complexity in decentralized training.
\newblock In {\em International Conference on Machine Learning}, pages 7111--7123. PMLR, 2021.

\bibitem{maclaurin2015gradient}
D.~Maclaurin, D.~Duvenaud, and R.~Adams.
\newblock Gradient-based hyperparameter optimization through reversible learning.
\newblock In {\em International Conference on Machine Learning}, pages 2113--2122. PMLR, 2015.

\bibitem{madry2018towards}
A.~Madry, A.~Makelov, L.~Schmidt, D.~Tsipras, and A.~Vladu.
\newblock Towards deep learning models resistant to adversarial attacks.
\newblock In {\em International Conference on Learning Representations}, 2018.

\bibitem{nedic2018network}
A.~Nedi{\'c}, A.~Olshevsky, and M.~G. Rabbat.
\newblock Network topology and communication-computation tradeoffs in decentralized optimization.
\newblock {\em Proceedings of the IEEE}, 106(5):953--976, 2018.

\bibitem{nedic2017achieving}
A.~Nedic, A.~Olshevsky, and W.~Shi.
\newblock Achieving geometric convergence for distributed optimization over time-varying graphs.
\newblock {\em SIAM Journal on Optimization}, 27(4):2597--2633, 2017.

\bibitem{nedic2009distributed}
A.~Nedic and A.~Ozdaglar.
\newblock Distributed subgradient methods for multi-agent optimization.
\newblock {\em IEEE Transactions on Automatic Control}, 54(1):48--61, 2009.

\bibitem{niu2023distributed}
Y.~Niu, J.~Xu, Y.~Sun, Y.~Huang, and L.~Chai.
\newblock Distributed stochastic bilevel optimization: Improved complexity and heterogeneity analysis.
\newblock {\em arXiv preprint arXiv:2312.14690}, 2023.

\bibitem{qu2018harnessing}
G.~Qu and N.~Li.
\newblock Harnessing smoothness to accelerate distributed optimization.
\newblock {\em IEEE Transactions on Control of Network Systems}, 5(3):1245--1260, 2018.

\bibitem{russakovsky2015imagenet}
O.~Russakovsky, J.~Deng, H.~Su, J.~Krause, S.~Satheesh, S.~Ma, Z.~Huang, A.~Karpathy, A.~Khosla, M.~Bernstein, et~al.
\newblock Imagenet large scale visual recognition challenge.
\newblock {\em International Journal of Computer Vision}, 115:211--252, 2015.

\bibitem{sayed2014adaptive}
A.~H. Sayed.
\newblock Adaptive networks.
\newblock {\em Proceedings of the IEEE}, 102(4):460--497, 2014.

\bibitem{shaban2019truncated}
A.~Shaban, C.-A. Cheng, N.~Hatch, and B.~Boots.
\newblock Truncated back-propagation for bilevel optimization.
\newblock In {\em The 22nd International Conference on Artificial Intelligence and Statistics}, pages 1723--1732. PMLR, 2019.

\bibitem{shi2015extra}
W.~Shi, Q.~Ling, G.~Wu, and W.~Yin.
\newblock {EXTRA}: An exact first-order algorithm for decentralized consensus optimization.
\newblock {\em SIAM Journal on Optimization}, 25(2):944--966, 2015.

\bibitem{tang2018d}
H.~Tang, X.~Lian, M.~Yan, C.~Zhang, and J.~Liu.
\newblock D$^2$: Decentralized training over decentralized data.
\newblock In {\em International Conference on Machine Learning}, pages 4848--4856, 2018.

\bibitem{vinyals2016matching}
O.~Vinyals, C.~Blundell, T.~Lillicrap, D.~Wierstra, et~al.
\newblock Matching networks for one shot learning.
\newblock {\em Advances in Neural Information Processing Systems}, 29, 2016.

\bibitem{xiao2017fashion}
H.~Xiao, K.~Rasul, and R.~Vollgraf.
\newblock Fashion-mnist: a novel image dataset for benchmarking machine learning algorithms.
\newblock {\em arXiv preprint arXiv:1708.07747}, 2017.

\bibitem{xu2021distributed}
J.~Xu, Y.~Tian, Y.~Sun, and G.~Scutari.
\newblock Distributed algorithms for composite optimization: Unified framework and convergence analysis.
\newblock {\em IEEE Transactions on Signal Processing}, 69:3555--3570, 2021.

\bibitem{xu2015augmented}
J.~Xu, S.~Zhu, Y.~C. Soh, and L.~Xie.
\newblock Augmented distributed gradient methods for multi-agent optimization under uncoordinated constant stepsizes.
\newblock In {\em IEEE Conference on Decision and Control (CDC)}, pages 2055--2060, Osaka, Japan, 2015.

\bibitem{yang2021provably}
J.~Yang, K.~Ji, and Y.~Liang.
\newblock Provably faster algorithms for bilevel optimization.
\newblock {\em Advances in Neural Information Processing Systems}, 34:13670--13682, 2021.

\bibitem{yang2022decentralized}
S.~Yang, X.~Zhang, and M.~Wang.
\newblock Decentralized gossip-based stochastic bilevel optimization over communication networks.
\newblock {\em Advances in Neural Information Processing Systems}, 35:238--252, 2022.

\bibitem{yuan2023removing}
K.~Yuan, S.~A. Alghunaim, and X.~Huang.
\newblock Removing data heterogeneity influence enhances network topology dependence of decentralized {SGD}.
\newblock {\em Journal of Machine Learning Research}, 24(280):1--53, 2023.

\bibitem{yuan2020influence}
K.~Yuan, S.~A. Alghunaim, B.~Ying, and A.~H. Sayed.
\newblock On the influence of bias-correction on distributed stochastic optimization.
\newblock {\em IEEE Transactions on Signal Processing}, 2020.

\bibitem{yuan2016convergence}
K.~Yuan, Q.~Ling, and W.~Yin.
\newblock On the convergence of decentralized gradient descent.
\newblock {\em SIAM Journal on Optimization}, 26(3):1835--1854, 2016.

\bibitem{yuan2018exact1}
K.~Yuan, B.~Ying, X.~Zhao, and A.~H. Sayed.
\newblock Exact dffusion for distributed optimization and learning -- {Part I: Algorithm development}.
\newblock {\em IEEE Transactions on Signal Processing}, 67(3):708 -- 723, 2018.

\bibitem{zhang2023communication}
Y.~Zhang, M.~T. Thai, J.~Wu, and H.~Gao.
\newblock On the communication complexity of decentralized bilevel optimization.
\newblock {\em arXiv preprint arXiv:2311.11342}, 2023.

\end{thebibliography}
\bibliographystyle{abbrv}
}

\newpage

\resettocdepth

\newpage
\appendix

\begin{center}
\LARGE
    \textbf{Appendix for ``SPARKLE: A Unified Single-Loop Primal-Dual   Framework for Decentralized Bilevel Optimization''}
\end{center}
\vspace{5mm}
\tableofcontents

\vspace{5mm}

\newpage
\section{More related works}
\label{app-related-work}

\paragraph{Bilevel optimization.}
Bilevel optimization presents substantial difficulties compared to single-level optimization due to its nested structure. Estimating hyper-gradient $\nabla\Phi(x)$ of the upper level involves solving lower-level problems and estimating the Hessian inverse, which requires additional calculations. Many algorithms and techniques have been proposed to solve the challenge. Approximate
Implicit Differentiation (AID)-based algorithms~\citep{domke2012generic,ghadimi2018approximation,grazzi2020iteration,ji2021bilevel} leverage the implicit gradient form of $\nabla\Phi(x)$, which entails solving a linear system to obtain the Hessian-inverse-vector product. Similarly, ~\citep{chen2021closing,hong2023two} utilize the Neumann series to handle the Hessian inverse. Iterative Differentiation (ITD)-based algorithms~\citep{franceschi2018bilevel,maclaurin2015gradient,domke2012generic,grazzi2020iteration,ji2021bilevel}  use iterative methods solving the lower-level problem and then estimate the hyper-gradient through automatic differentiation. However, these approaches introduce inner steps,  leading to extra computational overhead and memory spaces.~\cite{dagreou2022framework} proposes a single-level algorithm called SOBA, which approximating the Hessian-inverse-vector product by solving a quadratic programming problem. A recent
work~\cite{gao2024lancbio} utilizes the Krylov subspace technique and the Lanczos process to approximate it in deterministic scenarios. 
For stochastic bilevel optimization, various methods have been employed to improve the convergence rate, such as momentum~\cite{chen2022single,chen2023optimal} and variance reduction~\cite{yang2021provably,ji2021bilevel,guo2021randomized}.

\paragraph{Decentralized optimization.}
Decentralized optimization is developed to deal with large-scale optimization problems, where datasets are distributed among multiple agents. Without a central server, each agent only gets access to its own local data and communications are limited to its neighbors in a network. 
Compared with centralized algorithms,
decentralized ones preserve data privacy, and are more robust to contingencies in the communication network. However, due to the absence of a central server, decentralized optimization requires communication among agents, posing greater challenges for convergence, especially in the presence of severe data heterogeneity. 
To tackle this issue, various algorithms have emerged, such as decentralized gradient descent~\cite{nedic2009distributed,yuan2016convergence}, diffusion strategies~\cite{chen2012diffusion}, dual averaging~\cite{duchi2011dual}, EXTRA~\cite{shi2015extra}, Exact Diffusion (a.k.a.  D$^2$)~\cite{yuan2018exact1,li2017decentralized,tang2018d}, gradient tracking~\cite{xu2015augmented,di2016next,nedic2017achieving}, and decentralized ADMM ~\cite{chang2014multi}.
In stochastic scenarios, a common method for decentralized optimization
is the decentralized stochastic gradient descent (DSGD), which has gained
a lot of attentions recently. It has been proved to achieve linear speedup asymptotically and shares the same asymptotic rate with centralized stochastic gradient descent \cite{lian2017can}.

\section{More details of \ours}\label{AppA}
\subsection{Primal-dual deviation}\label{primal-dual deviation}

Here we provide a detailed motivation of the update framework \eqref{spdupdate} for decentralized single-level algorithms. First, we rewrite the single-level distributed optimization problem in the following equivalent form:
\begin{equation}\label{sl}
 \min _{x_i \in \mathbb{R}^d} f(x_1,...,x_n)=\frac{1}{n} \sum_{i=1}^n f_i\left(x_i\right),\quad \text{s.t. } \,x_1=...=x_n,
\end{equation}

where each $f_i$ is smooth and possibly non-convex. To simplify the notation, we assume that $d=1$ without loss of generality. Now we introduce three symmetric matrices $A,B,D$ such that  $A$ is a doubly stochastic communication matrix with $\rho(A)<1$, and $B,D$ satisfy $\mathrm{Null} B=\mathrm{Null} D=\mathrm{Span} \mathbf\{1_n\}$. 
In general, $B$ ($D$) determines the topology of a connected graph $\mathcal{G}_B$ ($\mathcal{G}_D$) over agents. The constraint $Bx=0$ ($Dx=0$) is equivalent to: 
\[x_i=x_j \text{ if } x_i,x_j \text{ are adjacent in } \mathcal{G}_B\, (\mathcal{G}_D).\] 
To simplify the derivation, we additionally assume that $A,B,D$ are pairwise commutative. Then for $x=(x_1,...,x_n)$, we have: 
\begin{equation}
  x_1=...=x_n\Leftrightarrow Bx=0 \Leftrightarrow Dx=0 \Leftrightarrow Ax=x.  
\end{equation}
Therefore, \eqref{sl} can be equivalently reformulated as
\begin{equation}\label{slr}
 \min _{x \in \mathbb{R}^n} f(Ax),\quad \text{s.t.} \,Bx=0.
\end{equation}
We construct the augmented Lagrangian function of the problem \eqref{slr} as follows: 
\begin{equation}\label{pdmotivation}
    \mathcal{L}_{\rho}(x,d)=f(Ax)+\left\langle d,Bx\right\rangle+\frac{\rho}{2}\|Dx\|^2,
\end{equation}
where $x$ denotes the primal variable, $d$ denotes the dual variable or Lagrangian multiplier associated with the consensus constraint, $\|Dx\|^2$ serves as the penalty term measuring the deviation from $Dx=0$, or equivalently $Bx=0$; $\rho > 0$ is the penalty coefficient.  
Though the introduction of matrices $A,D$
is essentially a matter of equivalent substitution, it enhances the universality of the algorithm framework we get.

Following classical primal-dual methods, we alternately perform gradient descent on 
$x$ and gradient ascent on $d$ in the $k$-th iteration:
\begin{equation}
\begin{aligned}
    x^{k+1}=x^k-\alpha(A\nabla f(Ax^k)+Bd^k+\rho D^2x^k),
    \quad d^{k+1}=d^k+\beta Bx^{k+1},  
\end{aligned}
\end{equation}
where $\alpha,\beta$ denote the step-sizes.  By making the  change of variables
\begin{equation}
\hat{x}^k=Ax^k,\,\hat{d}^k=\sqrt{\frac{\alpha}{\beta}}Ad^k, \,\widehat{B}=\sqrt{\alpha\beta}B, \,\widehat{C}=I-\alpha\rho D^2, \,\widehat{A}=A^2,
\end{equation}
we obtain
\begin{equation}\label{pdupdate}
\begin{aligned}
    \hat{x}^{k+1}=\widehat{C}\hat{x}^k-\alpha \widehat{A}\nabla f(\hat{x}^k)- \widehat{B}\hat{d}^k ,
\quad\hat{d}^{k+1}=\hat{d}^k+\widehat{B}\hat{x}^{k+1}.        
\end{aligned}
\end{equation}
One should note that the definition implies that $\widehat{A},\widehat{C}$ are doubly stochastic communication matrices under appropriate selections of $\alpha,\rho$. Finally, thanks to the introduction of moving-average iteration of \eqref{pdupdate}, we can obtain the framework \eqref{spdupdate} which serves as the foundation for our algorithm design. See more details in Section~\ref{Design}.

\subsection{Specific instances}\label{Specific instances}

\paragraph{Relation to some existing single-level  algorithm frameworks}
According to \eqref{pdupdate}, our framework at single-level is
\begin{equation}\label{frs}
 \mathbf{x}^{k+1}=\mathbf{C x}^k-\alpha \mathbf{A}\mathbf{g}^k-\mathbf{B}\mathbf{d}^k, 
\mathbf{d}^{k+1}=\mathbf{d}^k+\mathbf{B x}^{k+1}, k=0,1,...
\end{equation}
where $\alpha$ is the step-size, $\mathbf{g}^k$ denotes the estimated gradient at the $k$-th iteration, $\mathbf{d}$ serves as the dual variable.

Replacing $\mathbf{C}$ with  $\mathbf{CA}$, we get UDA\cite{alghunaim2020decentralized} , and equivalently, SUDA \cite{alghunaim2021unified}:
\begin{equation}
 \mathbf{x}^{k+1}=\mathbf{CAx}^k-\alpha \mathbf{A}\mathbf{g}^k-\mathbf{B }\mathbf{d}^k, 
\mathbf{d}^{k+1}=\mathbf{d}^k+\mathbf{B x}^{k+1}, k=0,1,...
\end{equation}

Therefore, following SUDA, we can also recover some common state-of-the-art heterogeneity methods as follows by selecting specific $\mathbf{A},\mathbf{B},\mathbf{C}$. First, from \eqref{frs} we get
\begin{equation}
\begin{aligned}
\mathbf{x}^{k+2}-\mathbf{x}^{k+1} & = \mathbf{C}(\mathbf{x}^{k+1}-\mathbf{x}^k)-\alpha \mathbf{A}\left(\mathbf{g}^{k+1}-\mathbf{g}^k\right)-\mathbf{B}\left(\mathbf{d}^{k+1}-\mathbf{d}^k\right) \\
& =  \mathbf{C}\left(\mathbf{x}^{k+1}-\mathbf{x}^k\right)-\alpha \mathbf{A}\left(\mathbf{g}^{k+1}-\mathbf{g}^k\right)-\mathbf{B}^2 \mathbf{x}^{k+1} .
\end{aligned}
\end{equation}
Thus, for $k \ge 0$ we have 
\begin{equation}
\mathbf{x}^{k+2}=\left(\mathbf{I}-\mathbf{B}^2+\mathbf{C}\right) \mathbf{x}^{k+1}- \mathbf{C} \mathbf{x}^k-\alpha \mathbf{A}\left(\mathbf{g}^{k+1}-\mathbf{g}^k\right),
\end{equation}
with $\mathbf{x}^1=\mathbf{C x}^0-\alpha \mathbf{A}\mathbf{g}^{0}$.

\paragraph{Some specific instances}
We next show that how to choose $\mathbf{A},\mathbf{B},\mathbf{C}$ to get some common heterogeneity methods.

\begin{itemize}
    \item ED: Taking $\mathbf{A}=\mathbf{W}, \mathbf{B}=(\mathbf{I}-\mathbf{W})^{1 / 2}$  and  $\mathbf{C}=\mathbf{W}$, we get ED:
    \begin{equation}
\mathbf{x}^{k+2}=\mathbf{W}\left(2 \mathbf{x}^{k+1}-\mathbf{x}^k-\alpha\left(\mathbf{g}^{k+1}-\mathbf{g}^k\right)\right),
    \end{equation}
with  $\mathbf{x}^1=\mathbf{W}\left(\mathbf{x}^0-\alpha \mathbf{g}^0\right)$. 
\item EXTRA: Taking $\mathbf{A}=\mathbf{I}, \mathbf{B}=(\mathbf{I}-\mathbf{W})^{1 / 2}$  with  $\mathbf{C}=\mathbf{W}$, we get EXTRA:
    \begin{equation}
\mathbf{x}^{k+2}=\mathbf{W}\left(2 \mathbf{x}^{k+1}-\mathbf{x}^k\right)-\alpha\left(\mathbf{g}^{k+1}-\mathbf{g}^k\right),
    \end{equation}
and  $\mathbf{x}^1=\mathbf{W}\mathbf{x}^0-\alpha \mathbf{g}^0$. 
\item Adapt-then-combine gradient tracking (ATC-GT):
The iteration  of ATC-GT  is
\begin{equation}
\begin{aligned}
 \mathbf{x}^{k+1}=\mathbf{W}\left(\mathbf{x}^k-\alpha \mathbf{h}^k\right) , \mathbf{h}^{k+1}=\mathbf{W}\left(\mathbf{h}^k+ \mathbf{g}^{k+1}-\mathbf{g}^k\right)
\end{aligned}
\end{equation}
with $\mathbf{h}^0=\mathbf{W} \mathbf{g}^0, \mathbf{x}^0=\mathbf{W} \mathbf{x}^0$ ($x^0_1=...=x^0_n$). It follows that for $k \ge 0$
\begin{equation}
\begin{aligned}
\mathbf{x}^{k+2}-\mathbf{W} \mathbf{x}^{k+1}  =\mathbf{W} \mathbf{x}^{k+1}-\mathbf{W}^2 \mathbf{x}^k-\alpha\mathbf{W} \left(\mathbf{h}^{k+1}-\mathbf{W}  \mathbf{h}^k\right). 
\end{aligned}
\end{equation}
 Then we obtain
 \begin{equation}
\mathbf{x}^{k+2}=2\mathbf{W}\mathbf{x}^{k+1}-\mathbf{W}^2 \mathbf{x}^k-\alpha \mathbf{W}^2\left(\mathbf{g}^{k+1}-\mathbf{g}^k\right),
\end{equation}
with $\mathbf{x}^1=\mathbf{W}^2 \mathbf{x}^0-\alpha\mathbf{W}^2 \mathbf{g}^0$. Thus, we can take $\mathbf{A}=\mathbf{\mathbf{W}^2},\mathbf{B}=(\mathbf{I}-\mathbf{W})^2,\mathbf{C}=\mathbf{W}^2$ to implement ATC-GT. 

\item Semi-ATC-GT:
The iteration  of Semi-ATC-GT  is
\begin{equation}
\begin{aligned}
 \mathbf{x}^{k+1}=\mathbf{W}\left(\mathbf{x}^k-\alpha \mathbf{h}^k\right) , \mathbf{h}^{k+1}=\mathbf{W}\mathbf{h}^k+ \mathbf{g}^{k+1}-\mathbf{g}^k
\end{aligned}
\end{equation}
with $\mathbf{h}^0=\mathbf{W} \mathbf{g}^0, \mathbf{x}^0=\mathbf{W} \mathbf{x}^0$ ($x^0_1=...=x^0_n$). Like ATC-GT, we have
 \begin{equation}
\mathbf{x}^{k+2}=2\mathbf{W} \mathbf{x}^{k+1}-\mathbf{W}^2\mathbf{x}^k-\alpha  \mathbf{W}\left(\mathbf{g}^{k+1}-\mathbf{g}^k\right),
\end{equation}
with $\mathbf{x}^1=\mathbf{W} ^2\mathbf{x}^0-\alpha\mathbf{W} \mathbf{g}^0$. Thus, we can take $\mathbf{A}=\mathbf{W},\mathbf{B}=(\mathbf{I}-\mathbf{W})^2,\mathbf{C}=\mathbf{W}^2$ to implement semi-ATC-GT. 

\item Non-ATC-GT:
The iteration  of Non-ATC-GT  is
\begin{equation}
\begin{aligned}
 \mathbf{x}^{k+1}=\mathbf{W}\mathbf{x}^k-\alpha \mathbf{h}^k , \mathbf{h}^{k+1}=\mathbf{W}\mathbf{h}^k+ \mathbf{g}^{k+1}-\mathbf{g}^k
\end{aligned}
\end{equation}
with $\mathbf{h}^0=\mathbf{W} \mathbf{g}^0, \mathbf{x}^0=\mathbf{W} \mathbf{x}^0$ ($x^0_1=...=x^0_n$). We have
 \begin{equation}
\mathbf{x}^{k+2}=2\mathbf{W} \mathbf{x}^{k+1}-\mathbf{W}^2 \mathbf{x}^k-\alpha  \left(\mathbf{g}^{k+1}-\mathbf{g}^k\right),
\end{equation}
with $\mathbf{x}^1=\mathbf{W} ^2\mathbf{x}^0-\alpha\mathbf{g}^0$. Thus, we can take $\mathbf{A}=\mathbf{I},\mathbf{B}=(\mathbf{I}-\mathbf{W})^2,\mathbf{C}=\mathbf{W}^2$ to implement Non-ATC-GT. 
\end{itemize}

\subsection{Implementation details}\label{Implementation details}

Given the update method \textbf{L}, we update the lower-level variable $y$ at the $k$-th ($k\ge0$) iteration as follows. For brevity, we define $y_i^{-1}=y_i^{0},v_i^{-1}=0,o_i^{0}=\sum_{j=1}^n(W_y)_{ij}v_j^0$.

\begin{equation}\label{y-imp}
    \begin{aligned}
        \begin{cases}y_i^{k+1}=\sum_{j=1}^n (W_y)_{ij}\left(2 y_j^{k}-y_j^{k-1}-\beta_k\left(v_i^k-v_i^{k-1}\right)\right)& \text { if } \textbf{L}=ED \\ y_i^{k+1}=\sum_{j=1}^n (W_y)_{ij}\left(2 y_j^{k}-y_j^{k-1}\right)-\beta_k\left(v_i^k-v_i^{k-1}\right) & \text { if } \textbf{L}=EXTRA \\ y_i^{k+1}  =\sum_{j=1}^n (W_y)_{ij}\left(y_j^k-\beta_k o_j^k\right), 
o_i^{k+1}  =\sum_{j=1}^n (W_y)_{ij}\left(o_j^k+v_i^{k+1}-v_i^{k}\right) & \text { if } \textbf{L}=GT \\ \cdots & \text { others }\end{cases}
    \end{aligned}
\end{equation}
Similarly, we update the auxiliary variable $z$ at the $k$-th ($k\ge0$) iteration as follows. For brevity, we define $z_i^{-1}=z_i^{0},p_i^{-1}=0,h_i^{0}=\sum_{j=1}^n(W_z)_{ij}p_j^0$. Note that we use the same method \textbf{L} to update $z$
 as we do for the lower-level variable 
$y$.
\begin{equation}\label{z-imp}
    \begin{aligned}
        \begin{cases}z_i^{k+1}=\sum_{j=1}^n (W_z)_{ij}\left(2 z_j^{k}-z_j^{k-1}-\gamma_k\left(p_i^k-p_i^{k-1}\right)\right)& \text { if } \textbf{L}=ED \\ z_i^{k+1}=\sum_{j=1}^n (W_z)_{ij}\left(2 z_j^{k}-z_j^{k-1}\right)-\gamma_k\left(p_i^k-p_i^{k-1}\right) & \text { if } \textbf{L}=EXTRA \\ z_i^{k+1}  =\sum_{j=1}^n (W_z)_{ij}\left(z_j^k-\gamma_k h_j^k\right), 
h_i^{k+1}  =\sum_{j=1}^n (W_z)_{ij}\left(h_j^k+p_i^{k+1}-p_i^{k}\right) & \text { if } \textbf{L}=GT \\ \cdots & \text { others }\end{cases}
    \end{aligned}
\end{equation}
Given the update method \textbf{U}, we update the upper-level variable $x$ at the $k$-th ($k\ge0$) iteration as follows. For brevity, we define $x_i^{-1}=x_i^{0},t_i^{0}=\sum_{j=1}^n(W_y)_{ij}r_j^1$.

\begin{equation}\label{x-imp}
    \begin{aligned}
        \begin{cases}x_i^{k+1}=\sum_{j=1}^n (W_x)_{ij}\left(2 x_j^{k}-x_j^{k-1}-\alpha_k\left(r_i^{k+1}-r_i^{k}\right)\right)& \text { if } \textbf{U}=ED \\ x_i^{k+1}=\sum_{j=1}^n (W_x)_{ij}\left(2 x_j^{k}-x_j^{k-1}\right)-\alpha_k\left(r_i^{k+1}-r_i^{k}\right) & \text { if } \textbf{U}=EXTRA \\ x_i^{k+1}  =\sum_{j=1}^n (W_x)_{ij}\left(x_j^k-\alpha_k t_j^k\right), 
t_i^{k+1}  =\sum_{j=1}^n (W_x)_{ij}\left(t_j^k+r_i^{k+2}-r_i^{k+1}\right) & \text { if } \textbf{U}=GT \\ \cdots & \text { others }\end{cases}
    \end{aligned}
\end{equation}
 Then the practical implementation of \ours with mixed strategies is 
\begin{algorithm}
  \caption{\ours-\textbf{L}-\textbf{U}}
  \label{I-II}
  \begin{algorithmic}
  \REQUIRE{Initialize ${x}_i^0={y}_i^0={z}_i^0={r}_i^0={0}$, step-sizes $\alpha_k,\beta_k,\gamma_k,\theta_k$}.
  \FOR{$k=0,1,\cdots,K-1$,  each agent $i$  (in parallel)}        
  \STATE Update $y_i^{k+1}$ according to \eqref{y-imp};
     \STATE Update $z_i^{k+1}$ according to \eqref{z-imp};
  \STATE${r}^{k+1}_i=(1-\theta_k)r_i^k+\theta_k u_i^k$ ;
  \STATE Update $x_i^{k+1}$ according to \eqref{x-imp}. 
  \ENDFOR
  \end{algorithmic}
\end{algorithm}

\section{Convergence analysis}\label{Proof}
\subsection{Proof of Theorem \ref{thm1}}
\subsubsection{Notations}

We use lowercase letters to represent vectors and uppercase letters to represent matrices. Stacked vectors $[x_1^{\top},...,x_n^{\top}]^{\top}$ is denoted by $\text{col}\{x_1,...,x_n\}$ for brevity. 
We denote a block diagonal matrix with diagonal block $M_i(1\le i\le l)$ by $\text{blkdiag}\{M_1,...,M_l\}$, and a diagonal matrix with diagonal elements $d_i(1\le i\le k)$ by $\text{diag}\{d_1,...,d_k\}$. The Kronecker product operator is denoted by
$\otimes$. For a variable $v$, we use $v^k_i$ to represent its components at $k$-th iteration and $i$-th agent. 

Moreover, we use an overbar \emph{above} an iterator to denote the average over all agents. For example, $\bar{x}^k=\sum_{i=1}^n x_i^k/n$. 
Upright bold symbols are used to denote stacked vectors or matrices across agents. For example, $\mathbf{x}^k\DefinedAs \text{col}\{x_1^k,...,x_n^k\}$, $\bar{\mathbf{x}}^k\DefinedAs \text{col}\{\bar{x}^k,...,\bar{x}^k\}\,(n\,\text{times})$, $\mathbf{W}_x\DefinedAs W_x\otimes I_{dim{(x)}}$. Denote the 2-norm of a matrix by $\|\cdot\|$.

Next, we define following $\sigma$-fields which will be used in our convergence analysis:\begin{align*}
\mathcal{F}_k&=\sigma\left(\mathbf{y}^0, \ldots, \mathbf{y}^{k+1},\mathbf{z}^0, \ldots, \mathbf{z}^{k+1},\mathbf{x}^{0}, \ldots, \mathbf{x}^k, \mathbf{r}^{0}, \ldots, \mathbf{r}^k\right),\\
\mathcal{U}_k&=\sigma\left(\mathbf{y}^0, \ldots, \mathbf{y}^{k+1},\mathbf{z}^0, \ldots, \mathbf{z}^k,\mathbf{x}^{0}, \ldots, \mathbf{x}^k, \mathbf{r}^{0}, \ldots, \mathbf{r}^k\right),\\
\mathcal{G}_k&=\sigma\left(\mathbf{y}^0, \ldots, \mathbf{y}^k,\mathbf{z}^0, \ldots, \mathbf{z}^k,\mathbf{x}^{0}, \ldots, \mathbf{x}^k, \mathbf{r}^{0}, \ldots, \mathbf{r}^k\right),
\end{align*}
and denote $\mathbb{E}[\cdot|\mathcal{F}_k]$ by $\mathbb{E}_k$, $\mathbb{E}[\cdot|\mathcal{U}_k]$ by $\widetilde{\mathbb{E}}_k$, $\mathbb{E}[\cdot|\mathcal{G}_k]$ by $\widehat{\mathbb{E}}_k$ for brevity. 

Define 
\begin{align*}
z^\star(x)&=\left(\sum_{i=1}^n\nabla_{22}^2 g_i\left(x, y^\star(x)\right)\right)^{-1} \left(\sum_{i=1}^n\nabla_2 f_i\left(x, y^\star(x)\right)\right),
\end{align*}
Then, for $k=0,1,\cdots$, define:
\begin{align*}
z_\star^{k+1}&=\left(\sum_{i=1}^n\nabla_{22}^2 g_i\left(\bar{x}^k, y^\star(\bar{x}^k)\right)\right)^{-1} \left(\sum_{i=1}^n\nabla_2 f_i\left(\bar{x}^k, y^\star(\bar{x}^k)\right)\right).
\end{align*}

For convenience, we define $\mathbf{x}^{-1}=\mathbf{x}^0,\mathbf{y}^{-1}=\mathbf{y}^0, y^\star(\bar{x}^{-1})=y^\star(\bar{x}^0),z_\star^{0}=z_\star^{1}.$

\subsubsection{Basic transformations}
We begin with conducting SUDA-like~\cite{alghunaim2021unified} transformations, which is fundamental of the following proofs.

Firstly, we define $\mathbf{t}^k$ to track the averaged stochastic gradients among agents as follows
\begin{equation}
  \label{def_s}
  \begin{aligned}
&\mathbf{t}^k_{y}=\mathbf{B}_y(\mathbf{d}^k_y-\mathbf{B}_y\mathbf{y}^k)+\beta\mathbf{A}_y\nabla_2 \mathbf{g}(\bar{\mathbf{x}}^k,\bar{\mathbf{y}}^k),
\\&\mathbf{t}^k_{z}=\mathbf{B}_z(\mathbf{d}^k_z-\mathbf{B}_z\mathbf{z}^k)+\gamma\mathbf{A}_z\mathbf{p}^k(\bar{\mathbf{x}}^k,\bar{\mathbf{y}}^{k+1}),
\\&\mathbf{t}^k_{x}=\mathbf{B}_x(\mathbf{d}_x^k-\mathbf{B}_x\mathbf{x}^k)+\alpha\mathbf{A}_x\widetilde{\nabla}\mathbf{\Phi}(\bar{\mathbf{x}}^k),
  \end{aligned}
\end{equation}
where
\begin{equation}
  \begin{aligned}
\mathbf{p}^k(\bar{\mathbf{x}}^k,\bar{\mathbf{y}}^{k+1})&=\text{col}\left\{\nabla_{22}^2g_i(\bar{x}^k,\bar{y}^{k+1})z_{\star}^k-\nabla_2 f_i(\bar{x}^k,\bar{y}^{k+1})\right\}_{i=1}^n,
\\\widetilde{\nabla}\mathbf{\Phi}(\bar{\mathbf{x}}^k)&=\text{col}\left\{\nabla_1 f_i(\bar{x}^k,y^{\star}(\bar{x}^k))-\nabla_{12}g_i(\bar{x}^k,y^{\star}(\bar{x}^k))z_{\star}^{k+1}\right\}_{i=1}^n.
\end{aligned}
\end{equation} 

Then
the iteration of $\bf{y},\bf{z},\bf{x}$ in Algorithm~\ref{D-SOBA-SUDA} can be written as:
  \begin{flalign}
\label{update_rule_1y}
\hspace{8mm}
  &\ \text{iteration of }\bf{y}:\quad \left\{\begin{aligned}
\mathbf{y}^{k+1}&=(\mathbf{C}_y-\mathbf{B}_y^2)\mathbf{y}^k-\mathbf{t}^k_y-\beta\mathbf{A}_y\left[\mathbf{v}^k-\nabla_2 \mathbf{g}(\bar{\mathbf{x}}^k,\bar{\mathbf{y}}^k)\right],\\
\mathbf{t}_y^{k+1}&=\mathbf{t}^k_{y}+\mathbf{B}_y^2\mathbf{y}^k+\beta\mathbf{A}_y\left[\nabla_2 \mathbf{g}(\bar{\mathbf{x}}^{k+1},\bar{\mathbf{y}}^{k+1})-\nabla_2 \mathbf{g}(\bar{\mathbf{x}}^k,\bar{\mathbf{y}}^k)\right],
  \end{aligned}\right. &
  \end{flalign}
  \begin{flalign}
\label{update_rule_1z}
\hspace{8mm}
  &\ \text{iteration of }\bf{z}:\quad \left\{\begin{aligned}
\mathbf{z}^{k+1}&=(\mathbf{C}_z-\mathbf{B}_z^2)\mathbf{z}^k-\mathbf{t}^k_z-\gamma\mathbf{A}_z\left[\mathbf{p}^k-\mathbf{p}^k(\bar{\mathbf{x}}^k,\bar{\mathbf{y}}^{k+1})\right],\\
\mathbf{t}_z^{k+1}&=\mathbf{t}^k_{z}+\mathbf{B}_z^2\mathbf{z}^k+\gamma\mathbf{A}_z\left[\mathbf{p}^{k+1}(\bar{\mathbf{x}}^{k+1},\bar{\mathbf{y}}^{k+2})-\mathbf{p}^k(\bar{\mathbf{x}}^k,\bar{\mathbf{y}}^{k+1})\right],
  \end{aligned}\right. &
  \end{flalign}
  \begin{flalign}
\label{update_rule_1x}
\hspace{8mm}
  &\ \text{iteration of }\bf{x}:\quad \left\{\begin{aligned}
\mathbf{x}^{k+1}&=(\mathbf{C}_x-\mathbf{B}_x^2)\mathbf{x}^k-\mathbf{t}^k_x-\alpha\mathbf{A}_x\left[\mathbf{r}^{k+1}-\widetilde{\nabla} \mathbf{\Phi}(\bar{\mathbf{x}}^k)\right],\\
\mathbf{t}_x^{k+1}&=\mathbf{t}^k_{x}+\mathbf{B}_x^2\mathbf{x}^k+\alpha\mathbf{A}_x\left[\widetilde{\nabla} \mathbf{\Phi}(\bar{\mathbf{x}}^{k+1})-\widetilde{\nabla} \mathbf{\Phi}(\bar{\mathbf{x}}^k)\right].
  \end{aligned}\right. &
  \end{flalign}

Next, we present the transformation of the  matrices $A,B,C$. For a communication matrix $W_s$ for the  variable $s\in\{x,y,z\}$
satisfying Assumption~\ref{net}, there exists an orthogonal matrix $U$ such that:
\begin{align*}
    W=U_s\hat{\Lambda}_s U_s^\top = \begin{bmatrix} \dfrac{1}{\sqrt{n}}\mathbf{1} & \hat{U}_s \end{bmatrix} \begin{bmatrix} 1 & 0 \\ 0 & {\Lambda}_s \end{bmatrix} \begin{bmatrix} \dfrac{1}{\sqrt{n}}\mathbf{1}^\top \\ \hat{U}_s^\top \end{bmatrix},
\end{align*}
where ${\Lambda}_s=\text{diag}\{\lambda_{si}\}_{i=2}^n$, $\hat{U}_s^\top\in\mathbb{R}^{n\times(n-1)}$ satisfies $\hat{U}_s\hat{U}_s^\top=I_n-\frac{1}{n}\mathbf{1}_n\mathbf{1}_n^\top$ and $\mathbf{1}_n^\top\hat{U}_s=0$. Then it follows that:
\begin{align*}
\mathbf{W}_s=\mathbf{U}_s\hat{\mathbf{\Lambda}}_s\mathbf{U}_s^\top = \begin{bmatrix} \dfrac{1}{\sqrt{n}}\mathbf{1}\otimes I_{dim(s)} & \hat{\bf{U}}_s \end{bmatrix} \begin{bmatrix} I_{dim(s)} & 0 \\ 0 & {\bf{\Lambda}_s} \end{bmatrix} \begin{bmatrix} \dfrac{1}{\sqrt{n}}\mathbf{1}^\top\otimes I_{dim(s)} \\ \hat{\bf{U}}_s^\top \end{bmatrix},
\end{align*}
where $dim(s)$ denotes the dimension of the corresponding variable, ${\mathbf{\Lambda}_s}={\Lambda}_s\otimes I_{dim(s)}\in\mathbb{R}^{d(n-1)\times [{dim(s)}\cdot (n-1)]}$, ${\bf{U}}_s\in\mathbb{R}^{[dim(s)\cdot n]\times [dim(s)\cdot n]}$ is an orthogonal matrix, and $\hat{\bf{U}}_s=\hat{U}_s\otimes I_{dim(s)}\in\mathbb{R}^{[{dim(s)}\cdot n]\times [{dim(s)}\cdot (n-1)]}$ satisfies:
\begin{align*}
    \hat{\bf{U}}_s^{\top}\hat{\bf{U}}_s={I}_{dim(s)\cdot (n-1)},\quad \hat{\bf{U}}_s\hat{\bf{U}}_s^{\top}=\left[I_n-\frac{1}{n}\bf{1}\bf{1}^\top\right]\otimes I_{dim(s)},\quad (\mathbf{1}^\top\otimes I_{dim(s)})\hat{\bf{U}}_s=\mathbf{0}.
\end{align*}
Now we add subscript $s$ for $\mathbf{W}_s$. 
Then, as $\mathbf{A}_s,\mathbf{B}_s^2,\mathbf{C}_s$ can be expressed as a polynomial of $\mathbf{W}_s$ for $s\in\{x,y,z\}$ according to Assumption~\ref{net}, we have the orthogonal decomposition:
\begin{equation}
\label{Matrix Decomposition}
    \begin{aligned}
\mathbf{A}_s&=\mathbf{U}_s\mathbf{\hat{\Lambda}}_{sa}\mathbf{U}_s^\top = \begin{bmatrix} \dfrac{1}{\sqrt{n}}\mathbf{1}\otimes I_{\text{dim}(s)} & \hat{\mathbf{U}}_s \end{bmatrix} \begin{bmatrix} I_{\text{dim}(s)} & \textbf{0} \\ \textbf{0} & {\mathbf{\Lambda}}_{sa} \end{bmatrix} \begin{bmatrix} \dfrac{1}{\sqrt{n}}\mathbf{1}^\top\otimes I_{\text{dim}(s)} \\ \hat{\mathbf{U}}_s^\top \end{bmatrix},\\
\mathbf{B}^2_s&=\mathbf{U}_s\mathbf{\hat{\Lambda}}_{sb}^2\mathbf{U}_s^\top = \begin{bmatrix} \dfrac{1}{\sqrt{n}}\mathbf{1}\otimes I_{\text{dim}(s)} & \hat{\mathbf{U}}_s \end{bmatrix} \begin{bmatrix} \textbf{0} & \textbf{0} \\ \textbf{0} & {\mathbf{\Lambda}}_{sb}^2 \end{bmatrix} \begin{bmatrix} \dfrac{1}{\sqrt{n}}\mathbf{1}^\top\otimes I_{\text{dim}(s)} \\ \hat{\mathbf{U}}_s^\top \end{bmatrix},\\
\mathbf{C}_s&=\mathbf{U}_s\mathbf{\hat{\Lambda}}_{sc}\mathbf{U}_s^\top = \begin{bmatrix} \dfrac{1}{\sqrt{n}}\mathbf{1}\otimes I_{\text{dim}(s)} & \hat{\mathbf{U}}_s \end{bmatrix} \begin{bmatrix} I_{\text{dim}(s)} & \textbf{0} \\ \textbf{0} & {\mathbf{\Lambda}}_{sc} \end{bmatrix} \begin{bmatrix} \dfrac{1}{\sqrt{n}}\mathbf{1}^\top\otimes I_{\text{dim}(s)} \\ \hat{\mathbf{U}}_s^\top \end{bmatrix},
    \end{aligned}
\end{equation}
where
\begin{small}
\begin{equation}
{\mathbf{\Lambda}}_{sa}=\underbrace{\text{diag}\{\lambda_{sa,i}\}_{i=2}^n}_{{\Lambda}_{sa}}\otimes I_{\text{dim}(s)},\quad{\mathbf{\Lambda}}_{sb}=\underbrace{\text{diag}\{\lambda_{sb,i}\}_{i=2}^n}_{{\Lambda}_{sb}}\otimes I_{\text{dim}(s)},\quad{\mathbf{\Lambda}}_{sc}=\underbrace{\text{diag}\{\lambda_{sc,i}\}_{i=2}^n}_{{\Lambda}_{sc}}\otimes I_{\text{dim}(s)}.
\end{equation}
\end{small}
Moreover, each ${\mathbf{\Lambda}}_{sb}$ is positive definite because of the null space condition in Assumption~\ref{net}. Then, multiplying both sides of \eqref{update_rule_1y}, \eqref{update_rule_1z} and  \eqref{update_rule_1x} by $\mathbf{U}_y^{\top},\mathbf{U}_z^{\top},\mathbf{U}_x^{\top}$ respectively, we get:

  \begin{flalign}
  \label{update_rule_2y}
  &\  \text{iter. of }\bf{y}:\left\{\begin{aligned}
\mathbf{U}_y^{\top}\mathbf{y}^{k+1}&=(\mathbf{\hat{\Lambda}}_{yc}-\mathbf{\hat{\Lambda}}_{yb}^2)\mathbf{U}_y^{\top}\mathbf{y}^k-\mathbf{U}_y^{\top}\mathbf{t}^k_y-\beta\mathbf{\hat{\Lambda}}_{ya}\mathbf{U}_y^{\top}\left[\mathbf{v}^k-\nabla_2 \mathbf{g}(\bar{\mathbf{x}}^k,\bar{\mathbf{y}}^k)\right],\\
\mathbf{U}_y^{\top}\mathbf{t}_y^{k+1}&=\mathbf{U}_y^{\top}\mathbf{t}^k_{y}+\mathbf{\hat{\Lambda}}_{yb}^2\mathbf{U}_y^{\top}\mathbf{y}^k+\beta\mathbf{\hat{\Lambda}}_{ya}\mathbf{U}_y^{\top}\left[\nabla_2 \mathbf{g}(\bar{\mathbf{x}}^{k+1},\bar{\mathbf{y}}^{k+1})-\nabla_2 \mathbf{g}(\bar{\mathbf{x}}^k,\bar{\mathbf{y}}^k)\right],
  \end{aligned}\right. &
 \end{flalign}
 
  \begin{flalign}
  \label{update_rule_2z}
  &\  \text{iter. of }\bf{z}: \left\{\begin{aligned}
\mathbf{U}_z^{\top}\mathbf{z}^{k+1}&=(\mathbf{\hat{\Lambda}}_{zc}-\mathbf{\hat{\Lambda}}_{zb}^2)\mathbf{U}_z^{\top}\mathbf{z}^k-\mathbf{U}_z^{\top}\mathbf{t}^k_z-\gamma\mathbf{\hat{\Lambda}}_{za}\mathbf{U}_z^{\top}\left[\mathbf{p}^k-\mathbf{p}^k(\bar{\mathbf{x}}^k,\bar{\mathbf{y}}^{k+1})\right],
\\\mathbf{U}_z^{\top}\mathbf{t}_z^{k+1}&=\mathbf{U}_z^{\top}\mathbf{t}^k_{z}+\mathbf{\hat{\Lambda}}_{zb}^2\mathbf{U}_z^{\top}\mathbf{z}^k+\gamma\mathbf{\hat{\Lambda}}_{za}\mathbf{U}_z^{\top}\left[\mathbf{p}^{k+1}(\bar{\mathbf{x}}^{k+1},\bar{\mathbf{y}}^{k+2})-\mathbf{p}^k(\bar{\mathbf{x}}^k,\bar{\mathbf{y}}^{k+1})\right].
  \end{aligned}\right. &
 \end{flalign}
 
  \begin{flalign}
  \label{update_rule_2x}
  &\  \text{iter. of }\bf{x}: \left\{\begin{aligned}
\mathbf{U}_x^{\top}\mathbf{x}^{k+1}&=(\mathbf{\hat{\Lambda}}_{xc}-\mathbf{\hat{\Lambda}}_{xb}^2)\mathbf{U}_x^{\top}\mathbf{x}^k-\mathbf{U}_x^{\top}\mathbf{t}^k_x-\alpha\mathbf{\hat{\Lambda}}_{xa}\mathbf{U}^{\top}\left[\mathbf{r}^{k+1}-
\widetilde{\nabla} \mathbf{\Phi}(\bar{\mathbf{x}}^k)\right],\\
\mathbf{U}_x^{\top}\mathbf{t}_x^{k+1}&=\mathbf{U}_x^{\top}\mathbf{t}^k_{x}+\mathbf{\hat{\Lambda}}_{xb}^2\mathbf{U}_x^{\top}\mathbf{x}^k+\alpha\mathbf{\hat{\Lambda}}_{xa}\mathbf{U}_x^{\top}\left[\widetilde{\nabla} \mathbf{\Phi}(\bar{\mathbf{x}}^{k+1})-\widetilde{\nabla} \mathbf{\Phi}(\bar{\mathbf{x}}^k)\right].
  \end{aligned}\right. &
 \end{flalign}

Then, due to Eq. \eqref{def_s}, we have:
  \begin{flalign}
  \label{express_s_y}
\hspace{14mm}
  &\  \begin{aligned}
      (\mathbf{1}^\top\otimes I_d)\mathbf{t}^k_{y}&=(\mathbf{1}^\top\otimes I_d)\left(\mathbf{B}_y(\mathbf{d}^k_y-\mathbf{B}_y\mathbf{y}^k)+\beta\mathbf{A}_y\nabla_2 \mathbf{g}(\bar{\mathbf{x}}^k,\bar{\mathbf{y}}^k)\right)\\
&=n\beta\nabla_2{g}(\bar{{x}}^k,\bar{{y}}^k).
  \end{aligned} &
 \end{flalign}
  \begin{flalign}
  \label{express_s_z}
\hspace{14mm}
  &\  \begin{aligned}
      (\mathbf{1}^\top\otimes I_d)\mathbf{t}^k_{z}&=(\mathbf{1}^\top\otimes I_d)\left(\mathbf{B}_z(\mathbf{d}^k_z-\mathbf{B}_z\mathbf{z}^k)+\gamma\mathbf{A}_z\mathbf{p}^k(\bar{\mathbf{x}}^k,\bar{\mathbf{y}}^{k+1})\right)\\
&=\gamma\sum_{i=1}^n\left[\nabla_{22}^2g_i(\bar{x}^k,\bar{y}^{k+1})z_{\star}^k-\nabla_2 f_i(\bar{x}^k,\bar{y}^{k+1})\right].
  \end{aligned} &
 \end{flalign}
  \begin{flalign}
  \label{express_s_x}
\hspace{14mm}
  &\  \begin{aligned}
      (\mathbf{1}^\top\otimes I_d)\mathbf{t}^k_{x}&=(\mathbf{1}^\top\otimes I_d)\left(\mathbf{B}_x(\mathbf{d}_x^k-\mathbf{B}_x\mathbf{x}^k)+\alpha\mathbf{A}_x\widetilde{\nabla}\mathbf{\Phi}(\bar{\mathbf{x}}^k)\right)\\
&=\alpha\sum_{i=1}^n\left[\nabla_1 f_i(\bar{x}^k,y^{\star}(\bar{x}^k))-\nabla_{12}g_i(\bar{x}^k,y^{\star}(\bar{x}^k))z_{\star}^{k+1}\right].
  \end{aligned} &
 \end{flalign}

Substituting \eqref{express_s_y}, \eqref{express_s_z}, \eqref{express_s_x} into \eqref{update_rule_2y}, \eqref{update_rule_2z}, \eqref{update_rule_2x}, respectively. Then use \eqref{Matrix Decomposition} and the structure of $\hat{\mathbf{U}}_y, \hat{\mathbf{U}}_z, \hat{\mathbf{U}}_x$, we have 
  \begin{flalign}
  \label{transform1_y}
  &\  \begin{aligned}
  \text{iter. of }\bf{y}: \left\{\begin{aligned}
\bar{y}^{k+1}&=\bar{y}^k-\beta\bar{v}^k,\\
\hat{\mathbf{U}}_y^{\top}\mathbf{y}^{k+1}&=({\mathbf{\Lambda}}_{yc}-{\mathbf{\Lambda}}_{yb}^2)\hat{\mathbf{U}}_y^{\top}\mathbf{y}^k-\hat{\mathbf{U}}_y^{\top}\mathbf{t}^k_y-\beta{\mathbf{\Lambda}}_{ya}\hat{\mathbf{U}}_y^{\top}\left[\mathbf{v}^k-\nabla_2 \mathbf{g}(\bar{\mathbf{x}}^k,\bar{\mathbf{y}}^k)\right],\\
\hat{\mathbf{U}}_y^{\top}\mathbf{t}_y^{k+1}&=\hat{\mathbf{U}}_y^{\top}\mathbf{t}^k_{y}+{\mathbf{\Lambda}}_{yb}^2\hat{\mathbf{U}}_y^{\top}\mathbf{y}^k+\beta{\mathbf{\Lambda}}_{ya}\hat{\mathbf{U}}_y^{\top}\left[\nabla_2 \mathbf{g}(\bar{\mathbf{x}}^{k+1},\bar{\mathbf{y}}^{k+1})-\nabla_2 \mathbf{g}(\bar{\mathbf{x}}^k,\bar{\mathbf{y}}^k)\right],
  \end{aligned}\right.
  \end{aligned} &
 \end{flalign}
 
  \begin{flalign}
  \label{transform1_z}
  &\  \begin{aligned}
  \text{iter. of }\bf{z}: \left\{\begin{aligned}
\bar{z}^{k+1}&=\bar{z}^k-\gamma\bar{p}^k,\\
\hat{\mathbf{U}}_z^{\top}\mathbf{z}^{k+1}&=({\mathbf{\Lambda}}_{zc}-{\mathbf{\Lambda}}_{zb}^2)\hat{\mathbf{U}}_z^{\top}\mathbf{z}^k-\hat{\mathbf{U}}_z^{\top}\mathbf{t}^k_z-\gamma{\mathbf{\Lambda}}_{za}\hat{\mathbf{U}}_z^{\top}\left[\mathbf{p}^k-\mathbf{p}^k(\bar{\mathbf{x}}^k,\bar{\mathbf{y}}^{k+1})\right],
\\\hat{\mathbf{U}}_z^{\top}\mathbf{t}_z^{k+1}&=\hat{\mathbf{U}}_z^{\top}\mathbf{t}^k_{z}+{\mathbf{\Lambda}}_{zb}^2\hat{\mathbf{U}}_z^{\top}\mathbf{z}^k+\gamma{\mathbf{\Lambda}}_{za}\hat{\mathbf{U}}_z^{\top}\left[\mathbf{p}^{k+1}(\bar{\mathbf{x}}^{k+1},\bar{\mathbf{y}}^{k+2})-\mathbf{p}^k(\bar{\mathbf{x}}^k,\bar{\mathbf{y}}^{k+1})\right],
  \end{aligned}\right.
  \end{aligned} &
 \end{flalign}
 
  \begin{flalign}
  \label{transform1_x}
  &\  \begin{aligned}
  \text{iter. of }\bf{x}: \left\{\begin{aligned}
\bar{x}^{k+1}&=\bar{x}^k-\alpha\bar{r}^{k+1},\\
\hat{\mathbf{U}}_x^{\top}\mathbf{x}^{k+1}&=(\hat{\mathbf{\Lambda}}_{xc}-\hat{\mathbf{\Lambda}}_{xb}^2)\hat{\mathbf{U}}_x^{\top}\mathbf{x}^k-\hat{\mathbf{U}}_x^{\top}\mathbf{t}^k_x-\alpha\hat{\mathbf{\Lambda}}_{xa}\hat{\mathbf{U}}_x^{\top}\left[\mathbf{r}^{k+1}-\widetilde{\nabla} \mathbf{\Phi}(\bar{\mathbf{x}}^k)\right],\\
\hat{\mathbf{U}}_x^{\top}\mathbf{t}_x^{k+1}&=\hat{\mathbf{U}}_x^{\top}\mathbf{t}^k_{x}+\hat{\mathbf{\Lambda}}_{xb}^2\hat{\mathbf{U}}_x^{\top}\mathbf{x}^k+\alpha\hat{\mathbf{\Lambda}}_{xa}\hat{\mathbf{U}}_x^{\top}\left[\widetilde{\nabla} \mathbf{\Phi}(\bar{\mathbf{x}}^{k+1})-\widetilde{\nabla} \mathbf{\Phi}(\bar{\mathbf{x}}^k)\right].
  \end{aligned}\right.
  \end{aligned} &
 \end{flalign}

The above three equations are equivalent to:
  \begin{flalign}
  \label{iteres_y}
\hspace{8mm}
  &\  \begin{aligned}
  {\left[\begin{array}{c}
{\hat{\mathbf{U}}_y}^{\top} \mathbf{y}^{k+1} \\
{{\boldsymbol{\Lambda}}}_{yb}^{-1} {\hat{\mathbf{U}}_y}^{\top} \mathbf{t}_y^{k+1}
\end{array}\right]=} & {\left[\begin{array}{cc}
 {{\boldsymbol{\Lambda}}}_{yc}-{{\boldsymbol{\Lambda}}}_{yb}^2 & -{{\boldsymbol{\Lambda}}}_{yb} \\
{{\boldsymbol{\Lambda}}}_{yb} & \mathbf{I}
\end{array}\right]\left[\begin{array}{c}
{\hat{\mathbf{U}}_y}^{\top} \mathbf{y}^k \\
{{\boldsymbol{\Lambda}}}_{yb}^{-1} {\hat{\mathbf{U}}_y}^{\top} \mathbf{t}^y_k
\end{array}\right] }  
\\&-\beta\left[\begin{array}{c}
{{\boldsymbol{\Lambda}}}_{ya} {\hat{\mathbf{U}}_y}^{\top}\left[\mathbf{v}^k-\nabla_2 \mathbf{g}(\bar{\mathbf{x}}^k,\bar{\mathbf{y}}^k)\right] \\
{{\boldsymbol{\Lambda}}}_{yb}^{-1} {{\boldsymbol{\Lambda}}}_{ya} {\hat{\mathbf{U}}_y}^{\top}\left[\nabla_2 \mathbf{g}(\bar{\mathbf{x}}^{k+1},\bar{\mathbf{y}}^{k+1})-\nabla_2 \mathbf{g}(\bar{\mathbf{x}}^k,\bar{\mathbf{y}}^k)\right]
\end{array}\right],
  \end{aligned} &
 \end{flalign}
 
  \begin{flalign}
  \label{iteres_z}
\hspace{8mm}
  &\  \begin{aligned}
{\left[\begin{array}{c}
{\hat{\mathbf{U}}_z}^{\top} \mathbf{z}^{k+1} \\
{{\boldsymbol{\Lambda}}}_{zb}^{-1} {\hat{\mathbf{U}}_z}^{\top} \mathbf{t}_z^{k+1}
\end{array}\right]=} & {\left[\begin{array}{cc}
 {{\boldsymbol{\Lambda}}}_{zc}-{{\boldsymbol{\Lambda}}}_{zb}^2 & -{{\boldsymbol{\Lambda}}}_{zb} \\
{{\boldsymbol{\Lambda}}}_{zb} & \mathbf{I}
\end{array}\right]\left[\begin{array}{c}
{\hat{\mathbf{U}}_z}^{\top} \mathbf{z}^k \\
{{\boldsymbol{\Lambda}}}_{zb}^{-1} {\hat{\mathbf{U}}_z}^{\top} \mathbf{t}^z_k
\end{array}\right] }  
\\&-\gamma\left[\begin{array}{c}
{{\boldsymbol{\Lambda}}}_{za} {\hat{\mathbf{U}}_z}^{\top}\left[\mathbf{p}^k-\mathbf{p}^k(\bar{\mathbf{x}}^k,\bar{\mathbf{y}}^{k+1})\right] \\
{{\boldsymbol{\Lambda}}}_{zb}^{-1} {{\boldsymbol{\Lambda}}}_{za} {\hat{\mathbf{U}}_z}^{\top}\left[\mathbf{p}^{k+1}(\bar{\mathbf{x}}^{k+1},\bar{\mathbf{y}}^{k+2})-\mathbf{p}^k(\bar{\mathbf{x}}^k,\bar{\mathbf{y}}^{k+1})\right]
\end{array}\right],
\end{aligned} &
 \end{flalign}
 
  \begin{flalign}
  \label{iteres_x}
\hspace{8mm}
  &\  \begin{aligned}
{\left[\begin{array}{c}
{\hat{\mathbf{U}}_x}^{\top} \mathbf{x}^{k+1} \\
{{\boldsymbol{\Lambda}}}_{xb}^{-1} {\hat{\mathbf{U}}_x}^{\top} \mathbf{t}_x^{k+1}
\end{array}\right]=} & {\left[\begin{array}{cc}
 {{\boldsymbol{\Lambda}}}_{xc}-{{\boldsymbol{\Lambda}}}_{xb}^2 & -{{\boldsymbol{\Lambda}}}_{xb} \\
{{\boldsymbol{\Lambda}}}_{xb} & \mathbf{I}
\end{array}\right]\left[\begin{array}{c}
{\hat{\mathbf{U}}_x}^{\top} \mathbf{x}^k \\
{{\boldsymbol{\Lambda}}}_{xb}^{-1} {\hat{\mathbf{U}}_x}^{\top} \mathbf{t}^x_k
\end{array}\right] }  
\\&-\alpha\left[\begin{array}{c}
{{\boldsymbol{\Lambda}}}_{xa} {\hat{\mathbf{U}}_x}^{\top}\left[\mathbf{r}^{k+1}-\widetilde{\nabla} \mathbf{\Phi}(\bar{\mathbf{x}}^k)\right] \\
{{\boldsymbol{\Lambda}}}_{xb}^{-1} {{\boldsymbol{\Lambda}}}_{xa} {\hat{\mathbf{U}}_x}^{\top}\left[\widetilde{\nabla} \mathbf{\Phi}(\bar{\mathbf{x}}^{k+1})-\widetilde{\nabla} \mathbf{\Phi}(\bar{\mathbf{x}}^k)\right]
\end{array}\right].
\end{aligned} &
 \end{flalign}

For $\mathbf{s}\in \{\mathbf{x},\mathbf{y},\mathbf{z}\}$, define:
\begin{equation}\label{M}
  \begin{aligned}
\mathbf{e}_s^k=\left[\begin{array}{c}
{\hat{\mathbf{U}}}_s^{\top} \mathbf{s}^k \\
{\mathbf{\Lambda}}_{sb}^{-1} {\hat{\mathbf{U}}_s}^{\top} \mathbf{t}^k_s
\end{array}\right],\quad \mathbf{M}_s=\left[\begin{array}{cc}
{\mathbf{\Lambda}}_{sc}-{\mathbf{\Lambda}}_{sb}^2 & -{\mathbf{\Lambda}}_{sb} \\
{\mathbf{\Lambda}}_{sb} & \mathbf{I}
\end{array}\right].
\end{aligned}
\end{equation} 
Then \eqref{iteres_y}, \eqref{iteres_z}, \eqref{iteres_x} are respectively equivalent to: 
\begin{equation} 
  \begin{aligned}
\mathbf{e}_y^{k+1}&=\mathbf{M}_y\mathbf{e}_y^k-\beta\left[\begin{array}{c}
{{\boldsymbol{\Lambda}}}_{ya} {\hat{\mathbf{U}}_y}^{\top}\left[\mathbf{v}^k-\nabla_2 \mathbf{g}(\bar{\mathbf{x}}^k,\bar{\mathbf{y}}^k)\right] \\
{{\boldsymbol{\Lambda}}}_{yb}^{-1} {{\boldsymbol{\Lambda}}}_{ya} {\hat{\mathbf{U}}_y}^{\top}\left[\nabla_2 \mathbf{g}(\bar{\mathbf{x}}^{k+1},\bar{\mathbf{y}}^{k+1})-\nabla_2 \mathbf{g}(\bar{\mathbf{x}}^k,\bar{\mathbf{y}}^k)\right]
\end{array}\right], \\
\mathbf{e}_z^{k+1}&=\mathbf{M}_z\mathbf{e}_z^k-\gamma\left[\begin{array}{c}
{{\boldsymbol{\Lambda}}}_{za} {\hat{\mathbf{U}}_z}^{\top}\left[\mathbf{p}^k-\mathbf{p}^k(\bar{\mathbf{x}}^k,\bar{\mathbf{y}}^{k+1})\right] \\
{{\boldsymbol{\Lambda}}}_{zb}^{-1} {{\boldsymbol{\Lambda}}}_{za} {\hat{\mathbf{U}}_z}^{\top}\left[\mathbf{p}^{k+1}(\bar{\mathbf{x}}^{k+1},\bar{\mathbf{y}}^{k+2})-\mathbf{p}^k(\bar{\mathbf{x}}^k,\bar{\mathbf{y}}^{k+1})\right]
\end{array}\right], \\
\mathbf{e}_x^{k+1}&=\mathbf{M}_x\mathbf{e}^x_k-\alpha\left[\begin{array}{c}
{{\boldsymbol{\Lambda}}}_{xa} {\hat{\mathbf{U}}_x}^{\top}\left[\mathbf{r}^{k+1}-\widetilde{\nabla} \mathbf{\Phi}(\bar{\mathbf{x}}^k)\right] \\
{{\boldsymbol{\Lambda}}}_{xb}^{-1} {{\boldsymbol{\Lambda}}}_{xa} {\hat{\mathbf{U}}_x}^{\top}\left[\widetilde{\nabla} \mathbf{\Phi}(\bar{\mathbf{x}}^{k+1})-\widetilde{\nabla} \mathbf{\Phi}(\bar{\mathbf{x}}^k)\right]
\end{array}\right].
\end{aligned}
\end{equation} 

 Assumption~\ref{net}, \ref{L_s}  imply that all eigenvalues of
 \begin{small}
\begin{equation}\label{tranL1}
\begin{aligned}
  &\left[\begin{array}{cc}
\text{diag}\{0,\Lambda_{sc}-\Lambda_{sb}^2\}& -\text{diag}\{0,\Lambda_{sb}\} \\\text{diag}\{0,\Lambda_{sb}\}
 & \text{diag}\{0,1,...,1\}
\end{array}\right] 
\\&=\left[\begin{array}{cc}
U_s^{\top}&  \\
 & U_s^{\top}
\end{array}\right] \left[\begin{array}{cc}
C_s-\frac{1}{n}{1}_n{1}_n^{\top} -B_s^2 &- B_s \\
B_s & I_n-\frac{1}{n}{1}_n{1}_n^{\top}
\end{array}\right]\left[\begin{array}{cc}
U_s&  \\
 & U_s
\end{array}\right]
\end{aligned}
\end{equation}
\end{small}
are strictly less than one in magnitude. Thus by symmetrically exchanging columns and rows of the matrix, we know that equivalently, all eigenvalues of
\begin{equation}\label{tranL2}
    \left[\begin{array}{cc}
\Lambda_{sc}-\Lambda_{sb}^2& -\Lambda_{sb} \\\Lambda_{sb}
 & I_{n-1}
\end{array}\right] \text{ and }
\mathbf{M}_s=\left[\begin{array}{cc}
\mathbf{\Lambda}_{sc}-\mathbf{\Lambda}_{sb}^2&- \mathbf{\Lambda}_{sb} \\\mathbf{\Lambda}_{sb}
 & I_{n-1}\otimes I_{dim(s)}
\end{array}\right]
\end{equation}
are strictly less than one in magnitude,.

Then according to Lemma~\ref{stable}, for $s\in \{x,y,z\}$, $\mathbf{M}_s$ has the similarity transformation:
\begin{equation}
  \begin{aligned}
\mathbf{M}_s=\mathbf{O}_s\mathbf{\Gamma}_s\mathbf{O}_s^{-1},
  \end{aligned}
\end{equation}
where $\mathbf{O}_s$ is invertible and $\|\mathbf{\Gamma}_s\|<1$. Moreover, we define $\hat{\mathbf{e}}_s^k=\mathbf{O}_s^{-1}\mathbf{e}_s^k$. It yields 

  \begin{align}\label{ey}
\hat{\mathbf{e}}_y^{k+1}&=\mathbf{\Gamma}_y\hat{\mathbf{e}}_y^k-\beta\mathbf{O}_y^{-1}\left[\begin{array}{c}
{\boldsymbol{\Lambda}}_{ya} \hat{\mathbf{U}}_y^{\top}\left[\mathbf{v}^k-\nabla_2 \mathbf{g}(\bar{\mathbf{x}}^k,\bar{\mathbf{y}}^k)\right] \\
{\boldsymbol{\Lambda}}_{yb}^{-1} {\boldsymbol{\Lambda}}_{ya}\hat {\mathbf{U}}_y^{\top}\left[\nabla_2 \mathbf{g}(\bar{\mathbf{x}}^{k+1},\bar{\mathbf{y}}^{k+1})-\nabla_2 \mathbf{g}(\bar{\mathbf{x}}^k,\bar{\mathbf{y}}^k)\right]
\end{array}\right], \\
\label{ez}
\hat{\mathbf{e}}_z^{k+1}&=\mathbf{\Gamma}_y\hat{\mathbf{e}}_z^k-\gamma\mathbf{O}_z^{-1}\left[\begin{array}{c}
{{\boldsymbol{\Lambda}}}_{za} {\hat{\mathbf{U}}_z}^{\top}\left[\mathbf{p}^k-\mathbf{p}^k(\bar{\mathbf{x}}^k,\bar{\mathbf{y}}^{k+1})\right] \\
{{\boldsymbol{\Lambda}}}_{zb}^{-1} {{\boldsymbol{\Lambda}}}_{za} {\hat{\mathbf{U}}_z}^{\top}\left[\mathbf{p}^{k+1}(\bar{\mathbf{x}}^{k+1},\bar{\mathbf{y}}^{k+2})-\mathbf{p}^k(\bar{\mathbf{x}}^k,\bar{\mathbf{y}}^{k+1})\right]
\end{array}\right], \\
\label{ex}
\hat{\mathbf{e}}_x^{k+1}&=\mathbf{\Gamma}_x\hat{\mathbf{e}}_x^k-\alpha\mathbf{O}_x^{-1}\left[\begin{array}{c}
{{\boldsymbol{\Lambda}}}_{xa} {\hat{\mathbf{U}}_x}^{\top}\left[\mathbf{r}^{k+1}-\widetilde{\nabla} \mathbf{\Phi}(\bar{\mathbf{x}}^k)\right] \\
{{\boldsymbol{\Lambda}}}_{xb}^{-1} {{\boldsymbol{\Lambda}}}_{xa} {\hat{\mathbf{U}}_x}^{\top}\left[\widetilde{\nabla} \mathbf{\Phi}(\bar{\mathbf{x}}^{k+1})-\widetilde{\nabla} \mathbf{\Phi}(\bar{\mathbf{x}}^k)\right]
\end{array}\right].
\end{align}

Then, for $\mathbf{s}\in \{\mathbf{x},\mathbf{y},\mathbf{z}\}$, the consensus errors between different agents have the upper bound of:
\begin{equation}
    \label{transformer_consensus}
\left\Vert \mathbf{s}^k-\bar{\mathbf{s}} ^k\right\Vert^2=\Vert {\hat{\mathbf{U}}}_s^{\top} \mathbf{s}^k \Vert^2\leq \Vert \mathbf{e}_{s}^k \Vert^2\leq \Vert \mathbf{O}_{s} \Vert^2\Vert \hat{\mathbf{e}}_{s}^k \Vert^2.
\end{equation}

Thus, we can define:
\begin{align}     \Delta_k&=\kappa^2\|\mathbf{O}_x\|^2\|\hat{\mathbf{e}}^k_x\|^2+\kappa^2\|\mathbf{O}_y\|^2\|\hat{\mathbf{e}}^{k+1}_y\|^2+\|\mathbf{O}_z\|^2\|\hat{\mathbf{e}}^{k+1}_z\|^2
    \end{align}
to measure the consensus error during the iteration. 

We also define
\begin{align}     I_k&=\|\bar{z}^{k+1}-z_{\star}^{k+1}\|^2+\kappa^2\|\bar{y}^{k+1}-y^{\star}(\bar{x}^k)\|^2,
    \end{align}
to measure the estimation accuracy of the lower- and auxiliary-level problems.

\subsubsection{Proof sketch}
Before proceeding with the formal proof, we first present the structure of the proof in Appendix \ref{Proof}.

\[\hspace{-40mm}
\begin{array}{cc}
\begin{tikzcd}
{\phantom{aaaaaaaaaaaaaaaaaaaaaaaa}\text{Bounded by each other}} \\
	{}
	\arrow[shift right=25, Rightarrow, from=2-1, to=1-1]
\end{tikzcd}
\\
\begin{tikzcd}
	{\text{Descent of }x} & \hyperref[Er]{\sum\mathbb{E}\|\bar{r}^k\|^2} \\
	{\text{Descent of }y} & \hyperref[ystar]{\sum\mathbb{E}\|\bar{y}^{k+1}-y^\star(\bar{x}^k)\|^2} & \hspace{-3mm}\hyperref[Deltasum] {\sum\mathbb{E}\left[\Delta_k\right]} & \hspace{-3mm}\hyperref[Isum]{\sum\mathbb{E}\left[I_k\right]} \\
	{\text{Descent of }z} &  \hyperref[zstar]{\sum\mathbb{E}\|\bar{z}^{k+1}-z_\star^{k+1}\|^2} \\
	{\text{Consensus of }y} & \hyperref[4]{\sum\mathbb{E}\|\hat{\mathbf{e}}_y^{k}\|^2} && \hspace{-8mm}\hyperref[DeltaI]{\sum\mathbb{E}\left[\frac{\Delta_k}{n}+I_k\right]} \\
	{\text{Consensus of }z} & \hyperref[6]{\sum\mathbb{E}\|\hat{\mathbf{e}}_z^{k}\|^2} & {} \\
	{\text{Consensus of }x} & \hyperref[10]{\sum\mathbb{E}\|\hat{\mathbf{e}}_x^{k}\|^2} && \hspace{-5mm}\hyperref[new r]{\sum\mathbb{E}\|\Phi(\bar{x}^k)\|^2} \\
	{\text{Hyper-gradient estimation}} &
 \hyperref[8]{\sum\mathbb{E}\|\mathbb{E}_k\mathbf{u}^k-\widetilde{\nabla}\boldsymbol{\Phi}(\bar{x}^k)\|^2} \\
	{\text{Hyper-gradient estimation}} & \hspace{-8mm}\hyperref[rphi]{\sum\mathbb{E}\|\mathbb{E}_k\mathbf{r}^{k+1}-\widetilde{\nabla}\boldsymbol{\Phi}(\bar{x}^k)\|^2} && \hspace{-1mm}{\text{Lemma \ref{Stochasticrate} }} \\
	{\text{Variance}} & \hyperref[uvar]{\sum\mathbb{E}\|\mathbb{E}_k\mathbf{u}^k-\mathbf{u}^k\|^2}
	\arrow[from=2-3, to=4-4]
	\arrow[from=2-4, to=4-4]
	\arrow[from=4-4, to=6-4]
	\arrow[shorten <=5pt, from=5-3, to=2-3]
	\arrow[shorten <=7pt, from=5-3, to=2-4]
	\arrow[shorten <=3pt, shorten >=3pt, from=5-3, to=6-4]
	\arrow[from=6-4, to=8-4]
\end{tikzcd}
\hspace{-52mm}\left.\phantom{\begin{array}{c}
     \\
     \\
     \\   
     \\
     \\
     \\
     \\
     \\
     \\
     \\
     \\
     \\   
     \\
     \\
     \\
     \\
     \\
     \\
     \\
     \\
     \\
     \\    
     \\
     \\
     \\
     \\
     \\
     \\
\end{array}}\right\}&
    \end{array}\]

\subsubsection{Technical lemmas}
\begin{lemma}\label{Lsmooth} Suppose Assumptions~\ref{smooth}  hold, we know $\nabla \Phi(x)$,$\widetilde{\nabla} \mathbf{\Phi}(x)$, $z^\star(x)$ and $y^\star(x)$ defined above are $L_{\nabla\Phi}$ ,  $\widetilde{L}$, $L_{z^\star}$, $L_{y^\star}$- Lipschitz continuous respectively with the constants satisfying:
\begin{subequations}
\begin{align}
L_{\nabla\Phi}&\le L_{f, 1}+\frac{2 L_{f, 1} L_{g, 1}+L_{g, 2} L_{f, 0}}{\mu_g}+\frac{2 L_{g, 1} L_{f, 0} L_{g, 2}+L_{g, 1}^2 L_{f, 1}}{\mu_g^2}+\frac{L_{g, 2} L_{g, 1}^2 L_{f, 0}}{\mu_g^3},\\
\widetilde{L}&\le L_{f, 1}+\frac{2 L_{f, 1} L_{g, 1}+L_{g, 2} L_{f, 0}}{\mu_g}+\frac{2 L_{g, 1} L_{f, 0} L_{g, 2}+L_{g, 1}^2 L_{f, 1}}{\mu_g^2}+\frac{L_{g, 2} L_{g, 1}^2 L_{f, 0}}{\mu_g^3},\\
L_{y^\star}&\le\frac{L_{g, 1}}{\mu_g},\\
L_{z^\star}&\le\sqrt{1+L_{y^\star}^2}\left(\frac{L_{f,1}}{\mu_g}+\frac{L_{f,0} L_{g,2}}{\mu_g^2}\right).
   \end{align} 
\end{subequations}
And we also have:
\begin{align}
\label{est:z*}
    \left\Vert z^\star(x)\right\Vert\leq\dfrac{L_{f,0}}{\mu_g},\quad\forall x\in\mathbb{R}^p.
\end{align}
\end{lemma}

\begin{proof}
    See Lemma 2.2 in~\cite{ghadimi2018approximation} and Lemma B.2 in ~\cite{chen2023optimal}.
\end{proof}

\begin{lemma}\label{des} Suppose that $g(x)$ is $\mu$-strongly convex and $L$-smooth. Then for any $x$ and $0<\alpha<\frac{2}{\mu+L}$,  we have
\[
\left\|x-\alpha \nabla g(x)-x^\star\right\| \leq(1-\alpha \mu)\left\|x-x^\star\right\|,
\]
where $x^\star=\argmin g(x)$.
\end{lemma}
\begin{proof}
     See Lemma 10 in~\cite{qu2018harnessing}.
\end{proof}

\begin{lemma}\label{stable}
Given diagonal matrices $A,B,C,D\in \mathbb{R}^{(n-1)\times (n-1)}$, and 
\[\mathbf{M}=\left[\begin{array}{cc}
A\otimes I_d & B\otimes I_d \\
C\otimes I_d & D\otimes I_d
\end{array}\right].\]
Suppose that the eigenvalues of $\mathbf{M}$  are strictly less than one in magnitude. Then there exist an invertible matrix $\mathbf{O}$ and a matrix $\mathbf{\Gamma}$ with $\|\mathbf{\Gamma}\|<1$, such that $\mathbf{M}$ has the similarity transformation:   
\[\mathbf{M}=\mathbf{O}\mathbf{\Gamma}\mathbf{O}^{-1}.\]
\end{lemma}
\begin{proof}
We provide the proof based on  Lemma 1 in~\cite{alghunaim2021unified}. Denote $A,B,C,D$ by $\text{diag}\{\lambda_{i}^a\}_{i=1}^{n-1}$, $\text{diag}\{\lambda_{i}^b\}_{i=1}^{n-1}$, $\text{diag}\{\lambda_{i}^c\}_{i=1}^{n-1}$, $\text{diag}\{\lambda_{i}^d\}_{i=1}^{n-1}$ respectively. Then, the matrix  $\mathbf{M}$ can be expressed as:
\begin{equation}
    \mathbf{M}=\left[\begin{array}{cc}
\text{diag}\{\lambda_{i}^a\}_{i=1}^{n-1} & \text{diag}\{\lambda_{i}^b\}_{i=1}^{n-1} \\\text{diag}
\{\lambda_{i}^c\}_{i=1}^{n-1} & \text{diag}\{\lambda_{i}^d\}_{i=1}^{n-1}
\end{array}\right]\otimes I_d.
\end{equation}

By performing symmetric column and row permutations, it suffices to show that each submatrix
\[
M_i=\left[\begin{array}{cc}
\lambda_{i}^a & \lambda_{i}^b \\
\lambda_{i}^c & \lambda_{i}^d
\end{array}\right]\in\mathbb{R}^{2\times 2}
\]
admits the similarity transformation $M_i=O_i\Gamma_iO_i^{-1}$, where $||\Gamma_i||<1$. We consider two cases based on the eigenvalues of $M_i$:
\begin{itemize}
    \item Case 1: $M_i$ has two distinct eigenvalues $\lambda_1,\lambda_2$.  In this case,  $M_i$  trivially admits the similarity transformation  $M_i=U\Lambda U^{-1}$, where $\Lambda=\text{diag}\{\lambda_1,\lambda_2\}$, and the eigenvalues $\lambda_1,\lambda_2$  
  are strictly less than one in magnitude.
\item Case 2: $M_i$ has repeated  eigenvalues $\lambda_1=\lambda_2=\lambda$.  By the Jordan decomposition, $M_i$ admits the similarity transformation: 
\[M_i=U J U^{-1}, J=\left[\begin{array}{cc}
\lambda & 1 \\
0 & \lambda
\end{array}\right].\]
To ensure the spectral norm  is less than one, we introduce a scaling matrix $E=\text{diag}\{1,\epsilon\}$ for some $0<\epsilon<1-|\lambda|$, and define $\hat{J}=E^{-1}JE$. Then we have $||\hat{J}||^2\le ||\hat{J}\hat{J}^*||_1\le(|\lambda|+\epsilon)^2<1$. Thus $M_i=UE \hat{J}(UE)^{-1}$ satisfies the required condition.
\end{itemize}
\end{proof}

\begin{remark}
Asserting the existence of $\mathbf{\Gamma}$ with $\|\mathbf{\Gamma}\|<1$, Lemma \ref{stable} only guarantees the convergence of \ours. However, to obtain a precise non-asymptotic convergence rate, one must construct appropriate $\mathbf{O}$ and $\mathbf{\Gamma}$. See more details in Appendix \ref{Essential matrix norms}.
\end{remark}

\subsubsection{Descent lemmas for the upper-level}
In this subsection, we estimate the upper bound of the errors induced by the moving average in hyper-gradient estimation, as well as  the upper bound of $\Vert\nabla\Phi(x)\Vert^2$ based on $I_k,\Delta_k$.  
\begin{lemma}\label{8}
  Suppose Assumptions~\ref{smooth}-~\ref{var} hold.
  We have:
  \begin{equation}
    \begin{aligned}\label{u,psi,A,B}
      \left\|\mathbb{E}_k\bar{u}^k -\nabla \Phi(\bar{x}^k)\right\| ^2\leq& \dfrac{20}{n}L^2(\Delta_k+nI_k),\\
      \left\|\mathbb{E}_k\mathbf{u}^k -\widetilde{\nabla} \mathbf{\Phi}(\bar{\mathbf{x}}^k)\right\| ^2
       \leq& 20L^2(\Delta_k+nI_k).
      \end{aligned}
      \end{equation}
 \begin{proof}

   Cauchy Schwartz inequality implies that:
  \begin{equation}\label{est:U,psi,A,B}
  \begin{aligned}
    &\left\|\mathbb{E}_k\mathbf{u}^k -\widetilde{\nabla} \mathbf{\Phi}\left(\bar{\mathbf{x}}^k\right)\right\| ^2
    \\\leq & 5\sum_{i=1}^n\left\|\nabla_1 f_i\left(x_{i}^k, y_{i}^{k+1}\right)-\nabla_1 f_i\left(\bar{x}^k, \bar{y}^{k+1}\right)\right\|^2+5 \sum_{i=1}^n\left\|\nabla_1 f_i\left(\bar{x}^k, \bar{y}^{k+1}\right)-\nabla_1 f_i\left(\bar{x}^k, y^{\star}(\bar{x}^k)\right)\right\|^2 \\
    &+  5  \sum_{i=1}^n\left\|\nabla_{12}^2 g_i\left(x_{i}^k, y_{i}^{k+1}\right)\left(z_{i}^{k+1} -z_{\star}^{k+1}\right)\right\|^2\\
    &+5  \sum_{i=1}^n\left\|\left(\nabla_{12}^2 g_i\left(x_{i}^k, y_{i}^{k+1}\right)-\nabla_{12}^2 g_i\left(\bar{x}^k, \bar{y}^{k+1}\right)\right) z_\star^{k+1}\right\|^2\\
    &+  5 \sum_{i=1}^n\left\|\left(\nabla_{12}^2 g_i\left(\bar{x}^k, \bar{y}^{k+1}\right)-\nabla_{12}^2 g_i\left(\bar{x}^k, y^\star(\bar{x}^k)\right)\right) z_\star^{k+1}\right\|^2
    \\ 
    \leq& 10\left(L_{f, 1}^2+\kappa^2L_{f, 0}^2\right) \left( \left\|\mathbf{x}^k-\bar{\mathbf{x}}^k \right\|^2+\left\|\mathbf{y}^{k+1}-\bar{\mathbf{y}}^{k+1} \right\|^2+\left\|\bar{\mathbf{y}}^{k+1}-\mathbf{y}^{\star}(\bar{\mathbf{x}}^k)\right\|^2\right) \\
    & +10 L_{g, 1}^2  \left(\left\|\mathbf{z}^{k+1}-\bar{\mathbf{z}}^{k+1} \right\|^2+\left\|\bar{\mathbf{z}}^{k+1} -\mathbf{z}_{\star}^{k+1}\right\|^2\right)            \\
    \leq&20L^2\left(\Delta_k+n I_k\right).
    \end{aligned}
  \end{equation}

For the term $\left\|\mathbb{E}_k\bar{u}^k -\nabla \Phi\left(\bar{x}^k\right)\right\| ^2$, we have:
  \begin{equation}\label{est:u,psi,A,B}
  \begin{aligned}
     \left\|\mathbb{E}_k\bar{u}^k -\nabla \Phi(\bar{x}^k)\right\| ^2  
     \leq  \frac{1}{n}\left\|\mathbb{E}_k\mathbf{u}^k -\widetilde{\nabla} \mathbf{\Phi}\left(\bar{\mathbf{x}}^k\right)\right\| ^2
    \leq 20 L^2\left(\frac{\Delta_k}{n}+I_k\right).
    \end{aligned}
  \end{equation}

 \end{proof} 
\end{lemma}

\begin{lemma}\label{uvar}
Suppose that Assumptions~\ref{smooth}-~\ref{var} hold. 
We have
\begin{equation}
\label{conclusion_lemma4}
  \begin{aligned}
&n^2\sum_{k=0}^K\mathbb{E}\left[\|\bar{u}^k-\mathbb{E}_k[\bar{u}^k]\|^2\right]=\sum_{k=0}^K\mathbb{E}\left[\|\mathbf{u}^k-\mathbb{E}_k[\mathbf{u}^k]\|^2\right]\\
    \leq &9\sigma_{g,2}^2\sum_{k=0}^K\left(\mathbb{E}\|\mathbf{z}^{k+1}-\bar{\mathbf{z}}^{k+1}\|^2+\mathbb{E}\|\bar{\mathbf{z}}^{k+1}-\mathbf{z}_{\star}^{k+1}\|^2\right)+3(K+1)n\left(\sigma_{f,1}^2+3\sigma_{g,2}^2\frac{L_{f,0}^2}{\mu_g^2}\right).
  \end{aligned}
\end{equation}

\begin{proof}

For $k\ge0$, Cauchy Schwartz inequality implies that
\begin{equation}
  \begin{aligned}
 &\frac{1}{3}\mathbb{E}_k\left[\|\mathbf{u}^k-\mathbb{E}_k[\mathbf{u}^k]\|^2\right]\\
 \le& \mathbb{E}_k\left[\sum_{i=1}^n\|\nabla_1 f_i(x_{i}^k,y_{i}^{k+1},\xi_{i}^k)-\nabla_1 f_i(x_{i}^k,y_{i}^{k+1})\|^2\right]
\\&+\mathbb{E}_k\left[\sum_{i=1}^n\left\|\left(\nabla_{12} g_i(x_{i}^k,y_{i}^{k+1},\zeta_{i}^k)-\nabla_{12} g_i(x_{i}^k,y_{i}^{k+1})\right)z_{i}^{k+1}\right\|^2\right]
\\\leq&n\sigma_{f,1}^2+\sigma_{g,2}^2\|\mathbf{z}^{k+1}\|^2
\\\leq&n\sigma_{f,1}^2+3\sigma_{g,2}^2\left(\|\mathbf{z}^{k+1}-\bar{\mathbf{z}}^{k+1}\|^2+\|\bar{\mathbf{z}}^{k+1}-\mathbf{z}_{\star}^{k+1}\|^2+\|\mathbf{z}_{\star}^{k+1}\|^2\right)
\\\leq&n\sigma_{f,1}^2+3\sigma_{g,2}^2\left(\|\mathbf{z}^{k+1}-\bar{\mathbf{z}}^{k+1}\|^2+\|\bar{\mathbf{z}}^{k+1}-\mathbf{z}_{\star}^{k+1}\|^2+n\frac{L_{f,0}^2}{\mu_g^2}\right).
\end{aligned}
\end{equation}

Then taking expectation and summation on both sides, we get
\begin{equation}
  \begin{aligned}
&\sum_{k=0}^K\mathbb{E}\left[\|\mathbf{u}^k-\mathbb{E}_k[\mathbf{u}^k]\|^2\right]\\
    \leq &9\sigma_{g,2}^2\sum_{k=0}^K\left(\mathbb{E}\|\mathbf{z}^{k+1}-\bar{\mathbf{z}}^{k+1}\|^2+\mathbb{E}\|\bar{\mathbf{z}}^{k+1}-\mathbf{z}_{\star}^{k+1}\|^2\right)+3(K+1)n\left(\sigma_{f,1}^2+3\sigma_{g,2}^2\frac{L_{f,0}^2}{\mu_g^2}\right).
  \end{aligned}
\end{equation}

Since samples among agents are independent, it follows that
\begin{equation}
  \begin{aligned}
\sum_{k=0}^K\mathbb{E}_k\left[\|\bar{u}^k-\mathbb{E}_k[\bar{u}^k]\|^2\right]= \frac{1}{n^2}\sum_{k=0}^K\mathbb{E}_k\left[\|\mathbf{u}^k-\mathbb{E}_k[\mathbf{u}^k]\|^2\right].
  \end{aligned}
\end{equation}
Taking expectations, we get the conclusion.
\end{proof}
\end{lemma}

\begin{lemma}\label{Er}
  Suppose that Assumptions~\ref{smooth}-~\ref{var}, and 
 Lemmas~\ref{8},~\ref{uvar} hold. If 
 \begin{equation}\label{l11}
      \alpha\leq \frac{1}{2L_{\nabla \Phi}},
 \end{equation}
 
we have
\begin{equation}\label{barr}
  \begin{aligned}
    &\frac{1}{4}\sum_{k=0}^K\mathbb{E}\left\|\bar{r}^{k+1}\right\|^2\\
    \le&\frac{\Phi(\bar{x}_0)-\inf\Phi}{\alpha}+10\left(L^2+\frac{\theta\sigma_{g,2}^2}{n}\right)\sum_{k=0}^K\mathbb{E}\left(\frac{\Delta_k}{n}+I_k\right)+\frac{3\theta}{n}(K+1)\left(\sigma_{f,1}^2+2\sigma_{g,2}^2\frac{L_{f,0}^2}{\mu_g^2}\right).
  \end{aligned}
\end{equation}  

  \begin{proof}
    The $L_{\nabla \Phi}$-smoothness of $\Phi$ indicates that
    \begin{equation}\label{psi}
      \begin{aligned}
      &\mathbb{E}_k[\Phi\left(\bar{x}^{k+1}\right)]-\Phi(\bar{x}^k) \\
      \leq&\left\langle\nabla \Phi(\bar{x}^k),\left(-\alpha \mathbb{E}_k[\bar{r}^{k+1}]\right)\right\rangle+\frac{L_{\nabla \Phi} \alpha^2}{2}\mathbb{E}_k\|\bar{r}^{k+1}\|^2\\
      =&\left\langle\nabla \Phi(\bar{x}^k)-\mathbb{E}_k[\bar{u}^k],-\alpha \mathbb{E}_k[\bar{r}^{k+1}]\right\rangle+\frac{L_{\nabla \Phi}}{2}\alpha^2\mathbb{E}_k\|\bar{r}^{k+1}\|^2 -\alpha\left\langle\mathbb{E}_k[\bar{u}^k],\mathbb{E}_k[\bar{r}^{k+1}]\right\rangle.
  \end{aligned}
\end{equation}
Then, due to $\mathbb{E}_k[\bar{u}^k]=\theta^{-1}(\mathbb{E}_k[\bar{r}^{k+1}]-(1-\theta)\bar{r}^k)$, we have: 
    \begin{equation}
      \begin{aligned}
&\mathbb{E}_k[\Phi\left(\bar{x}^{k+1}\right)]-\Phi(\bar{x}^k) \\
      \leq&\frac{\alpha}{2}\|\nabla \Phi(\bar{x}^k)-\mathbb{E}_k[\bar{u}^k]\|^2+\frac{\alpha}{2}\|\mathbb{E}_k[\bar{r}^{k+1}]\|^2 \\
      &+\frac{L_{\nabla \Phi}}{2}\alpha^2\mathbb{E}_k\|\bar{r}^{k+1}\|^2-\alpha\left\langle\mathbb{E}_k[\bar{u}^k],\mathbb{E}_k[\bar{r}^{k+1}]\right\rangle\\
      =&\frac{\alpha}{2}\|\nabla \Phi(\bar{x}^k)-\mathbb{E}_k[\bar{u}^k]\|^2+(-\frac{\alpha}{2}+\frac{L_{\nabla \Phi}}{2}\alpha^2)\mathbb{E}_k\|\bar{r}^{k+1}\|^2-\frac{\alpha(1-\theta)}{2\theta}\|\mathbb{E}_k[\bar{r}^{k+1}]-\bar{r}^k\|^2\\
      &+\frac{\alpha(1-\theta)}{2\theta}\left(\|\bar{r}^k\|^2 -\mathbb{E}_k\|\bar{r}^{k+1}\|^2\right)+\frac{\alpha}{2\theta}\mathbb{E}_k\|\bar{r}^{k+1}-\mathbb{E}_k[\bar{r}^{k+1}]\|^2\\
      \leq&\frac{\alpha}{2}\|\nabla \Phi(\bar{x}^k)-\mathbb{E}_k[\bar{u}^k]\|^2+(-\frac{\alpha}{2}+\frac{L_{\nabla \Phi}}{2}\alpha^2)\mathbb{E}_k\|\bar{r}^{k+1}\|^2+\frac{\alpha\theta}{2}\mathbb{E}_k\|\bar{u}^k-\mathbb{E}_k[\bar{u}^k]\|^2\\
      &+\frac{\alpha(1-\theta)}{2\theta}\left(\|\bar{r}^k\|^2 -\mathbb{E}_k\|\bar{r}^{k+1}\|^2\right),
  \end{aligned}
\end{equation}
where the first equality uses $2\left\langle\bar{r}^k,\mathbb{E}_k[\bar{r}^{k+1}]\right\rangle=\|\bar{r}^k\|^2+\|\mathbb{E}_k[\bar{r}^{k+1}]\|^2-\|\bar{r}^k-\mathbb{E}_k[\bar{r}^{k+1}]\|^2$ and $\mathbb{E}_k\|\bar{r}^{k+1}\|^2=\|\mathbb{E}_k[\bar{r}^{k+1}]\|^2+\mathbb{E}_k\|\bar{r}^{k+1}-\mathbb{E}_k[\bar{r}^{k+1}]\|^2$.

Taking expectation and summation, and using $\alpha\le\frac{1}{2L_{\nabla\Psi}}$,  we get
\begin{equation}\label{phi1}
  \begin{aligned}
    &\inf\Phi-\Phi(\bar{x}_0)\\
    \leq&\frac{\alpha}{2}\sum_{k=0}^K\mathbb{E}\|\nabla \Phi(\bar{x}^k)-\mathbb{E}_k[\bar{u}^k]\|^2-\frac{\alpha}{4}\sum_{k=0}^K\mathbb{E}\|\bar{r}^{k+1}\|^2 
    +\frac{\alpha\theta}{2}\sum_{k=0}^K\mathbb{E}\left[\mathbb{E}_k\|\bar{u}^k-\mathbb{E}_k[\bar{u}^k]\|^2\right].
  \end{aligned}
\end{equation}  
Since samples of different agents are independent, we have \[\mathbb{E}_k\|\bar{u}^k-\mathbb{E}_k[\bar{u}^k]\|^2=\frac{1}{n^2}\mathbb{E}_k\|\mathbf{u}^k-\mathbb{E}_k[\mathbf{u}^k]\|^2.\]

Combining it with the conclusion of Lemma~\ref{8} and~\ref{uvar}, we get from \eqref{phi1} that 
\begin{equation}
  \begin{aligned}
    &\frac{\alpha}{4}\sum_{k=0}^K\mathbb{E}\left\|\bar{r}^{k+1}\right\|^2\\
    \leq&\Phi(\bar{x}_0)-\inf\Phi+\frac{\alpha}{2}\sum_{k=0}^K\mathbb{E}\|\nabla \Phi(\bar{x}^k)-\mathbb{E}_k[\bar{u}^k]\|^2+\frac{\alpha\theta}{2}\sum_{k=0}^K\mathbb{E}\left[\mathbb{E}_k\|\bar{u}^k-\mathbb{E}_k[\bar{u}^k]\|^2\right]\\
    \le&\Phi(\bar{x}_0)-\inf\Phi+10\alpha\left(L^2+\frac{\theta\sigma_{g,2}^2}{n}\right)\sum_{k=0}^K\mathbb{E}\left(\frac{\Delta_k}{n}+I_k\right)+\frac{3\alpha\theta}{n}(K+1)\left(\sigma_{f,1}^2+2\sigma_{g,2}^2\frac{L_{f,0}^2}{\mu_g^2}\right).
  \end{aligned}
\end{equation} 

\end{proof}
\end{lemma}

\begin{lemma} \label{rphi}
Suppose that Assumptions~\ref{smooth}-~\ref{var} hold, then we have
\begin{equation}
  \begin{aligned}
 &\sum_{k=0}^K\mathbb{E}\left[\left\|\mathbb{E}_k[\mathbf{r}^{k+1}]-\widetilde{\nabla}\mathbf{\Phi}(\bar{\mathbf{x}}^k)\right\|^2\right]\\
 \le& \frac{1-\theta}{\theta}\left\|\widetilde{\nabla}\mathbf{\Phi}(\bar{\mathbf{x}}^{0})\right\|^2
+2\sum_{k=0}^K\mathbb{E}\left[\left\|\mathbb{E}_k[\mathbf{u}^k]-\widetilde{\nabla}\mathbf{\Phi}(\bar{\mathbf{x}}^k)\right\|^2\right] \\
&+\frac{2\widetilde{L}^2(1-\theta)^2}{\theta^2}\sum_{k=0}^{K-1}\mathbb{E}\left[\left\|\bar{\mathbf{x}}^{k+1}-\bar{\mathbf{x}}^k\right\|^2\right]+(1-\theta)\theta\sum_{k=0}^K\mathbb{E}\left[\left\|\mathbf{u}^k-\mathbb{E}_k[\mathbf{u}^k]\right\|^2\right].
\end{aligned}
\end{equation}
    \begin{proof}
    We define $\mathbf{u}^{-1}=\mathbf{0}$ for brevity. From the definition of $\mathbb{E}_k$, we have :
\begin{equation}
\label{est_rk_1}
  \begin{aligned}        
   & \mathbb{E}_{k-1}\left[\left\|\mathbf{r}^k-\widetilde{\nabla}\mathbf{\Phi}(\bar{\mathbf{x}}^{k-1})\right\|^2\right]
    \\=&\mathbb{E}_{k-1}\left[\left\|\mathbb{E}_{k-1}[\mathbf{r}^k]-\widetilde{\nabla}\mathbf{\Phi}(\bar{\mathbf{x}}^{k-1})\right\|^2\right]+\mathbb{E}_{k-1}\left[\left\|\mathbf{r}^k-\mathbb{E}_{k-1}[\mathbf{r}^k]\right\|^2\right]  \\
=&\mathbb{E}_{k-1}\left[\left\|\mathbb{E}_{k-1}[\mathbf{r}^k]-\widetilde{\nabla}\mathbf{\Phi}(\bar{\mathbf{x}}^{k-1})\right\|^2\right]+\theta^2\mathbb{E}_{k-1}\left[\left\|\mathbf{u}^{k-1}-\mathbb{E}_{k-1}[\mathbf{u}^{k-1}]\right\|^2\right].
    \end{aligned}
\end{equation}
Jensen's inequality implies that
\begin{equation}
\label{est_rk_2}
  \begin{aligned}
&\mathbb{E}_k
\left[
\|\mathbb{E}_k[\mathbf{r}^{k+1}]-\widetilde{\nabla}\mathbf{\Phi}(\bar{\mathbf{x}}^k)\|^2\right]\\
\le& (1-\theta)\mathbb{E}_k\left[\left\|\mathbf{r}^k-\widetilde{\nabla}\mathbf{\Phi}(\bar{\mathbf{x}}^{k-1})\right\|^2\right]\\
&+\theta\mathbb{E}_k\left[\left\|\left( \mathbb{E}_k[\mathbf{u}^k]-\widetilde{\nabla}\mathbf{\Phi}(\bar{\mathbf{x}}^k)\right) +\theta^{-1}(1-\theta)\left(\widetilde{\nabla}\mathbf{\Phi}(\bar{\mathbf{x}}^{k-1})-\widetilde{\nabla}\mathbf{\Phi}(\bar{\mathbf{x}}^k)\right)\right\|^2\right]\\
\le&(1-\theta)\mathbb{E}_k\left[\left\|\mathbf{r}^k-\widetilde{\nabla}\mathbf{\Phi}(\bar{\mathbf{x}}^{k-1})\right\|^2\right]+2\theta\mathbb{E}_k\left[\left\|\mathbb{E}_k[\mathbf{u}^k]-\widetilde{\nabla}\mathbf{\Phi}(\bar{\mathbf{x}}^k)\right\|^2\right]\\ 
&+\dfrac{2(1-\theta)^2}{\theta}\mathbb{E}_k\left[\left\|\widetilde{\nabla}\mathbf{\Phi}(\bar{\mathbf{x}}^{k-1})-\widetilde{\nabla}\mathbf{\Phi}(\bar{\mathbf{x}}^k)\right\|^2\right].
\end{aligned}
\end{equation}

Substituting \eqref{est_rk_1} into \eqref{est_rk_2} , and  taking expectation and summation on both sides, we get:
\begin{equation}
  \begin{aligned}
&\theta \sum_{k=0}^K\mathbb{E}\left[\left\|\mathbb{E}_{k-1}[\mathbf{r}^k]-\widetilde{\nabla}\mathbf{\Phi}(\bar{\mathbf{x}}^{k-1})\right\|^2\right] \\\le& \mathbb{E}\left[\left\|\mathbf{r}^0-\widetilde{\nabla}\mathbf{\Phi}(\bar{\mathbf{x}}^{-1})\right\|^2\right]-\mathbb{E}\left[\left\|\mathbb{E}_K[\mathbf{r}^{K+1}]-\widetilde{\nabla}\mathbf{\Phi}(\bar{\mathbf{x}}^k)\right\|^2\right]
+2\theta\sum_{k=0}^K\mathbb{E}\left[\left\|\mathbb{E}_k[\mathbf{u}^k]-\widetilde{\nabla}\mathbf{\Phi}(\bar{\mathbf{x}}^k)\right\|^2\right] \\&+\dfrac{2(1-\theta)^2}{\theta}\sum_{k=0}^K\mathbb{E}\left[\left\|\widetilde{\nabla}\mathbf{\Phi}(\bar{\mathbf{x}}^{k-1})-\widetilde{\nabla}\mathbf{\Phi}(\bar{\mathbf{x}}^k)\right\|^2\right]+(1-\theta)\theta^2\sum_{k=0}^K\mathbb{E}\left[\left\|\mathbf{u}^k-\mathbb{E}_k[\mathbf{u}^k]\right\|^2\right].
\end{aligned}
\end{equation}
Finally, note that $\mathbf{x}^{-1}=\mathbf{x}^{0}$, $\mathbf{r}^{0}=\mathbf{0}$, and $\mathbb{E}_{-1}=\mathbb{E}_0$. Subtracting $\theta\mathbb{E}\left[\left\|\mathbb{E}_{-1}[\mathbf{r}^{0}]-\widetilde{\nabla}\mathbf{\Phi}(\bar{\mathbf{x}}^{-1})\right\|^2\right] =\theta\left\|\widetilde{\nabla}\mathbf{\Phi}(\bar{\mathbf{x}}^{0})\right\|^2$ from both sides of this equation, we get: 
\begin{equation}
  \begin{aligned}
&\sum_{k=0}^K\mathbb{E}\left[\left\|\mathbb{E}_k[\mathbf{r}^{k+1}]-\widetilde{\nabla}\mathbf{\Phi}(\bar{\mathbf{x}}^k)\right\|^2\right] 
\\\le &\frac{1-\theta}{\theta}\left\|\widetilde{\nabla}\mathbf{\Phi}(\bar{\mathbf{x}}^{0})\right\|^2
+2\sum_{k=0}^K\mathbb{E}\left[\left\|\mathbb{E}_k[\mathbf{u}^k]-\widetilde{\nabla}\mathbf{\Phi}(\bar{\mathbf{x}}^k)\right\|^2\right]\\
&+\frac{2\widetilde{L}^2(1-\theta)^2}{\theta^2}\sum_{k=0}^{K-1}\mathbb{E}\left[\left\|\bar{\mathbf{x}}^{k+1}-\bar{\mathbf{x}}^k\right\|^2\right]+(1-\theta)\theta\sum_{k=0}^K\mathbb{E}\left[\left\|\mathbf{u}^k-\mathbb{E}_k[\mathbf{u}^k]\right\|^2\right].
\end{aligned}
\end{equation}
\end{proof}
\end{lemma}
\begin{lemma}[Descent lemma]\label{new r}
  Suppose that  Assumptions~\ref{smooth}-~\ref{var} and Lemmas~\ref{8},~\ref{uvar} hold. If 
\begin{equation}\label{l15}
    \frac{\alpha^2}{\theta^2}(1-\theta)\le \frac{1}{32L_{\nabla\Phi}^2},\, \alpha\le \frac{1}{10 L_{\nabla\Phi}},
\end{equation}

then we have
  \begin{equation}
  \begin{aligned}
\sum_{k=0}^K\mathbb{E}\|\nabla\Phi(\bar{x}^k)\|^2
\lesssim& \frac{\Phi(\bar{x}_0)-\inf \Phi}{\alpha}
 +
\left(L^2 +\left(\theta(1-\theta)+L_{\nabla\Phi}\alpha\theta^2\right)\sigma_{g,2}^2\right)\sum_{k=0}^K\mathbb{E}\left[\frac{\Delta_k}{n}+I_k\right]
 \\&+(K+1)\left(\theta(1-\theta)+L_{\nabla\Phi}\alpha\theta^2\right) (\sigma_{f,1}^2+\kappa^2\sigma_{g,2}^2)
+\frac{(1-\theta)^2}{\theta}\|\nabla \Phi(\bar{x}^{0})\|^2.
  \end{aligned}
\end{equation}
  \begin{proof}
    The $L_{\nabla \Phi}$-smoothness of $\Phi$ indicates that
    \begin{equation}
      \begin{aligned}
    &\mathbb{E}_k[\Phi(\bar{x}^{k+1})]-\Phi(\bar{x}^k) \\
    \leq&\left \langle\nabla \Phi(\bar{x}^k) ,-\alpha\mathbb{E}_k[\bar{r}^{k+1}]\right \rangle+\frac{L_{\nabla \Phi} \alpha^2}{2}\mathbb{E}_k\left\|\bar{r}^{k+1}\right\|^2 
    \\=&-\alpha \left \langle\nabla \Phi(\bar{x}^k) , \mathbb{E}_k[\bar{r}^{k+1}]-\nabla \Phi(\bar{x}^k)\right \rangle-\alpha \|\nabla \Phi(\bar{x}^k)\|^2+\frac{L_{\nabla \Phi}}{2}\alpha^2\mathbb{E}_k\left\|\bar{r}^{k+1}\right\|^2 
    \\\leq&-\frac{\alpha}{2}\|\nabla \Phi(\bar{x}^k)\|^2 +\frac{\alpha}{2}\|\mathbb{E}_k[\bar{r}^{k+1}]-\nabla \Phi(\bar{x}^k)\|^2+\frac{L_{\nabla \Phi}}{2}\alpha^2\mathbb{E}_k\left\|\bar{r}^{k+1}\right\|^2. 
  \end{aligned}
\end{equation}
Taking expectation and summation on both sides, we get:
\begin{equation}\label{sumpsi}
  \begin{aligned}
&\sum_{k=0}^K\alpha\mathbb{E}\|\nabla\Phi(\bar{x}^k)\|^2\\
\leq&2(\Phi(\bar{x}^0)-\inf \Phi)+\sum_{k=0}^K{\alpha}\mathbb{E}\|\mathbb{E}_k[\bar{r}^{k+1}]-\nabla \Phi(\bar{x}^k)\|^2+\sum_{k=0}^K{L_{\nabla \Phi}}\alpha^2\mathbb{E}\left\|\bar{r}^{k+1}\right\|^2. 
  \end{aligned}
\end{equation}

Define auxiliary series $m^k$ as:
\begin{equation}
  \begin{aligned}
    m^0=\bar{r}^{0}=0, m^{k+1}=(1-\theta)m^k+\theta\nabla \Phi(\bar{x}^k).
  \end{aligned}
\end{equation}

Note that
\begin{equation}\label{bdr}
  \begin{aligned}
    \mathbb{E}_k\left\|\bar{r}^{k+1}\right\|^2&=\left\|\mathbb{E}_k\bar{r}^{k+1}\right\|^2+\mathbb{E}_k\left\|\bar{r}^{k+1}-\mathbb{E}_k\bar{r}^{k+1}\right\|^2
    \\&\leq2\|\mathbb{E}_k[\bar{r}^{k+1}]-\nabla \Phi(\bar{x}^k)\|^2+2\|\nabla \Phi(\bar{x}^k)\|^2+\theta^2\mathbb{E}_k\|\bar{u}^k-\mathbb{E}_k\bar{u}^k \|^2.
  \end{aligned}
\end{equation}

Then using the Jenson's Inequality, we get:
\begin{equation}
  \begin{aligned}
    \|\mathbb{E}_k\bar{r}^{k+1}-m^{k+1} \|^2&=\|(1-\theta)(\bar{r}^k-m^k)+\theta(\mathbb{E}_k\bar{u}^k-\nabla \Phi(\bar{x}^k))\|^2\\&\leq (1-\theta)\|\bar{r}^k-m^k\|^2+\theta\|\mathbb{E}_k\bar{u}^k-\nabla \Phi(\bar{x}^k)\|^2.
  \end{aligned}
\end{equation}  
It follows that for $k\ge0$
\begin{equation}
  \begin{aligned}
    &\mathbb{E}\|\mathbb{E}_k\bar{r}^{k+1}-m^{k+1} \|^2
    \\\leq& (1-\theta)\mathbb{E}\|\mathbb{E}_{k-1}[\bar{r}^k]-m^k\|^2+(1-\theta)\theta^2\mathbb{E}\|\bar{u}^{k-1}-\mathbb{E}_{k-1}\bar{u}^{k-1} \|^2+\theta\mathbb{E}\|\mathbb{E}_k\bar{u}^k-\nabla \Phi(\bar{x}^k)\|^2,
  \end{aligned}
\end{equation}
where for brevity we define $\bar{u}^{-1}=0$. 

Taking the summation on both sides from $k=0$ to $K$, we get
\begin{equation}\label{rmk}
  \begin{aligned}
\sum_{k=0}^K\theta\mathbb{E}\|\mathbb{E}_k\bar{r}^{k+1}-m^{k+1} \|^2&\leq \sum_{k=0}^{K-1}\theta\mathbb{E}\|\mathbb{E}_k\bar{r}^{k+1}-m^{k+1} \|^2+\mathbb{E}\|\mathbb{E}_K\bar{r}^{K+1}-m^{K+1} \|^2
\\&\leq\sum_{k=0}^K\theta\mathbb{E}\|\mathbb{E}_k\bar{u}^k-\nabla \Phi(\bar{x}^k)\|^2+\sum_{k=0}^{K-1}(1-\theta)\theta^2\mathbb{E}\|\bar{u}^k-\mathbb{E}_k\bar{u}^k \|^2.
  \end{aligned}
\end{equation}   
On the other hand, due to the definition of $m^k$ and Jenson's Inequality, we have:
\begin{equation}
  \begin{aligned}
    \|m^{k+1}-\nabla \Phi(\bar{x}^k)\|^2&=\|(1-\theta)(m^k-\nabla \Phi(\bar{x}^k))\|^2\\
    &=(1-\theta)^2\|m^k-\nabla \Phi(\bar{x}^{k-1})+\nabla \Phi(\bar{x}^{k-1})-\nabla \Phi(\bar{x}^k)\|^2\\
    &\leq (1-\theta)\|m^k-\nabla \Phi(\bar{x}^{k-1})\|^2+\frac{(1-\theta)^2}{\theta}L_{\nabla \Phi}^2\alpha^2\|\bar{r}^k\|^2.
  \end{aligned}
\end{equation} 
Taking the summation, we get
\begin{equation}\label{rpsi}
  \begin{aligned}
    \sum_{k=0}^K\theta\|m^{k+1}-\nabla \Phi(\bar{x}^k)\|^2&\leq \|m^0-\nabla \Phi(\bar{x}^{-1})\|^2+\sum_{k=0}^K\frac{(1-\theta)^2}{\theta}L_{\nabla \Phi}^2\alpha^2\|\bar{r}^k\|^2
    \\&=(1-\theta)^2\|\nabla \Phi(\bar{x}^{0})\|^2+\sum_{k=0}^K\frac{(1-\theta)^2}{\theta}L_{\nabla \Phi}^2\alpha^2\|\bar{r}^k\|^2.
  \end{aligned}
\end{equation} 
Combining  \eqref{rmk} and \eqref{rpsi}, we obtain:
\begin{small}
\begin{equation}\label{3940}
  \begin{aligned}
&\sum_{k=0}^K\theta\mathbb{E}\|\mathbb{E}_k\bar{r}^{k+1}-\nabla \Phi(\bar{x}^k)\|^2\\
\leq& 2\sum_{k=0}^K\theta\mathbb{E}\|\mathbb{E}_k\bar{r}^{k+1}-m^{k+1} \|^2+2\sum_{k=0}^K\theta\|m^{k+1}-\nabla \Phi(\bar{x}^k)\|^2
\\\leq&2\sum_{k=0}^K\theta\mathbb{E}\|\mathbb{E}_k\bar{u}^k-\nabla \Phi(\bar{x}^k)\|^2+2\sum_{k=0}^{K-1}(1-\theta)\theta^2\mathbb{E}\|\bar{u}^k-\mathbb{E}_k\bar{u}^k \|^2\\&+2\sum_{k=0}^K\frac{(1-\theta)^2}{\theta}L_{\nabla \Phi}^2\alpha^2\mathbb{E}\|\bar{r}^k\|^2
+2(1-\theta)^2\|\nabla \Phi\left(\bar{x}^{0}\right)\|^2
\\\leq& 2\sum_{k=0}^K\theta\mathbb{E}\|\mathbb{E}_k\bar{u}^k-\nabla \Phi(\bar{x}^k)\|^2+2\sum_{k=0}^{K-1}\left(1+2\frac{1-\theta}{\theta}L_{\nabla \Phi}^2\alpha^2\right)(1-\theta)\theta^2\mathbb{E}\|\bar{u}^k-\mathbb{E}_k\bar{u}^k \|^2
\\&+2(1-\theta)^2\|\nabla \Phi\left(\bar{x}^{0}\right)\|^2
+2\sum_{k=0}^{K-1}\frac{(1-\theta)^2}{\theta}L_{\nabla \Phi}^2\alpha^2
\left(2\mathbb{E}\|\mathbb{E}_k[\bar{r}^{k+1}]-\nabla \Phi(\bar{x}^k)\|^2+2\mathbb{E}\|\nabla \Phi(\bar{x}^k)\|^2\right),
  \end{aligned}
\end{equation}
\end{small}
where the last inequality uses \eqref{bdr}.

 \eqref{l15} indicates that $4\dfrac{1-\theta}{\theta}L_{\nabla \Phi}^2\alpha^2\le \dfrac{\theta}{8}$. Subtracting 
 \[2\sum_{k=0}^{K-1}\frac{(1-\theta)^2}{\theta}L_{\nabla \Phi}^2\alpha^2
\cdot2\mathbb{E}\|\mathbb{E}_k[\bar{r}^{k+1}]-\nabla \Phi(\bar{x}^k)\|^2\] from both sides of \eqref{3940}, we have: 
\begin{equation}\label{rpsiu}
  \begin{aligned}
&\sum_{k=0}^K\theta\mathbb{E}\|\mathbb{E}_k\bar{r}^{k+1}-\nabla \Phi(\bar{x}^k)\|^2
\\\leq& 4\sum_{k=0}^K\theta\mathbb{E}\|\mathbb{E}_k\bar{u}^k-\nabla \Phi(\bar{x}^k)\|^2+8\sum_{k=0}^{K-1}(1-\theta)\theta^2\mathbb{E}\|\bar{u}^k-\mathbb{E}_k\bar{u}^k \|^2 +4(1-\theta)^2\|\nabla \Phi(\bar{x}^{0})\|^2
\\&+\frac{\theta}{4}\sum_{k=0}^{K-1}\mathbb{E}\|\nabla \Phi(\bar{x}^k)\|^2.
  \end{aligned}
\end{equation}
Substituting \eqref{bdr}, \eqref{rpsiu} into \eqref{sumpsi}, we get:
\begin{equation}
\label{descent_lem_1}
  \begin{aligned}
&\sum_{k=0}^K\alpha\mathbb{E}\|\nabla\Phi(\bar{x}^k)\|^2\\\leq &2(\Phi(\bar{x}^0)-\inf \Phi)+\sum_{k=0}^K(\alpha+2{L_{\nabla \Phi}}\alpha^2)\mathbb{E}\|\mathbb{E}_k[\bar{r}^{k+1}]-\nabla \Phi(\bar{x}^k)\|^2+\sum_{k=0}^K2{L_{\nabla \Phi}}\alpha^2\mathbb{E}\|\nabla\Phi(\bar{x}^k)\|^2\\
&+\sum_{k=0}^K2{L_{\nabla \Phi}}\alpha^2\theta^2\mathbb{E}_k\|\bar{u}^k-\mathbb{E}_k\bar{u}^k \|^2\\
\leq& 2(\Phi(\bar{x}_0)-\inf \Phi)+ 5\alpha\sum_{k=0}^K\mathbb{E}\|\mathbb{E}_k\bar{u}^k-\nabla \Phi(\bar{x}^k)\|^2+5\frac{\alpha}{\theta}(1-\theta)^2\|\nabla \Phi\left(\bar{x}^{0}\right)\|^2\\
&+\sum_{k=0}^K\left(10\frac{\alpha}{\theta}(1-\theta)+2L_{\nabla\Phi}\alpha^2\right)\theta^2\mathbb{E}\|\bar{u}^k-\mathbb{E}_k\bar{u}^k \|^2+\frac{\alpha}{2}\sum_{k=0}^{K-1}\mathbb{E}\|\nabla \Phi(\bar{x}^k)\|^2,
  \end{aligned}
\end{equation}
where the last inequality uses $\alpha\le \frac{1}{10L_{\nabla\Phi}}$. 

Subtracting $\frac{\alpha}{2}\sum_{k=0}^{K-1}\mathbb{E}\|\nabla \Phi(\bar{x}^k)\|^2$ from both sides of \eqref{descent_lem_1}, and substituting \eqref{u,psi,A,B}, \eqref{conclusion_lemma4} into it,  we get:
  \begin{align}
&\frac{1}{2}\sum_{k=0}^K\alpha\mathbb{E}\|\nabla\Phi(\bar{x}^k)\|^2
\\
\leq&2(\Phi(\bar{x}^0)-\inf \Phi)
 + 100\alpha L^2\sum_{k=0}^K\mathbb{E}\left[\frac{\Delta_k}{n}+I_k\right]+5\frac{\alpha}{\theta}(1-\theta)^2\|\nabla \Phi\left(\bar{x}^{0}\right)\|^2
 \\&+\frac{\left(10\frac{\alpha}{\theta}(1-\theta)+2L_{\nabla\Phi}\alpha^2\right)}{n^2}\theta^2\cdot9\sigma_{g,2}^2\sum_{k=0}^K\left(\mathbb{E}\|\mathbf{z}^{k+1}-\bar{\mathbf{z}}^{k+1}\|^2+\mathbb{E}\|\bar{\mathbf{z}}^{k+1}-\mathbf{z}_{\star}^{k+1}\|^2\right)
 \\&+\frac{\left(10\frac{\alpha}{\theta}(1-\theta)+2L_{\nabla\Phi}\alpha^2\right)}{n^2}\theta^2\cdot3(K+1)n\left(\sigma_{f,1}^2+3\sigma_{g,2}^2\frac{L_{f,0}^2}{\mu_g^2}\right)
\\\le&2(\Phi(\bar{x}_0)-\inf \Phi)
 + \left(100\alpha L^2+9\left(10\alpha\theta(1-\theta)+2L_{\nabla\Phi}\alpha^2\theta^2\right)\frac{\sigma_{g,2}^2}{n}\right)\sum_{k=0}^K\mathbb{E}\left[\frac{\Delta_k}{n}+I_k\right]
 \\&+3(K+1)\left(10\alpha(1-\theta)+2L_{\nabla\Phi}\alpha^2\theta\right)\frac{\theta}{n}\left(\sigma_{f,1}^2+3\sigma_{g,2}^2\frac{L_{f,0}^2}{\mu_g^2}\right)\\
 &+5\frac{\alpha}{\theta}(1-\theta)^2\|\nabla \Phi\left(\bar{x}^{0}\right)\|^2.
  \end{align}

Finally, multiplying $\dfrac{2}{\alpha}$ on both sides, we get:
  \begin{equation}
  \begin{aligned}
\sum_{k=0}^K\mathbb{E}\|\nabla\Phi(\bar{x}^k)\|^2
\lesssim& \frac{\Phi(\bar{x}_0)-\inf \Phi}{\alpha}
 +
\left(L^2 +\left(\theta(1-\theta)+L_{\nabla\Phi}\alpha\theta^2\right)\frac{\sigma_{g,2}^2}{n}\right)\sum_{k=0}^K\mathbb{E}\left[\frac{\Delta_k}{n}+I_k\right]
 \\&+\frac{K+1}{n}\left(\theta(1-\theta)+L_{\nabla\Phi}\alpha\theta^2\right) (\sigma_{f,1}^2+\kappa^2\sigma_{g,2}^2)
+\frac{(1-\theta)^2}{\theta}\|\nabla \Phi(\bar{x}^{0})\|^2.
  \end{aligned}
\end{equation}

\end{proof}
\end{lemma}

\subsubsection{Descent lemmas for the lower- and auxiliary-level}
The following lemmas present the error analysis of the estimation of $y^\star(\bar{x}^k)$ and $z_\star^k$, i.e., the term $I_k$:
\begin{lemma}[Estimation error of $y^\star(x)$]
\label{ystar}
Suppose Assumptions~\ref{smooth}-~\ref{var}hold, and: 
\begin{equation}\label{betady}
\beta \leq \frac{\mu_g}{32L_{g,1}^2}.  
\end{equation}

Then we have the estimation error of $y^\star$:
\begin{equation}\label{ystar1}
  \begin{aligned}
&\|\bar{y}^{0}-y^{\star}(\bar{x}^0)\|^2+\sum_{k=0}^K\mathbb{E}[\|\bar{y}^{k+1}-y^{\star}(\bar{x}^k)\|^2]
\\\leq& \frac{4}{\beta\mu_g}\|\bar{y}^{0}-y^{\star}(\bar{x}^{0})\|^2+\sum_{k=1}^K\frac{6\alpha^2L_{y^{\star}}^2}{\beta^2\mu_g^2}\mathbb{E}\|\bar{r}^k\|^2
+\sum_{k=1}^K\frac{6}{\mu_g^2}L_{g,1}^2\mathbb{E}\left[\frac{\|
\mathbf{O}_x\|^2\|\hat{\mathbf{e}}_x^k\|^2+\|
\mathbf{O}_y\|^2\|\hat{\mathbf{e}}_y^k\|^2}{n}\right]
\\&+\frac{4K\beta\sigma_{g,1}^2}{n\mu_g},
  \end{aligned}
\end{equation}
and\begin{equation}\label{ly2}
\begin{aligned}
&\sum_{k=0}^K\mathbb{E}\left[\|\bar{y}^{k+1}-\bar{y}^k\|^2\right]\\
\le&\frac{\beta^2L_{g,1}^2}{n}\left(4+\frac{48L_{g,1}^2}{\mu_g^2}\right)\sum_{k=1}^K\mathbb{E}\left(\|
\mathbf{O}_x\|^2\|\hat{\mathbf{e}}^k_x\|^2+\|
\mathbf{O}_y\|^2\|\hat{\mathbf{e}}^k_y\|^2\right)+\frac{48\alpha^2L_{g,1}^2}{\mu_g^2}L_{y^{\star}}^2\sum_{k=1}^K\mathbb{E}\|\bar{r}^k\|^2
\\&+\frac{3(K+1)\beta^2}{n}\sigma_{g,1}^2
+\frac{32\beta L_{g,1}^2}{\mu_g}\|\bar{y}^0-y^{\star}(\bar{x}^{0})\|^2.
\end{aligned}
\end{equation}
\begin{proof}
For each $k\ge0$, due to the independence of samples, we have:
 \begin{equation}
  \begin{aligned}
  &\widehat{\mathbb{E}}_k[\|\bar{y}^{k+1}-y^\star(\bar{x}^k)\|^2]
  =\widehat{\mathbb{E}}_k[\|\bar{y}^k-\beta\bar{v}^k-y^\star(\bar{x}^k)\|^2]\\
  =&\widehat{\mathbb{E}}_k\left[\left\Vert\bar{y}^k-\beta\frac{1}{n}\sum_{i=1}^n\nabla_2g_i(x_{i}^k,y_{i}^k)-y^\star(\bar{x}^k)+\beta\frac{1}{n}\sum_{i=1}^n\left(\nabla_2g_i(x_{i}^k,y_{i}^k)-v_i^k\right)\right\Vert^2\right]\\
  \leq &\left\|\bar{y}^k-\beta\frac{1}{n}\sum_{i=1}^n\nabla_2g_i(\bar{x}^k,\bar{y}^k)-y^\star(\bar{x}^k)+\beta\frac{1}{n}\sum_{i=1}^n(\nabla_2g_i(\bar{x}^k,\bar{y}^k)-\nabla_2g_i(x_{i}^k,y_{i}^k))\right\|^2+\beta^2\frac{\sigma_{g,1}^2}{n}.
  \end{aligned}
 \end{equation} 

Then, 
 \begin{equation}
  \begin{aligned}
  &\widehat{\mathbb{E}}_k[\|\bar{y}^{k+1}-y^\star(\bar{x}^k)\|^2]\\
  \leq &\left(1+\frac{\beta\mu_g}{2}\right)\left\|\bar{y}^k-\beta\frac{1}{n}\sum_{i=1}^n\nabla_2g_i(\bar{x}^k,\bar{y}^k)-y^\star(\bar{x}^k)\right\|^2\\
  &+\beta^2\left(1+\frac{2}{\beta\mu_g}\right)\left\|\frac{1}{n}\sum_{i=1}^n(\nabla_2g_i(\bar{x}^k,\bar{y}^k)-\nabla_2g_i(x_{i}^k,y_{i}^k))\right\|^2+\beta^2\frac{\sigma_{g,1}^2}{n}\\
  \leq &\left(1+\frac{\beta\mu_g}{2}\right)(1-\beta\mu_g)^2\|\bar{y}^k-y^\star(\bar{x}^k)\|^2\\
  &+\beta^2\left(1+\frac{2}{\beta\mu_g}\right)L_{g,1}^2\left(\frac{\|\mathbf{x}^k-\bar{\mathbf{x}}^k\|^2}{n}+\frac{\|\mathbf{y}^k-\bar{\mathbf{y}}^k\|^2}{n}\right)+\beta^2\frac{\sigma_{g,1}^2}{n}\\
  \leq &\left(1-{\beta\mu_g}\right)\left[\left(1+\frac{\beta\mu_g}{2}\right)\|\bar{y}^k-y^\star(\bar{x}^{k-1})\|^2+\left(1+\frac{2}{\beta\mu_g}\right)\|y^\star(\bar{x}^k)-y^\star(\bar{x}^{k-1})\|^2\right]\\
  &+\beta^2\left(1+\frac{2}{\beta\mu_g}\right)L_{g,1}^2\left(\frac{\|\mathbf{x}^k-\bar{\mathbf{x}}^k\|^2}{n}+\frac{\|\mathbf{y}^k-\bar{\mathbf{y}}^k\|^2}{n}\right)+\beta^2\frac{\sigma_{g,1}^2}{n}\\
  \leq &\left(1-\frac{\beta\mu_g}{2}\right)\|\bar{y}^k-y^\star(\bar{x}^{k-1})\|^2+\frac{3}{\beta\mu_g}L_{y^\star}^2\|\bar{x}^k-\bar{x}^{k-1}\|^2\\
  &+\frac{3\beta}{\mu_g}L_{g,1}^2\left(\frac{\|\mathbf{x}^k-\bar{\mathbf{x}}^k\|^2}{n}+\frac{\|\mathbf{y}^k-\bar{\mathbf{y}}^k\|^2}{n}\right)+\beta^2\frac{\sigma_{g,1}^2}{n},
  \end{aligned}
 \end{equation} 
 where the first and the third inequality is due to the Jenson's inequality, the second inequality holds according to Lemma~\ref{des} and the fact that $\beta \leq \frac{\mu_g}{32L_{g,1}^2}\le\frac{1}{3(\mu_g+L_{g,1})}$, and the last inequality uses $\beta\mu_g\leq\frac{1}{3}$. Taking the summation and expectation on the both sides, we get:
\begin{equation}
  \begin{aligned}
&\sum_{k=0}^K\frac{\beta\mu_g}{2}\mathbb{E}[\|\bar{y}^k-y^\star(\bar{x}^{k-1})\|^2]+\mathbb{E}[\|\bar{y}^{k+1}-y^\star(\bar{x}^k)\|^2]
\\\leq &\mathbb{E}\|\bar{y}^{0}-y^\star(\bar{x}^{0})\|^2+\sum_{k=0}^K\mathbb{E}\left[\frac{3\alpha^2}{\beta\mu_g}L_{y^\star}^2\|\bar{r}^k\|^2
+\frac{3\beta}{\mu_g}L_{g,1}^2\left(\frac{\|\mathbf{x}^k-\bar{\mathbf{x}}^k\|^2}{n}+\frac{\|\mathbf{y}^k-\bar{\mathbf{y}}^k\|^2}{n}\right)+\beta^2\frac{\sigma_{g,1}^2}{n}\right].
  \end{aligned}
\end{equation} 
Using \eqref{transformer_consensus} and the fact that $\mathbf{x}^0,\mathbf{y}^0$ is consensual, it follows that: 
\begin{equation}
\label{est_y_1}
  \begin{aligned}
\sum_{k=0}^{K+1}\frac{\beta\mu_g}{2}\mathbb{E}[\|\bar{y}^k-y^\star(\bar{x}^{k-1})\|^2]
\leq&2\|\bar{y}^{0}-y^\star(\bar{x}^{0})\|^2+\sum_{k=1}^K\frac{3\alpha^2}{\beta\mu_g}L_{y^{\star}}^2\mathbb{E}\|\bar{r}^k\|^2\\
&+\sum_{k=1}^K\frac{3\beta}{\mu_g}L_{g,1}^2\mathbb{E}\left[\frac{\|
\mathbf{O}_x\|^2\|\hat{\mathbf{e}}_x^k\|^2+\|
\mathbf{O}_y\|^2\|\hat{\mathbf{e}}_y^k\|^2}{n}\right]+\frac{2K\beta^2\sigma_{g,1}^2}{n}.
  \end{aligned}
\end{equation}

On the other hand, 
\begin{equation}
  \begin{aligned}
&\widehat{\mathbb{E}}_k\left[\|\bar{y}^{k+1}-\bar{y}^k\|^2\right]
\\\le&\beta^2\left\|\frac{1}{n}\sum_{i=1}^n\nabla_2 g_i(x_{i}^{k},y_{i}^{k})\right\|^2+\frac{\beta^2}{n}\sigma_{g,1}^2
\\
\le&2\beta^2\left(\left\|\frac{1}{n}\sum_{i=1}^n\left(\nabla_2 g_i(x_{i}^k,y_{i}^k)-\nabla_2 g_i(\bar{x}^k,\bar{y}^k)\right)\right\|^2+\left\|\nabla_2g(\bar{x}^k,\bar{y}^k)-\nabla_2g(\bar{x}^k,y^{\star}(\bar{x}^k))\right\|^2\right)\\
&+\frac{\beta^2}{n}\sigma_{g,1}^2\\
\le&\frac{2\beta^2 L_{g,1}^2}{n}\left(\|\mathbf{x}^k-\bar{\mathbf{x}}^k\|^2+\|\mathbf{y}^k-\bar{\mathbf{y}}^k\|^2+2\|\bar{\mathbf{y}}^{k+1}-\mathbf{y}^{\star}(\bar{\mathbf{x}}^k)\|^2+2\|\bar{\mathbf{y}}^{k+1}-\bar{\mathbf{y}}^k\|^2\right)+\frac{\beta^2}{n}\sigma_{g,1}^2,
\end{aligned}
\end{equation}
where the second inequality uses $\nabla_2g(\bar{x}^k,y^{\star}(\bar{x}^k))=0$.

Note that $\beta^2\le  \frac{\mu_g^2}{32L_{g,1}^4}\le\frac{1}{8L_{g,1}^2}$. Subtracting $2\beta^2L_{g,1}^2\|\bar{y}^{k+1}-\bar{y}^k\|$ on both sides, and taking expectation and summation, we get:

\begin{equation}
\begin{aligned}
&\sum_{k=0}^K\mathbb{E}\left[\|\bar{y}^{k+1}-\bar{y}^k\|^2\right]\\
\le&\sum_{k=0}^K\left[\frac{4\beta^2L_{g,1}^2}{n}\mathbb{E}\left(\|\mathbf{x}^k-\bar{\mathbf{x}}^k\|^2+\|\mathbf{y}^k-\bar{\mathbf{y}}^k\|^2\right)+\frac{8\beta^2L_{g,1}^2}{n}\mathbb{E}\|\bar{\mathbf{y}}^{k+1}-\mathbf{y}^{\star}(\bar{\mathbf{x}}^k)\|^2+\frac{2\beta^2}{n}\sigma_{g,1}^2\right]\\
\le&\frac{4\beta^2L_{g,1}^2}{n}\sum_{k=1}^K\mathbb{E}\left(\|
\mathbf{O}_x\|^2\|\hat{\mathbf{e}}_k^x\|^2+\|
\mathbf{O}_y\|^2\|\hat{\mathbf{e}}_k^y\|^2\right)+\frac{8\beta^2L_{g,1}^2}{n}\sum_{k=0}^K\mathbb{E}\|\bar{\mathbf{y}}^{k+1}-\mathbf{y}^{\star}(\bar{\mathbf{x}}^k)\|^2\\
&+\frac{2(K+1)\beta^2}{n}\sigma_{g,1}^2\\
\le&\frac{4\beta^2L_{g,1}^2}{n}\sum_{k=1}^K\mathbb{E}\left(\|
\mathbf{O}_x\|^2\|\hat{\mathbf{e}}^k_x\|^2+\|
\mathbf{O}_y\|^2\|\hat{\mathbf{e}}^k_y\|^2\right)+\frac{2(K+1)\beta^2}{n}\sigma_{g,1}^2\\
&+8\beta^2L_{g,1}^2\left(\frac{4}{\beta\mu_g}\|\bar{y}_{0}-y^{\star}(\bar{x}^{0})\|^2+\sum_{k=1}^K\frac{6\alpha^2}{\beta^2\mu_g^2}L_{y^{\star}}^2\mathbb{E}\|\bar{r}^k\|^2\right)\\
&+8\beta^2L_{g,1}^2\left(
\sum_{k=1}^K\frac{6}{\mu_g^2}L_{g,1}^2\mathbb{E}\left[\frac{\|
\mathbf{O}_x\|^2\|\hat{\mathbf{e}}_x^k\|^2+\|
\mathbf{O}_y\|^2\|\hat{\mathbf{e}}_y^k\|^2}{n}\right]+\frac{4K\beta\sigma_{g,1}^2}{n\mu_g}\right)\\
\le& \frac{\beta^2L_{g,1}^2}{n}\left(4+\frac{48L_{g,1}^2}{\mu_g^2}\right)\sum_{k=1}^K\mathbb{E}\left(\|\mathbf{O}_x\|^2\|\hat{\mathbf{e}}^k_x\|^2+\|
\mathbf{O}_y\|^2\|\hat{\mathbf{e}}^k_y\|^2\right)
+\frac{48\alpha^2L_{g,1}^2}{\mu_g^2}L_{y^{\star}}^2\sum_{k=1}^K\mathbb{E}\|\bar{r}^k\|^2\\
&+\frac{3(K+1)\beta^2}{n}\sigma_{g,1}^2+\frac{32\beta L_{g,1}^2}{\mu_g}\|\bar{y}^0-y^{\star}(\bar{x}^{0})\|^2.
\end{aligned}
\end{equation}
where the second inequality holds since $\mathbf{x}^0,\mathbf{y}^0$ are consensual, the third inequality uses \eqref{est_y_1}, and the last inequality holds since  
$\beta\le\frac{\mu_g}{32L_{g,1}^2}$.
\end{proof}
\end{lemma}

\begin{lemma}[Estimation error of $z^\star(x)$]
\label{zstar}
  Suppose that Assumptions~\ref{smooth}-~\ref{var}hold, and 
  \begin{equation}\label{gammaz}
      \gamma<\min\left\{\frac{1}{\mu_g},\frac{nL_{g,1}^2}{\mu_g^2\sigma_{g,2}^2},\frac{n\mu_g}{36\sigma_{g,2}^2}\right\}.
  \end{equation}
  We have:
   \begin{equation}\label{ly3}
  \begin{aligned}
&\sum_{k=0}^{K+1}\mathbb{E}\|\bar{z}^{k}-z_\star^{k}\|^2\\
\leq& \sum_{k=0}^K\frac{9\alpha^2L_{z^\star}^2}{\gamma^2\mu_g^2}\mathbb{E}\|\bar{r}^k\|^2
\\&+72\kappa^2\sum_{k=1}^K\mathbb{E}\left[\frac{\kappa^2\|\mathbf{O}_x\|^2\|\hat{\mathbf{e}}^k_x\|^2+\|\mathbf{O}_z\|^2\|\hat{\mathbf{e}}^{k}_z\|^2}{n}\right]+72\kappa^2\sum_{k=0}^K\mathbb{E}\left[\frac{\kappa^2\|\mathbf{O}_y\|^2\|\hat{\mathbf{e}}^{k+1}_y\|^2}{n}\right]
\\&+\sum_{k=0}^K72\kappa^4\mathbb{E}\left[\|\bar{y}^{k+1}-y^\star(\bar{x}^k)\|^2\right]
+\frac{3\|z^1_{\star}\|^2}{\mu_g\gamma}+\frac{6(K+1)\gamma}{\mu_gn}\left(3\sigma_{g,2}^2\frac{L_{f,0}^2}{\mu_g^2}+\sigma_{f,1}^2\right).
  \end{aligned}
\end{equation}
\begin{proof}
For each $k\ge0$, note that  $z^k_\star=\nabla_{22}^2g(\bar{x}^k,y^{\star}(\bar{x}^k))^{-1}\nabla_2 f_2(\bar{x}^k,y^{\star}(\bar{x}^k))$, we have:
\begin{align*}
&\widetilde{\mathbb{E}}_k[\bar{z}^{k+1}]-z_\star^{k+1}=\bar{z}^k-\frac{\gamma}{n}\sum_{i=1}^n\left(\nabla_{22}^2 g_i(x_{i}^k,y_{i}^{k+1})z_{i}^k-\nabla_2 f_i(x_{i}^k,y_{i}^{k+1})\right)-z_\star^{k+1}\\
=&\left[I-\frac{\gamma}{n}\sum_{i=1}^n\nabla_{22}^2 g_i(x_i^k,y_i^{k+1})\right](\bar{z}^k-z^{k+1}_\star)
+\dfrac{\gamma}{n}\sum_{i=1}^n\left[\nabla_2 f_i(x_{i}^k,y_{i}^{k+1})-\nabla_2f_i(\bar{x}^k,y^{\star}(\bar{x}^k))\right]
\\&+\frac{\gamma}{n}\sum_{i=1}^n\nabla_{22}^2 g_i(x_i^k,y_i^{k+1})(\bar{z}^k-z^k_i)
+\frac{\gamma}{n}\sum_{i=1}^n\left[\nabla_{22}^2 g_i(\bar{x}^k,y^{\star}(\bar{x}^k))-\nabla_{22}^2 g_i(x_{i}^k,y_{i}^{k+1})\right]z_\star^{k+1}
.
\end{align*}

Then we have:
\begin{equation}
    \begin{aligned}  &\left\|\widetilde{\mathbb{E}}_k[\bar{z}^{k+1}]-z_\star^{k+1}\right\|^2\\
  \leq&(1+\gamma\mu_g)\left\Vert\left[I-\frac{\gamma}{n}\sum_{i=1}^n\nabla_{22}^2 g_i(x_i^k,y_i^{k+1})\right](\bar{z}^k-z^{k+1}_{\star})\right\Vert^2\\
  &+3\gamma^2\left(1+\frac{1}{\gamma\mu_g}\right)\left\Vert\dfrac{1}{n}\sum_{i=1}^n\left[\nabla_2 f_i(x_{i}^k,y_{i}^{k+1})-\nabla_2f_i(\bar{x}^k,y^{\star}(\bar{x}^k))\right]\right\Vert^2\\
  &+3\gamma^2\left(1+\frac{1}{\gamma\mu_g}\right)\left\Vert\dfrac{1}{n}\sum_{i=1}^n\left[\nabla_{22}^2 g_i(\bar{x}^k,y^{\star}(\bar{x}^k))-\nabla_{22}^2 g_i(x_{i}^k,y_{i}^{k+1})\right]z_{\star}^{k+1}\right\Vert^2\\
  &+3\gamma^2\left(1+\frac{1}{\gamma\mu_g}\right)\left\Vert\frac{1}{n}\sum_{i=1}^n\nabla_{22}^2 g_i(x_i^k,y_i^{k+1})(\bar{z}^k-z^k_i)\right\Vert^2\\
  \leq&(1+\gamma\mu_g)(1-\gamma\mu_g)^2\|\bar{z}^k-z_\star^{k+1}\|^2\\
  &+\frac{6\gamma}{\mu_g}\left(L_{g,1}^2\frac{\|\mathbf{z}^k-\bar{\mathbf{z}}^k\|}{n}+\left(\frac{L_{g,2}^2L_{f,0}^2}{\mu_g^2}+L_{f,1}^2\right)\left(\frac{\|\mathbf{x}^k-\bar{\mathbf{x}}^k\|^2}{n}+\frac{\|\mathbf{y}^{k+1}-\mathbf{y}^\star(\bar{x}^k)\|^2}{n}\right)\right)\\
  \leq&(1-\gamma\mu_g)\left(1+\frac{\gamma\mu_g}{2}\right)\|\bar{z}^k-z_\star^k\|^2+\left(1+\frac{2}{\gamma\mu_g}\right)\|z_\star^k-z_\star^{k+1}\|^2
  +\frac{6\gamma}{\mu_g}L_{g,1}^2\frac{\|\mathbf{z}^k-\bar{\mathbf{z}}^k\|}{n}\\
  &+\frac{12\gamma}{\mu_g}\left(\frac{L_{g,2}^2L_{f,0}^2}{\mu_g^2}+L_{f,1}^2\right)\left(\frac{\|\mathbf{x}^k-\bar{\mathbf{x}}^k\|^2}{n}+\frac{\|\mathbf{y}^{k+1}-\bar{\mathbf{y}}^{k+1}\|^2}{n}+\|\bar{y}^{k+1}-y^\star(\bar{x}^k)\|^2\right)\\
  \leq&(1-\frac{\gamma\mu_g}{2})\|\bar{z}^k-z_\star^k\|^2+\frac{3\alpha^2L_{z^\star}^2}{\gamma\mu_g}\|\bar{r}^k\|^2
  +\frac{6\gamma}{\mu_g}L_{g,1}^2\frac{\|\mathbf{z}^k-\bar{\mathbf{z}}^k\|}{n}\\
  &+\frac{12\gamma}{\mu_g}\left(\frac{L_{g,2}^2L_{f,0}^2}{\mu_g^2}+L_{f,1}^2\right)\left(\frac{\|\mathbf{x}^k-\bar{\mathbf{x}}^k\|^2}{n}+\frac{\|\mathbf{y}^{k+1}-\bar{\mathbf{y}}^{k+1}\|^2}{n}+\|\bar{y}^{k+1}-y^\star(\bar{x}^k)\|^2\right)
\end{aligned}
\end{equation}
where the first and third inequality uses Jensen's inequality and Cauchy Schwartz inequality, the second inequality holds due to Assumption~\ref{smooth} and  $\gamma\mu_g<1$, the last inequality holds since $z^{\star}(x)$ is $L_{z^{\star}}$ Lipschitz continuous.

Moreover, the independence of samples implies that
\begin{equation}
  \begin{aligned}
   \widetilde{\mathbb{E}}_k\left\|\bar{z}^{k+1}-\widetilde{\mathbb{E}}_k[\bar{z}^{k+1}]\right\|^2&=\gamma^2\widetilde{\mathbb{E}}_k\left\|\frac{1}{n}\sum_{i=1}^n(H_{i}^k-\widetilde{\mathbb{E}}_k[H_{i}^k])z_{i}^k+\frac{1}{n}\sum_i(b_{i}^k-\widetilde{\mathbb{E}}_k[b_{i}^k])\right\|^2\\
   &\leq \frac{2\gamma^2}{n}\left(\sigma_{g,2}^2\frac{\|\mathbf{z}^k\|^2}{n}+\sigma_{f,1}^2\right)\\
   &\leq\frac{2\gamma^2}{n}\left(3\sigma_{g,2}^2\left(\frac{\|\mathbf{z}^k-\bar{\mathbf{z}}_k\|^2}{n}+\|\bar{z}^k-z_\star^k\|^2+\frac{L_{f,0}^2}{\mu_g^2}\right)+\sigma_{f,1}^2\right).
  \end{aligned}
\end{equation}
As $\gamma$ satisfies
\begin{equation}
  \begin{aligned}
    \frac{6\sigma_{g,2}^2\gamma^2}{n}\leq \frac{6\gamma L_{g,1}^2}{\mu_{g}^2},\quad\frac{6\sigma_{g,2}^2\gamma^2}{n}\leq\frac{\gamma\mu_g}{6},
  \end{aligned}
\end{equation} 
we get: 
\begin{equation}
  \begin{aligned}
&\widetilde{\mathbb{E}}_k[\|\bar{z}^{k+1}-z_\star^{k+1}\|^2]=\widetilde{\mathbb{E}}_k\|\widetilde{\mathbb{E}}_k[\bar{z}^{k+1}]-z_\star^{k+1}\|^2+\widetilde{\mathbb{E}}_k\|\bar{z}^{k+1}-\widetilde{\mathbb{E}}_k[\bar{z}^{k+1}]\|^2
\\\leq&\left(1-\frac{\gamma\mu_g}{3}\right)\|\bar{z}^k-z_\star^k\|^2+\frac{3\alpha^2L_{z^\star}^2}{\gamma\mu_g}\|\bar{r}^k\|^2
  +\frac{12\gamma}{\mu_g}L_{g,1}^2\frac{\|\mathbf{z}^k-\bar{\mathbf{z}}^k\|^2}{n}+\frac{2\gamma^2}{n}\left(3\sigma_{g,2}^2\frac{L_{f,0}^2}{\mu_g^2}+\sigma_{f,1}^2\right)
  \\&+\frac{12\gamma}{\mu_g}\left(\frac{L_{g,2}^2L_{f,0}^2}{\mu_g^2}+L_{f,1}^2\right)\left(\frac{\|\mathbf{x}^k-\bar{\mathbf{x}}^k\|^2}{n}+\frac{\|\mathbf{y}^{k+1}-\bar{\mathbf{y}}^{k+1}\|^2}{n}+\|\bar{y}^{k+1}-y^\star(\bar{x}^k)\|^2\right).
\end{aligned}
\end{equation} 
Taking expectation and summation on both sides, we get
\begin{equation}
  \begin{aligned}
&\sum_{k=0}^K\frac{\gamma\mu_g}{3}\mathbb{E}\|\bar{z}^k-z_\star^k\|^2+\mathbb{E}\|\bar{z}^{K+1}-z_\star^{K+1}\|^2\\
\leq&\mathbb{E}\|\bar{z}^{0}-z_\star^{0}\|^2+\sum_{k=0}^K\left[\frac{3\alpha^2L_{z^\star}^2}{\gamma\mu_g}\mathbb{E}\|\bar{r}^k\|^2+\frac{12\gamma }{\mu_g}L_{g,1}^2\frac{\mathbb{E}\|\mathbf{z}^k-\bar{\mathbf{z}}^k\|^2}{n}+ \frac{2\gamma^2}{n}\left(3\sigma_{g,2}^2\frac{L_{f,0}^2}{\mu_g^2}+\sigma_{f,1}^2\right)\right]\\
&+\sum_{k=0}^K\frac{12\gamma}{\mu_g}\left(\frac{L_{g,2}^2L_{f,0}^2}{\mu_g^2}+L_{f,1}^2\right)\mathbb{E}\left[\frac{\|\mathbf{x}^k-\bar{\mathbf{x}}^k\|^2}{n}+\frac{\|\mathbf{y}^{k+1}-\bar{\mathbf{y}}^{k+1}\|^2}{n}+\|\bar{y}^{k+1}-y^\star(\bar{x}^k)\|^2\right].
  \end{aligned}
\end{equation}

It follows that
 \begin{equation}
  \begin{aligned}
&\sum_{k=0}^{K+1}\mathbb{E}\|\bar{z}^{k}-z_\star^{k}\|^2\\
\leq& \sum_{k=0}^K\frac{9\alpha^2L_{z^\star}^2}{\gamma^2\mu_g^2}\mathbb{E}\|\bar{r}^k\|^2
\\&+72\kappa^2\sum_{k=1}^K\mathbb{E}\left[\frac{\kappa^2\|\mathbf{O}_x\|^2\|\hat{\mathbf{e}}^k_x\|^2+\|\mathbf{O}_z\|^2\|\hat{\mathbf{e}}^{k}_z\|^2}{n}\right]+72\kappa^2\sum_{k=0}^K\mathbb{E}\left[\frac{\kappa^2\|\mathbf{O}_y\|^2\|\hat{\mathbf{e}}^{k+1}_y\|^2}{n}\right]
\\&+\sum_{k=0}^K72\kappa^4\mathbb{E}\left[\|\bar{y}^{k+1}-y^\star(\bar{x}^k)\|^2\right]
+\frac{3\|z^1_{\star}\|^2}{\mu_g\gamma}+\frac{6(K+1)\gamma}{\mu_gn}\left(3\sigma_{g,2}^2\frac{L_{f,0}^2}{\mu_g^2}+\sigma_{f,1}^2\right),
  \end{aligned}
\end{equation}
since $z^0_\star=z^1_\star$ and $\mathbf{z}^0$ is consensual.
  \end{proof}
\end{lemma}

Then, we combine the results in Lemmas~\ref{ystar},~\ref{zstar} and give an upper bound of $\mathbb{E}[I_k]$:
\begin{lemma}\label{Isum}
Suppose that Lemmas~\ref{ystar} and ~\ref{zstar} hold. Then we have:
\begin{equation}  \label{I}
\begin{aligned}
\sum_{k=-1}^K\mathbb{E}[I_k]
\leq& 
\left(\frac{9\alpha^2L_{z^\star}^2}{\gamma^2\mu_g^2}+\frac{438\kappa^4\alpha^2}{\beta^2\mu_g^2}L_{y^{\star}}^2\right)\sum_{k=0}^K\mathbb{E}\|\bar{r}^k\|^2
+510\kappa^4\sum_{k=0}^K\mathbb{E}\left[\frac{\Delta_k}{n}\right]+\frac{3\|z^1_{\star}\|^2}{\mu_g\gamma}
\\&+\frac{6(K+1)\gamma}{\mu_gn}\left(3\sigma_{g,2}^2\frac{L_{f,0}^2}{\mu_g^2}+\sigma_{f,1}^2\right)
+73\kappa^4\left(\frac{4}{\beta\mu_g}\|\bar{y}^{0}-y^{\star}(\bar{x}^{0})\|^2
+\frac{4K\sigma_{g,1}^2}{n\mu_g}\beta\right).
 \end{aligned}
\end{equation}
\begin{remark}
Here $I_{-1}=\|\bar{z}^{0}-z_\star^{0}\|^2+\kappa^2\|\bar{y}^{0}-y^\star(\bar{x}^{-1})\|^2$. The aim of introducing this term is to simplify the subsequent proofs of other lemmas.
\end{remark}
    \begin{proof}
Lemma~\ref{zstar} implies that:
\begin{equation}
  \begin{aligned}
&\sum_{k=0}^{K+1}\mathbb{E}\|\bar{z}^{k}-z_\star^{k}\|^2\\
\leq& \sum_{k=0}^K\frac{9\alpha^2L_{z^\star}^2}{\gamma^2\mu_g^2}\mathbb{E}\|\bar{r}^k\|^2
\\&+72\kappa^2\sum_{k=1}^K\mathbb{E}\left[\frac{\kappa^2\|\mathbf{O}_x\|^2\|\hat{\mathbf{e}}^k_x\|^2+\|\mathbf{O}_z\|^2\|\hat{\mathbf{e}}^{k}_z\|^2}{n}\right]+72\kappa^2\sum_{k=0}^K\mathbb{E}\left[\frac{\kappa^2\|\mathbf{O}_y\|^2\|\hat{\mathbf{e}}^{k+1}_y\|^2}{n}\right]
\\&+\sum_{k=0}^K72\kappa^4\mathbb{E}\left[\|\bar{y}^{k+1}-y^\star(\bar{x}^k)\|^2\right]
+\frac{3\|z^1_{\star}\|^2}{\mu_g\gamma}+\frac{6(K+1)\gamma}{\mu_gn}\left(3\sigma_{g,2}^2\frac{L_{f,0}^2}{\mu_g^2}+\sigma_{f,1}^2\right)
  \end{aligned}
\end{equation}
Then, using Lemma~\ref{ystar}, we have:
\begin{equation}
  \begin{aligned}
&\sum_{k=-1}^K\mathbb{E}[I_k]=\sum_{k=0}^{K+1}\mathbb{E}\|\bar{z}^{k}-z_\star^{k}\|^2+\kappa^2\sum_{k=0}^{K+1}\mathbb{E}\|\bar{y}^{k}-y^\star(\bar{x}^{k-1})\|^2\\
\leq& \sum_{k=0}^K\frac{9\alpha^2L_{z^\star}^2}{\gamma^2\mu_g^2}\mathbb{E}\|\bar{r}^k\|^2
\\&+72\kappa^2\sum_{k=1}^K\mathbb{E}\left[\frac{\kappa^2\|\mathbf{O}_x\|^2\|\hat{\mathbf{e}}^k_x\|^2+\|\mathbf{O}_z\|^2\|\hat{\mathbf{e}}^{k}_z\|^2}{n}\right]+72\kappa^2\sum_{k=0}^K\mathbb{E}\left[\frac{\kappa^2\|\mathbf{O}_y\|^2\|\hat{\mathbf{e}}^{k+1}_y\|^2}{n}\right]\\
&+ \frac{6(K+1)\gamma}{\mu_gn}\left(3\sigma_{g,2}^2\frac{L_{f,0}^2}{\mu_g^2}+\sigma_{f,1}^2\right)+73\kappa^4\left[\frac{4}{\beta\mu_g}\|\bar{y}^{0}-y^{\star}(\bar{x}^{0})\|^2+\sum_{k=1}^K\frac{6\alpha^2}{\beta^2\mu_g^2}L_{y^{\star}}^2\mathbb{E}\|\bar{r}^k\|^2\right]
\\&+\frac{3\|z^1_{\star}\|^2}{\mu_g\gamma}+73\kappa^4\left[\sum_{k=1}^K\frac{6}{\mu_g^2}L_{g,1}^2\mathbb{E}\left[\frac{\|\mathbf{O}_x\|^2\|\hat{\mathbf{e}}_x^k\|^2+\|\mathbf{O}_y\|^2\|\hat{\mathbf{e}}_y^k\|^2}{n}\right]+\frac{4K\sigma_{g,1}^2}{n\mu_g}\beta\right]
\\
\leq& \sum_{k=0}^K\left(\frac{9\alpha^2L_{z^\star}^2}{\gamma^2\mu_g^2}+\frac{438\kappa^4\alpha^2}{\beta^2\mu_g^2}L_{y^{\star}}^2\right)\mathbb{E}\|\bar{r}^k\|^2+\sum_{k=0}^K510\kappa^4\mathbb{E}\left[\frac{\Delta_k}{n}\right]+\frac{3\|z^1_{\star}\|^2}{\mu_g\gamma}
\\
&+ \frac{6(K+1)\gamma}{\mu_gn}\left(3\sigma_{g,2}^2\frac{L_{f,0}^2}{\mu_g^2}+\sigma_{f,1}^2\right)+73\kappa^4\left(\frac{4}{\beta\mu_g}\|\bar{y}^{0}-y^{\star}(\bar{x}^{0})\|^2+\frac{4K\sigma_{g,1}^2}{n\mu_g}\beta\right).
  \end{aligned}
\end{equation}
    \end{proof}
\end{lemma}

\subsubsection{Consensus error analysis}
In this subsection we aim to bound the consensus errors of $y,z,x$ (\ie the terms $\|\hat{\mathbf{e}}_y^k\|^2$, $\|\hat{\mathbf{e}}_z^k\|^2$, and $\|\hat{\mathbf{e}}_x^k\|^2$).
\begin{lemma}[Consensus error of $y$]\label{4}
Suppose that  Assumptions~\ref{smooth}-~\ref{var} hold, and 

\begin{equation}\label{betay}
\beta^2\le \frac{(1-\|\mathbf{\Gamma}_y\|)^2}{8L_{g,1}^2\|\mathbf{O}_y^{-1}\|^2\|\mathbf{O}_y\|^2\| {\boldsymbol{\Lambda}}_{ya}\|^2}.   
\end{equation}

We have
\begin{equation}\label{ley}
\begin{aligned}
\sum_{k=0}^{K+1}\mathbb{E}\|\hat{\mathbf{e}}_y^k\|^2
\le&3\sum_{k=0}^K\frac{\beta^2\|\mathbf{O}_y^{-1}\|^2}{(1-\|\mathbf{\Gamma}_y\|)^2}\|{\boldsymbol{\Lambda}}_{yb}^{-1}\|^2\| {\boldsymbol{\Lambda}}_{ya}\|^2L_{g,1}^2\mathbb{E}\left[
\|\bar{\mathbf{x}}^{k+1}-\bar{\mathbf{x}}^k\|^2+\|\bar{\mathbf{y}}^{k+1}-\bar{\mathbf{y}}^k\|^2\right]
\\&+\frac{\|\mathbf{O}_x\|^2}{3\|\mathbf{O}_y\|^2}\sum_{k=0}^K\mathbb{E}
\|\hat{\mathbf{e}}_x^k\|^2
+\frac{3(K+1)\beta^2\|\mathbf{O}_y^{-1}\|^2 \|\mathbf{\Lambda}_{ya}\|^2}{1-\|\mathbf{\Gamma}_y\|} n\sigma_{g,1}^2+\frac{2\mathbb{E}\|\hat{\mathbf{e}}_y^{0}\|^2}{1-\|\mathbf{\Gamma}_y\|}.
\end{aligned}
\end{equation}
\begin{proof}

Firstly, the term $\|\hat{\mathbf{e}}_y^{k+1}\|^2$ can be deformed as
\begin{scriptsize}
\begin{equation}\label{a0}
  \begin{aligned}
&\|\hat{\mathbf{e}}_y^{k+1}\|^2\\
=&\left\|\mathbf{\Gamma}_y\hat{\mathbf{e}}_y^k-\beta\mathbf{O}_y^{-1}\left[\begin{array}{c}
{\boldsymbol{\Lambda}}_{ya} {\hat{\mathbf{U}}_y}^{\top}\left[\widehat{\mathbb{E}}_k[\mathbf{v}^k]-\nabla_2 \mathbf{g}(\bar{\mathbf{x}}^k,\bar{\mathbf{y}}^k)\right] \\
{\boldsymbol{\Lambda}}_{yb}^{-1} {\boldsymbol{\Lambda}}_{ya} {\hat{\mathbf{U}}_y}^{\top}\left[\nabla_2 \mathbf{g}(\bar{\mathbf{x}}^{k+1},\bar{\mathbf{y}}^{k+1})-\nabla_2 \mathbf{g}(\bar{\mathbf{x}}^k,\bar{\mathbf{y}}^k)\right]
\end{array}\right]-\beta\mathbf{O}_y^{-1}\left[\begin{array}{c}
{\boldsymbol{\Lambda}}_{ya} {\hat{\mathbf{U}}_y}^{\top}\left[\mathbf{v}^k-\widehat{\mathbb{E}}_k[\mathbf{v}^k]\right] \\\mathbf{0} 
\end{array}\right]\right\|^2 \\
=&\left\|\mathbf{\Gamma}_y\hat{\mathbf{e}}_y^k-\beta\mathbf{O}_y^{-1}\left[\begin{array}{c}
{\boldsymbol{\Lambda}}_{ya} {\hat{\mathbf{U}}_y}^{\top}\left[\widehat{\mathbb{E}}_k[\mathbf{v}^k]-\nabla_2 \mathbf{g}(\bar{\mathbf{x}}^k,\bar{\mathbf{y}}^k)\right] \\
{\boldsymbol{\Lambda}}_{yb}^{-1} {\boldsymbol{\Lambda}}_{ya} {\hat{\mathbf{U}}_y}^{\top}\left[\nabla_2 \mathbf{g}(\bar{\mathbf{x}}^{k+1},\bar{\mathbf{y}}^{k+1})-\nabla_2 \mathbf{g}(\bar{\mathbf{x}}^k,\bar{\mathbf{y}}^k)\right]
\end{array}\right]\right\|^2
\\
&+\beta^2\left\|\mathbf{O}_y^{-1}\left[\begin{array}{c}
{\boldsymbol{\Lambda}}_{ya} {\hat{\mathbf{U}}_y}^{\top}\left[\mathbf{v}^k-\widehat{\mathbb{E}}_k[\mathbf{v}^k]\right] \\\mathbf{0} 
\end{array}\right]\right\|^2 -2\left\langle \mathbf{\Gamma}_y\hat{\mathbf{e}}_y^k,\beta\mathbf{O}_y^{-1}\left[\begin{array}{c}
{\boldsymbol{\Lambda}}_{ya} {\hat{\mathbf{U}}_y}^{\top}\left[\mathbf{v}^k-\widehat{\mathbb{E}}_k[\mathbf{v}^k]\right] \\\mathbf{0} 
\end{array}\right] \right\rangle \\
&+2\beta^2\left\langle \mathbf{O}_y^{-1}\left[\begin{array}{c}
{\boldsymbol{\Lambda}}_{ya} {\hat{\mathbf{U}}_y}^{\top}\left[\widehat{\mathbb{E}}_k[\mathbf{v}^k]-\nabla_2 \mathbf{g}(\bar{\mathbf{x}}^k,\bar{\mathbf{y}}^k)\right] \\
{\boldsymbol{\Lambda}}_{yb}^{-1} {\boldsymbol{\Lambda}}_{ya} {\hat{\mathbf{U}}_y}^{\top}\left[\nabla_2 \mathbf{g}(\bar{\mathbf{x}}^{k+1},\bar{\mathbf{y}}^{k+1})-\nabla_2 \mathbf{g}(\bar{\mathbf{x}}^k,\bar{\mathbf{y}}^k)\right]
\end{array}\right],\mathbf{O}_y^{-1}\left[\begin{array}{c}
{\boldsymbol{\Lambda}}_{ya} {\hat{\mathbf{U}}_y}^{\top}\left[\mathbf{v}^k-\widehat{\mathbb{E}}_k[\mathbf{v}^k]\right] \\\mathbf{0} 
\end{array}\right] \right\rangle
\end{aligned}
\end{equation}
\end{scriptsize}
due to Eq.\eqref{ey}. Then, for the first term in the right-hand side of \eqref{a0}, we have:
\begin{equation}\label{a1}
\begin{aligned}
&\widehat{\mathbb{E}}_k\left[\left\|\mathbf{\Gamma}_y\hat{\mathbf{e}}_y^k-\beta\mathbf{O}_y^{-1}\left[\begin{array}{c}
{\boldsymbol{\Lambda}}_{ya} {\hat{\mathbf{U}}_y}^{\top}\left[\widehat{\mathbb{E}}_k[\mathbf{v}^k]-\nabla_2 \mathbf{g}(\bar{\mathbf{x}}^k,\bar{\mathbf{y}}^k)\right] \\
{\boldsymbol{\Lambda}}_{yb}^{-1} {\boldsymbol{\Lambda}}_{ya} {\hat{\mathbf{U}}_y}^{\top}\left[\nabla_2 \mathbf{g}(\bar{\mathbf{x}}^{k+1},\bar{\mathbf{y}}^{k+1})-\nabla_2 \mathbf{g}(\bar{\mathbf{x}}^k,\bar{\mathbf{y}}^k)\right]
\end{array}\right]\right\|^2\right]
\\\le&\|\mathbf{\Gamma}_y\|\|\hat{\mathbf{e}}_y^k\|^2+\frac{\beta^2\|\mathbf{O}_y^{-1}\|^2}{1-\|\mathbf{\Gamma}_y\|}\widehat{\mathbb{E}}_k\left[\left\|
\left[\begin{array}{c}
{\boldsymbol{\Lambda}}_{ya} {\hat{\mathbf{U}}_y}^{\top}\left[\widehat{\mathbb{E}}_k[\mathbf{v}^k]-\nabla_2 \mathbf{g}(\bar{\mathbf{x}}^k,\bar{\mathbf{y}}^k)\right] \\
{\boldsymbol{\Lambda}}_{yb}^{-1} {\boldsymbol{\Lambda}}_{ya} {\hat{\mathbf{U}}_y}^{\top}\left[\nabla_2 \mathbf{g}(\bar{\mathbf{x}}^{k+1},\bar{\mathbf{y}}^{k+1})-\nabla_2 \mathbf{g}(\bar{\mathbf{x}}^k,\bar{\mathbf{y}}^k)\right]
\end{array}\right]\right\|^2\right]
\\\le&\|\mathbf{\Gamma}_y\|\|\hat{\mathbf{e}}_y^k\|^2+\frac{\beta^2\|\mathbf{O}_y^{-1}\|^2}{1-\|\mathbf{\Gamma}_y\|}\cdot\|{\boldsymbol{\Lambda}}_{ya}\|^2\|\nabla_2 \mathbf{g}(\mathbf{x}^k,\mathbf{y}^k)-\nabla_2 \mathbf{g}(\bar{\mathbf{x}}^k,\bar{\mathbf{y}}^k)\|^2
\\&+\frac{\beta^2\|\mathbf{O}_y^{-1}\|^2}{1-\|\mathbf{\Gamma}_y\|}\|{\boldsymbol{\Lambda}}_{yb}^{-1}\|^2\| {\boldsymbol{\Lambda}}_{ya}\|^2\widehat{\mathbb{E}}_k\left[\|\nabla_2 \mathbf{g}(\bar{\mathbf{x}}^{k+1},\bar{\mathbf{y}}^{k+1})-\nabla_2 \mathbf{g}(\bar{\mathbf{x}}^k,\bar{\mathbf{y}}^k)\|^2\right]
\\\le&\|\mathbf{\Gamma}_y\|\|\hat{\mathbf{e}}_y^k\|^2+\frac{\beta^2\|\mathbf{O}_y^{-1}\|^2}{1-\|\mathbf{\Gamma}_y\|}\cdot\|{\boldsymbol{\Lambda}}_{ya}\|^2L_{g,1}^2\left(
\|\mathbf{x}^k-\bar{\mathbf{x}}^k\|^2+\|\mathbf{y}^k-\bar{\mathbf{y}}^k\|^2\right)
\\&+\frac{\beta^2\|\mathbf{O}_y^{-1}\|^2}{1-\|\mathbf{\Gamma}_y\|}\|{\boldsymbol{\Lambda}}_{yb}^{-1}\|^2\| {\boldsymbol{\Lambda}}_{ya}\|^2L_{g,1}^2\widehat{\mathbb{E}}_k\left[
\|\bar{\mathbf{x}}^{k+1}-\bar{\mathbf{x}}^k\|^2+\|\bar{\mathbf{y}}^{k+1}-\bar{\mathbf{y}}^k\|^2\right],
\end{aligned}
\end{equation}
where the first inequality uses the Jenson's inequality, the second inequality hold since $\|\hat{\mathbf{U}}_y^{\top}\|\le1$.

For the second term, we have:
\begin{equation}\label{a2}
\begin{aligned}
\widehat{\mathbb{E}}_k\left[\left\|\mathbf{O}_y^{-1}\left[\begin{array}{c}
{\boldsymbol{\Lambda}}_{ya} {\hat{\mathbf{U}}_y}^{\top}\left[\mathbf{v}^k-\widehat{\mathbb{E}}_k[\mathbf{v}^k]\right] \\\mathbf{0} 
\end{array}\right]\right\|^2 \right]\le\|\mathbf{O}_y^{-1}\|^2 \|\mathbf{\Lambda}_{ya}\|^2 n\sigma_{g,1}^2.
\end{aligned}
\end{equation}

For the third them, we have:
\begin{equation}\label{a3}
\begin{aligned}
&\widehat{\mathbb{E}}_k\left[\left\langle \mathbf{\Gamma}_y\hat{\mathbf{e}}_y^k,\beta\mathbf{O}_y^{-1}\left[\begin{array}{c}
{\boldsymbol{\Lambda}}_{ya} {\hat{\mathbf{U}}_y}^{\top}\left[\mathbf{v}^k-\widehat{\mathbb{E}}_k[\mathbf{v}^k]\right] \\\mathbf{0} 
\end{array}\right] \right\rangle\right]=0.
\end{aligned}
\end{equation}

Next, for the last term, we have:
\begin{small}
\begin{equation}\label{a4}
\begin{aligned}
&\widehat{\mathbb{E}}_k\left\langle \mathbf{O}_y^{-1}\left[\begin{array}{c}
{\boldsymbol{\Lambda}}_{ya} {\hat{\mathbf{U}}_y}^{\top}\left[\widehat{\mathbb{E}}_k[\mathbf{v}^k]-\nabla_2 \mathbf{g}(\bar{\mathbf{x}}^k,\bar{\mathbf{y}}^k)\right] \\
{\boldsymbol{\Lambda}}_{yb}^{-1} {\boldsymbol{\Lambda}}_{ya} {\hat{\mathbf{U}}_y}^{\top}\left[\nabla_2 \mathbf{g}(\bar{\mathbf{x}}^{k+1},\bar{\mathbf{y}}^{k+1})-\nabla_2 \mathbf{g}(\bar{\mathbf{x}}^k,\bar{\mathbf{y}}^k)\right]
\end{array}\right], \mathbf{O}_y^{-1}\left[\begin{array}{c}
{\boldsymbol{\Lambda}}_{ya} {\hat{\mathbf{U}}_y}^{\top}\left[\mathbf{v}^k-\widehat{\mathbb{E}}_k[\mathbf{v}^k]\right] \\\mathbf{0} 
\end{array}\right] \right\rangle\\
\le&\frac{1}{2}\|\mathbf{O}_y^{-1}\|^2\widehat{\mathbb{E}}_k\left\| \left[\begin{array}{c}
{\boldsymbol{\Lambda}}_{ya} {\hat{\mathbf{U}}_y}^{\top}\left[\widehat{\mathbb{E}}_k[\mathbf{v}^k]-\nabla_2 \mathbf{g}(\bar{\mathbf{x}}^k,\bar{\mathbf{y}}^k)\right] \\
{\boldsymbol{\Lambda}}_{yb}^{-1} {\boldsymbol{\Lambda}}_{ya} {\hat{\mathbf{U}}_y}^{\top}\left[\nabla_2 \mathbf{g}(\bar{\mathbf{x}}^{k+1},\bar{\mathbf{y}}^{k+1})-\nabla_2 \mathbf{g}(\bar{\mathbf{x}}^k,\bar{\mathbf{y}}^k)\right]
\end{array}\right]\right\|^2\\
&+ \frac{1}{2}\|\mathbf{O}_y^{-1}\|^2\widehat{\mathbb{E}}_k\left\|\left[\begin{array}{c}
{\boldsymbol{\Lambda}}_{ya} {\hat{\mathbf{U}}_y}^{\top}\left[\mathbf{v}^k-\widehat{\mathbb{E}}_k[\mathbf{v}^k]\right] \\\mathbf{0} 
\end{array}\right] \right\|^2\\
\le&\frac{1}{2}\beta^2\|\mathbf{O}_y^{-1}\|^2\|{\boldsymbol{\Lambda}}_{ya}\|^2L_{g,1}^2\widehat{\mathbb{E}}_k\left[
\|\mathbf{x}^k-\bar{\mathbf{x}}^k\|^2+\|\mathbf{y}^k-\bar{\mathbf{y}}^k\|^2\right]\\
&+\frac{1}{2}\beta^2\|\mathbf{O}_y^{-1}\|^2\|{\boldsymbol{\Lambda}}_{yb}^{-1}\|^2\| {\boldsymbol{\Lambda}}_{ya}\|^2L_{g,1}^2\widehat{\mathbb{E}}_k\left[\|\bar{\mathbf{x}}^{k+1}-\bar{\mathbf{x}}^k\|^2+\|\bar{\mathbf{y}}^{k+1}-\bar{\mathbf{y}}^k\|^2\right]\\
&+\frac{1}{2}\beta^2\|\mathbf{O}_y^{-1}\|^2\|\mathbf{\Lambda}_{ya}\|^2 n\sigma_{g,1}^2.
\end{aligned}
\end{equation}
\end{small}

Taking expectations  on both sides of \eqref{a0}, and plugging \eqref{a1}, \eqref{a2}, \eqref{a3},  \eqref{a4} into it, we obtain:
\begin{equation}
\begin{aligned}
&\mathbb{E}\left[\|\hat{\mathbf{e}}_y^{k+1}\|^2\right]\\
\le&2\frac{\beta^2\|\mathbf{O}_y^{-1}\|^2}{1-\|\mathbf{\Gamma}_y\|}\|{\boldsymbol{\Lambda}}_{yb}^{-1}\|^2\| {\boldsymbol{\Lambda}}_{ya}\|^2L_{g,1}^2\mathbb{E}\left[
\|\bar{\mathbf{x}}^{k+1}-\bar{\mathbf{x}}^k\|^2+\|\bar{\mathbf{y}}^{k+1}-\bar{\mathbf{y}}^k\|^2\right]+\|\mathbf{\Gamma}_y\|\mathbb{E}\|\hat{\mathbf{e}}_y^k\|^2
\\&+2\frac{\beta^2\|\mathbf{O}_y^{-1}\|^2}{1-\|\mathbf{\Gamma}_y\|}\|\| {\boldsymbol{\Lambda}}_{ya}\|^2L_{g,1}^2\mathbb{E}\left[
\|\mathbf{x}^k-\bar{\mathbf{x}}^k\|^2+\|\mathbf{y}^k-\bar{\mathbf{y}}^k\|^2\right]
+2\beta^2\|\mathbf{O}_y^{-1}\|^2 \|\mathbf{\Lambda}_{ya}\|^2 n\sigma_{g,1}^2.
\end{aligned}
\end{equation}

Taking summation over $k$ and using $\|\mathbf{x}^k-\bar{\mathbf{x}}^k\|^2\le\|\mathbf{O}_x\|^2\|\hat{\mathbf{e}}_x^k\|^2$, $\|\mathbf{y}^k-\bar{\mathbf{y}}^k\|^2\le\|\mathbf{O}_y\|^2\|\hat{\mathbf{e}}_y^k\|^2$, we get: 

\begin{equation}
\begin{aligned}
&(1-\|\mathbf{\Gamma}_y\|)\sum_{k=0}^K\mathbb{E}\left[\|\hat{\mathbf{e}}_y^k\|^2\right]\\
\le&2\sum_{k=0}^K\frac{\beta^2\|\mathbf{O}_y^{-1}\|^2}{1-\|\mathbf{\Gamma}_y\|}\|{\boldsymbol{\Lambda}}_{yb}^{-1}\|^2\| {\boldsymbol{\Lambda}}_{ya}\|^2L_{g,1}^2\mathbb{E}\left[
\|\bar{\mathbf{x}}^{k+1}-\bar{\mathbf{x}}^k\|^2+\|\bar{\mathbf{y}}^{k+1}-\bar{\mathbf{y}}^k\|^2\right]
\\&+\mathbb{E}\|\hat{\mathbf{e}}_y^{0}\|^2-\mathbb{E}\|\hat{\mathbf{e}}_y^{k+1}\|^2
+2\sum_{k=0}^K\frac{\beta^2\|\mathbf{O}_y^{-1}\|^2}{1-\|\mathbf{\Gamma}_y\|}\|\| {\boldsymbol{\Lambda}}_{ya}\|^2L_{g,1}^2\mathbb{E}\left[
\|\mathbf{O}_x\|^2\|\hat{\mathbf{e}}_x^k\|^2+\|\mathbf{O}_y\|^2\|\hat{\mathbf{e}}_y^k\|^2\right]
\\
&+2(K+1)\beta^2\|\mathbf{O}_y^{-1}\|^2 \|\mathbf{\Lambda}_{ya}\|^2 n\sigma_{g,1}^2\\
\le&2\sum_{k=0}^K\frac{\beta^2\|\mathbf{O}_y^{-1}\|^2}{1-\|\mathbf{\Gamma}_y\|}\|{\boldsymbol{\Lambda}}_{yb}^{-1}\|^2\| {\boldsymbol{\Lambda}}_{ya}\|^2L_{g,1}^2\mathbb{E}\left[
\|\bar{\mathbf{x}}^{k+1}-\bar{\mathbf{x}}^k\|^2+\|\bar{\mathbf{y}}^{k+1}-\bar{\mathbf{y}}^k\|^2\right]
\\&+\mathbb{E}\|\hat{\mathbf{e}}_y^{0}\|^2-\mathbb{E}\|\hat{\mathbf{e}}_y^{k+1}\|^2
+\frac{1-\|\mathbf{\Gamma}_y\|}{4\|\mathbf{O}_y\|^2}\sum_{k=0}^K\mathbb{E}\left[
\|\mathbf{O}_x\|^2\|\hat{\mathbf{e}}_x^k\|^2+\|\mathbf{O}_y\|^2\|\hat{\mathbf{e}}_y^k\|^2\right]\\
&+2(K+1)\beta^2\|\mathbf{O}_y^{-1}\|^2 \|\mathbf{\Lambda}_{ya}\|^2 n\sigma_{g,1}^2,
\end{aligned}
\end{equation}
where the last inequality uses  $\beta^2\le \dfrac{(1-\|\mathbf{\Gamma}_y\|)^2}{8L_{g,1}^2\|\mathbf{O}_y^{-1}\|^2\|\mathbf{O}_y\|^2\| {\boldsymbol{\Lambda}}_{ya}\|^2}$.

It follows that 
\begin{equation}
\begin{aligned}
\sum_{k=0}^{K+1}\mathbb{E}\left[\|\hat{\mathbf{e}}_y^k\|^2\right]\le&3\sum_{k=0}^K\frac{\beta^2\|\mathbf{O}_y^{-1}\|^2}{(1-\|\mathbf{\Gamma}_y\|)^2}\|{\boldsymbol{\Lambda}}_{yb}^{-1}\|^2\| {\boldsymbol{\Lambda}}_{ya}\|^2L_{g,1}^2\mathbb{E}\left[
\|\bar{\mathbf{x}}^{k+1}-\bar{\mathbf{x}}^k\|^2+\|\bar{\mathbf{y}}^{k+1}-\bar{\mathbf{y}}^k\|^2\right]
\\&+\frac{\|\mathbf{O}_x\|^2}{3\|\mathbf{O}_y\|^2}\sum_{k=0}^K\mathbb{E}\left[
\|\hat{\mathbf{e}}_x^k\|^2\right]
+\frac{3(K+1)\beta^2\|\mathbf{O}_y^{-1}\|^2 \|\mathbf{\Lambda}_{ya}\|^2}{1-\|\mathbf{\Gamma}_y\|} n\sigma_{g,1}^2+\frac{2\mathbb{E}\|\hat{\mathbf{e}}_y^{0}\|^2}{1-\|\mathbf{\Gamma}_y\|}.
\end{aligned}
\end{equation}

\end{proof}
\end{lemma}

\begin{lemma}[Consensus error of $z$]\label{6}

Suppose that Assumptions~\ref{smooth}-~\ref{var} hold, and $\gamma$ satisfies
\begin{equation}\label{gammaez}
\frac{6\gamma^2\|\mathbf{O}_z^{-1}\|^2\|\mathbf{O}_z\|^2\|{\boldsymbol{\Lambda}}_{za}\|^2}{1-\|\mathbf{\Gamma}_z\|}\cdot(2L^2+2(1-\|\mathbf{\Gamma}_z\|)\sigma_{g,2}^2)\le\frac{1-\|\mathbf{\Gamma}_z\|}{4}.
\end{equation}
We have
\begin{equation}\label{lez}
  \begin{aligned}
&\sum_{k=0}^{K+1}\mathbb{E}\left[\|\hat{\mathbf{e}}_z^k\|^2\right]\\
\le&\frac{16\gamma^2(L_{g,1}^2+(1-\|\mathbf{\Gamma}_z\|)\sigma_{g,2}^2)\|\mathbf{O}_z^{-1}\|^2\|{\boldsymbol{\Lambda}}_{za}\|^2}{(1-\|\mathbf{\Gamma}_z\|)^2}
\sum_{k=0}^K\mathbb{E}\left[\|\bar{\mathbf{z}}^k-\mathbf{z}_{\star}^k\|^2\right]+\frac{2\mathbb{E}[\|\hat{\mathbf{e}}_z^{0}\|^2]}{1-\|\mathbf{\Gamma}_z\|}\\
&+\frac{8\gamma^2\|\mathbf{O}_z^{-1}\|^2\|{\boldsymbol{\Lambda}}_{zb}^{-1}\|^2\|{\boldsymbol{\Lambda}}_{za}\|^2}{(1-\|\mathbf{\Gamma}_z\|)^2}\sum_{k=0}^K\mathbb{E}\left[\left(L_{f,1}^2+L_{g,2}^2\frac{L_{f,0}^2}{\mu_g^2}+L_{g,1}^2L_{z^{\star}}^2\right)\|\bar{\mathbf{x}}^{k+1}-\bar{\mathbf{x}}^k\|^2\right]\\
&+\frac{8\gamma^2\|\mathbf{O}_z^{-1}\|^2\|{\boldsymbol{\Lambda}}_{zb}^{-1}\|^2\|{\boldsymbol{\Lambda}}_{za}\|^2}{(1-\|\mathbf{\Gamma}_z\|)^2}\sum_{k=0}^K\mathbb{E}\left[\left(L_{f,1}^2+L_{g,2}^2\frac{L_{f,0}^2}{\mu_g^2}\right)\|\bar{\mathbf{y}}^{k+2}-\bar{\mathbf{y}}^{k+1}\|^2\right]\\
&+16(K+1)n\gamma^2\frac{\|\mathbf{O}_z^{-1}\|^2\|{\boldsymbol{\Lambda}}_{za}\|^2}{1-\|\mathbf{\Gamma}_z\|}\left(\frac{L_{f,0}^2}{\mu_g^2}\sigma_{g,2}^2+\sigma_{f,1}^2\right)\\
&+\frac{\kappa^2}{3\|\mathbf{O}_z\|^2}\sum_{k=0}^K\mathbb{E}(\|\mathbf{O}_x\|^2\|\hat{\mathbf{e}}_x^k\|^2+\|\mathbf{O}_y\|^2\|\hat{\mathbf{e}}_y^{k+1}\|^2).
\end{aligned}
\end{equation}

\begin{proof}
Firstly, Eq. \eqref{ez} implies that:
\begin{equation}
\begin{aligned}
\hat{\mathbf{e}}_z^{k+1}=&\mathbf{\Gamma}_z\hat{\mathbf{e}}_z^k-\gamma\mathbf{O}_z^{-1}\left[\begin{array}{c}
{\boldsymbol{\Lambda}}_{za} {\hat{\mathbf{U}}}_z^{\top}\left[\widetilde{\mathbb{E}}_k[\mathbf{p}^k]-\mathbf{p}^k(\bar{\mathbf{x}}^k,\bar{\mathbf{y}}^{k+1})\right] \\
{\boldsymbol{\Lambda}}_{zb}^{-1} {\boldsymbol{\Lambda}}_{za} {\hat{\mathbf{U}}}_z^{\top}\left[\mathbf{p}^{k+1}(\bar{\mathbf{x}}^{k+1},\bar{\mathbf{y}}^{k+2})-\mathbf{p}^k(\bar{\mathbf{x}}^k,\bar{\mathbf{y}}^{k+1})\right]
\end{array}\right]
\\&+
\gamma\mathbf{O}_z^{-1}\left[\begin{array}{c}
{\boldsymbol{\Lambda}}_{za} {\hat{\mathbf{U}}}_z^{\top}\left[\widetilde{\mathbb{E}}_k[\mathbf{p}^k]-\mathbf{p}^k\right] \\
0
\end{array}\right].
\end{aligned}
\end{equation}
Then using Cauchy Schwartz inequality, we get \begin{small}
\begin{equation}\label{consensus_z_1}
  \begin{aligned}
&\|\hat{\mathbf{e}}_z^{k+1}\|^2\\
\leq&\left\|\mathbf{\Gamma}_z\hat{\mathbf{e}}_z^k-\gamma\mathbf{O}_z^{-1}\left[\begin{array}{c}
{\boldsymbol{\Lambda}}_{za} {\hat{\mathbf{U}}}_z^{\top}\left[\widetilde{\mathbb{E}}_k[\mathbf{p}^k]-\mathbf{p}^k(\bar{\mathbf{x}}^k,\bar{\mathbf{y}}^{k+1})\right] \\
{\boldsymbol{\Lambda}}_{zb}^{-1} {\boldsymbol{\Lambda}}_{za} {\hat{\mathbf{U}}}_z^{\top}\left[\mathbf{p}^{k+1}(\bar{\mathbf{x}}^{k+1},\bar{\mathbf{y}}^{k+2})-\mathbf{p}^k(\bar{\mathbf{x}}^k,\bar{\mathbf{y}}^{k+1})\right]
\end{array}\right]\right\|^2\\
&+\gamma^2\left\|\mathbf{O}_z^{-1}\left[\begin{array}{c}
{\boldsymbol{\Lambda}}_{za} {\hat{\mathbf{U}}}_z^{\top}\left[\widetilde{\mathbb{E}}_k[\mathbf{p}^k]-\mathbf{p}^k\right] \\
0
\end{array}\right]\right\|^2-2\left\langle \mathbf{\Gamma}_z\hat{\mathbf{e}}_z^k,\gamma\mathbf{O}_z^{-1}\left[\begin{array}{c}
{\boldsymbol{\Lambda}}_{za} {\hat{\mathbf{U}}}_z^{\top}\left[\widetilde{\mathbb{E}}_k[\mathbf{p}^k]-\mathbf{p}^k\right] \\
0
\end{array}\right] \right\rangle\\
&+\gamma^2\left\Vert\mathbf{O}_z^{-1}\left[\begin{array}{c}
{\boldsymbol{\Lambda}}_{za} {\hat{\mathbf{U}}}_z^{\top}\left[\widetilde{\mathbb{E}}_k[\mathbf{p}^k]-\mathbf{p}^k(\bar{\mathbf{x}}^k,\bar{\mathbf{y}}^{k+1})\right] \\
{\boldsymbol{\Lambda}}_{zb}^{-1} {\boldsymbol{\Lambda}}_{za} {\hat{\mathbf{U}}}_z^{\top}\left[\mathbf{p}^{k+1}(\bar{\mathbf{x}}^{k+1},\bar{\mathbf{y}}^{k+2})-\mathbf{p}^k(\bar{\mathbf{x}}^k,\bar{\mathbf{y}}^{k+1})\right]
\end{array}\right]\right\Vert^2\\
&+\gamma^2\left\Vert\mathbf{O}_z^{-1}\left[\begin{array}{c}
{\boldsymbol{\Lambda}}_{za} {\hat{\mathbf{U}}}_z^{\top}\left[\widetilde{\mathbb{E}}_k[\mathbf{p}^k]-\mathbf{p}^k\right] \\
0
\end{array}\right]\right\Vert^2
\end{aligned}
\end{equation}
\end{small}
To obtain the upper bound of the right-hand side of the above equation, we first estimate some individual terms in it as follows.
Note that:
\begin{equation}\label{p1}
  \begin{aligned}
&\widetilde{\mathbb{E}}_k\|\widetilde{\mathbb{E}}_k[\mathbf{p}^k]-\mathbf{p}^k(\bar{\mathbf{x}}^k,\bar{\mathbf{y}}^{k+1})\|^2
\\=&\sum_{i=1}^n\widetilde{\mathbb{E}}_k\left\|\nabla_{22}^2g_i(x_{i}^k,y_{i}^{k+1})z_{i}^k-\nabla_2 f_i(x_{i}^k,y_{i}^{k+1})-\left(\nabla_{22}^2g_i(\bar{x}^k,\bar{y}^{k+1})z_k^{\star}-\nabla_2 f_i(\bar{x}^k,\bar{y}^{k+1})\right)\right\|^2\\
\leq & 3\sum_{i=1}^n\widetilde{\mathbb{E}}_k\left\|\nabla_{22}^2g_i(x_{i}^k,y_{i}^{k+1})(z_{i}^k-z^k_{\star})\right\|^2+3\sum_{i=1}^n\widetilde{\mathbb{E}}_k\left\|(\nabla_{22}^2g_i(x_{i}^k,y_{i}^{k+1})-\nabla_{22}^2g_i(\bar{x}^k,\bar{y}^{k+1}))z^k_{\star}\right\|^2\\
&+3\sum_{i=1}^n\widetilde{\mathbb{E}}_k\left\|\nabla_2 f_i(\bar{x}^k,\bar{y}^{k+1})-\nabla_2 f_i(x_{i}^k,y_{i}^{k+1})\right\|^2\\
\le&6L_{g,1}^2(\|\mathbf{z}^k-\bar{\mathbf{z}}^k\|^2+\|\bar{\mathbf{z}}^k-\mathbf{z}_{\star}^k\|^2)+3\left(L_{g,2}^2\frac{L_{f,0}^2}{\mu_g^2}+L_{f,1}^2\right)\left(\|\mathbf{x}^k-\bar{\mathbf{x}}^k\|^2+\|\mathbf{y}^{k+1}-\bar{\mathbf{y}}^{k+1}\|^2\right)
\end{aligned}
\end{equation}
and
\begin{equation}\label{p2}
  \begin{aligned}
&\widetilde{\mathbb{E}}_k\|\mathbf{p}^{k+1}(\bar{\mathbf{x}}^{k+1},\bar{\mathbf{y}}^{k+2})-\mathbf{p}^k(\bar{\mathbf{x}}^k,\bar{\mathbf{y}}^{k+1})\|^2
\\=&\sum_{i=1}^n\widetilde{\mathbb{E}}_k\|\nabla_{22}^2g_i(\bar{x}^{k+1},\bar{y}^{k+2})z_{\star}^{k+1}-\nabla_2 f_i(\bar{x}^{k+1},\bar{y}^{k+2})-\nabla_{22}^2g_i(\bar{x}^k,\bar{y}^{k+1})z_{\star}^k+\nabla_2 f_i(\bar{x}^k,\bar{y}^{k+1})\|^2\\
\leq&3\sum_{i=1}^n\widetilde{\mathbb{E}}_k\|(\nabla_{22}^2g_i(\bar{x}^{k+1},\bar{y}^{k+2})-\nabla_{22}^2g_i(\bar{x}^k,\bar{y}^{k+1}))z_{\star}^{k+1}\|^2\\
&+3\sum_{i=1}^n\widetilde{\mathbb{E}}_k\|\nabla_{22}^2g_i(\bar{x}^k,\bar{y}^{k+1})(z_{\star}^{k+1}-z_{\star}^k)\|^2+3\sum_{i=1}^n\widetilde{\mathbb{E}}_k\|\nabla_2 f_i(\bar{x}^{k+1},\bar{y}^{k+2})-\nabla_2 f_i(\bar{x}^k,\bar{y}^{k+1})\|^2\\
\le&3\widetilde{\mathbb{E}}_k\left[\left(L_{f,1}^2+L_{g,2}^2\frac{L_{f,0}^2}{\mu_g^2}\right)(\|\bar{\mathbf{x}}^{k+1}-\bar{\mathbf{x}}^k\|^2+\|\bar{\mathbf{y}}^{k+2}-\bar{\mathbf{y}}^{k+1}\|^2)+L_{g,1}^2L_{z^{\star}}^2\|\bar{\mathbf{x}}^k-\bar{\mathbf{x}}^{k-1}\|^2\right].
\end{aligned}
\end{equation}

Then we present the bound of the right-hand side of \eqref{consensus_z_1}. For the first term, we have the following evaluations:
\begin{equation}\label{z1}
  \begin{aligned}
&\widetilde{\mathbb{E}}_k\left[\left\|\mathbf{\Gamma}_z\hat{\mathbf{e}}_z^k-\gamma\mathbf{O}_z^{-1}\left[\begin{array}{c}
{\boldsymbol{\Lambda}}_{za} {\hat{\mathbf{U}}}_z^{\top}\left[\widetilde{\mathbb{E}}_k[\mathbf{p}^k]-\mathbf{p}^k(\bar{\mathbf{x}}^k,\bar{\mathbf{y}}^{k+1})\right] \\
{\boldsymbol{\Lambda}}_{zb}^{-1} {\boldsymbol{\Lambda}}_{za} {\hat{\mathbf{U}}}_z^{\top}\left[\mathbf{p}^{k+1}(\bar{\mathbf{x}}^{k+1},\bar{\mathbf{y}}^{k+2})-\mathbf{p}^k(\bar{\mathbf{x}}^k,\bar{\mathbf{y}}^{k+1})\right]
\end{array}\right]\right\|\right]\\
\le&\|\mathbf{\Gamma}_z\|\|\hat{\mathbf{e}}_z^k\|^2
+\frac{\gamma^2\|\mathbf{O}_z^{-1}\|^2}{1-\|\mathbf{\Gamma}_z\|}\widetilde{\mathbb{E}}_k
\left[\left\|\left[\begin{array}{c}
{\boldsymbol{\Lambda}}_{za} {\hat{\mathbf{U}}}_z^{\top}\left[\widetilde{\mathbb{E}}_k[\mathbf{p}^k]-\mathbf{p}^k(\bar{\mathbf{x}}^k,\bar{\mathbf{y}}^{k+1})\right] \\
{\boldsymbol{\Lambda}}_{zb}^{-1} {\boldsymbol{\Lambda}}_{za} {\hat{\mathbf{U}}}_z^{\top}\left[\mathbf{p}^{k+1}(\bar{\mathbf{x}}^{k+1},\bar{\mathbf{y}}^{k+2})-\mathbf{p}^k(\bar{\mathbf{x}}^k,\bar{\mathbf{y}}^{k+1})\right]
\end{array}\right]\right\|^2\right]\\
\le&\|\mathbf{\Gamma}_z\|\|\hat{\mathbf{e}}_z^k\|^2
+\frac{\gamma^2\|\mathbf{O}_z^{-1}\|^2\|{\boldsymbol{\Lambda}}_{za}\|^2}{1-\|\mathbf{\Gamma}_z\|}\widetilde{\mathbb{E}}_k\left[\|\widetilde{\mathbb{E}}_k[\mathbf{p}^k]-\mathbf{p}^k(\bar{\mathbf{x}}^k,\bar{\mathbf{y}}^{k+1})\|^2\right]\\
&+\frac{\gamma^2\|\mathbf{O}_z^{-1}\|^2\|{\boldsymbol{\Lambda}}_{za}\|^2\|{\boldsymbol{\Lambda}}_{zb}^{-1}\|^2}{1-\|\mathbf{\Gamma}_z\|}\widetilde{\mathbb{E}}_k\left[\|\mathbf{p}^{k+1}(\bar{\mathbf{x}}^{k+1},\bar{\mathbf{y}}^{k+2})-\mathbf{p}^k(\bar{\mathbf{x}}^k,\bar{\mathbf{y}}^{k+1})\|^2\right]
\end{aligned}
\end{equation}
where the first inequality use Jensen's inequality and the second inequality use $\|\hat{\mathbf{U}}_z^{\top}\|\le1$.

For the second term, since $\mathbf{z}^k,\mathbf{y}^{k+1}\in\mathcal{U}_k$, we have:
\begin{equation}\label{z2}
  \begin{aligned}
&\widetilde{\mathbb{E}}_k\left\|\mathbf{O}_z^{-1}\left[\begin{array}{c}
{\boldsymbol{\Lambda}}_{za} {\hat{\mathbf{U}}}_z^{\top}\left[\widetilde{\mathbb{E}}_k[\mathbf{p}^k]-\mathbf{p}^k\right]\\
0
\end{array}\right]\right\|^2\\
\le& \|\mathbf{O}_z^{-1}\|^2\|{\boldsymbol{\Lambda}}_{za}\|^2\widetilde{\mathbb{E}}_k\left[\|\widetilde{\mathbb{E}}_k[\mathbf{p}^k]-\mathbf{p}^k\right\|^2]\\
\le&2\|\mathbf{O}_z^{-1}\|^2\|{\boldsymbol{\Lambda}}_{za}\|^2(\|\mathbf{z}^k\|^2\sigma_{g,2}^2+n\sigma_{f,1}^2)\\
\le&6\|\mathbf{O}_z^{-1}\|^2\|{\boldsymbol{\Lambda}}_{za}\|^2\left(\left(\|\mathbf{z}^k-\bar{\mathbf{z}}^k\|^2+\|\bar{\mathbf{z}}^k-\mathbf{z}^k_
{\star}\|^2+n\frac{L_{f,1}^2}{\mu_g^2}\right)\sigma_{g,2}^2+n\sigma_{f,1}^2\right).
\end{aligned}
\end{equation}

For the third term, we have:
\begin{equation}\label{z3}
  \begin{aligned}
&2\widetilde{\mathbb{E}}_k\left\langle \mathbf{\Gamma}_z\hat{\mathbf{e}}_z^k,\gamma\mathbf{O}_z^{-1}\left[\begin{array}{c}
{\boldsymbol{\Lambda}}_{za} {\hat{\mathbf{U}}}_z^{\top}\left[\widetilde{\mathbb{E}}_k[\mathbf{p}^k]-\mathbf{p}^k\right] \\
0
\end{array}\right] \right\rangle=0,
  \end{aligned}
\end{equation}
since $\hat{\mathbf{e}}_z^k\in \mathcal{U}_k$.

Next, for the last two terms, we have:

\begin{equation}\label{z4}
  \begin{aligned}
&\widetilde{\mathbb{E}}_k\left[\left\Vert\mathbf{O}_z^{-1}\left[\begin{array}{c}
{\boldsymbol{\Lambda}}_{za} {\hat{\mathbf{U}}}_z^{\top}\left[\widetilde{\mathbb{E}}_k[\mathbf{p}^k]-\mathbf{p}^k(\bar{\mathbf{x}}^k,\bar{\mathbf{y}}^{k+1})\right] \\
{\boldsymbol{\Lambda}}_{zb}^{-1} {\boldsymbol{\Lambda}}_{za} {\hat{\mathbf{U}}}_z^{\top}\left[\mathbf{p}^{k+1}(\bar{\mathbf{x}}^{k+1},\bar{\mathbf{y}}^{k+2})-\mathbf{p}^k(\bar{\mathbf{x}}^k,\bar{\mathbf{y}}^{k+1})\right]
\end{array}\right]\right\Vert^2\right]\\
\le&\|\mathbf{O}_z^{-1}\|^2\|{\boldsymbol{\Lambda}}_{za}\|^2\widetilde{\mathbb{E}}_k\left[\|\widetilde{\mathbb{E}}_k[\mathbf{p}^k]-\mathbf{p}^k(\bar{\mathbf{x}}^k,\bar{\mathbf{y}}^{k+1})\|^2\right]\\
&+\|\mathbf{O}_z^{-1}\|^2\|{\boldsymbol{\Lambda}}_{za}\|^2\|{\boldsymbol{\Lambda}}_{zb}^{-1}\|^2 \widetilde{\mathbb{E}}_k
\left[\|\mathbf{p}^{k+1}(\bar{\mathbf{x}}^{k+1},\bar{\mathbf{y}}^{k+2})-\mathbf{p}^k(\bar{\mathbf{x}}^k,\bar{\mathbf{y}}^{k+1})\|^2
\right].
  \end{aligned}
\end{equation}

\begin{equation}\label{z5}
  \begin{aligned}
&\widetilde{\mathbb{E}}_k\left[\left\Vert\mathbf{O}_z^{-1}\left[\begin{array}{c}
{\boldsymbol{\Lambda}}_{za} {\hat{\mathbf{U}}}_z^{\top}\left[\widetilde{\mathbb{E}}_k[\mathbf{p}^k]-\mathbf{p}^k\right] \\
0
\end{array}\right]\right\Vert^2\right]\le\|\mathbf{O}_z^{-1}\|^2\|{\boldsymbol{\Lambda}}_{za}\|^2\widetilde{\mathbb{E}}_k
\left[\|\widetilde{\mathbb{E}}_k[\mathbf{p}^k]-\mathbf{p}^k\|^2\right].
  \end{aligned}
\end{equation}

Taking the expectation $\widetilde{\mathbb{E}}_k$ on both sides of \eqref{consensus_z_1} and plugging \eqref{z1}, \eqref{z2}, \eqref{z3}, \eqref{z4} and \eqref{z5} into it,  we obtain:
\begin{equation}
  \begin{aligned}
&\widetilde{\mathbb{E}}_k\left[\|\hat{\mathbf{e}}_z^{k+1}\|^2\right]\\
\le&\|\mathbf{\Gamma}_z\|\|\hat{\mathbf{e}}_z^k\|^2
+\frac{2\gamma^2\|\mathbf{O}_z^{-1}\|^2\|{\boldsymbol{\Lambda}}_{za}\|^2}{1-\|\mathbf{\Gamma}_z\|}\widetilde{\mathbb{E}}_k
\left[
\|\widetilde{\mathbb{E}}_k[\mathbf{p}^k]-\mathbf{p}^k(\bar{\mathbf{x}}^k,\bar{\mathbf{y}}^{k+1})\|^2\right]\\
&+\frac{2\gamma^2\|\mathbf{O}_z^{-1}\|^2\|{\boldsymbol{\Lambda}}_{za}\|^2\|{\boldsymbol{\Lambda}}_{zb}^{-1}\|^2}{1-\|\mathbf{\Gamma}_z\|}\widetilde{\mathbb{E}}_k
\left[\|\mathbf{p}^{k+1}(\bar{\mathbf{x}}^{k+1},\bar{\mathbf{y}}^{k+2})-\mathbf{p}^k(\bar{\mathbf{x}}^k,\bar{\mathbf{y}}^{k+1})\|^2
\right]
\\&+2\gamma^2\|\mathbf{O}_z^{-1}\|^2\|{\boldsymbol{\Lambda}}_{za}\|^2\widetilde{\mathbb{E}}_k
\|\widetilde{\mathbb{E}}_k[\mathbf{p}^k]-\mathbf{p}^k\|^2\\
\le&\|\mathbf{\Gamma}_z\|\|\hat{\mathbf{e}}_z^k\|^2+12n\gamma^2\|\mathbf{O}_z^{-1}\|^2\|{\boldsymbol{\Lambda}}_{za}\|^2\left(\frac{L_{f,0}^2}{\mu_g^2}\sigma_{g,2}^2+\sigma_{f,1}^2\right)
\\&+\frac{12\gamma^2\|\mathbf{O}_z^{-1}\|^2\|{\boldsymbol{\Lambda}}_{za}\|^2}{1-\|\mathbf{\Gamma}_z\|}(L_{g,1}^2+(1-\|\mathbf{\Gamma}_z\|)\sigma_{g,2}^2)
\widetilde{\mathbb{E}}_k\left[\|\mathbf{O}_z\|^2\|\hat{\mathbf{e}}_z^k\|^2+\|\bar{\mathbf{z}}^k-\mathbf{z}_{\star}^k\|^2\right]
\\&+\frac{6\gamma^2\|\mathbf{O}_z^{-1}\|^2\|{\boldsymbol{\Lambda}}_{za}\|^2}{1-\|\mathbf{\Gamma}_z\|}\left(L_{g,2}^2\frac{L_{f,0}^2}{\mu_g^2}+L_{f,1}^2\right)\widetilde{\mathbb{E}}_k\left[\|\mathbf{O}_x\|^2\|\hat{\mathbf{e}}_x^k\|^2+\|\mathbf{O}_y\|^2\|\hat{\mathbf{e}}_y^{k+1}\|^2\right]\\
&+\frac{6\gamma^2\|\mathbf{O}_z^{-1}\|^2\|{\boldsymbol{\Lambda}}_{zb}^{-1}\|^2\|{\boldsymbol{\Lambda}}_{za}\|^2}{1-\|\mathbf{\Gamma}_z\|}\left(L_{f,1}^2+L_{g,2}^2\frac{L_{f,0}^2}{\mu_g^2}\right)\widetilde{\mathbb{E}}_k\left[\|\bar{\mathbf{x}}^{k+1}-\bar{\mathbf{x}}^k\|^2+\|\bar{\mathbf{y}}^{k+2}-\bar{\mathbf{y}}^{k+1}\|^2\right]\\
&+\frac{6\gamma^2\|\mathbf{O}_z^{-1}\|^2\|{\boldsymbol{\Lambda}}_{zb}^{-1}\|^2\|{\boldsymbol{\Lambda}}_{za}\|^2}{1-\|\mathbf{\Gamma}_z\|}L_{g,1}^2L_{z^{\star}}^2
\widetilde{\mathbb{E}}_k\left[\|\bar{\mathbf{x}}^k-\bar{\mathbf{x}}^{k-1}\|^2\right],
  \end{aligned}
\end{equation}
where the second inequality uses \eqref{transformer_consensus}, \eqref{p1}, \eqref{p2}, and \eqref{z2}.

Thanks to
\begin{equation}
  \begin{aligned}
\frac{6\gamma^2\|\mathbf{O}_z^{-1}\|^2\|\mathbf{O}_z\|^2\|{\boldsymbol{\Lambda}}_{za}\|^2}{1-\|\mathbf{\Gamma}_z\|}\cdot(2L^2+2(1-\|\mathbf{\Gamma}_z\|)\sigma_{g,2}^2)\le\frac{1-\|\mathbf{\Gamma}_z\|}{4},
\end{aligned}
\end{equation}
we have:
\begin{equation}
  \begin{aligned}
&\widetilde{\mathbb{E}}_k\left[\|\hat{\mathbf{e}}_z^{k+1}\|^2\right]
\\\le&\frac{1+3\|\mathbf{\Gamma}_z\|}{4}\|\hat{\mathbf{e}}_z^k\|^2
+\frac{12\gamma^2(L_{g,1}^2+(1-\|\mathbf{\Gamma}_z\|)\sigma_{g,2}^2)\|\mathbf{O}_z^{-1}\|^2\|{\boldsymbol{\Lambda}}_{za}\|^2}{1-\|\mathbf{\Gamma}_z\|}
\widetilde{\mathbb{E}}_k\left[\|\bar{\mathbf{z}}^k-\mathbf{z}_{\star}^k\|^2\right]
\\&+12n\gamma^2\|\mathbf{O}_z^{-1}\|^2\|{\boldsymbol{\Lambda}}_{za}\|^2\left(\frac{L_{f,0}^2}{\mu_g^2}\sigma_{g,2}^2+\sigma_{f,1}^2\right)\\&+\frac{1-\|\mathbf{\Gamma}_z\|}{4\|\mathbf{O}_z\|^2}\kappa^2(\|\mathbf{O}_x\|^2\|\hat{\mathbf{e}}_x^k\|^2+\|\mathbf{O}_y\|^2\|\hat{\mathbf{e}}_y^{k+1}\|^2)
\\&+\frac{6\gamma^2\|\mathbf{O}_z^{-1}\|^2\|{\boldsymbol{\Lambda}}_{zb}^{-1}\|^2\|{\boldsymbol{\Lambda}}_{za}\|^2}{1-\|\mathbf{\Gamma}_z\|}\left(L_{f,1}^2+L_{g,2}^2\frac{L_{f,0}^2}{\mu_g^2}\right)\widetilde{\mathbb{E}}_k\left[\|\bar{\mathbf{x}}^{k+1}-\bar{\mathbf{x}}^k\|^2+\|\bar{\mathbf{y}}^{k+2}-\bar{\mathbf{y}}^{k+1}\|^2\right]\\
&+\frac{6\gamma^2\|\mathbf{O}_z^{-1}\|^2\|{\boldsymbol{\Lambda}}_{zb}^{-1}\|^2\|{\boldsymbol{\Lambda}}_{za}\|^2}{1-\|\mathbf{\Gamma}_z\|}L_{g,1}^2L_{z^{\star}}^2
\widetilde{\mathbb{E}}_k\left[\|\bar{\mathbf{x}}^k-\bar{\mathbf{x}}^{k-1}\|^2\right].
  \end{aligned}
\end{equation}

Taking summation and expectation on both sides, we get:
\begin{equation}
  \begin{aligned}
&\frac{3}{4}(1-\|\mathbf{\Gamma}_z\|)\sum_{k=0}^K\mathbb{E}\left[\|\hat{\mathbf{e}}_z^k\|^2\right]
\\\le&\mathbb{E}\|\hat{\mathbf{e}}_z^{0}\|^2-\mathbb{E}\left[\|\hat{\mathbf{e}}_z^{k+1}\|^2\right]\\&
+\frac{12\gamma^2(L_{g,1}^2+(1-\|\mathbf{\Gamma}_z\|)\sigma_{g,2}^2)\|\mathbf{O}_z^{-1}\|^2\|{\boldsymbol{\Lambda}}_{za}\|^2}{1-\|\mathbf{\Gamma}_z\|}
\sum_{k=0}^K\mathbb{E}\left[\|\bar{\mathbf{z}}^k-\mathbf{z}_{\star}^k\|^2\right]
\\&+\frac{1-\|\mathbf{\Gamma}_z\|}{4\|\mathbf{O}_z\|^2}\kappa^2\sum_{k=0}^K\mathbb{E}(\|\mathbf{O}_x\|^2\|\hat{\mathbf{e}}_x^k\|^2+\|\mathbf{O}_y\|^2\|\hat{\mathbf{e}}_y^{k+1}\|^2)\\
&+12n\gamma^2\|\mathbf{O}_z^{-1}\|^2\|{\boldsymbol{\Lambda}}_{za}\|^2\left(\frac{L_{f,0}^2}{\mu_g^2}\sigma_{g,2}^2+\sigma_{f,1}^2\right)
\\&+\frac{6\gamma^2\|\mathbf{O}_z^{-1}\|^2\|{\boldsymbol{\Lambda}}_{zb}^{-1}\|^2\|{\boldsymbol{\Lambda}}_{za}\|^2}{1-\|\mathbf{\Gamma}_z\|}
\sum_{k=0}^K\mathbb{E}\left[\left(L_{f,1}^2+L_{g,2}^2\frac{L_{f,0}^2}{\mu_g^2}+L_{g,1}^2L_{z^{\star}}^2\right)\|\bar{\mathbf{x}}^{k+1}-\bar{\mathbf{x}}^k\|^2\right]\\
&+\frac{6\gamma^2\|\mathbf{O}_z^{-1}\|^2\|{\boldsymbol{\Lambda}}_{zb}^{-1}\|^2\|{\boldsymbol{\Lambda}}_{za}\|^2}{1-\|\mathbf{\Gamma}_z\|}\sum_{k=0}^K\mathbb{E}\left[\left(L_{f,1}^2+L_{g,2}^2\frac{L_{f,0}^2}{\mu_g^2}\right)\|\bar{\mathbf{y}}^{k+2}-\bar{\mathbf{y}}^{k+1}\|^2\right].
\end{aligned}
\end{equation}

Thus,
  \begin{align*}
&\sum_{k=0}^{K+1}\mathbb{E}\left[\|\hat{\mathbf{e}}_z^k\|^2\right]\\
\le&
\frac{16\gamma^2(L_{g,1}^2+(1-\|\mathbf{\Gamma}_z\|)\sigma_{g,2}^2)\|\mathbf{O}_z^{-1}\|^2\|{\boldsymbol{\Lambda}}_{za}\|^2}{(1-\|\mathbf{\Gamma}_z\|)^2}
\sum_{k=0}^K\mathbb{E}\left[\|\bar{\mathbf{z}}^k-\mathbf{z}_{\star}^k\|^2\right]+\frac{2\mathbb{E}[\|\hat{\mathbf{e}}_z^{0}\|^2]}{1-\|\mathbf{\Gamma}_z\|}
\\&+\frac{8\gamma^2\|\mathbf{O}_z^{-1}\|^2\|{\boldsymbol{\Lambda}}_{zb}^{-1}\|^2\|{\boldsymbol{\Lambda}}_{za}\|^2}{(1-\|\mathbf{\Gamma}_z\|)^2}
\sum_{k=0}^K\mathbb{E}\left[\left(L_{f,1}^2+L_{g,2}^2\frac{L_{f,0}^2}{\mu_g^2}+L_{g,1}^2L_{z^{\star}}^2\right)\|\bar{\mathbf{x}}^{k+1}-\bar{\mathbf{x}}^k\|^2\right]\\
&+\frac{8\gamma^2\|\mathbf{O}_z^{-1}\|^2\|{\boldsymbol{\Lambda}}_{zb}^{-1}\|^2\|{\boldsymbol{\Lambda}}_{za}\|^2}{(1-\|\mathbf{\Gamma}_z\|)^2}\sum_{k=0}^K\mathbb{E}\left[\left(L_{f,1}^2+L_{g,2}^2\frac{L_{f,0}^2}{\mu_g^2}\right)\|\bar{\mathbf{y}}^{k+2}-\bar{\mathbf{y}}^{k+1}\|^2\right]\\
\\&+16(K+1)n\gamma^2\frac{\|\mathbf{O}_z^{-1}\|^2\|{\boldsymbol{\Lambda}}_{za}\|^2}{1-\|\mathbf{\Gamma}_z\|}\left(\frac{L_{f,0}^2}{\mu_g^2}\sigma_{g,2}^2+\sigma_{f,1}^2\right)\\
&+\frac{\kappa^2}{3\|\mathbf{O}_z\|^2}\sum_{k=0}^K\mathbb{E}(\|\mathbf{O}_x\|^2\|\hat{\mathbf{e}}_x^k\|^2+\|\mathbf{O}_y\|^2\|\hat{\mathbf{e}}_y^{k+1}\|^2).
\end{align*}
\end{proof}
\end{lemma}

\begin{lemma}[Consensus error of $x$]\label{10}
Suppose that Assumptions~\ref{smooth}-~\ref{var} and Lemmas~\ref{8},~\ref{uvar}, and~\ref{rphi} hold. We have
\begin{equation}\label{lex}
  \begin{aligned}
    &\sum_{k=0}^{K+1}\mathbb{E} \|\hat{\mathbf{e}}_x^k\|^2\\
    \le&\frac{\mathbb{E}\|\hat{\mathbf{e}}_x^{0}\|^2}{1-\|\mathbf{\Gamma}_x\|}
+\frac{2\alpha^2\|\mathbf{O}_x^{-1}\|^2\|{\boldsymbol{\Lambda}}_{xa}\|^2}{(1-\|\mathbf{\Gamma}_x\|)^2}
\left[
\frac{1-\theta}{\theta}\left\|\widetilde{\nabla}\mathbf{\Phi}(\bar{\mathbf{x}}^{0})\right\|^2\right]\\
&+\frac{6n\alpha^2\theta(K+1)\|\mathbf{O}_x^{-1}\|^2\|{\boldsymbol{\Lambda}}_{xa}\|^2}{1-\|\mathbf{\Gamma}_x\|}\left(\theta+\frac{1-\theta}{1-\|\mathbf{\Gamma}_x\|}\right)\left(\sigma_{f,1}^2+3\sigma_{g,2}^2\frac{L_{f,0}^2}{\mu_g^2}\right)
\\&+\frac{\alpha^2\|\mathbf{O}_x^{-1}\|^2\|{\boldsymbol{\Lambda}}_{xa}\|^2}{1-\|\mathbf{\Gamma}_x\|}\left(\frac{80L^2}{1-\|\mathbf{\Gamma}_x\|}+18\theta\left(\theta+\frac{1-\theta}{1-\|\mathbf{\Gamma}_x\|}\right)\sigma_{g,2}^2 \right)\sum_{k=0}^K\mathbb{E}\left[\Delta_k+nI_k\right]\\&+\frac{2\widetilde{L}^2\alpha^2\|\mathbf{O}_x^{-1}\|^2\|{\boldsymbol{\Lambda}}_{xa}\|^2}{(1-\|\mathbf{\Gamma}_x\|)^2}\left( 
\|{\boldsymbol{\Lambda}}_{xb}^{-1}\|^2+\frac{2(1-\theta)^2}{\theta^2}
\right)\sum_{k=0}^{K-1}\mathbb{E}\left[\left\|\bar{\mathbf{x}}^{k+1}-\bar{\mathbf{x}}^k\right\|^2\right].
 \end{aligned}
\end{equation}

\begin{proof}
Firstly, the term $\|\hat{\mathbf{e}}_x^{k+1}\|^2$ can be deformed as
\begin{small}
\begin{equation}\label{x0}
  \begin{aligned}
&\|\hat{\mathbf{e}}_x^{k+1}\|^2\\
=&\left\|\mathbf{\Gamma}_x\hat{\mathbf{e}}_x^k-\alpha\mathbf{O}_x^{-1}\left[\begin{array}{c}
{\boldsymbol{\Lambda}}_{xa} \hat{\mathbf{U}}_x^{\top}\left[\mathbf{r}^{k+1}-\widetilde{\nabla}\mathbf{\Phi}(\bar{\mathbf{x}}^k)\right] \\
{\boldsymbol{\Lambda}}_{xb}^{-1} {\boldsymbol{\Lambda}}_{xa} \hat{\mathbf{U}}_x^{\top}\left[\widetilde{\nabla}\mathbf{\Phi}(\bar{\mathbf{x}}^{k+1})-\widetilde{\nabla}\mathbf{\Phi}(\bar{\mathbf{x}}^k)\right]
\end{array}\right]\right\|^2\\
=&\left\|\mathbf{\Gamma}_x\hat{\mathbf{e}}_x^k-\alpha\mathbf{O}_x^{-1}\left[\begin{array}{c}
{\boldsymbol{\Lambda}}_{xa} \hat{\mathbf{U}}_x^{\top}\left[\mathbb{E}_k[\mathbf{r}^{k+1}]-\widetilde{\nabla}\mathbf{\Phi}(\bar{\mathbf{x}}^k)\right] \\
{\boldsymbol{\Lambda}}_{xb}^{-1} {\boldsymbol{\Lambda}}_{xa} \hat{\mathbf{U}}_x^{\top}\left[\widetilde{\nabla}\mathbf{\Phi}(\bar{\mathbf{x}}^{k+1})-\widetilde{\nabla}\mathbf{\Phi}(\bar{\mathbf{x}}^k)\right]
\end{array}\right]\right\|^2\\
&+\alpha^2\mathbb{E}_k\left[\left\|\mathbf{O}_x^{-1}\left[\begin{array}{c}
{\boldsymbol{\Lambda}}_{xa} \hat{\mathbf{U}}_x^{\top}\left[\mathbb{E}_k[\mathbf{r}^{k+1}]-\mathbf{r}^{k+1}\right] \\
0
\end{array}\right]\right\|^2\right]\\
&-2\left\langle \mathbf{\Gamma}_x\hat{\mathbf{e}}_x^k,\alpha\mathbf{O}_x^{-1}\left[\begin{array}{c}
{\boldsymbol{\Lambda}}_{xa} \hat{\mathbf{U}}_x^{\top}\left[\mathbb{E}_k[\mathbf{r}^{k+1}]-\mathbf{r}^{k+1}\right] \\
0
\end{array}\right] \right\rangle\\
&+2\alpha^2\left\langle \mathbf{O}_x^{-1}\left[\begin{array}{c}
{\boldsymbol{\Lambda}}_{xa} \hat{\mathbf{U}}_x^{\top}\left[\mathbb{E}_k[\mathbf{r}^{k+1}]-\widetilde{\nabla}\mathbf{\Phi}(\bar{\mathbf{x}}^k)\right] \\
{\boldsymbol{\Lambda}}_{xb}^{-1} {\boldsymbol{\Lambda}}_{xa} \hat{\mathbf{U}}_x^{\top}\left[\widetilde{\nabla}\mathbf{\Phi}(\bar{\mathbf{x}}^{k+1})-\widetilde{\nabla}\mathbf{\Phi}(\bar{\mathbf{x}}^k)\right]
\end{array}\right],\mathbf{O}_x^{-1}\left[\begin{array}{c}
{\boldsymbol{\Lambda}}_{xa} \hat{\mathbf{U}}_x^{\top}\left[\mathbb{E}_k[\mathbf{r}^{k+1}]-\mathbf{r}^{k+1}\right] \\
0
\end{array}\right] \right\rangle
\end{aligned}
\end{equation}
\end{small}
due to Eq. \eqref{ex}.

Then, for the first term of the right-hand side of \eqref{x0}, we use Jensen's Inequality and get:
\begin{equation}\label{x1}
  \begin{aligned}
&\mathbb{E}_k\left[\left\|\mathbf{\Gamma}\hat{\mathbf{e}}_x^k-\alpha\mathbf{O}_x^{-1}\left[\begin{array}{c}
{\boldsymbol{\Lambda}}_{xa} \hat{\mathbf{U}}_x^{\top}\left[\mathbb{E}_k[\mathbf{r}^{k+1}]-\widetilde{\nabla}\mathbf{\Phi}(\bar{\mathbf{x}}^k)\right] \\
{\boldsymbol{\Lambda}}_{xb}^{-1} {\boldsymbol{\Lambda}}_{xa} \hat{\mathbf{U}}_x^{\top}\left[\widetilde{\nabla}\mathbf{\Phi}(\bar{\mathbf{x}}^{k+1})-\widetilde{\nabla}\mathbf{\Phi}(\bar{\mathbf{x}}^k)\right]
\end{array}\right]\right\|^2\right]
\\\le&\|\mathbf{\Gamma}_x\|\|\hat{\mathbf{e}}_x^k\|^2
+\frac{\alpha^2\|\mathbf{O}_x^{-1}\|^2}{1-\|\mathbf{\Gamma}_x\|}\mathbb{E}_k
\left[\left\|\left[\begin{array}{c}
{\boldsymbol{\Lambda}}_{xa} \hat{\mathbf{U}}_x^{\top}\left[\mathbb{E}_k[\mathbf{r}^{k+1}]-\widetilde{\nabla}\mathbf{\Phi}(\bar{\mathbf{x}}^k)\right] \\
{\boldsymbol{\Lambda}}_{xb}^{-1} {\boldsymbol{\Lambda}}_{xa} \hat{\mathbf{U}}_x^{\top}\left[\widetilde{\nabla}\mathbf{\Phi}(\bar{\mathbf{x}}^{k+1})-\widetilde{\nabla}\mathbf{\Phi}(\bar{\mathbf{x}}^k)\right]
\end{array}\right]\right\|^2\right]\\
\le&\|\mathbf{\Gamma}_x\|\|\hat{\mathbf{e}}_x^k\|^2
+\frac{\alpha^2\|\mathbf{O}_x^{-1}\|^2\|{\boldsymbol{\Lambda}}_{xa}\|^2}{1-\|\mathbf{\Gamma}_x\|}\mathbb{E}_k
\left[
\left\|\mathbb{E}_k[\mathbf{r}^{k+1}]-\widetilde{\nabla}\mathbf{\Phi}(\bar{\mathbf{x}}^k)\right\|^2\right],\\
&+\frac{\alpha^2\|\mathbf{O}_x^{-1}\|^2\|{\boldsymbol{\Lambda}}_{xa}\|^2\|{\boldsymbol{\Lambda}}_{xb}^{-1}\|^2}{1-\|\mathbf{\Gamma}_x\|}\mathbb{E}_k
\left[\left\|\widetilde{\nabla}\mathbf{\Phi}(\bar{\mathbf{x}}^{k+1})-\widetilde{\nabla}\mathbf{\Phi}(\bar{\mathbf{x}}^k)\right\|^2
\right].
\end{aligned}
\end{equation}

For the second term in the right-hand side of \eqref{x0}, we have:
\begin{equation}\label{x3}
  \begin{aligned}
\mathbb{E}_k\left\|\mathbf{O}_x^{-1}\left[\begin{array}{c}
{\boldsymbol{\Lambda}}_{xa} \hat{\mathbf{U}}_x^{\top}\left[\mathbb{E}_k[\mathbf{r}^{k+1}]-\mathbf{r}^{k+1}\right]\\
0
\end{array}\right]\right\|^2
\le&\|\mathbf{O}_x^{-1}\|^2\|{\boldsymbol{\Lambda}}_{xa}\|^2\mathbb{E}_k\|\mathbb{E}_k[\mathbf{r}^{k+1}]-\mathbf{r}^{k+1}\|^2\\
=&\theta^2\|\mathbf{O}_x^{-1}\|^2\|{\boldsymbol{\Lambda}}_{xa}\|^2\mathbb{E}_k\|\mathbb{E}_k[\mathbf{u}^k]-\mathbf{u}^k\|^2.
\end{aligned}
\end{equation}
Like \eqref{z3}, we have:
\begin{equation}\label{x2}
  \begin{aligned}
&\mathbb{E}_k\left[\left\langle \mathbf{\Gamma}_x\hat{\mathbf{e}}_x^k,\alpha\mathbf{O}_x^{-1}\left[\begin{array}{c}
{\boldsymbol{\Lambda}}_{xa} \hat{\mathbf{U}}_x^{\top}\left[\mathbb{E}_k[\mathbf{r}^{k+1}]-\mathbf{r}^{k+1}\right] \\
0
\end{array}\right] \right\rangle\right]
=0.
  \end{aligned}
\end{equation}

Next, for the last term, we have:
\begin{small}
\begin{equation}\label{x4}
  \begin{aligned}
&2\alpha^2\mathbb{E}_k\left\langle \mathbf{O}_x^{-1}\left[\begin{array}{c}
{\boldsymbol{\Lambda}}_{xa} \hat{\mathbf{U}}_x^{\top}\left[\mathbb{E}_k[\mathbf{r}^{k+1}]-\widetilde{\nabla}\mathbf{\Phi}(\bar{\mathbf{x}}^k)\right] \\
{\boldsymbol{\Lambda}}_{xb}^{-1} {\boldsymbol{\Lambda}}_{xa} \hat{\mathbf{U}}_x^{\top}\left[\widetilde{\nabla}\mathbf{\Phi}(\bar{\mathbf{x}}^{k+1})-\widetilde{\nabla}\mathbf{\Phi}(\bar{\mathbf{x}}^k)\right]
\end{array}\right],\mathbf{O}_x^{-1}\left[\begin{array}{c}
{\boldsymbol{\Lambda}}_{xa} \hat{\mathbf{U}}_x^{\top}\left[\mathbb{E}_k[\mathbf{r}^{k+1}]-\mathbf{r}^{k+1}\right] \\
0
\end{array}\right] \right\rangle\\
\le&\alpha
^2\|\mathbf{O}_x^{-1}\|^2\|{\boldsymbol{\Lambda}}_{xa}\|^2\mathbb{E}_k
\left[
\|\mathbb{E}_k[\mathbf{r}^{k+1}]-\widetilde{\nabla}\mathbf{\Phi}(\bar{\mathbf{x}}^k)\|^2+\|\mathbb{E}_k[\mathbf{r}^{k+1}]-\mathbf{r}^{k+1}\|^2\right]\\
&+\alpha
^2\|\mathbf{O}_x^{-1}\|^2\|{\boldsymbol{\Lambda}}_{xa}\|^2\|{\boldsymbol{\Lambda}}_{xb}^{-1}\|^2\mathbb{E}_k
\left[
 \|\widetilde{\nabla}\mathbf{\Phi}(\bar{\mathbf{x}}^{k+1})-\widetilde{\nabla}\mathbf{\Phi}(\bar{\mathbf{x}}^k)\|^2\right].
  \end{aligned}
\end{equation}
\end{small}

Taking the expectation on both sides of \eqref{x0}, and plugging \eqref{x1}, \eqref{x3}, \eqref{x2}, \eqref{x4} into it, we obtain:
\begin{equation}
  \begin{aligned}
 &\mathbb{E}_k \|\hat{\mathbf{e}}_x^{k+1}\|^2\\
 \le&\|\mathbf{\Gamma}_x\|\|\hat{\mathbf{e}}_x^k\|^2+2\alpha^2\theta^2\|\mathbf{O}_x^{-1}\|^2\|{\boldsymbol{\Lambda}}_{xa}\|^2\mathbb{E}_k\|\mathbb{E}_k[\mathbf{u}^k]-\mathbf{u}^k\|^2\\
&+\frac{2\alpha^2\|\mathbf{O}_x^{-1}\|^2\|{\boldsymbol{\Lambda}}_{xa}\|^2}{1-\|\mathbf{\Gamma}_x\|}\mathbb{E}_k
\left[
\|\mathbb{E}_k[\mathbf{r}^{k+1}]-\widetilde{\nabla}\mathbf{\Phi}(\bar{\mathbf{x}}^k)\|^2
+\|{\boldsymbol{\Lambda}}_{xb}^{-1}\|^2 \|\widetilde{\nabla}\mathbf{\Phi}(\bar{\mathbf{x}}^{k+1})-\widetilde{\nabla}\mathbf{\Phi}(\bar{\mathbf{x}}^k)\|^2
\right].
 \end{aligned}
\end{equation} 
Taking expectation and summation on both sides, we obtain:
\begin{equation}
  \begin{aligned}
  &(1-\|\mathbf{\Gamma}_x\|) \sum_{k=0}^K\mathbb{E} \|\hat{\mathbf{e}}_x^k\|^2\\
  \le& \mathbb{E}\|\hat{\mathbf{e}}_x^{0}\|^2- \mathbb{E}\|\hat{\mathbf{e}}_x^{k+1}\|^2+\frac{2\widetilde{L}^2\alpha^2\|\mathbf{O}_x^{-1}\|^2\|{\boldsymbol{\Lambda}}_{xb}^{-1}\|^2\|{\boldsymbol{\Lambda}}_{xa}\|^2}{1-\|\mathbf{\Gamma}_x\|}\sum_{k=0}^K\mathbb{E} \|\bar{\mathbf{x}}^{k+1}-\bar{\mathbf{x}}^k\|^2
\\&+\frac{2\alpha^2\|\mathbf{O}_x^{-1}\|^2\|{\boldsymbol{\Lambda}}_{xa}\|^2}{1-\|\mathbf{\Gamma}_x\|}
\left(
\frac{1-\theta}{\theta}\left\|\widetilde{\nabla}\mathbf{\Phi}(\bar{\mathbf{x}}^{0})\right\|^2
+2\sum_{k=0}^K\mathbb{E}\left[\left\|\mathbb{E}_k[\mathbf{u}^k]-\widetilde{\nabla}\mathbf{\Phi}(\bar{\mathbf{x}}^k)\right\|^2\right]\right)
\\&+\frac{2\alpha^2\|\mathbf{O}_x^{-1}\|^2\|{\boldsymbol{\Lambda}}_{xa}\|^2}{1-\|\mathbf{\Gamma}_x\|}\cdot\frac{2\widetilde{L}^2(1-\theta)^2}{\theta^2}\sum_{k=0}^{K-1}\mathbb{E}\left[\left\|\bar{\mathbf{x}}^{k+1}-\bar{\mathbf{x}}^k\right\|^2\right]
\\&+\left[2\alpha^2\theta^2\|\mathbf{O}_x^{-1}\|^2\|{\boldsymbol{\Lambda}}_{xa}\|^2+\frac{2\alpha^2\|\mathbf{O}_x^{-1}\|^2\|{\boldsymbol{\Lambda}}_{xa}\|^2}{1-\|\mathbf{\Gamma}_x\|}\theta(1-\theta)\right]
\sum_{k=0}^K\mathbb{E}\|\mathbb{E}_k[\mathbf{u}^k]-\mathbf{u}^k\|^2\\
\le& \mathbb{E}\|\hat{\mathbf{e}}_x^{0}\|^2- \mathbb{E}\|\hat{\mathbf{e}}_x^{k+1}\|^2
+\frac{2\alpha^2\|\mathbf{O}_x^{-1}\|^2\|{\boldsymbol{\Lambda}}_{xa}\|^2}{1-\|\mathbf{\Gamma}_x\|}
\left[
\frac{1-\theta}{\theta}\left\|\widetilde{\nabla}\mathbf{\Phi}(\bar{\mathbf{x}}^{0})\right\|^2\right]
\\&+6(K+1)n\alpha^2\theta\|\mathbf{O}_x^{-1}\|^2\|{\boldsymbol{\Lambda}}_{xa}\|^2\left(\theta+\frac{1-\theta}{1-\|\mathbf{\Gamma}_x\|}\right)\left(\sigma_{f,1}^2+3\sigma_{g,2}^2\frac{L_{f,0}^2}{\mu_g^2}\right)
\\&+\alpha^2\|\mathbf{O}_x^{-1}\|^2\|{\boldsymbol{\Lambda}}_{xa}\|^2\left(\frac{80L^2}{1-\|\mathbf{\Gamma}_x\|}+18\theta\left(\theta+\frac{1-\theta}{1-\|\mathbf{\Gamma}_x\|}\right)\sigma_{g,2}^2 \right)\sum_{k=0}^K\mathbb{E}\left[\Delta_k+nI_k\right]\\&+\frac{2\widetilde{L}^2\alpha^2\|\mathbf{O}_x^{-1}\|^2\|{\boldsymbol{\Lambda}}_{xa}\|^2}{1-\|\mathbf{\Gamma}_x\|}\left( 
\|{\boldsymbol{\Lambda}}_{xb}^{-1}\|^2+\frac{2(1-\theta)^2}{\theta^2}
\right)\sum_{k=0}^{K-1}\mathbb{E}\left[\left\|\bar{\mathbf{x}}^{k+1}-\bar{\mathbf{x}}^k\right\|^2\right].
 \end{aligned}
\end{equation}

where the first inequality uses $\widetilde{L}$- Lipschitz continuity of $\widetilde{\nabla}{\mathbf{\Phi}}$ ,  Lemma~\ref{rphi}, and the second inequality uses Lemma~\ref{8} , Lemma~\ref{uvar} and
\[
\|\mathbf{z}^{k+1}-\bar{\mathbf{z}}^{k+1}\|^2\le \Delta_k, \quad \| \bar{\mathbf{z}}^{k+1}-\mathbf{z}^{k+1}_{\star}\|^2\le nI_k.
\]

Hence we get 
  \begin{align*}
    &\sum_{k=0}^{K+1}\mathbb{E} \|\hat{\mathbf{e}}_x^k\|^2\\
    \le&\frac{\mathbb{E}\|\hat{\mathbf{e}}_x^{0}\|^2}{1-\|\mathbf{\Gamma}_x\|}
+\frac{2\alpha^2\|\mathbf{O}_x^{-1}\|^2\|{\boldsymbol{\Lambda}}_{xa}\|^2}{(1-\|\mathbf{\Gamma}_x\|)^2}
\left[
\frac{1-\theta}{\theta}\left\|\widetilde{\nabla}\mathbf{\Phi}(\bar{\mathbf{x}}^{0})\right\|^2\right]\\
&+\frac{6n\alpha^2\theta(K+1)\|\mathbf{O}_x^{-1}\|^2\|{\boldsymbol{\Lambda}}_{xa}\|^2}{1-\|\mathbf{\Gamma}_x\|}\left(\theta+\frac{1-\theta}{1-\|\mathbf{\Gamma}_x\|}\right)\left(\sigma_{f,1}^2+3\sigma_{g,2}^2\frac{L_{f,0}^2}{\mu_g^2}\right)
\\&+\frac{\alpha^2\|\mathbf{O}_x^{-1}\|^2\|{\boldsymbol{\Lambda}}_{xa}\|^2}{1-\|\mathbf{\Gamma}_x\|}\left(\frac{80L^2}{1-\|\mathbf{\Gamma}_x\|}+18\theta\left(\theta+\frac{1-\theta}{1-\|\mathbf{\Gamma}_x\|}\right)\sigma_{g,2}^2 \right)\sum_{k=0}^K\mathbb{E}\left[\Delta_k+nI_k\right]\\&+\frac{2\widetilde{L}^2\alpha^2\|\mathbf{O}_x^{-1}\|^2\|{\boldsymbol{\Lambda}}_{xa}\|^2}{(1-\|\mathbf{\Gamma}_x\|)^2}\left( 
\|{\boldsymbol{\Lambda}}_{xb}^{-1}\|^2+\frac{2(1-\theta)^2}{\theta^2}
\right)\sum_{k=0}^{K-1}\mathbb{E}\left[\left\|\bar{\mathbf{x}}^{k+1}-\bar{\mathbf{x}}^k\right\|^2\right].
 \end{align*}

\end{proof}
\end{lemma}

The following lemma gather the consensus analysis of $x,y,z$ together:
\begin{lemma}\label{Deltasum}
Take
\begin{small}
\begin{subequations}
    \begin{align}
\eta_1=&\frac{3\kappa^2\beta^2\|\mathbf{O}_y\|^2\|\mathbf{O}_y^{-1}\|^2}{(1-\|\mathbf{\Gamma}_y\|)^2}\|{\boldsymbol{\Lambda}}_{yb}^{-1}\|^2\| {\boldsymbol{\Lambda}}_{ya}\|^2L_{g,1}^2+16\gamma^2L^2\left(2\kappa^2+L_{z^{\star}}^2\right)\frac{\|\mathbf{O}_z\|^2\|\|\mathbf{O}_z^{-1}\|^2\|{\boldsymbol{\Lambda}}_{za}\|^2\|{\boldsymbol{\Lambda}}_{zb}^{-1}\|^2}{(1-\|\mathbf{\Gamma}_z\|)^2}
\\&+4\kappa^2\widetilde{L}^2\left(1+\frac{(1-\theta)^2}{\theta^2\|{\boldsymbol{\Lambda}}_{xb}^{-1}\|^2}\right)\alpha^2\frac{\|\mathbf{O}_x\|^2\|\mathbf{O}_x^{-1}\|^2\|{\boldsymbol{\Lambda}}_{xa}\|^2\|{\boldsymbol{\Lambda}}_{xb}^{-1}\|^2}{(1-\|\mathbf{\Gamma}_x\|)^2},   \\
\eta_2=&3\kappa^2L_{g,1}^2\beta^2\frac{\|\mathbf{O}_y\|^2\|\mathbf{O}_y^{-1}\|^2\|{\boldsymbol{\Lambda}}_{ya}\|^2\|{\boldsymbol{\Lambda}}_{yb}^{-1}\|^2}{(1-\|\mathbf{\Gamma}_y\|)^2}+16L^2\left(2\kappa^2+L_{z^{\star}}^2\right)\gamma^2\frac{\|\mathbf{O}_z\|^2\|\mathbf{O}_z^{-1}\|^2\|{\boldsymbol{\Lambda}}_{za}\|^2\|{\boldsymbol{\Lambda}}_{zb}^{-1}\|^2}{(1-\|\mathbf{\Gamma}_z\|)^2}.  
    \end{align}
\end{subequations}
\end{small}
Suppose that Assumptions~\ref{smooth}-~\ref{var} and Lemmas~\ref{4},~\ref{6},~\ref{10} hold, and $\alpha,\beta$ satisfy
\begin{small}
\begin{equation}\label{l12}
\begin{aligned}
    \alpha^2&\le\frac{(1-\|\mathbf{\Gamma}_x\|)^2}{24\kappa^2\|\mathbf{O}_x^{-1}\|^2\|\mathbf{O}_x\|^2\|{\boldsymbol{\Lambda}}_{xa}\|^2\left[80L^2+18\theta(1-\|\mathbf{\Gamma}_x\|)\left(\theta+\frac{1-\theta}{1-\|\mathbf{\Gamma}_x\|}\right)\sigma_{g,2}^2\right]},\\
     \eta_2\beta^2&\le\frac{1}{1248L_{g,1}^2}.
\end{aligned}
\end{equation}
\end{small}

We have:
\begin{small}

\begin{align} \label{Delta}
&\frac{1}{4}\sum_{k=0}^K\mathbb{E}[\Delta_k]
\\\le&(\eta_1+48\kappa^2L_{y^{\star}}^2\eta_2)\alpha^2\sum_{k=0}^K\mathbb{E}\|\bar{\mathbf{r}}^{k+1}\|^2+\frac{32L_{g,1}^2\eta_2\beta}{\mu_g}\|\bar{\mathbf{y}}^0-\mathbf{y}^{\star}(\bar{x}^{0})\|^2+3(K+1)\eta_2\beta^2\sigma_{g,1}^2
\\&+\frac{\kappa^2\|\mathbf{O}_x^{-1}\|^2\|\mathbf{O}_x\|^2\|{\boldsymbol{\Lambda}}_{xa}\|^2\alpha^2}{1-\|\mathbf{\Gamma}_x\|}\left(\frac{80L^2}{1-\|\mathbf{\Gamma}_x\|}+18\theta\left(\theta+\frac{1-\theta}{1-\|\mathbf{\Gamma}_x\|}\right)\sigma_{g,2}^2\right)\sum_{k=0}^K\mathbb{E}[nI_k]
\\&+\frac{\|\mathbf{O}_z\|^2\|\mathbf{O}_z^{-1}\|^2\|{\boldsymbol{\Lambda}}_{za}\|^2}{1-\|\mathbf{\Gamma}_z\|}\cdot\frac{16\gamma^2(L_{g,1}^2+(1-\|\mathbf{\Gamma}_z\|)\sigma_{g,2}^2)}{1-\|\mathbf{\Gamma}_z\|}\sum_{k=-1}^K\mathbb{E}[nI_k]
\\&+\frac{3\kappa^2\beta^2(K+1)\|\mathbf{O}_y\|^2\|\mathbf{O}_y^{-1}\|^2 \|\mathbf{\Lambda}_{ya}\|^2}{1-\|\mathbf{\Gamma}_y\|} n\sigma_{g,1}^2+\frac{2\kappa^2\|\mathbf{O}_y\|^2\mathbb{E}\|\hat{\mathbf{e}}_y^{0}\|^2}{1-\|\mathbf{\Gamma}_y\|}+\frac{2\|\mathbf{O}_z\|^2\mathbb{E}\|\hat{\mathbf{e}}_z^{0}\|^2}{1-\|\mathbf{\Gamma}_z\|}
\\&+
\frac{\kappa^2\|\mathbf{O}_x\|^2\mathbb{E}\|\hat{\mathbf{e}}_x^{0}\|^2}{1-\|\mathbf{\Gamma}_x\|}+\frac{2\kappa^2\alpha^2\|\mathbf{O}_x\|^2\|\mathbf{O}_x^{-1}\|^2\|{\boldsymbol{\Lambda}}_{xa}\|^2}{(1-\|\mathbf{\Gamma}_x\|)^2}
\left[
\frac{1-\theta}{\theta}\left\|\widetilde{\nabla}\mathbf{\Phi}(\bar{\mathbf{x}}^{0})\right\|^2\right]
\\&+16(K+1)n\gamma^2\frac{\|\mathbf{O}_z\|^2\|\mathbf{O}_z^{-1}\|^2\|{\boldsymbol{\Lambda}}_{za}\|^2}{1-\|\mathbf{\Gamma}_z\|}\left(\frac{L_{f,0}^2}{\mu_g^2}\sigma_{g,2}^2+\sigma_{f,1}^2\right)
\\&+24(K+1)n\kappa^2\alpha^2\theta\left(\theta+\frac{1-\theta}{1-\|\mathbf{\Gamma}_x\|}\right)\frac{\|\mathbf{O}_x\|^2\|\mathbf{O}_x^{-1}\|^2\|{\boldsymbol{\Lambda}}_{xa}\|^2}{1-\|\mathbf{\Gamma}_x\|}\left(\frac{L_{f,0}^2}{\mu_g^2}\sigma_{g,2}^2+\sigma_{f,1}^2\right).
 \end{align}
  
\end{small}

    \begin{proof}
    Adding \eqref{ley}, \eqref{lez} and \eqref{lex} together, we get:
\begin{small}
  \begin{align*}
&\kappa^2\|\mathbf{O}_x\|^2\sum_{k=0}^{K+1}\mathbb{E}[\|\hat{\mathbf{e}}_x^k\|^2]+\kappa^2\|\mathbf{O}_y\|^2\sum_{k=0}^{K+1}\mathbb{E}[\|\hat{\mathbf{e}}_y^k\|^2]+\|\mathbf{O}_z\|^2\sum_{k=0}^{K+1}\mathbb{E}[\|\hat{\mathbf{e}}_z^k\|^2]
\\\le&3\kappa^2\sum_{k=0}^K\frac{\beta^2\|\mathbf{O}_y\|^2\|\mathbf{O}_y^{-1}\|^2}{(1-\|\mathbf{\Gamma}_y\|)^2}\|{\boldsymbol{\Lambda}}_{yb}^{-1}\|^2\| {\boldsymbol{\Lambda}}_{ya}\|^2L_{g,1}^2\mathbb{E}\left[
\|\bar{\mathbf{x}}^{k+1}-\bar{\mathbf{x}}^k\|^2+\|\bar{\mathbf{y}}^{k+1}-\bar{\mathbf{y}}^k\|^2\right]
\\&+\frac{\kappa^2\|\mathbf{O}_x\|^2}{3}\sum_{k=0}^K\mathbb{E}\left[
\|\hat{\mathbf{e}}_x^k\|^2\right]
+\frac{3\kappa^2(K+1)\beta^2\|\mathbf{O}_y\|^2\|\mathbf{O}_y^{-1}\|^2 \|\mathbf{\Lambda}_{ya}\|^2}{1-\|\mathbf{\Gamma}_y\|} n\sigma_{g,1}^2+\frac{2\kappa^2\|\mathbf{O}_y\|^2\mathbb{E}\|\hat{\mathbf{e}}_y^{0}\|^2}{1-\|\mathbf{\Gamma}_y\|}
\\&+\frac{16\gamma^2(L_{g,1}^2+(1-\|\mathbf{\Gamma}\|)\sigma_{g,2}^2)\|\mathbf{O}_z^{-1}\|^2\|\mathbf{O}_z\|^2\|{\boldsymbol{\Lambda}}_{za}\|^2}{(1-\|\mathbf{\Gamma}_z\|)^2}
\sum_{k=0}^K\mathbb{E}\left[\|\bar{\mathbf{z}}^k-\mathbf{z}_{\star}^k\|^2\right]\\
&+\frac{\kappa^2}{3}\sum_{k=0}^K\mathbb{E}(\|\mathbf{O}_x\|^2\|\hat{\mathbf{e}}_x^k\|^2+\|\mathbf{O}_y\|^2\|\hat{\mathbf{e}}_y^{k+1}\|^2)
\\&+\frac{8\gamma^2(2\kappa^2+L_{z^{\star}}^2)L^2\|\mathbf{O}_z\|^2\|\mathbf{O}_z^{-1}\|^2\|{\boldsymbol{\Lambda}}_{zb}^{-1}\|^2\|{\boldsymbol{\Lambda}}_{za}\|^2}{(1-\|\mathbf{\Gamma}_z\|)^2}
\sum_{k=0}^K\mathbb{E}\left[\|\bar{\mathbf{x}}^{k+1}-\bar{\mathbf{x}}^k\|^2\right]
\\&+\frac{16\gamma^2\kappa^2L^2\|\mathbf{O}_z\|^2\|\mathbf{O}_z^{-1}\|^2\|{\boldsymbol{\Lambda}}_{zb}^{-1}\|^2\|{\boldsymbol{\Lambda}}_{za}\|^2}{(1-\|\mathbf{\Gamma}_z\|)^2}
\sum_{k=0}^K\mathbb{E}\left[\|\bar{\mathbf{y}}^{k+2}-\bar{\mathbf{y}}^{k+1}\|^2\right]
\\&+16(K+1)n\gamma^2\frac{\|\mathbf{O}_z\|^2\|\mathbf{O}_z^{-1}\|^2\|{\boldsymbol{\Lambda}}_{za}\|^2}{1-\|\mathbf{\Gamma}_z\|}\left(\frac{L_{f,0}^2}{\mu_g^2}\sigma_{g,2}^2+\sigma_{f,1}^2\right)+\frac{2\|\mathbf{O}_z\|^2\mathbb{E}\|\hat{\mathbf{e}}_z^{0}\|^2}{1-\|\mathbf{\Gamma}_z\|}
\\&+\frac{\kappa^2\|\mathbf{O}_x\|^2\mathbb{E}\|\hat{\mathbf{e}}_x^{0}\|^2}{1-\|\mathbf{\Gamma}_x\|}
+\frac{2\kappa^2\alpha^2\|\mathbf{O}_x\|^2\|\mathbf{O}_x^{-1}\|^2\|{\boldsymbol{\Lambda}}_{xa}\|^2}{(1-\|\mathbf{\Gamma}_x\|)^2}
\left[
\frac{1-\theta}{\theta}\left\|\widetilde{\nabla}\mathbf{\Phi}(\bar{\mathbf{x}}^{0})\right\|^2\right]
\\&+\frac{6n\kappa^2\alpha^2\theta(K+1)\|\mathbf{O}_x\|^2\|\mathbf{O}_x^{-1}\|^2\|{\boldsymbol{\Lambda}}_{xa}\|^2}{1-\|\mathbf{\Gamma}_x\|}\left(\theta+\frac{1-\theta}{1-\|\mathbf{\Gamma}_x\|}\right)\left(\sigma_{f,1}^2+3\sigma_{g,2}^2\frac{L_{f,0}^2}{\mu_g^2}\right)
\\&+\frac{\kappa^2\alpha^2\|\mathbf{O}_x\|^2\|\mathbf{O}_x^{-1}\|^2\|{\boldsymbol{\Lambda}}_{xa}\|^2}{1-\|\mathbf{\Gamma}_x\|}\left(\frac{80L^2}{1-\|\mathbf{\Gamma}_x\|}+18\theta\left(\theta+\frac{1-\theta}{1-\|\mathbf{\Gamma}_x\|}\right)\sigma_{g,2}^2\right)\sum_{k=0}^K\left[\Delta_k+nI_k\right]\\&+\frac{2\kappa^2\widetilde{L}^2\alpha^2\|\mathbf{O}_x\|^2\|\mathbf{O}_x^{-1}\|^2\|{\boldsymbol{\Lambda}}_{xa}\|^2}{(1-\|\mathbf{\Gamma}_x\|)^2}\left( 
\|{\boldsymbol{\Lambda}}_{xb}^{-1}\|^2+\frac{2(1-\theta)^2}{\theta^2}
\right)\sum_{k=0}^{K-1}\mathbb{E}\left[\left\|\bar{\mathbf{x}}^{k+1}-\bar{\mathbf{x}}^k\right\|^2\right]\\
\le&\frac{3}{4}\sum_{k=0}^{K+1}\mathbb{E}\left(\kappa^2\|\mathbf{O}_x\|^2\|\mathbf{\hat{e}}^k_x\|^2+\kappa^2\|\mathbf{O}_y\|^2\|\|\mathbf{\hat{e}}^k_y\|^2+\|\mathbf{O}_z\|^2\|\|\mathbf{\hat{e}}^k_z\|^2\right)
\\&+(\eta_1+48\kappa^2L_{y^{\star}}^2\eta_2)\alpha^2\sum_{k=0}^K\mathbb{E}\|\bar{\mathbf{r}}^{k+1}\|^2
+\eta_2\left(\frac{32L_{g,1}^2\beta}{\mu_g}\|\bar{\mathbf{y}}^0-\mathbf{y}^{\star}(\bar{x}^{0})\|^2+3(K+1)\beta^2\sigma_{g,1}^2\right)
\\&+\frac{\kappa^2\|\mathbf{O}_x^{-1}\|^2\|\mathbf{O}_x\|^2\|{\boldsymbol{\Lambda}}_{xa}\|^2\alpha^2}{1-\|\mathbf{\Gamma}_x\|}\left(\frac{80L^2}{1-\|\mathbf{\Gamma}_x\|}+18\theta\left(\theta+\frac{1-\theta}{1-\|\mathbf{\Gamma}_x\|}\right)\sigma_{g,2}^2\right)\sum_{k=0}^K\mathbb{E}[nI_k]
\\&+\frac{\|\mathbf{O}_z\|^2\|\mathbf{O}_z^{-1}\|^2\|{\boldsymbol{\Lambda}}_{za}\|^2}{1-\|\mathbf{\Gamma}_z\|}\cdot\frac{16\gamma^2(L_{g,1}^2+(1-\|\mathbf{\Gamma}_z\|)\sigma_{g,2}^2)}{1-\|\mathbf{\Gamma}_z\|}\sum_{k=-1}^K\mathbb{E}[nI_k]
\\&+\frac{3\kappa^2\beta^2(K+1)\|\mathbf{O}_y\|^2\|\mathbf{O}_y^{-1}\|^2 \|\mathbf{\Lambda}_{ya}\|^2}{1-\|\mathbf{\Gamma}_y\|} n\sigma_{g,1}^2+\frac{2\kappa^2\|\mathbf{O}_y\|^2\mathbb{E}\|\hat{\mathbf{e}}_y^{0}\|^2}{1-\|\mathbf{\Gamma}_y\|}+\frac{2\|\mathbf{O}_z\|^2\mathbb{E}\|\hat{\mathbf{e}}_z^{0}\|^2}{1-\|\mathbf{\Gamma}_z\|}
\\&+
\frac{\kappa^2\|\mathbf{O}_x\|^2\mathbb{E}\|\hat{\mathbf{e}}_x^{0}\|^2}{1-\|\mathbf{\Gamma}_x\|}+\frac{2\kappa^2\alpha^2\|\mathbf{O}_x\|^2\|\mathbf{O}_x^{-1}\|^2\|{\boldsymbol{\Lambda}}_{xa}\|^2}{(1-\|\mathbf{\Gamma}_x\|)^2}
\left[
\frac{1-\theta}{\theta}\left\|\widetilde{\nabla}\mathbf{\Phi}(\bar{\mathbf{x}}^{0})\right\|^2\right]
\\&+16(K+1)n\gamma^2\frac{\|\mathbf{O}_z\|^2\|\mathbf{O}_z^{-1}\|^2\|{\boldsymbol{\Lambda}}_{za}\|^2}{1-\|\mathbf{\Gamma}_z\|}\left(\frac{L_{f,0}^2}{\mu_g^2}\sigma_{g,2}^2+\sigma_{f,1}^2\right)
\\&+24(K+1)n\kappa^2\alpha^2\theta\left(\theta+\frac{1-\theta}{1-\|\mathbf{\Gamma}_x\|}\right)\frac{\|\mathbf{O}_x\|^2\|\mathbf{O}_x^{-1}\|^2\|{\boldsymbol{\Lambda}}_{xa}\|^2}{1-\|\mathbf{\Gamma}_x\|}\left(\frac{L_{f,0}^2}{\mu_g^2}\sigma_{g,2}^2+\sigma_{f,1}^2\right).
   \end{align*}
\end{small}

where the second inequality uses \eqref{ly2} and
\[\eta_2L_{g,1}^2\beta^2\cdot52+\frac{\kappa^2\alpha^2\|\mathbf{O}_x^{-1}\|^2\|\mathbf{O}_x\|^2\|{\boldsymbol{\Lambda}}_{xa}\|^2}{1-\|\mathbf{\Gamma}_x\|}\left(\frac{80L^2}{1-\|\mathbf{\Gamma}_x\|}+18\theta\left(\theta+\frac{1-\theta}{1-\|\mathbf{\Gamma}_x\|}\right)\sigma_{g,2}^2 \right)\le\frac{1}{12}.\]

Hence:
\begin{align*}
&\frac{1}{4}\sum_{k=0}^K\mathbb{E}[\Delta_k]\\\le&(\eta_1+48\kappa^2L_{y^{\star}}^2\eta_2)\alpha^2\sum_{k=0}^K\mathbb{E}\|\bar{\mathbf{r}}^{k+1}\|^2+\frac{32L_{g,1}^2\eta_2\beta}{\mu_g}\|\bar{\mathbf{y}}^0-\mathbf{y}^{\star}(\bar{x}^{0})\|^2+3(K+1)\eta_2\beta^2\sigma_{g,1}^2
\\&+\frac{\kappa^2\|\mathbf{O}_x^{-1}\|^2\|\mathbf{O}_x\|^2\|{\boldsymbol{\Lambda}}_{xa}\|^2\alpha^2}{1-\|\mathbf{\Gamma}_x\|}\left(\frac{80L^2}{1-\|\mathbf{\Gamma}_x\|}+18\theta\left(\theta+\frac{1-\theta}{1-\|\mathbf{\Gamma}_x\|}\right)\sigma_{g,2}^2\right)\sum_{k=0}^K\mathbb{E}[nI_k]
\\&+\frac{\|\mathbf{O}_z\|^2\|\mathbf{O}_z^{-1}\|^2\|{\boldsymbol{\Lambda}}_{za}\|^2}{1-\|\mathbf{\Gamma}_z\|}\cdot\frac{16\gamma^2(L_{g,1}^2+(1-\|\mathbf{\Gamma}_z\|)\sigma_{g,2}^2)}{1-\|\mathbf{\Gamma}_z\|}\sum_{k=-1}^K\mathbb{E}[nI_k]
\\&+\frac{3\kappa^2\beta^2(K+1)\|\mathbf{O}_y\|^2\|\mathbf{O}_y^{-1}\|^2 \|\mathbf{\Lambda}_{ya}\|^2}{1-\|\mathbf{\Gamma}_y\|} n\sigma_{g,1}^2+\frac{2\kappa^2\|\mathbf{O}_y\|^2\mathbb{E}\|\hat{\mathbf{e}}_y^{0}\|^2}{1-\|\mathbf{\Gamma}_y\|}+\frac{2\|\mathbf{O}_z\|^2\mathbb{E}\|\hat{\mathbf{e}}_z^{0}\|^2}{1-\|\mathbf{\Gamma}_z\|}
\\&+
\frac{\kappa^2\|\mathbf{O}_x\|^2\mathbb{E}\|\hat{\mathbf{e}}_x^{0}\|^2}{1-\|\mathbf{\Gamma}_x\|}+\frac{2\kappa^2\alpha^2\|\mathbf{O}_x\|^2\|\mathbf{O}_x^{-1}\|^2\|{\boldsymbol{\Lambda}}_{xa}\|^2}{(1-\|\mathbf{\Gamma}_x\|)^2}
\left[
\frac{1-\theta}{\theta}\left\|\widetilde{\nabla}\mathbf{\Phi}(\bar{\mathbf{x}}^{0})\right\|^2\right]
\\&+16(K+1)n\gamma^2\frac{\|\mathbf{O}_z\|^2\|\mathbf{O}_z^{-1}\|^2\|{\boldsymbol{\Lambda}}_{za}\|^2}{1-\|\mathbf{\Gamma}_z\|}\left(\frac{L_{f,0}^2}{\mu_g^2}\sigma_{g,2}^2+\sigma_{f,1}^2\right)
\\&+24(K+1)n\kappa^2\alpha^2\theta\left(\theta+\frac{1-\theta}{1-\|\mathbf{\Gamma}_x\|}\right)\frac{\|\mathbf{O}_x\|^2\|\mathbf{O}_x^{-1}\|^2\|{\boldsymbol{\Lambda}}_{xa}\|^2}{1-\|\mathbf{\Gamma}_x\|}\left(\frac{L_{f,0}^2}{\mu_g^2}\sigma_{g,2}^2+\sigma_{f,1}^2\right).
\end{align*}
\end{proof}
\end{lemma}

\subsubsection{Proof of the main theorem}
Before giving the final result of the convergence analysis, we present the following Lemma that combines the results in the analysis of $I_k$ and $\Delta_k$:
\begin{lemma}\label{DeltaI}
Suppose that Assumptions~\ref{smooth}-~\ref{var} and Lemmas~\ref{Er},~\ref{Isum},~\ref{Deltasum} hold. If $\alpha,\beta,\gamma,\theta$ satisfy
\begin{equation}\label{stepsize14}
    \begin{aligned}
\alpha^2&\le \frac{(1-\|\mathbf{\Gamma}_x\|)^2}{16\|\mathbf{O}_x\|^2\|\mathbf{O}_x^{-1}\|^2\|{\boldsymbol{\Lambda}}_{xa}\|^2\left[80L^2+18\theta\left(\theta+\frac{1-\theta}{1-\|\mathbf{\Gamma}_x\|}\right)(1-\|\mathbf{\Gamma}_x\|)\sigma_{g,2}^2\right]\cdot2040\kappa^6},\\
\beta^2\eta_2&\le \frac{1}{1024 L_{g,1}^2},\\
\gamma^2&\le \frac{(1-\|\mathbf{\Gamma}_z\|)^2}{256\|\mathbf{O}_z\|^2\|\mathbf{O}_z^{-1}\|^2\|{\boldsymbol{\Lambda}}_{za}\|^2(L_{g,1}^2+(1-\|\mathbf{\Gamma}_z\|)\sigma_{g,2}^2)\cdot2040\kappa^4}, 
    \end{aligned}
\end{equation}
and
\begin{align}\label{stepsize15}
40\left[
4(\eta_1+48\kappa^2L_{y^{\star}}^2\eta_2)\alpha^2+\frac{1}{1020\kappa^4}\left(\frac{9\alpha^2L_{z^\star}^2}{\gamma^2\mu_g^2}+\frac{438\kappa^4\alpha^2}{\beta^2\mu_g^2}L_{y^{\star}}^2\right)\right]\left(L^2+\frac{\theta\sigma_{g,2}^2}{n}\right)\le\frac{1}{4080\kappa^4},
\end{align}

then we have:
\begin{equation}
    \begin{aligned}
\sum_{k=0}^K\mathbb{E}[\Delta_k+nI_k]
        \lesssim
&\kappa^4
(\eta_1+\kappa^2L_{y^{\star}}^2\eta_2)\alpha^2\left(\frac{n(\Phi(\bar{x}_0)-\inf\Phi)}{\alpha}+\theta(K+1)(\sigma_{f,1}^2+\kappa^2\sigma_{g,2}^2)\right)
\\&+\left(\frac{\alpha^2L_{z^\star}^2}{\gamma^2\mu_g^2}+\frac{\kappa^4\alpha^2}{\beta^2\mu_g^2}L_{y^{\star}}^2\right)\left(\frac{n(\Phi(\bar{x}_0)-\inf\Phi)}{\alpha}+\theta(K+1)(\sigma_{f,1}^2+\kappa^2\sigma_{g,2}^2)\right)
\\&+\frac{\kappa^6\beta^2(K+1)\|\mathbf{O}_y\|^2\|\mathbf{O}_y^{-1}\|^2 \|\mathbf{\Lambda}_{ya}\|^2}{1-\|\mathbf{\Gamma}_y\|} n\sigma_{g,1}^2+ \kappa^4\eta_2\beta^2(K+1)\sigma_{g,1}^2
\\&+\frac{\kappa^6\|\mathbf{O}_y\|^2\mathbb{E}\|\hat{\mathbf{e}}_y^{0}\|^2}{1-\|\mathbf{\Gamma}_y\|}+\frac{\kappa^4\|\mathbf{O}_z\|^2\mathbb{E}\|\hat{\mathbf{e}}_z^{0}\|^2}{1-\|\mathbf{\Gamma}_z\|}+\frac{\kappa^6\|\mathbf{O}_x\|^2\mathbb{E}\|\hat{\mathbf{e}}_x^{0}\|^2}{1-\|\mathbf{\Gamma}_x\|}\\
&+\frac{\kappa^6\alpha^2\|\mathbf{O}_x\|^2\|\mathbf{O}_x^{-1}\|^2\|{\boldsymbol{\Lambda}}_{xa}\|^2}{(1-\|\mathbf{\Gamma}_x\|)^2}\cdot
\frac{1-\theta}{\theta}\left\|\widetilde{\nabla}\mathbf{\Phi}(\bar{\mathbf{x}}^{0})\right\|^2
\\&+(K+1)\kappa^4\gamma^2\frac{\|\mathbf{O}_z\|^2\|\mathbf{O}_z^{-1}\|^2\|{\boldsymbol{\Lambda}}_{za}\|^2n}{1-\|\mathbf{\Gamma}_z\|}(\sigma_{f,1}^2+\kappa^2\sigma_{g,2}^2)
\\&+(K+1)\kappa^6\alpha^2\theta\left(\theta+\frac{1-\theta}{1-\|\mathbf{\Gamma}_x\|}\right)\frac{\|\mathbf{O}_x\|^2\|\mathbf{O}_x^{-1}\|^2\|{\boldsymbol{\Lambda}}_{xa}\|^2n}{1-\|\mathbf{\Gamma}_x\|}(\sigma_{f,1}^2+\kappa^2\sigma_{g,2}^2)
\\&+\frac{(K+1)\gamma}{\mu_g}(\sigma_{f,1}^2+\kappa^2\sigma_{g,2}^2)+\frac{n\|z^1_{\star}\|^2}{\mu_g\gamma}
+\kappa^4\left(\frac{\|\bar{\mathbf{y}}^0-\mathbf{y}^{\star}(\bar{x}^{0})\|^2}{\beta\mu_g}
+\frac{K\sigma_{g,1}^2}{\mu_g}\beta\right).
    \end{aligned}
\end{equation}

    \begin{proof}

Combining \eqref{I} and \eqref{Delta}, we obtain
\begin{equation}\label{deltaI1}
\begin{aligned}
&\sum_{k=0}^K\mathbb{E}[\Delta_k]+\frac{1}{1020\kappa^4}\sum_{k=-1}^K\mathbb{E}[nI_k]
\\\le&4(\eta_1+48\kappa^2L_{y^{\star}}^2\eta_2)\alpha^2\sum_{k=0}^K\mathbb{E}\|\bar{\mathbf{r}}^{k+1}\|^2+\frac{128L_{g,1}^2\eta_2\beta}{\mu_g}\|\bar{\mathbf{y}}^0-\mathbf{y}^{\star}(\bar{x}^{0})\|^2+12(K+1)\eta_2\beta^2\sigma_{g,1}^2
\\&+4\frac{\kappa^2\|\mathbf{O}_x^{-1}\|^2\|\mathbf{O}_x\|^2\|{\boldsymbol{\Lambda}}_{xa}\|^2\alpha^2}{1-\|\mathbf{\Gamma}_x\|}\left(\frac{80L^2}{1-\|\mathbf{\Gamma}_x\|}+18\theta\left(\theta+\frac{1-\theta}{1-\|\mathbf{\Gamma}_x\|}\right)\sigma_{g,2}^2\right)\sum_{k=0}^K\mathbb{E}[nI_k]\\
&+\frac{4\|\mathbf{O}_z\|^2\|\mathbf{O}_z^{-1}\|^2\|{\boldsymbol{\Lambda}}_{za}\|^2}{1-\|\mathbf{\Gamma}_z\|}\cdot\frac{16\gamma^2(L_{g,1}^2+(1-\|\mathbf{\Gamma}_z\|)\sigma_{g,2}^2)}{1-\|\mathbf{\Gamma}_z\|}\sum_{k=-1}^K\mathbb{E}[nI_k]
\\&+\frac{12\kappa^2\beta^2(K+1)\|\mathbf{O}_y\|^2\|\mathbf{O}_y^{-1}\|^2 \|\mathbf{\Lambda}_{ya}\|^2}{1-\|\mathbf{\Gamma}_y\|} n\sigma_{g,1}^2+\frac{8\kappa^2\|\mathbf{O}_y\|^2\mathbb{E}\|\hat{\mathbf{e}}_y^{0}\|^2}{1-\|\mathbf{\Gamma}_y\|}+\frac{8\|\mathbf{O}_z\|^2\mathbb{E}\|\hat{\mathbf{e}}_z^{0}\|^2}{1-\|\mathbf{\Gamma}_z\|}
\\&+
\frac{4\kappa^2\|\mathbf{O}_x\|^2\mathbb{E}\|\hat{\mathbf{e}}_x^{0}\|^2}{1-\|\mathbf{\Gamma}_x\|}+\frac{8\kappa^2\alpha^2\|\mathbf{O}_x\|^2\|\mathbf{O}_x^{-1}\|^2\|{\boldsymbol{\Lambda}}_{xa}\|^2}{(1-\|\mathbf{\Gamma}_x\|)^2}
\left[
\frac{1-\theta}{\theta}\left\|\widetilde{\nabla}\mathbf{\Phi}(\bar{\mathbf{x}}^{0})\right\|^2\right]
\\&+64(K+1)n\gamma^2\frac{\|\mathbf{O}_z\|^2\|\mathbf{O}_z^{-1}\|^2\|{\boldsymbol{\Lambda}}_{za}\|^2}{1-\|\mathbf{\Gamma}_z\|}\left(\frac{L_{f,0}^2}{\mu_g^2}\sigma_{g,2}^2+\sigma_{f,1}^2\right)
\\&+96(K+1)n\kappa^2\alpha^2\theta\left(\theta+\frac{1-\theta}{1-\|\mathbf{\Gamma}_x\|}\right)\frac{\|\mathbf{O}_x\|^2\|\mathbf{O}_x^{-1}\|^2\|{\boldsymbol{\Lambda}}_{xa}\|^2}{1-\|\mathbf{\Gamma}_x\|}\left(\frac{L_{f,0}^2}{\mu_g^2}\sigma_{g,2}^2+\sigma_{f,1}^2\right)
\\&+\frac{1}{1020\kappa^4}\left(\frac{9\alpha^2L_{z^\star}^2}{\gamma^2\mu_g^2}+\frac{438\kappa^4\alpha^2}{\beta^2\mu_g^2}L_{y^{\star}}^2\right)\sum_{k=0}^K\mathbb{E}\|\bar{\mathbf{r}}^k\|^2
+\frac{1}{2}\sum_{k=0}^K\mathbb{E}\left[{\Delta_k}\right]+\frac{1}{1020\kappa^4}\cdot\frac{3n\|z^1_{\star}\|^2}{\mu_g\gamma}
\\&+\frac{(K+1)}{1020\kappa^4} \frac{6\gamma}{\mu_g}\left(3\sigma_{g,2}^2\frac{L_{f,0}^2}{\mu_g^2}+\sigma_{f,1}^2\right)+\frac{73\kappa^4}{1020\kappa^4}\left(\frac{4}{\beta\mu_g}\|\bar{\mathbf{y}}^0-\mathbf{y}^{\star}(\bar{x}^{0})\|^2
+\frac{4K\sigma_{g,1}^2}{\mu_g}\beta\right).
\end{aligned}
\end{equation}

Subtracting the term
\begin{align}
&4\frac{\kappa^2\|\mathbf{O}_x^{-1}\|^2\|\mathbf{O}_x\|^2\|{\boldsymbol{\Lambda}}_{xa}\|^2\alpha^2}{1-\|\mathbf{\Gamma}_x\|}\left(\frac{80L^2}{1-\|\mathbf{\Gamma}_x\|}+18\theta\left(\theta+\frac{1-\theta}{1-\|\mathbf{\Gamma}_x\|}\right)\sigma_{g,2}^2\right)\sum_{k=0}^K\mathbb{E}[n I_k]
\\+&\frac{4\|\mathbf{O}_z\|^2\|\mathbf{O}_z^{-1}\|^2\|{\boldsymbol{\Lambda}}_{za}\|^2}{1-\|\mathbf{\Gamma}_z\|}\frac{16\gamma^2(L_{g,1}^2+(1-\|\mathbf{\Gamma}_z\|)\sigma_{g,2}^2)}{1-\|\mathbf{\Gamma}_z\|}\sum_{k=-1}^K\mathbb{E}[n I_k]+\frac{1}{2}\sum_{k=0}^K\mathbb{E}[\Delta_k]
\end{align}
from both sides of \eqref{deltaI1} and using the restriction of $\alpha,\gamma$ in \eqref{stepsize14}, we can get:
\begin{align*}
\label{combine}
&\frac{1}{2040\kappa^4}\left( \sum_{k=0}^K\mathbb{E}[\Delta_k]+\sum_{k=-1}^K\mathbb{E}[nI_k]\right)
\\\le&4(\eta_1+48\kappa^2L_{y^{\star}}^2\eta_2)\alpha^2\sum_{k=0}^K\mathbb{E}\|\bar{\mathbf{r}}^{k+1}\|^2+\frac{128L_{g,1}^2\eta_2\beta}{\mu_g}\|\bar{\mathbf{y}}^0-\mathbf{y}^{\star}(\bar{x}^{0})\|^2+12(K+1)\eta_2\beta^2\sigma_{g,1}^2
\\&+\frac{12\kappa^2\beta^2(K+1)\|\mathbf{O}_y\|^2\|\mathbf{O}_y^{-1}\|^2 \|\mathbf{\Lambda}_{ya}\|^2}{1-\|\mathbf{\Gamma}_y\|} n\sigma_{g,1}^2+\frac{8\kappa^2\|\mathbf{O}_y\|^2\mathbb{E}\|\hat{\mathbf{e}}_y^{0}\|^2}{1-\|\mathbf{\Gamma}_y\|}+\frac{8\|\mathbf{O}_z\|^2\mathbb{E}\|\hat{\mathbf{e}}_z^{0}\|^2}{1-\|\mathbf{\Gamma}_z\|}
\\&+\frac{4\kappa^2\|\mathbf{O}_x\|^2\mathbb{E}\|\hat{\mathbf{e}}_x^{0}\|^2}{1-\|\mathbf{\Gamma}_x\|}+\frac{8\kappa^2\alpha^2\|\mathbf{O}_x\|^2\|\mathbf{O}_x^{-1}\|^2\|{\boldsymbol{\Lambda}}_{xa}\|^2}{(1-\|\mathbf{\Gamma}_x\|)^2}
\left[
\frac{1-\theta}{\theta}\left\|\widetilde{\nabla}\mathbf{\Phi}(\bar{\mathbf{x}}^{0})\right\|^2\right]
\\&+64(K+1)n\gamma^2\frac{\|\mathbf{O}_z\|^2\|\mathbf{O}_z^{-1}\|^2\|{\boldsymbol{\Lambda}}_{za}\|^2}{1-\|\mathbf{\Gamma}_z\|}\left(\frac{L_{f,0}^2}{\mu_g^2}\sigma_{g,2}^2+\sigma_{f,1}^2\right)
\\&+96(K+1)n\kappa^2\alpha^2\theta\left(\theta+\frac{1-\theta}{1-\|\mathbf{\Gamma}_x\|}\right)\frac{\|\mathbf{O}_x\|^2\|\mathbf{O}_x^{-1}\|^2\|{\boldsymbol{\Lambda}}_{xa}\|^2}{1-\|\mathbf{\Gamma}_x\|}\left(\frac{L_{f,0}^2}{\mu_g^2}\sigma_{g,2}^2+\sigma_{f,1}^2\right)
\\&+\frac{1}{1020\kappa^4}\left(\frac{9\alpha^2L_{z^\star}^2}{\gamma^2\mu_g^2}+\frac{438\kappa^4\alpha^2}{\beta^2\mu_g^2}L_{y^{\star}}^2\right)\sum_{k=0}^K\mathbb{E}\|\bar{\mathbf{r}}^k\|^2+\frac{(K+1)}{1020\kappa^4}\cdot\frac{6\gamma}{\mu_g}\left(3\sigma_{g,2}^2\frac{L_{f,0}^2}{\mu_g^2}+\sigma_{f,1}^2\right)
\\&+\frac{1}{1020\kappa^4}\cdot\frac{3n\|z^1_{\star}\|^2}{\mu_g\gamma}
+\frac{1}{1020\kappa^4}\cdot73\kappa^4\left(\frac{4}{\beta\mu_g}\|\bar{\mathbf{y}}^0-\mathbf{y}^{\star}(\bar{x}^{0})\|^2
+\frac{4K\sigma_{g,1}^2}{\mu_g}\beta\right)
\\\le& \left[
4(\eta_1+48\kappa^2L_{y^{\star}}^2\eta_2)\alpha^2+\frac{1}{1020\kappa^4}\left(\frac{9\alpha^2L_{z^\star}^2}{\gamma^2\mu_g^2}+\frac{438\kappa^4\alpha^2}{\beta^2\mu_g^2}L_{y^{\star}}^2\right)\right]\sum_{k=0}^K\mathbb{E}\|\bar{\mathbf{r}}^{k+1}\|^2
\\&+12(K+1)\eta_2\beta^2\sigma_{g,1}^2
+\frac{12\kappa^2\beta^2(K+1)\|\mathbf{O}_y\|^2\|\mathbf{O}_y^{-1}\|^2 \|\mathbf{\Lambda}_{ya}\|^2}{1-\|\mathbf{\Gamma}_y\|} n\sigma_{g,1}^2+\frac{8\kappa^2\|\mathbf{O}_y\|^2\mathbb{E}\|\hat{\mathbf{e}}_y^{0}\|^2}{1-\|\mathbf{\Gamma}_y\|}
\\&+\frac{8\|\mathbf{O}_z\|^2\mathbb{E}\|\hat{\mathbf{e}}_z^{0}\|^2}{1-\|\mathbf{\Gamma}_z\|}+
\frac{4\kappa^2\|\mathbf{O}_x\|^2\mathbb{E}\|\hat{\mathbf{e}}_x^{0}\|^2}{1-\|\mathbf{\Gamma}_x\|}+\frac{8\kappa^2\alpha^2\|\mathbf{O}_x\|^2\|\mathbf{O}_x^{-1}\|^2\|{\boldsymbol{\Lambda}}_{xa}\|^2}{(1-\|\mathbf{\Gamma}_x\|)^2}
\left[
\frac{1-\theta}{\theta}\left\|\widetilde{\nabla}\mathbf{\Phi}(\bar{\mathbf{x}}^{0})\right\|^2\right]
\\&+64(K+1)n\gamma^2\frac{\|\mathbf{O}_z\|^2\|\mathbf{O}_z^{-1}\|^2\|{\boldsymbol{\Lambda}}_{za}\|^2}{1-\|\mathbf{\Gamma}_z\|}\left(\frac{L_{f,0}^2}{\mu_g^2}\sigma_{g,2}^2+\sigma_{f,1}^2\right)
\\&+96(K+1)n\kappa^2\alpha^2\theta\left(\theta+\frac{1-\theta}{1-\|\mathbf{\Gamma}_x\|}\right)\frac{\|\mathbf{O}_x\|^2\|\mathbf{O}_x^{-1}\|^2\|{\boldsymbol{\Lambda}}_{xa}\|^2}{1-\|\mathbf{\Gamma}_x\|}\left(\frac{L_{f,0}^2}{\mu_g^2}\sigma_{g,2}^2+\sigma_{f,1}^2\right)
\\&+(K+1)\frac{18\gamma}{\mu_g\cdot1020\kappa^4}\left(\frac{L_{f,0}^2}{\mu_g^2}\sigma_{g,2}^2+\sigma_{f,1}^2\right)
\\&+\frac{1}{1020\kappa^4}\cdot\frac{3n\|z^1_{\star}\|^2}{\mu_g\gamma}
+\frac{1}{1020\kappa^4}\cdot73\kappa^4\left(\frac{8}{\beta\mu_g}\|\bar{\mathbf{y}}^0-\mathbf{y}^{\star}(\bar{x}^{0})\|^2
+\frac{4K\sigma_{g,1}^2}{\mu_g}\beta\right),
\end{align*}

where the second inequality holds since
\[\frac{128L_{g,1}^2\eta_2\beta}{\mu_g}\le \frac{1}{8\beta \mu_g}.
\]

Then taking \eqref{barr} into the concern, we know:
\begin{small}
\begin{align*}
        &\frac{1}{2040\kappa^4}\left( \sum_{k=0}^K\mathbb{E}[\Delta_k]+\sum_{k=-1}^K\mathbb{E}[nI_k]\right)
        \\\le&40\left[
4(\eta_1+48\kappa^2L_{y^{\star}}^2\eta_2)\alpha^2+\frac{1}{1020\kappa^4}\left(\frac{9\alpha^2L_{z^\star}^2}{\gamma^2\mu_g^2}+\frac{438\kappa^4\alpha^2}{\beta^2\mu_g^2}L_{y^{\star}}^2\right)\right]\left(L^2+\frac{\theta\sigma_{g,2}^2}{n}\right)\sum_{k=0}^K\mathbb{E}[\Delta_k+nI_k]\\
&+4\left[
4(\eta_1+48\kappa^2L_{y^{\star}}^2\eta_2)\alpha^2+\frac{1}{1020\kappa^4}\left(\frac{9\alpha^2L_{z^\star}^2}{\gamma^2\mu_g^2}+\frac{438\kappa^4\alpha^2}{\beta^2\mu_g^2}L_{y^{\star}}^2\right)\right]\frac{n(\Phi(\bar{x}_0)-\inf\Phi)}{\alpha}\\
&+12\theta(K+1)\left[
4(\eta_1+48\kappa^2L_{y^{\star}}^2\eta_2)\alpha^2+\frac{1}{1020\kappa^4}\left(\frac{9\alpha^2L_{z^\star}^2}{\gamma^2\mu_g^2}+\frac{438\kappa^4\alpha^2}{\beta^2\mu_g^2}L_{y^{\star}}^2\right)\right]\left(\sigma_{f,1}^2+2\sigma_{g,2}^2\frac{L_{f,0}^2}{\mu_g^2}\right)\\
&+12(K+1)\eta_2\beta^2\sigma_{g,1}^2
+\frac{12\kappa^2\beta^2(K+1)\|\mathbf{O}_y\|^2\|\mathbf{O}_y^{-1}\|^2 \|\mathbf{\Lambda}_{ya}\|^2}{1-\|\mathbf{\Gamma}_y\|} n\sigma_{g,1}^2+\frac{8\kappa^2\|\mathbf{O}_y\|^2\mathbb{E}\|\hat{\mathbf{e}}_y^{0}\|^2}{1-\|\mathbf{\Gamma}_y\|}
\\&+\frac{8\|\mathbf{O}_z\|^2\mathbb{E}\|\hat{\mathbf{e}}_z^{0}\|^2}{1-\|\mathbf{\Gamma}_z\|}+
\frac{4\kappa^2\|\mathbf{O}_x\|^2\mathbb{E}\|\hat{\mathbf{e}}_x^{0}\|^2}{1-\|\mathbf{\Gamma}_x\|}+\frac{8\kappa^2\alpha^2\|\mathbf{O}_x\|^2\|\mathbf{O}_x^{-1}\|^2\|{\boldsymbol{\Lambda}}_{xa}\|^2}{(1-\|\mathbf{\Gamma}_x\|)^2}
\left[
\frac{1-\theta}{\theta}\left\|\widetilde{\nabla}\mathbf{\Phi}(\bar{\mathbf{x}}^{0})\right\|^2\right]
\\&+64(K+1)n\gamma^2\frac{\|\mathbf{O}_z\|^2\|\mathbf{O}_z^{-1}\|^2\|{\boldsymbol{\Lambda}}_{za}\|^2}{1-\|\mathbf{\Gamma}_z\|}\left(\frac{L_{f,0}^2}{\mu_g^2}\sigma_{g,2}^2+\sigma_{f,1}^2\right)
\\&+96(K+1)n\kappa^2\alpha^2\theta\left(\theta+\frac{1-\theta}{1-\|\mathbf{\Gamma}_x\|}\right)\frac{\|\mathbf{O}_x\|^2\|\mathbf{O}_x^{-1}\|^2\|{\boldsymbol{\Lambda}}_{xa}\|^2}{1-\|\mathbf{\Gamma}_x\|}\left(\frac{L_{f,0}^2}{\mu_g^2}\sigma_{g,2}^2+\sigma_{f,1}^2\right)
\\&+(K+1)\frac{18\gamma}{\mu_g\cdot1020\kappa^4}\left(\frac{L_{f,0}^2}{\mu_g^2}\sigma_{g,2}^2+\sigma_{f,1}^2\right)
\\&+\frac{1}{1020\kappa^4}\cdot\frac{3n\|z^1_{\star}\|^2}{\mu_g\gamma}
+\frac{1}{1020\kappa^4}\cdot73\kappa^4\left(\frac{8}{\beta\mu_g}\|\bar{\mathbf{y}}^0-\mathbf{y}^{\star}(\bar{x}^{0})\|^2
+\frac{4K\sigma_{g,1}^2}{\mu_g}\beta\right),
\end{align*}
\end{small}

Since \[40\left[
4(\eta_1+48\kappa^2L_{y^{\star}}^2\eta_2)\alpha^2+\frac{1}{1020\kappa^4}\left(\frac{9\alpha^2L_{z^\star}^2}{\gamma^2\mu_g^2}+\frac{438\kappa^4\alpha^2}{\beta^2\mu_g^2}L_{y^{\star}}^2\right)\right]\left(L^2+\frac{\theta\sigma_{g,2}^2}{n}\right)\le\frac{1}{4080\kappa^4},\]
it follows that
\begin{equation}
    \begin{aligned}
&\sum_{k=0}^K\mathbb{E}[\Delta_k+nI_k]
        \\\lesssim
&\left[\kappa^4
(\eta_1+\kappa^2L_{y^{\star}}^2\eta_2)\alpha^2+\left(\frac{\alpha^2L_{z^\star}^2}{\gamma^2\mu_g^2}+\frac{\kappa^4\alpha^2}{\beta^2\mu_g^2}L_{y^{\star}}^2\right)\right]\frac{n(\Phi(\bar{x}_0)-\inf\Phi)}{\alpha}
\\&+\theta(K+1)\left[\kappa^4
(\eta_1+\kappa^2L_{y^{\star}}^2\eta_2)\alpha^2+\left(\frac{\alpha^2L_{z^\star}^2}{\gamma^2\mu_g^2}+\frac{\kappa^4\alpha^2}{\beta^2\mu_g^2}L_{y^{\star}}^2\right)\right](\sigma_{f,1}^2+\kappa^2\sigma_{g,2}^2)
\\&+\frac{\kappa^6\beta^2(K+1)\|\mathbf{O}_y\|^2\|\mathbf{O}_y^{-1}\|^2 \|\mathbf{\Lambda}_{ya}\|^2}{1-\|\mathbf{\Gamma}_y\|} n\sigma_{g,1}^2+ \kappa^4\eta_2\beta^2(K+1)\sigma_{g,1}^2+\frac{\kappa^6\|\mathbf{O}_y\|^2\mathbb{E}\|\hat{\mathbf{e}}_y^{0}\|^2}{1-\|\mathbf{\Gamma}_y\|}
\\&+\frac{\kappa^4\|\mathbf{O}_z\|^2\mathbb{E}\|\hat{\mathbf{e}}_z^{0}\|^2}{1-\|\mathbf{\Gamma}_z\|}+\frac{\kappa^6\|\mathbf{O}_x\|^2\mathbb{E}\|\hat{\mathbf{e}}_x^{0}\|^2}{1-\|\mathbf{\Gamma}_x\|}
+\frac{\kappa^6\alpha^2\|\mathbf{O}_x\|^2\|\mathbf{O}_x^{-1}\|^2\|{\boldsymbol{\Lambda}}_{xa}\|^2}{(1-\|\mathbf{\Gamma}_x\|)^2}
\frac{1-\theta}{\theta}\left\|\widetilde{\nabla}\mathbf{\Phi}(\bar{\mathbf{x}}^{0})\right\|^2
\\&+\left[\kappa^4\gamma^2\frac{\|\mathbf{O}_z\|^2\|\mathbf{O}_z^{-1}\|^2\|{\boldsymbol{\Lambda}}_{za}\|^2n}{1-\|\mathbf{\Gamma}_z\|}+\frac{\gamma}{\mu_g}\right](K+1)(\sigma_{f,1}^2+\kappa^2\sigma_{g,2}^2)
\\&+(K+1)\kappa^6\alpha^2\theta\left(\theta+\frac{1-\theta}{1-\|\mathbf{\Gamma}_x\|}\right)\frac{\|\mathbf{O}_x\|^2\|\mathbf{O}_x^{-1}\|^2\|{\boldsymbol{\Lambda}}_{xa}\|^2n}{1-\|\mathbf{\Gamma}_x\|}(\sigma_{f,1}^2+\kappa^2\sigma_{g,2}^2)
\\&+\frac{n\|z^1_{\star}\|^2}{\mu_g\gamma}
+\kappa^4\left(\frac{1}{\beta\mu_g}\|\bar{\mathbf{y}}^0-\mathbf{y}^{\star}(\bar{x}^{0})\|^2
+\frac{K\sigma_{g,1}^2}{\mu_g}\beta\right).
    \end{aligned}
\end{equation}
Then, we finish the proof of this lemma.
  \end{proof}
\end{lemma}

Finally, we can give the proof of Lemma~\ref{Stochasticrate}, which is a detailed version of Theorem~\ref{thm1}:

\begin{lemma}[Detailed version of Theorem~\ref{thm1}]\label{Stochasticrate}
    Suppose that  Assumptions~\ref{smooth}-~\ref{var} hold. Then there exist constant step-sizes $\alpha,\,\beta,\,\gamma,\,\theta$, such that
  \begin{small}  
  \begin{equation}
  \begin{aligned}
&\frac{1}{K+1}\sum_{k=0}^K\mathbb{E}\|\nabla\Phi(\bar{x}^k)\|^2
\\\lesssim&\frac{\kappa^5\sigma}{\sqrt{{nK}}}+\kappa^{\frac{16}{3}}\left[\left(\frac{\|\mathbf{O}_y\|^2\|\mathbf{O}_y^{-1}\|^2\|{\boldsymbol{\Lambda}}_{ya}\|^2}{1-\|\mathbf{\Gamma}_y\|}\right)^{\frac{1}{3}}+\left(\frac{\|\mathbf{O}_z\|^2\|\mathbf{O}_z^{-1}\|^2\|{\boldsymbol{\Lambda}}_{za}\|^2}{1-\|\mathbf{\Gamma}_z\|}\right)^{\frac{1}{3}}\right]\frac{\sigma^{\frac{2}{3}}}{K^{\frac{2}{3}}}
\\&+\kappa^{\frac{7}{2}}\left(\frac{\|\mathbf{O}_x\|\|\mathbf{O}_x^{-1}\|\|{\boldsymbol{\Lambda}}_{xa}\|}{1-\|\mathbf{\Gamma}_y\|}\right)^{\frac{1}{2}}\frac{\sigma^{\frac{1}{2}}}{K^{\frac{3}{4}}}
\\&+\left[\kappa^{\frac{26}{5}}\left(\frac{\|\mathbf{O}_y\|^2\|\mathbf{O}_y^{-1}\|^2\|{\boldsymbol{\Lambda}}_{ya}\|^2\|{\boldsymbol{\Lambda}}_{yb}^{-1}\|^2}{n(1-\|\mathbf{\Gamma}_y\|)^2}\right)^{\frac{1}{5}}+\kappa^{6}\left(\frac{\|\mathbf{O}_z\|^2\|\mathbf{O}_z^{-1}\|^2\|{\boldsymbol{\Lambda}}_{za}\|^2\|{\boldsymbol{\Lambda}}_{zb}^{-1}\|^2}{n(1-\|\mathbf{\Gamma}_z\|)^2}\right)^{\frac{1}{5}}\right]\frac{\sigma^{\frac{2}{5}}}{K^{\frac{4}{5}}}
\\&+\left[\kappa^{\frac{16}{3}}\left(\frac{\|\mathbf{O}_y\|^2\|\mathbf{O}_y^{-1}\|^2\|{\boldsymbol{\Lambda}}_{ya}\|^2\|{\boldsymbol{\Lambda}}_{yb}^{-1}\|^2\zeta^y_0}{1-\|\mathbf{\Gamma}_y\|}\right)^{\frac{1}{3}}
+\kappa^{\frac{14}{3}}\left(\frac{\|\mathbf{O}_z\|^2\|\mathbf{O}_z^{-1}\|^2\|{\boldsymbol{\Lambda}}_{za}\|^2\|{\boldsymbol{\Lambda}}_{zb}^{-1}\|^2\zeta^z_0}{1-\|\mathbf{\Gamma}_z\|}\right)^{\frac{1}{3}}\right.\\&\left.\quad+\kappa^{\frac{8}{3}}\left(\frac{\|\mathbf{O}_x\|^2\|\mathbf{O}_x^{-1}\|^2\|{\boldsymbol{\Lambda}}_{xa}\|^2\|{\boldsymbol{\Lambda}}_{xb}^{-1}\|^2\zeta^x_0}{1-\|\mathbf{\Gamma}_x\|}\right)^{\frac{1}{3}}\right]\frac{1}{K}
+\left(\kappa C_\alpha+\kappa^4C_\theta\right)\frac{1}{K},
  \end{aligned}
\end{equation}
\end{small}
where $\sigma=\max\{\sigma_{f,1},\sigma_{g,1},\sigma_{g,2}\}$, $C_\alpha, C_\theta$ are defined as:
\begin{small}  
\begin{align*}
    C_\alpha=&L_{\nabla \Phi}+\kappa^3\frac{\|\mathbf{O}_x\|\|\mathbf{O}_x^{-1}\|\|{\boldsymbol{\Lambda}}_{xa}\|L}{1-\|\mathbf{\Gamma}_x\|}
+\kappa^3L\left(\frac{\|\mathbf{O}_x\|\|\mathbf{O}_x^{-1}\|\|{\boldsymbol{\Lambda}}_{xa}\|\|{\boldsymbol{\Lambda}}_{xb}^{-1}\|}{1-\|\mathbf{\Gamma}_x\|}\right)^{\frac{1}{2}}+\kappa^4\left(\frac{L_{g,1}^2}{\mu_g}+\frac{\sigma_{g,1}^2}{n\mu_g}\right)
\\&+\kappa^4\frac{\|\mathbf{O}_y\|\|\mathbf{O}_y^{-1}\|\|{\boldsymbol{\Lambda}}_{ya}\|L_{g,1}}{1-\|\mathbf{\Gamma}_y\|}
+\kappa^\frac{9}{2}L_{g,1}\left(\frac{\|\mathbf{O}_y\|\|\mathbf{O}_y^{-1}\|\|{\boldsymbol{\Lambda}}_{ya}\|\|{\boldsymbol{\Lambda}}_{yb}^{-1}\|}{1-\|\mathbf{\Gamma}_y\|}\right)^{\frac{1}{2}}+\kappa^4\left(\mu_g+\frac{\mu_g^2\sigma_{g,1}^2}{nL_{g,1}^2}\right)
\\&+\kappa^{6}\frac{\|\mathbf{O}_z\|\|\mathbf{O}_z^{-1}\|\|{\boldsymbol{\Lambda}}_{za}\|\sqrt{L^2+(1-\|\mathbf{\Gamma}_z\|)\sigma_{g,2}^2}}{1-\|\mathbf{\Gamma}_z\|}+\kappa^{\frac{11}{2}}L\left(\frac{\|\mathbf{O}_z\|\|\mathbf{O}_z^{-1}\|\|{\boldsymbol{\Lambda}}_{za}\|\|{\boldsymbol{\Lambda}}_{zb}^{-1}\|}{1-\|\mathbf{\Gamma}_z\|}\right)^{\frac{1}{2}},
\\C_\theta=&\frac{\sigma
_{g,2}^2}{ n L_{g,1}^2}+\frac{\sigma
_{g,2}^2}{ L^2}+1
\end{align*}
\end{small}

    \begin{proof}
    Take $L_1=L^2 +\left(\theta(1-\theta)+L_{\nabla\Phi}\alpha\theta^2\right)\dfrac{\sigma_{g,2}^2}{n}$ 
and use the conclusion of Lemmas~\ref{new r} and~\ref{DeltaI}, we get: 
\begin{small}  
  \begin{equation}\label{f1}
  \begin{aligned}
&\frac{1}{K+1}\sum_{k=0}^K\mathbb{E}\|\nabla\Phi(\bar{x}^k)\|^2
\\ \lesssim& \frac{\Phi(\bar{x}_0)-\inf \Phi}{\alpha(K+1)}+\frac{1}{n}\left(\theta(1-\theta)+L_{\nabla\Phi}\alpha\theta^2\right) (\sigma_{f,1}^2+\kappa^2\sigma_{g,2}^2)
+\frac{(1-\theta)^2}{\theta(K+1)}\|\nabla \Phi\left(\bar{x}^{0}\right)\|^2
 \\&+
L_1\left[\kappa^4
(\eta_1+\kappa^2L_{y^{\star}}^2\eta_2)\alpha^2+\left(\frac{\alpha^2L_{z^\star}^2}{\gamma^2\mu_g^2}+\frac{\kappa^4\alpha^2}{\beta^2\mu_g^2}L_{y^{\star}}^2\right)\right]\left(\frac{\Phi(\bar{x}_0)-\inf\Phi}{\alpha (K+1)}+\frac{\theta}{n}(\sigma_{f,1}^2+\kappa^2\sigma_{g,2}^2)\right)
\\&+L_1\frac{\kappa^6\beta^2\|\mathbf{O}_y\|^2\|\mathbf{O}_y^{-1}\|^2 \|\mathbf{\Lambda}_{ya}\|^2}{1-\|\mathbf{\Gamma}_y\|} \sigma_{g,1}^2+L_1 \kappa^4\eta_2\beta^2\frac{\sigma_{g,1}^2}{n}
\\&+\frac{L_1}{K+1}\left[\frac{\kappa^6\|\mathbf{O}_y\|^2\mathbb{E}\|\hat{\mathbf{e}}_y^{0}\|^2}{n(1-\|\mathbf{\Gamma}_y\|)}+\frac{\kappa^4\|\mathbf{O}_z\|^2\mathbb{E}\|\hat{\mathbf{e}}_z^{0}\|^2}{n(1-\|\mathbf{\Gamma}_z\|)}+\frac{\kappa^6\|\mathbf{O}_x\|^2\mathbb{E}\|\hat{\mathbf{e}}_x^{0}\|^2}{n(1-\|\mathbf{\Gamma}_x\|)}\right]\\
&+L_1\frac{\kappa^6\alpha^2\|\mathbf{O}_x\|^2\|\mathbf{O}_x^{-1}\|^2\|{\boldsymbol{\Lambda}}_{xa}\|^2}{(K+1)(1-\|\mathbf{\Gamma}_x\|)^2}\cdot\frac{1-\theta}{\theta}\cdot\frac{\left\|\widetilde{\nabla}\mathbf{\Phi}(\bar{\mathbf{x}}^{0})\right\|^2}{n}
\\&+L_1\left[\kappa^4\gamma^2\frac{\|\mathbf{O}_z\|^2\|\mathbf{O}_z^{-1}\|^2\|{\boldsymbol{\Lambda}}_{za}\|^2}{1-\|\mathbf{\Gamma}_z\|}+\frac{\gamma}{n\mu_g}\right](\sigma_{f,1}^2+\kappa^2\sigma_{g,2}^2)
\\&+L_1\kappa^6\alpha^2\theta\left(\theta+\frac{1-\theta}{1-\|\mathbf{\Gamma}_x\|}\right)\frac{\|\mathbf{O}_x\|^2\|\mathbf{O}_x^{-1}\|^2\|{\boldsymbol{\Lambda}}_{xa}\|^2}{1-\|\mathbf{\Gamma}_x\|}(\sigma_{f,1}^2+\kappa^2\sigma_{g,2}^2)
\\&+L_1\frac{\|z^1_{\star}\|^2}{(K+1)\mu_g\gamma}
+L_1\kappa^4\left(\frac{1}{\beta\mu_g(K+1)}\|\bar{y}_{0}-y^{\star}(\bar{x}^{0})\|^2
+\frac{\sigma_{g,1}^2}{n\mu_g}\beta\right).
  \end{aligned}
\end{equation}
\end{small}

Define:
\begin{equation}
  \begin{aligned}
\zeta^y_0=&\frac{1}{n}\sum_{i=1}^n\|\nabla_2 g_i(\bar{x}_0,\bar{y}_0)-\nabla_2 g(\bar{x}_0,\bar{y}_0)\|^2,
\\\zeta^z_0=&\frac{1}{n}\sum_{i=1}^n\mathbb{E}\left[\|\nabla_{22}^2 g_i(\bar{x}_0,\bar{y}_1)-\nabla_{22}^2 g(\bar{x}_0,\bar{y}_1)\|^2\|z_{\star}^{1}\|^2+\|\nabla_{2}g_i(\bar{x}_0,\bar{y}_1)-\nabla_{2} g(\bar{x}_0,\bar{y}_1)\|^2\right],
\\\zeta^x_0=&\frac{1}{n}\sum_{i=1}^n\|\nabla_{1} f_i(\bar{x}_0,{y}^{\star}(\bar{x}_0))-\nabla_{1}f(\bar{x}_0,{y}^{\star}(\bar{x}_0))\|^2\\
&+\frac{1}{n}\sum_{i=1}^n\|\nabla_{12}^2g_i(\bar{x}_0,{y}^{\star}(\bar{x}_0))-\nabla_{12} ^2g(\bar{x}_0,{y}^{\star}(\bar{x}_0))\|^2\|z_{\star}^{1}\|^2,
\\\hat{\zeta}_0=&\frac{1}{n}\left\|\widetilde{\nabla}\mathbf{\Phi}(\bar{\mathbf{x}}^{0})\right\|^2.
\end{aligned}
\end{equation} 
Then  we take:
\begin{align}\label{stepsizeconstants}
\alpha_1&=\kappa^{-4}\sqrt{\frac{n}{K\sigma^2}},\\
\alpha_{x,2}&=\left(\frac{(1-\|\mathbf{\Gamma}_x\|)^2}{\kappa^{10}K\|\mathbf{O}_x\|^2\|\mathbf{O}_x^{-1}\|^2\|{\boldsymbol{\Lambda}}_{xa}\|^2\sigma^2}\right)^{\frac{1}{4}}\\
\alpha_{y,2}&=\left(\frac{1-\|\mathbf{\Gamma}_y\|}{\kappa^{13}K\|\mathbf{O}_y\|^2\|\mathbf{O}_y^{-1}\|^2\|{\boldsymbol{\Lambda}}_{ya}\|^2\sigma_{g,1}^2}\right)^{\frac{1}{3}},\\ 
\alpha_{y,3}&=\left(\frac{n(1-\|\mathbf{\Gamma}_y\|)^2}{\kappa^{21}K\|\mathbf{O}_y\|^2\|\mathbf{O}_y^{-1}\|^2\|{\boldsymbol{\Lambda}}_{ya}\|^2\|{\boldsymbol{\Lambda}}_{yb}^{-1}\|^2\sigma_{g,1}^2}\right)^{\frac{1}{5}},
\\\alpha_{z,2}&=\left(\frac{1-\|\mathbf{\Gamma}_z\|}{\kappa^{13}K\|\mathbf{O}_z\|^2\|\mathbf{O}_z^{-1}\|^2\|{\boldsymbol{\Lambda}}_{za}\|^2\sigma^2}\right)^{\frac{1}{3}} 
,\\ 
\alpha_{z,3}&=\left(\frac{n(1-\|\mathbf{\Gamma}_z\|)^2}{\kappa^{25}K\|\mathbf{O}_z\|^2\|\mathbf{O}_z^{-1}\|^2\|{\boldsymbol{\Lambda}}_{za}\|^2\|{\boldsymbol{\Lambda}}_{zb}^{-1}\|^2\sigma_{g,1}^2}\right)^{\frac{1}{5}},
\\\alpha_{yb,2}&=\left(\frac{1-\|\mathbf{\Gamma}_y\|}{\kappa^{13}\|\mathbf{O}_y\|^2\|\mathbf{O}_y^{-1}\|^2\|{\boldsymbol{\Lambda}}_{ya}\|^2\|{\boldsymbol{\Lambda}}_{yb}^{-1}\|^2\zeta^y_0}\right)^{\frac{1}{3}},\\ 
\alpha_{zb,2}&=\left(\frac{1-\|\mathbf{\Gamma}_z\|}{\kappa^{11}\|\mathbf{O}_z\|^2\|\mathbf{O}_z^{-1}\|^2\|{\boldsymbol{\Lambda}}_{za}\|^2\|{\boldsymbol{\Lambda}}_{zb}^{-1}\|^2\zeta^z_0}\right)^{\frac{1}{3}},
\\\alpha_{xb,2}&=\left(\frac{1-\|\mathbf{\Gamma}_x\|}{\kappa^{5}\|\mathbf{O}_x\|^2\|\mathbf{O}_x^{-1}\|^2\|{\boldsymbol{\Lambda}}_{xa}\|^2\|{\boldsymbol{\Lambda}}_{xb}^{-1}\|^2\zeta^x_0}\right)^{\frac{1}{3}},\\
\theta_1&=\left(\frac{n\kappa^2\hat{\zeta_0}}{K\sigma^2}\right)^{\frac{1}{2}},\\
\theta_2&=\kappa^3\alpha_{x,2},
 \end{align}

and
\begin{equation}\label{step_size}
  \begin{aligned}
  \theta=&\left(C_\theta+\frac{1}{\theta_1}+\frac{1}{\theta_2}\right)^{-1},
\\\alpha=&
\Theta\left(C_\alpha+\frac{\sqrt{1-\theta}}{\theta}\kappa^3+\frac{1}{\alpha_1}+\frac{1}{\alpha_{y,2}}+\frac{1}{\alpha_{y,3}}+\frac{1}{\alpha_{z,3}}+\frac{1}{\alpha_{yb,2}}+\frac{1}{\alpha_{zb,2}}+\frac{1}{\alpha_{xb,2}}+\frac{1}{\alpha_{z,2}}+\frac{1}{\alpha_{x,2}}\right)^{-1},
\\\beta=&\Theta\left(\kappa^4\alpha\right),
\\\gamma=&\Theta\left(\kappa^4\alpha\right),
 \end{aligned}
\end{equation}
\newpage
It yields $L_1=\Theta(L^2)$, and \eqref{l15}, \eqref{betady}, \eqref{gammaz}, \eqref{l11}, \eqref{betay}, \eqref{gammaez}, \eqref{l12}, \eqref{stepsize14}, and \eqref{stepsize15} hold. It implies that the restrictions on the step-sizes $\alpha,\beta,\gamma,\theta$ in all previous lemma conditions hold. Thus all previous lemmas hold. We obtain:
  \begin{align}
&\frac{1}{K+1}\sum_{k=0}^K\mathbb{E}\|\nabla\Phi(\bar{x}^k)\|^2
\\ \lesssim& \frac{\Phi(\bar{x}_0)-\inf \Phi}{\alpha K}+\frac{\theta}{n}(\sigma_{f,1}^2+\kappa^2\sigma_{g,2}^2)
+\frac{\hat{\zeta}_0}{\theta K}+\frac{\kappa^6\beta^2\|\mathbf{O}_y\|^2\|\mathbf{O}_y^{-1}\|^2 \|\mathbf{\Lambda}_{ya}\|^2}{1-\|\mathbf{\Gamma}_y\|} \sigma_{g,1}^2+ \eta_2\kappa^4\beta^2\frac{\sigma_{g,1}^2}{n}
\\&+\frac{1}{K}\left[\frac{\kappa^6\|\mathbf{O}_y\|^2\mathbb{E}\|\hat{\mathbf{e}}_y^{0}\|^2}{n(1-\|\mathbf{\Gamma}_y\|)}+\frac{\kappa^4\|\mathbf{O}_z\|^2\mathbb{E}\|\hat{\mathbf{e}}_z^{0}\|^2}{n(1-\|\mathbf{\Gamma}_z\|)}+\frac{\kappa^6\|\mathbf{O}_x\|^2\mathbb{E}\|\hat{\mathbf{e}}_x^{0}\|^2}{n(1-\|\mathbf{\Gamma}_x\|)}\right]
\\&+\left[\kappa^4\gamma^2\frac{\|\mathbf{O}_z\|^2\|\mathbf{O}_z^{-1}\|^2\|{\boldsymbol{\Lambda}}_{za}\|^2}{1-\|\mathbf{\Gamma}_z\|}+\frac{\gamma}{n\mu_g}\right](\sigma_{f,1}^2+\kappa^2\sigma_{g,2}^2)
\\&+\kappa^6\alpha^2\theta\left(\theta+\frac{1-\theta}{1-\|\mathbf{\Gamma}_x\|}\right)\frac{\|\mathbf{O}_x\|^2\|\mathbf{O}_x^{-1}\|^2\|{\boldsymbol{\Lambda}}_{xa}\|^2}{1-\|\mathbf{\Gamma}_x\|}(\sigma_{f,1}^2+\kappa^2\sigma_{g,2}^2)
\\&+\frac{\|z^1_{\star}\|^2}{(K+1)\mu_g\gamma}
+\kappa^4\left(\frac{1}{\beta\mu_g(K+1)}\|\bar{y}_{0}-y^{\star}(\bar{x}^{0})\|^2
+\frac{\sigma_{g,1}^2}{n\mu_g}\beta\right)
\\ \lesssim& \frac{\theta}{n}(\sigma_{f,1}^2+\kappa^2\sigma_{g,2}^2)
+\frac{1}{\theta K}+\frac{\kappa}{\alpha K}
+\frac{\kappa^9\sigma_{g,1}^2}{n}\alpha
+\frac{\kappa^{14}\alpha^2\|\mathbf{O}_y\|^2\|\mathbf{O}_y^{-1}\|^2 \|\mathbf{\Lambda}_{ya}\|^2}{1-\|\mathbf{\Gamma}_y\|} \sigma_{g,1}^2
\\&+ \left(\kappa^{10}\frac{\|\mathbf{O}_y\|^2\|\mathbf{O}_y^{-1}\|^2\|{\boldsymbol{\Lambda}}_{ya}\|^2\|{\boldsymbol{\Lambda}}_{yb}^{-1}\|^2}{(1-\|\mathbf{\Gamma}_y\|)^2}+\kappa^{14}\frac{\|\mathbf{O}_z\|^2\|\mathbf{O}_z^{-1}\|^2\|{\boldsymbol{\Lambda}}_{za}\|^2\|{\boldsymbol{\Lambda}}_{zb}^{-1}\|^2}{(1-\|\mathbf{\Gamma}_z\|)^2}\right)\kappa^{12}\alpha^4\frac{\sigma_{g,1}^2}{n}
\\&+\alpha^2\frac{\kappa^{14}\|\mathbf{O}_y\|^2\|\mathbf{O}_y^{-1}\|^2\|{\boldsymbol{\Lambda}}_{ya}\|^2\|{\boldsymbol{\Lambda}}_{yb}^{-1}\|^2\zeta^y_0}{K(1-\|\mathbf{\Gamma}_y\|)}+\alpha^2\frac{\kappa^{12}\|\mathbf{O}_z\|^2\|\mathbf{O}_z^{-1}\|^2\|{\boldsymbol{\Lambda}}_{za}\|^2\|{\boldsymbol{\Lambda}}_{zb}^{-1}\|^2\zeta^z_0}{K(1-\|\mathbf{\Gamma}_z\|)}
\\&+\alpha^2\frac{\kappa^6\|\mathbf{O}_x\|^2\|\mathbf{O}_x^{-1}\|^2\|{\boldsymbol{\Lambda}}_{xa}\|^2\|{\boldsymbol{\Lambda}}_{xb}^{-1}\|^2\zeta^x_0}{K(1-\|\mathbf{\Gamma}_x\|)}
\\&+\left[\kappa^{12}\alpha^2\frac{\|\mathbf{O}_z\|^2\|\mathbf{O}_z^{-1}\|^2\|{\boldsymbol{\Lambda}}_{za}\|^2}{1-\|\mathbf{\Gamma}_z\|}+\kappa^6\alpha^2\theta\frac{\|\mathbf{O}_x\|^2\|\mathbf{O}_x^{-1}\|^2\|{\boldsymbol{\Lambda}}_{xa}\|^2}{(1-\|\mathbf{\Gamma}_x\|)^2}+\frac{\kappa^5\alpha}{n}\right](\sigma_{f,1}^2+\kappa^2\sigma_{g,2}^2)\\ 
\lesssim& \frac{\theta_1}{n}(\sigma_{f,1}^2+\kappa^2\sigma_{g,2}^2)
+\frac{\kappa^4}{\theta_1 K}+\frac{C_\theta\kappa^4}{K}+\frac{\kappa}{\alpha_{1} K}
+\frac{\kappa^9\sigma_{g,1}^2}{n}\alpha_1+\frac{\kappa^5\alpha_1}{n}(\sigma_{f,1}^2+\kappa^2\sigma_{g,2}^2)
\\&+\frac{\kappa^{14}\alpha_{y,2}^2\|\mathbf{O}_y\|^2\|\mathbf{O}_y^{-1}\|^2 \|\mathbf{\Lambda}_{ya}\|^2}{1-\|\mathbf{\Gamma}_y\|} \sigma_{g,1}^2+\frac{\kappa}{\alpha_{y,2}K}
\\&+ \kappa^{10}\frac{\|\mathbf{O}_y\|^2\|\mathbf{O}_y^{-1}\|^2\|{\boldsymbol{\Lambda}}_{ya}\|^2\|{\boldsymbol{\Lambda}}_{yb}^{-1}\|^2}{(1-\|\mathbf{\Gamma}_y\|)^2}\kappa^{12}\alpha_{y,3}^4\frac{\sigma_{g,1}^2}{n}+\frac{\kappa}{\alpha_{y,3}K}\\
&+\kappa^{14}\frac{\|\mathbf{O}_z\|^2\|\mathbf{O}_z^{-1}\|^2\|{\boldsymbol{\Lambda}}_{za}\|^2\|{\boldsymbol{\Lambda}}_{zb}^{-1}\|^2}{(1-\|\mathbf{\Gamma}_z\|)^2}\kappa^{12}\alpha_{z,3}^4\frac{\sigma_{g,1}^2}{n}+\frac{\kappa}{\alpha_{z,3}K}
\\&+\alpha_{yb,2}^2\frac{\kappa^{14}\|\mathbf{O}_y\|^2\|\mathbf{O}_y^{-1}\|^2\|{\boldsymbol{\Lambda}}_{ya}\|^2\|{\boldsymbol{\Lambda}}_{yb}^{-1}\|^2\zeta^y_0}{K(1-\|\mathbf{\Gamma}_y\|)}+\frac{\kappa}{\alpha_{yb,2}K}\\
&+\alpha_{zb,2}^2\frac{\kappa^{12}\|\mathbf{O}_z\|^2\|\mathbf{O}_z^{-1}\|^2\|{\boldsymbol{\Lambda}}_{za}\|^2\|{\boldsymbol{\Lambda}}_{zb}^{-1}\|^2\zeta^z_0}{K(1-\|\mathbf{\Gamma}_z\|)}+\frac{\kappa}{\alpha_{zb,2}K}
\\&+\alpha_{xb,2}^2\frac{\kappa^6\|\mathbf{O}_x\|^2\|\mathbf{O}_x^{-1}\|^2\|{\boldsymbol{\Lambda}}_{xa}\|^2\|{\boldsymbol{\Lambda}}_{xb}^{-1}\|^2\zeta^x_0}{K(1-\|\mathbf{\Gamma}_x\|)}+\frac{\kappa}{\alpha_{xb,2}K}\\
&+\kappa^{12}\alpha_{z,2}^2\frac{\|\mathbf{O}_z\|^2\|\mathbf{O}_z^{-1}\|^2\|{\boldsymbol{\Lambda}}_{za}\|^2}{1-\|\mathbf{\Gamma}_z\|}(\sigma_{f,1}^2+\kappa^2\sigma_{g,2}^2)+\frac{\kappa}{\alpha_{z,2}K}
\\&+\kappa^6\alpha_{x,2}^2\theta_2\frac{\|\mathbf{O}_x\|^2\|\mathbf{O}_x^{-1}\|^2\|{\boldsymbol{\Lambda}}_{xa}\|^2}{(1-\|\mathbf{\Gamma}_x\|)^2}(\sigma_{f,1}^2+\kappa^2\sigma_{g,2}^2)+\frac{\kappa}{\alpha_{x,2}K}+\frac{\kappa^4}{\theta_{2}K},
\end{align}
where the last inequality uses \eqref{step_size}. 

Finally, substituting \eqref{stepsizeconstants} and  \eqref{step_size} into the last inequality, we can get:
\begin{equation}
  \begin{aligned}
&\frac{1}{K+1}\sum_{k=0}^K\mathbb{E}\|\nabla\Phi(\bar{x}^k)\|^2
\\\lesssim&\frac{\kappa^5\sigma}{\sqrt{{nK}}}+\kappa^{\frac{16}{3}}\left[\left(\frac{\|\mathbf{O}_y\|^2\|\mathbf{O}_y^{-1}\|^2\|{\boldsymbol{\Lambda}}_{ya}\|^2}{1-\|\mathbf{\Gamma}_y\|}\right)^{\frac{1}{3}}+\left(\frac{\|\mathbf{O}_z\|^2\|\mathbf{O}_z^{-1}\|^2\|{\boldsymbol{\Lambda}}_{za}\|^2}{1-\|\mathbf{\Gamma}_z\|}\right)^{\frac{1}{3}}\right]\frac{\sigma^{\frac{2}{3}}}{K^{\frac{2}{3}}}
\\&+\kappa^{\frac{7}{2}}\left(\frac{\|\mathbf{O}_x\|\|\mathbf{O}_x^{-1}\|\|{\boldsymbol{\Lambda}}_{xa}\|}{1-\|\mathbf{\Gamma}_y\|}\right)^{\frac{1}{2}}\frac{\sigma^{\frac{1}{2}}}{K^{\frac{3}{4}}}
\\&+\left[\kappa^{\frac{26}{5}}\left(\frac{\|\mathbf{O}_y\|^2\|\mathbf{O}_y^{-1}\|^2\|{\boldsymbol{\Lambda}}_{ya}\|^2\|{\boldsymbol{\Lambda}}_{yb}^{-1}\|^2}{n(1-\|\mathbf{\Gamma}_y\|)^2}\right)^{\frac{1}{5}}+\kappa^{6}\left(\frac{\|\mathbf{O}_z\|^2\|\mathbf{O}_z^{-1}\|^2\|{\boldsymbol{\Lambda}}_{za}\|^2\|{\boldsymbol{\Lambda}}_{zb}^{-1}\|^2}{n(1-\|\mathbf{\Gamma}_z\|)^2}\right)^{\frac{1}{5}}\right]\frac{\sigma^{\frac{2}{5}}}{K^{\frac{4}{5}}}
\\&+\left[\kappa^{\frac{16}{3}}\left(\frac{\|\mathbf{O}_y\|^2\|\mathbf{O}_y^{-1}\|^2\|{\boldsymbol{\Lambda}}_{ya}\|^2\|{\boldsymbol{\Lambda}}_{yb}^{-1}\|^2\zeta^y_0}{1-\|\mathbf{\Gamma}_y\|}\right)^{\frac{1}{3}}
+\kappa^{\frac{14}{3}}\left(\frac{\|\mathbf{O}_z\|^2\|\mathbf{O}_z^{-1}\|^2\|{\boldsymbol{\Lambda}}_{za}\|^2\|{\boldsymbol{\Lambda}}_{zb}^{-1}\|^2\zeta^z_0}{1-\|\mathbf{\Gamma}_z\|}\right)^{\frac{1}{3}}\right.\\&\left.\quad+\kappa^{\frac{8}{3}}\left(\frac{\|\mathbf{O}_x\|^2\|\mathbf{O}_x^{-1}\|^2\|{\boldsymbol{\Lambda}}_{xa}\|^2\|{\boldsymbol{\Lambda}}_{xb}^{-1}\|^2\zeta^x_0}{1-\|\mathbf{\Gamma}_x\|}\right)^{\frac{1}{3}}\right]\frac{1}{K}
+\left(\kappa C_\alpha+\kappa^4C_\theta\right)\frac{1}{K},
  \end{aligned}
\end{equation}
where $\sigma=\max\{\sigma_{f,1},\sigma_{g,1},\sigma_{g,2}\}$.
    \end{proof}
\end{lemma}
\begin{remark}
    From the proof of Lemma \ref{Stochasticrate}, the impact of the moving average technique on variance reduction becomes evident. The term $\frac{\theta}{n}\sigma^2$ absorb  $\alpha^2\eta_1\sigma^2$, which includes the high order term $\alpha^4\sigma^2$. Additionally, compared to $y,z$, the quadratic term related to $\sigma^2$ of $x$ has an extra term $\theta$ multiplied in the numerator ($\alpha^2\theta\sigma^2$). These details reduce the impacts of noise to terms related to $x$, confirming the conclusion that terms related to $y,z$ dominate the rate in precious sections. Notably, taking $\theta<1$ is indispensable our proof. If we take $\theta=1$, there would be a constant term $\frac{1}{n}\sigma^2$ in the convergence rate (see the first inequality of \eqref{f1}), since the coefficient $\alpha^2/\beta^2+\alpha^2/\gamma^2=\mathcal{O}(1)$. This would not guarantee the convergence of \ours.
\end{remark}

\subsection{Analysis of consensus error and transient iteration complexity}
From Lemma~\ref{Stochasticrate}, we can immediately obtain the transient time complexity of Algorithm~\ref{D-SOBA-SUDA}. Here we omit the impacts of the condition number $\kappa$.

\begin{lemma}\label{formaltransienttime}The transient time complexity of Algorithm~\ref{D-SOBA-SUDA} has an upper bound of: 
     \begin{equation}\label{cor:trans}
    \begin{aligned}
\max&\left\{n^3 \left(\frac{\|\mathbf{O}_y\|^2\|\mathbf{O}_y^{-1}\|^2}{1-\|\mathbf{\Gamma}_y\|}\right)^2\|{\boldsymbol{\Lambda}}_{ya}\|^2 , n^3\left(\frac{\|\mathbf{O}_z\|^2\|\mathbf{O}_z^{-1}\|^2}{1-\|\mathbf{\Gamma}_z\|}\right)^2\|{\boldsymbol{\Lambda}}_{za}\|^2 ,\right.\\
&\quad\left.  n^2 \left(\frac{\|\mathbf{O}_x\|\|\mathbf{O}_x^{-1}\|}{1-\|\mathbf{\Gamma}_x\|}\right)^2\|{\boldsymbol{\Lambda}}_{xa}\|^2, n\left(\frac{\|\mathbf{O}_y\|\|\mathbf{O}_y^{-1}\|\|{\boldsymbol{\Lambda}}_{yb}^{-1}\|}{1-\|\mathbf{\Gamma}_y\|}\right)^{\frac{4}{3}}\|{\boldsymbol{\Lambda}}_{ya}\|,\right.\\&\quad\left.  n\left(\frac{\|\mathbf{O}_z\|\|\mathbf{O}_z^{-1}\|\|{\boldsymbol{\Lambda}}_{zb}^{-1}\|}{1-\|\mathbf{\Gamma}_z\|}\right)^{\frac{4}{3}}\|{\boldsymbol{\Lambda}}_{za}\|, n\left(\frac{\|\mathbf{O}_x\|^2\|\mathbf{O}_x^{-1}\|^2\|{\boldsymbol{\Lambda}}_{xa}\|^2\|{\boldsymbol{\Lambda}}_{xb}^{-1}\|^2}{1-\|\mathbf{\Gamma}_x\|}\right)^{\frac{2}{3}}, \right.\\&\quad\left. n\frac{\|\mathbf{O}_x\|\|\mathbf{O}_x^{-1}\|\|{\boldsymbol{\Lambda}}_{xa}\|\|{\boldsymbol{\Lambda}}_{xb}^{-1}\|}{1-\|\mathbf{\Gamma}_x\|},n\right\}.
    \end{aligned}
    \end{equation}

\begin{proof}
    According to lemma \ref{Stochasticrate}, \ours achieves linear speedup if:

\begin{align}
\frac{1}{\sqrt{{nK}}}\gtrsim&\left[\left(\frac{\|\mathbf{O}_y\|^2\|\mathbf{O}_y^{-1}\|^2\|{\boldsymbol{\Lambda}}_{ya}\|^2}{1-\|\mathbf{\Gamma}_y\|}\right)^{\frac{1}{3}}+\left(\frac{\|\mathbf{O}_z\|^2\|\mathbf{O}_z^{-1}\|^2\|{\boldsymbol{\Lambda}}_{za}\|^2}{1-\|\mathbf{\Gamma}_z\|}\right)^{\frac{1}{3}}\right]\frac{1}{K^{\frac{2}{3}}}
\\&+\left(\frac{\|\mathbf{O}_x\|\|\mathbf{O}_x^{-1}\|\|{\boldsymbol{\Lambda}}_{xa}\|}{1-\|\mathbf{\Gamma}_y\|}\right)^{\frac{1}{2}}\frac{1}{K^{\frac{3}{4}}}
\\&+\left[\left(\frac{\|\mathbf{O}_y\|^2\|\mathbf{O}_y^{-1}\|^2\|{\boldsymbol{\Lambda}}_{ya}\|^2\|{\boldsymbol{\Lambda}}_{yb}^{-1}\|^2}{n(1-\|\mathbf{\Gamma}_y\|)^2}\right)^{\frac{1}{5}}+\left(\frac{\|\mathbf{O}_z\|^2\|\mathbf{O}_z^{-1}\|^2\|{\boldsymbol{\Lambda}}_{za}\|^2\|{\boldsymbol{\Lambda}}_{zb}^{-1}\|^2}{n(1-\|\mathbf{\Gamma}_z\|)^2}\right)^{\frac{1}{5}}\right]\frac{1}{K^{\frac{4}{5}}}
\\&+\left[\left(\frac{\|\mathbf{O}_y\|^2\|\mathbf{O}_y^{-1}\|^2\|{\boldsymbol{\Lambda}}_{ya}\|^2\|{\boldsymbol{\Lambda}}_{yb}^{-1}\|^2\zeta^y_0}{1-\|\mathbf{\Gamma}_y\|}\right)^{\frac{1}{3}}
+\left(\frac{\|\mathbf{O}_z\|^2\|\mathbf{O}_z^{-1}\|^2\|{\boldsymbol{\Lambda}}_{za}\|^2\|{\boldsymbol{\Lambda}}_{zb}^{-1}\|^2\zeta^z_0}{1-\|\mathbf{\Gamma}_z\|}\right)^{\frac{1}{3}}\right.\\
&\left.\quad+\left(\frac{\|\mathbf{O}_x\|^2\|\mathbf{O}_x^{-1}\|^2\|{\boldsymbol{\Lambda}}_{xa}\|^2\|{\boldsymbol{\Lambda}}_{xb}^{-1}\|^2\zeta^x_0}{1-\|\mathbf{\Gamma}_x\|}\right)^{\frac{1}{3}}\right]\frac{1}{K}
+\left(C_\alpha+C_\theta\right)\frac{1}{K}.
\end{align}

It holds when $K$ satisfies:

\begin{align}
\left(\frac{\|\mathbf{O}_y\|^2\|\mathbf{O}_y^{-1}\|^2\|{\boldsymbol{\Lambda}}_{ya}\|^2}{1-\|\mathbf{\Gamma}_y\|}\right)^{\frac{1}{3}}\frac{1}{K^{\frac{2}{3}}}&\lesssim\frac{1}{\sqrt{{nK}}},\\
\left(\frac{\|\mathbf{O}_z\|^2\|\mathbf{O}_z^{-1}\|^2\|{\boldsymbol{\Lambda}}_{za}\|^2}{1-\|\mathbf{\Gamma}_z\|}\right)^{\frac{1}{3}}\frac{1}{K^{\frac{2}{3}}}&\lesssim\frac{1}{\sqrt{{nK}}},\\
\left(\frac{\|\mathbf{O}_y\|^2\|\mathbf{O}_y^{-1}\|^2\|{\boldsymbol{\Lambda}}_{ya}\|^2\|{\boldsymbol{\Lambda}}_{yb}^{-1}\|^2}{n(1-\|\mathbf{\Gamma}_y\|)^2}\right)^{\frac{1}{5}}\frac{1}{K^{\frac{4}{5}}}&\lesssim\frac{1}{\sqrt{{nK}}},\\
\left(\frac{\|\mathbf{O}_x\|\|\mathbf{O}_x^{-1}\|\|{\boldsymbol{\Lambda}}_{xa}\|}{1-\|\mathbf{\Gamma}_y\|}\right)^{\frac{1}{2}}\frac{1}{K^{\frac{3}{4}}}&\lesssim\frac{1}{\sqrt{{nK}}},\\
\left(\frac{\|\mathbf{O}_z\|^2\|\mathbf{O}_z^{-1}\|^2\|{\boldsymbol{\Lambda}}_{za}\|^2\|{\boldsymbol{\Lambda}}_{zb}^{-1}\|^2}{n(1-\|\mathbf{\Gamma}_z\|)^2}\right)^{\frac{1}{5}}\frac{1}{K^{\frac{4}{5}}}&\lesssim\frac{1}{\sqrt{{nK}}},\\
\left(\frac{\|\mathbf{O}_y\|^2\|\mathbf{O}_y^{-1}\|^2\|{\boldsymbol{\Lambda}}_{ya}\|^2\|{\boldsymbol{\Lambda}}_{yb}^{-1}\|^2\zeta^y_0}{1-\|\mathbf{\Gamma}_y\|}\right)^{\frac{1}{3}}\frac{1}{K}&\lesssim\frac{1}{\sqrt{{nK}}},\\
\left(\frac{\|\mathbf{O}_z\|^2\|\mathbf{O}_z^{-1}\|^2\|{\boldsymbol{\Lambda}}_{za}\|^2\|{\boldsymbol{\Lambda}}_{zb}^{-1}\|^2\zeta^z_0}{1-\|\mathbf{\Gamma}_z\|}\right)^{\frac{1}{3}}\frac{1}{K}&\lesssim\frac{1}{\sqrt{{nK}}},\\
\left(\frac{\|\mathbf{O}_x\|^2\|\mathbf{O}_x^{-1}\|^2\|{\boldsymbol{\Lambda}}_{xa}\|^2\|{\boldsymbol{\Lambda}}_{xb}^{-1}\|^2\zeta^x_0}{1-\|\mathbf{\Gamma}_x\|}\right)^{\frac{1}{3}}\frac{1}{K}&\lesssim\frac{1}{\sqrt{{nK}}},\\
\left(C_\alpha+C_\theta\right)\frac{1}{K}&\lesssim\frac{1}{\sqrt{nK}}.
\end{align}

Then we get:
\begin{small}
\begin{align*}\label{cor:trans1}
K\gtrsim \max&\left\{n^3 \left(\frac{\|\mathbf{O}_y\|^2\|\mathbf{O}_y^{-1}\|^2}{1-\|\mathbf{\Gamma}_y\|}\right)^2\|{\boldsymbol{\Lambda}}_{ya}\|^2 , n^3\left(\frac{\|\mathbf{O}_z\|^2\|\mathbf{O}_z^{-1}\|^2}{1-\|\mathbf{\Gamma}_z\|}\right)^2\|{\boldsymbol{\Lambda}}_{za}\|^2 ,\right.\\
&\quad\left.  n^2 \left(\frac{\|\mathbf{O}_x\|\|\mathbf{O}_x^{-1}\|}{1-\|\mathbf{\Gamma}_x\|}\right)^2\|{\boldsymbol{\Lambda}}_{xa}\|^2, n\left(\frac{\|\mathbf{O}_y\|\|\mathbf{O}_y^{-1}\|\|{\boldsymbol{\Lambda}}_{yb}^{-1}\|}{1-\|\mathbf{\Gamma}_y\|}\right)^{\frac{4}{3}}\|{\boldsymbol{\Lambda}}_{ya}\|,\right.\\&\quad\left.  n\left(\frac{\|\mathbf{O}_z\|\|\mathbf{O}_z^{-1}\|\|{\boldsymbol{\Lambda}}_{zb}^{-1}\|}{1-\|\mathbf{\Gamma}_z\|}\right)^{\frac{4}{3}}\|{\boldsymbol{\Lambda}}_{za}\|, n\left(\frac{\|\mathbf{O}_x\|^2\|\mathbf{O}_x^{-1}\|^2\|{\boldsymbol{\Lambda}}_{xa}\|^2\|{\boldsymbol{\Lambda}}_{xb}^{-1}\|^2}{1-\|\mathbf{\Gamma}_x\|}\right)^{\frac{2}{3}}, \right.\\&\quad\left. n\frac{\|\mathbf{O}_x\|\|\mathbf{O}_x^{-1}\|\|{\boldsymbol{\Lambda}}_{xa}\|\|{\boldsymbol{\Lambda}}_{xb}^{-1}\|}{1-\|\mathbf{\Gamma}_x\|},n\right\}.
\end{align*}
\end{small}
\end{proof}

\end{lemma}

\subsubsection{Consensus Error}\label{Consensus-Error}
\begin{lemma}\label{Consensus-Error-lemma}
Suppose that Assumptions ~\ref{smooth}-~\ref{var} hold. Then there exist 
constant step-sizes $\alpha,\beta,\gamma,\theta$, such that Lemma \ref{Stochasticrate} holds and 
\begin{equation}
\begin{aligned} 
&\frac{1}{K}\sum_{k=0}^K\mathbb{E}\left[\frac{\Vert \mathbf{x}^k-\bar{\mathbf{x}}^k\Vert ^2}{n}+\frac{\Vert \mathbf{y}^k-\bar{\mathbf{y}}^k\Vert ^2}{n}\right]\\
\lesssim_K &\frac{n}{K}\left(\frac{\Vert \mathbf{O}_z\Vert ^2\Vert \mathbf{O}_z^{-1}\Vert ^2\Vert {\mathbf{\Lambda}}_{za}\Vert ^2}{1-\Vert \mathbf{\Gamma}_z\Vert }+\frac{\Vert \mathbf{O}_y\Vert ^2\Vert \mathbf{O}_y^{-1}\Vert ^2\Vert {\mathbf{\Lambda}}_{ya}\Vert ^2}{1-\Vert \mathbf{\Gamma}_y\Vert }\right),
\end{aligned}
\end{equation}
where $\lesssim_K$ denotes the 
the asymptotic rate when $K\to \infty$.
\begin{proof}

Suppose $\alpha$, $\beta$, $\gamma$, and $\theta$ satisfy the constraints given in  \eqref{stepsizeconstants} and  \eqref{step_size}, which ensures that Theorem \ref{thm1} (Lemma \ref{Stochasticrate}) holds.

For clarity, we define the constants:
\begin{equation}  
\begin{aligned}
c_1=\frac{9\alpha^2L_{z^\star}^2}{\gamma^2\mu_g^2}+\frac{438\kappa^4\alpha^2}{\beta^2\mu_g^2}L_{y^{\star}}^2, 
\quad c_2=10\left(L^2+\frac{\theta\sigma_{g,2}^2}{n}\right).
\end{aligned}
\end{equation}
Then there exist $\alpha$, $\beta$, $\gamma$, and $\theta$ that satisfy the constraints in \eqref{stepsizeconstants} and \eqref{step_size}, and also:
\begin{equation}\label{c12}
    c_1\le 0.01 L^{-2}, \quad c_2\le 11 L^2.
\end{equation}

We take such values for step-sizes in the following proof.

We proceed by substituting \eqref{barr} into \eqref{I}, yielding:
\begin{equation}  
\begin{aligned}
\sum_{k=-1}^K\mathbb{E}[I_k]
\leq& 
4 c_1\left(\frac{\Phi(\bar{x}_0)-\inf\Phi}{\alpha}+c_2\sum_{k=0}^{K-1}\mathbb{E}\left(\frac{\Delta_k}{n}+I_k\right)+\frac{3\theta}{n}K\left(\sigma_{f,1}^2+2\sigma_{g,2}^2\frac{L_{f,0}^2}{\mu_g^2}\right)\right)
\\&+510\kappa^4\sum_{k=0}^K\mathbb{E}\left[\frac{\Delta_k}{n}\right]+\frac{3\Vert  z^1_{\star}\Vert ^2}{\mu_g\gamma}
\\&+\frac{6(K+1)\gamma}{\mu_gn}\left(3\sigma_{g,2}^2\frac{L_{f,0}^2}{\mu_g^2}+\sigma_{f,1}^2\right)
+73\kappa^4\left(\frac{4}{\beta\mu_g}\Vert \bar{y}^{0}-y^{\star}(\bar{x}^{0})\Vert ^2
+\frac{4K\sigma_{g,1}^2}{n\mu_g}\beta\right).
 \end{aligned}
\end{equation}
Subtracting $4c_1c_2\sum_{k=0}^{K-1}\mathbb{E}[I_k]$ from both sides, we get: 
\begin{equation}  \label{ir}
\begin{aligned}
\sum_{k=-1}^K\mathbb{E}[I_k]
\lesssim& 
\frac{\Phi(\bar{x}_0)-\inf\Phi}{\alpha}+\frac{\theta}{n}K\left(\sigma_{f,1}^2+\sigma_{g,2}^2\frac{L_{f,0}^2}{\mu_g^2}\right)
+\kappa^4\sum_{k=0}^K\mathbb{E}\left[\frac{\Delta_k}{n}\right]+\frac{\Vert z^1_{\star}\Vert ^2}{\mu_g\gamma}
\\&+\frac{K\gamma}{\mu_gn}\left(\sigma_{g,2}^2\frac{L_{f,0}^2}{\mu_g^2}+\sigma_{f,1}^2\right)
+\kappa^4\left(\frac{1}{\beta\mu_g}\Vert \bar{y}^{0}-y^{\star}(\bar{x}^{0})\Vert ^2
+\frac{K\sigma_{g,1}^2}{n\mu_g}\beta\right).
 \end{aligned}
\end{equation}

Substituting \eqref{I} into \eqref{barr}, we obtain:
\begin{equation}
  \begin{aligned}
    &\frac{1}{4}\sum_{k=0}^K\mathbb{E}\left\Vert \bar{r}^{k+1}\right\Vert ^2\\
    \le&\frac{\Phi(\bar{x}_0)-\inf\Phi}{\alpha}+c_2\sum_{k=0}^K\mathbb{E}\left[\frac{\Delta_k}{n}\right]+c_2c_1\sum_{k=0}^K\mathbb{E}\Vert \bar{r}^k\Vert ^2
    \\&+c_2
    \left[
510\kappa^4\sum_{k=0}^K\mathbb{E}\left[\frac{\Delta_k}{n}\right]+\frac{3\Vert z^1_{\star}\Vert ^2}{\mu_g\gamma}
+\frac{6(K+1)\gamma}{\mu_gn}\left(3\sigma_{g,2}^2\frac{L_{f,0}^2}{\mu_g^2}+\sigma_{f,1}^2\right)
\right.\\& \quad\left.+73\kappa^4\left(\frac{4}{\beta\mu_g}\Vert \bar{y}^{0}-y^{\star}(\bar{x}^{0})\Vert ^2
+\frac{4K\sigma_{g,1}^2}{n\mu_g}\beta\right)\right]
\\&+\frac{3\theta}{n}(K+1)\left(\sigma_{f,1}^2+2\sigma_{g,2}^2\frac{L_{f,0}^2}{\mu_g^2}\right).
  \end{aligned}
\end{equation}  
Subtracting $c_2c_1\sum_{k=0}^K\mathbb{E}\Vert \bar{r}^k\Vert ^2$ from both sides, we get 
\begin{equation}\label{ri}
  \begin{aligned}
&\sum_{k=0}^K\mathbb{E}\left\Vert \bar{r}^{k+1}\right\Vert ^2\\
\lesssim&\frac{\Phi(\bar{x}_0)-\inf\Phi}{\alpha}+\kappa^4\sum_{k=0}^K\mathbb{E}\left[\frac{\Delta_k}{n}\right]+\frac{\theta}{n}K\left(\sigma_{f,1}^2+\sigma_{g,2}^2\frac{L_{f,0}^2}{\mu_g^2}\right)
\\&+\frac{\Vert z^1_{\star}\Vert ^2}{\mu_g\gamma}
+\frac{K\gamma}{\mu_gn}\left(\sigma_{g,2}^2\frac{L_{f,0}^2}{\mu_g^2}+\sigma_{f,1}^2\right)
+\kappa^4\left(\frac{1}{\beta\mu_g}\Vert \bar{y}^{0}-y^{\star}(\bar{x}^{0})\Vert ^2
+\frac{K\sigma_{g,1}^2}{n\mu_g}\beta\right).
  \end{aligned}
\end{equation}  
Taking \[\eta_3=\left(\frac{\kappa^2\Vert \mathbf{O}_x^{-1}\Vert ^2\Vert \mathbf{O}_x\Vert ^2\Vert {\mathbf{\Lambda}}_{xa}\Vert ^2\alpha^2}{(1-\Vert \mathbf{\Gamma}_x\Vert )^2}
+\frac{\Vert \mathbf{O}_z\Vert ^2\Vert \mathbf{O}_z^{-1}\Vert ^2\Vert {\mathbf{\Lambda}}_{za}\Vert ^2}{1-\Vert \mathbf{\Gamma}_z\Vert }\cdot\frac{\gamma^2(L_{g,1}^2+(1-\Vert \mathbf{\Gamma}_z\Vert )\sigma_{g,2}^2)}{1-\Vert \mathbf{\Gamma}_z\Vert }\right),\]
and combining previous results with \eqref{Delta}, we obtain
\begin{equation}\label{Delta2}
\begin{aligned} 
&\sum_{k=0}^K\mathbb{E}\left[\Delta_k\right]
\\\lesssim&(\eta_1+\kappa^2L_{y^{\star}}^2\eta_2)\alpha^2\sum_{k=0}^K\mathbb{E}\Vert \bar{\mathbf{r}}^{k+1}\Vert ^2+\kappa\eta_2\beta\Vert \bar{\mathbf{y}}^0-\mathbf{y}^{\star}(\bar{x}^{0})\Vert ^2+K\eta_2\beta^2\sigma_{g,1}^2+\eta_3\sum_{k=-1}^K\mathbb{E}[nI_k]
\\&+\frac{\kappa^2\beta^2K\Vert \mathbf{O}_y\Vert ^2\Vert \mathbf{O}_y^{-1}\Vert ^2 \Vert \mathbf{\Lambda}_{ya}\Vert ^2}{1-\Vert \mathbf{\Gamma}_y\Vert } n\sigma_{g,1}^2+\frac{\kappa^2\Vert \mathbf{O}_y\Vert ^2\mathbb{E}\Vert \hat{\mathbf{e}}_y^{0}\Vert ^2}{1-\Vert \mathbf{\Gamma}_y\Vert }+\frac{\Vert \mathbf{O}_z\Vert ^2\mathbb{E}\Vert \hat{\mathbf{e}}_z^{0}\Vert ^2}{1-\Vert \mathbf{\Gamma}_z\Vert }
\\&+
\frac{\kappa^2\Vert \mathbf{O}_x\Vert ^2\mathbb{E}\Vert \hat{\mathbf{e}}_x^{0}\Vert ^2}{1-\Vert \mathbf{\Gamma}_x\Vert }+\frac{\kappa^2\alpha^2\Vert \mathbf{O}_x\Vert ^2\Vert \mathbf{O}_x^{-1}\Vert ^2\Vert {\mathbf{\Lambda}}_{xa}\Vert ^2}{\theta(1-\Vert \mathbf{\Gamma}_x\Vert )^2}
\left\Vert \widetilde{\nabla}\mathbf{\Phi}(\bar{\mathbf{x}}^{0})\right\Vert ^2
\\&+Kn\gamma^2\frac{\Vert \mathbf{O}_z\Vert ^2\Vert \mathbf{O}_z^{-1}\Vert ^2\Vert {\mathbf{\Lambda}}_{za}\Vert ^2}{1-\Vert \mathbf{\Gamma}_z\Vert }\left(\kappa^2\sigma_{g,2}^2+\sigma_{f,1}^2\right)
\\&+Kn\kappa^2\alpha^2\theta\frac{\Vert \mathbf{O}_x\Vert ^2\Vert \mathbf{O}_x^{-1}\Vert ^2\Vert {\mathbf{\Lambda}}_{xa}\Vert ^2}{(1-\Vert \mathbf{\Gamma}_x\Vert )^2}\left(\kappa^2\sigma_{g,2}^2+\sigma_{f,1}^2\right)
\\\lesssim&\left[(\eta_1+\kappa^2L_{y^{\star}}^2\eta_2)\alpha^2+\eta_3\right]\cdot\kappa^4\sum_{k=0}^K\mathbb{E}\left[\Delta_k\right]+\kappa\eta_2\beta\Vert \bar{\mathbf{y}}^0-\mathbf{y}^{\star}(\bar{x}^{0})\Vert ^2+K\eta_2\beta^2\sigma_{g,1}^2
\\&+n\left[(\eta_1+\kappa^2L_{y^{\star}}^2\eta_2)\alpha^2+\eta_3\right]\left[\frac{1}{\alpha}+\frac{\theta}{n}K\left(\sigma_{f,1}^2+\kappa^2\sigma_{g,2}^2\right)
+\frac{1}{\mu_g\gamma}
+\frac{K\gamma}{\mu_gn}\left(\kappa^2\sigma_{g,2}^2+\sigma_{f,1}^2\right)
\right.\\&\quad \left.+\kappa^4\left(\frac{1}{\beta\mu_g}
+\frac{K\sigma_{g,1}^2}{n\mu_g}\beta\right)\right]
+\frac{\kappa^2\beta^2K\Vert \mathbf{O}_y\Vert ^2\Vert \mathbf{O}_y^{-1}\Vert ^2 \Vert \mathbf{\Lambda}_{ya}\Vert ^2}{1-\Vert \mathbf{\Gamma}_y\Vert } n\sigma_{g,1}^2
\\&+\frac{\kappa^2\alpha^2\Vert \mathbf{O}_x\Vert ^2\Vert \mathbf{O}_x^{-1}\Vert ^2\Vert {\mathbf{\Lambda}}_{xa}\Vert ^2}{\theta(1-\Vert \mathbf{\Gamma}_x\Vert )^2}
\left\Vert \widetilde{\nabla}\mathbf{\Phi}(\bar{\mathbf{x}}^{0})\right\Vert ^2
\\&+\frac{\kappa^2\Vert \mathbf{O}_y\Vert ^2\mathbb{E}\Vert \hat{\mathbf{e}}_y^{0}\Vert ^2}{1-\Vert \mathbf{\Gamma}_y\Vert }+\frac{\Vert \mathbf{O}_z\Vert ^2\mathbb{E}\Vert \hat{\mathbf{e}}_z^{0}\Vert ^2}{1-\Vert \mathbf{\Gamma}_z\Vert }
+\frac{\kappa^2\Vert \mathbf{O}_x\Vert ^2\mathbb{E}\Vert \hat{\mathbf{e}}_x^{0}\Vert ^2}{1-\Vert \mathbf{\Gamma}_x\Vert }
\\&+Kn\gamma^2\frac{\Vert \mathbf{O}_z\Vert ^2\Vert \mathbf{O}_z^{-1}\Vert ^2\Vert {\mathbf{\Lambda}}_{za}\Vert ^2}{1-\Vert \mathbf{\Gamma}_z\Vert }\left(\kappa^2\sigma_{g,2}^2+\sigma_{f,1}^2\right)
\\&+Kn\kappa^2\alpha^2\theta\frac{\Vert \mathbf{O}_x\Vert ^2\Vert \mathbf{O}_x^{-1}\Vert ^2\Vert {\mathbf{\Lambda}}_{xa}\Vert ^2}{(1-\Vert \mathbf{\Gamma}_x\Vert )^2}\left(\kappa^2\sigma_{g,2}^2+\sigma_{f,1}^2\right).
 \end{aligned}
\end{equation}

\eqref{stepsizeconstants} and \eqref{step_size} imply that
\begin{equation}
    \begin{aligned}
    &\eta_1\lesssim \kappa^2+\kappa^2\frac{\Vert \mathbf{O}_x\Vert ^2\Vert \mathbf{O}_x^{-1}\Vert ^2{\Vert\mathbf{\Lambda}}_{xa}\Vert ^2 }{(1-\Vert \mathbf{\Gamma}_x\Vert )^2},\quad \eta_2\lesssim \kappa^2,
    \\& (\eta_1+\kappa^2L_{y^{\star}}^2\eta_2)\alpha^2\lesssim \kappa^{-4},\quad \eta_3\lesssim \kappa^{-4}
    \end{aligned}
\end{equation}
where $\eta_1,\eta_2$ are defined in Lemma \ref{Deltasum}. 

Then taking $\alpha,\beta,\gamma,\theta$ such that \eqref{stepsizeconstants},  \eqref{step_size}, \eqref{c12} hold and $\kappa^4[(\eta_1+\kappa^2L_{y^{\star}}^2\eta_2)\alpha^2+\eta_3]$ is a sufficiently small constant,  we can derive the following result:
\begin{equation}\label{Con-error}
\begin{aligned} 
&\frac{1}{K}\sum_{k=0}^K\mathbb{E}\left[\frac{\Delta_k}{n}\right]
\\\lesssim&\frac{\kappa\eta_2\beta}{K}+\eta_2\beta^2\frac{\sigma_{g,1}^2}{n}
\\&+\left[(\eta_1+\kappa^2L_{y^{\star}}^2\eta_2)\alpha^2+\eta_3\right]\left[\frac{1}{\alpha K}+\frac{\theta}{n}\left(\sigma_{f,1}^2+\kappa^2\sigma_{g,2}^2\right)
+\frac{1}{\mu_g\gamma K}
+\frac{\gamma}{\mu_gn}\left(\kappa^2\sigma_{g,2}^2+\sigma_{f,1}^2\right)
\right]\\&+\left[(\eta_1+\kappa^2L_{y^{\star}}^2\eta_2)\alpha^2+\eta_3\right]\kappa^4\left(\frac{1}{\beta\mu_g K}
+\frac{\sigma_{g,1}^2}{n\mu_g}\beta\right)
+\frac{\kappa^2\beta^2\Vert \mathbf{O}_y\Vert ^2\Vert \mathbf{O}_y^{-1}\Vert ^2 \Vert \mathbf{\Lambda}_{ya}\Vert ^2}{1-\Vert \mathbf{\Gamma}_y\Vert } \sigma_{g,1}^2\\&+\frac{\kappa^2\alpha^2\Vert \mathbf{O}_x\Vert ^2\Vert \mathbf{O}_x^{-1}\Vert ^2\Vert {\mathbf{\Lambda}}_{xa}\Vert ^2}{\theta K(1-\Vert \mathbf{\Gamma}_x\Vert )^2}
+\frac{\kappa^2\Vert \mathbf{O}_y\Vert ^2\mathbb{E}\Vert \hat{\mathbf{e}}_y^{0}\Vert ^2}{(1-\Vert \mathbf{\Gamma}_y\Vert )Kn}+\frac{\Vert \mathbf{O}_z\Vert ^2\mathbb{E}\Vert \hat{\mathbf{e}}_z^{0}\Vert ^2}{(1-\Vert \mathbf{\Gamma}_z\Vert )Kn}
+\frac{\kappa^2\Vert \mathbf{O}_x\Vert ^2\mathbb{E}\Vert \hat{\mathbf{e}}_x^{0}\Vert ^2}{(1-\Vert \mathbf{\Gamma}_x\Vert )Kn}
\\&+\gamma^2\frac{\Vert \mathbf{O}_z\Vert ^2\Vert \mathbf{O}_z^{-1}\Vert ^2\Vert {\mathbf{\Lambda}}_{za}\Vert ^2}{1-\Vert \mathbf{\Gamma}_z\Vert }\left(\kappa^2\sigma_{g,2}^2+\sigma_{f,1}^2\right)
+\kappa^2\alpha^2\theta\frac{\Vert \mathbf{O}_x\Vert ^2\Vert \mathbf{O}_x^{-1}\Vert ^2\Vert {\mathbf{\Lambda}}_{xa}\Vert ^2}{(1-\Vert \mathbf{\Gamma}_x\Vert )^2}\left(\kappa^2\sigma_{g,2}^2+\sigma_{f,1}^2\right)
\\\lesssim&\frac{\kappa^5\eta_2\alpha}{K}+\kappa^{10}\alpha^2\frac{\sigma_{g,1}^2}{n}+\frac{\kappa}{ K}\left[(\eta_1+\kappa^2L_{y^{\star}}^2\eta_2)\alpha+\frac{\eta_3}{\alpha}\right]
\\&+\left[(\eta_1+\kappa^2L_{y^{\star}}^2\eta_2)\alpha^2+\eta_3\right]\left[\frac{\theta}{n}\left(\sigma_{f,1}^2+\kappa^2\sigma_{g,2}^2\right)
+\frac{\kappa^5\alpha}{n}\left(\kappa^2\sigma_{g,2}^2+\sigma_{f,1}^2\right)
+\kappa^9\frac{\sigma_{g,1}^2}{n}\alpha\right]
\\&+\frac{\kappa^2\Vert \mathbf{O}_y\Vert ^2\Vert \mathbf{O}_y^{-1}\Vert ^2 \Vert \mathbf{\Lambda}_{ya}\Vert ^2}{1-\Vert \mathbf{\Gamma}_y\Vert } \sigma_{g,1}^2\kappa^8\alpha^2+\frac{\kappa^{-1}\alpha\Vert \mathbf{O}_x\Vert ^2\Vert \mathbf{O}_x^{-1}\Vert ^2\Vert {\mathbf{\Lambda}}_{xa}\Vert ^2}{ K(1-\Vert \mathbf{\Gamma}_x\Vert )^2}
\\&+\alpha^2\frac{\kappa^{10}\Vert \mathbf{O}_y\Vert ^2\Vert \mathbf{O}_y^{-1}\Vert ^2\Vert {\mathbf{\Lambda}}_{ya}\Vert ^2\Vert {\mathbf{\Lambda}}_{yb}^{-1}\Vert ^2\zeta^y_0}{K(1-\Vert \mathbf{\Gamma}_y\Vert )}+\alpha^2\frac{\kappa^{8}\Vert \mathbf{O}_z\Vert ^2\Vert \mathbf{O}_z^{-1}\Vert ^2\Vert {\mathbf{\Lambda}}_{za}\Vert ^2\Vert {\mathbf{\Lambda}}_{zb}^{-1}\Vert ^2\zeta^z_0}{K(1-\Vert \mathbf{\Gamma}_z\Vert )}
\\&+\alpha^2\frac{\kappa^2\Vert \mathbf{O}_x\Vert ^2\Vert \mathbf{O}_x^{-1}\Vert ^2\Vert {\mathbf{\Lambda}}_{xa}\Vert ^2\Vert {\mathbf{\Lambda}}_{xb}^{-1}\Vert ^2\zeta^x_0}{K(1-\Vert \mathbf{\Gamma}_x\Vert )}
\\&+\kappa^8\alpha^2\frac{\Vert \mathbf{O}_z\Vert ^2\Vert \mathbf{O}_z^{-1}\Vert ^2\Vert {\mathbf{\Lambda}}_{za}\Vert ^2}{1-\Vert \mathbf{\Gamma}_z\Vert }\left(\kappa^2\sigma_{g,2}^2+\sigma_{f,1}^2\right)
\\&+\kappa^2\alpha^2\theta\frac{\Vert \mathbf{O}_x\Vert ^2\Vert \mathbf{O}_x^{-1}\Vert ^2\Vert {\mathbf{\Lambda}}_{xa}\Vert ^2}{(1-\Vert \mathbf{\Gamma}_x\Vert )^2}\left(\kappa^2\sigma_{g,2}^2+\sigma_{f,1}^2\right).
 \end{aligned}
\end{equation}

From \eqref{stepsizeconstants} and \eqref{step_size}, we can determine the  asymptotic orders for $\alpha,\beta,\gamma$ and $\theta$ when $K\to \infty$
\begin{equation}
    \alpha=\mathcal{O}\left(\kappa^{-4}\sqrt{\frac{n}{K\sigma^2}}\right),\quad  \beta=\mathcal{O}\left(\sqrt{\frac{n}{K\sigma^2}}\right),\quad  \gamma=\mathcal{O}\left(\sqrt{\frac{n}{K\sigma^2}}\right),\quad  \theta=\mathcal{O}\left(\kappa\sqrt{\frac{n}{K\sigma^2}}\right).
\end{equation}

Then we get
\begin{equation}\label{Con-error-result}
\begin{aligned} 
&\frac{1}{K}\sum_{k=0}^K\mathbb{E}\left[\frac{\Delta_k}{n}\right]
\lesssim_K \frac{\kappa^2n}{K}\left(\frac{\Vert \mathbf{O}_z\Vert ^2\Vert \mathbf{O}_z^{-1}\Vert ^2\Vert {\mathbf{\Lambda}}_{za}\Vert ^2}{1-\Vert \mathbf{\Gamma}_z\Vert }+\frac{\Vert \mathbf{O}_y\Vert ^2\Vert \mathbf{O}_y^{-1}\Vert ^2\Vert {\mathbf{\Lambda}}_{ya}\Vert ^2}{1-\Vert \mathbf{\Gamma}_y\Vert }\right),
\end{aligned}
\end{equation}
where $\lesssim_K$ denotes the 
the asymptotic rate when $K\to \infty$.

Then using \eqref{transformer_consensus} and the definition of $\Delta_k$, we get 
\begin{equation}\label{Con-error-result-xy}
\begin{aligned} 
&\frac{1}{K}\sum_{k=0}^K\mathbb{E}\left[\frac{\Vert \mathbf{x}^k-\bar{\mathbf{x}}^k\Vert ^2}{n}+\frac{\Vert \mathbf{y}^k-\bar{\mathbf{y}}^k\Vert ^2}{n}\right]
\\
\lesssim_K&\frac{n}{K}\left(\frac{\Vert \mathbf{O}_z\Vert ^2\Vert \mathbf{O}_z^{-1}\Vert ^2\Vert {\mathbf{\Lambda}}_{za}\Vert ^2}{1-\Vert \mathbf{\Gamma}_z\Vert }+\frac{\Vert \mathbf{O}_y\Vert ^2\Vert \mathbf{O}_y^{-1}\Vert ^2\Vert {\mathbf{\Lambda}}_{ya}\Vert ^2}{1-\Vert \mathbf{\Gamma}_y\Vert }\right).
\end{aligned}
\end{equation}
In particular, the corresponding result of SPARKLE variants that using EXTRA, ED or GT is 
\begin{equation}
\begin{aligned}
&\frac{1}{K}\sum_{k=0}^K\mathbb{E}\left[\frac{\Vert \mathbf{x}^k-\bar{\mathbf{x}}^k\Vert ^2}{n}+\frac{\Vert \mathbf{y}^k-\bar{\mathbf{y}}^k\Vert ^2}{n}\right]
\lesssim_K \frac{n}{K}\left(\frac{1}{1-\rho_y}+\frac{1}{1-\rho_z}\right),
\end{aligned}
\end{equation}
where $\rho_y,\rho_z$ are spectrum gaps of relevant 
 mixing matrices.

\end{proof}
\end{lemma}

\subsubsection{Essential matrix norms for analysis}\label{Essential matrix norms}
Common heterogeneity-correction algorithms, including ED, EXTRA and GT, satisfy Assumption \ref{L_s}, according to transformations \eqref{tranL1}, \eqref{tranL2} and discussions in  ~\citep[Appendix B.2]{alghunaim2021unified}. Then Lemma \ref{stable} ensures that $\left\Vert\mathbf{\Gamma}\right\Vert<1$.
From Lemma~\ref{formaltransienttime}, the transient time complexity depends on the coefficients $\|\mathbf{O}\|^2$, $\|\mathbf{O}^{-1}\|^2$, $\|\mathbf{\Lambda}_a\|^2$, $\|\mathbf{\Lambda}_b^{-1}\|^2$, and $\|\mathbf{\Gamma}\|^2$. The solution of these matrices is constructive.  Table~\ref{table_Suda_coeffient} presents the upper bounds of these coefficients with different communication modes. Please refer to ~\citep[Appendix B.2]{alghunaim2021unified} for more details about the construction of these matrices and the computation of relevant norms. It is required that $W$ is positive definite  for ED, EXTRA, and we denote the smallest nonzero eigenvalue of $W$ by  $\underline{\rho}$. $\underline{\rho}$ can view as a constant. Otherwise we replace $\mathbf{W}$ with $t\mathbf{I}+(1-t)\mathbf{W}$ for some constant $t\in(0,1)$ (e.g. $t=1/2$). 

Substituting values of $\|\mathbf{O}_s\|,\|\mathbf{O}_s^{-1}\|,\|{\boldsymbol{\Lambda}}_{sa}\|,\|{\boldsymbol{\Lambda}}_{sb}^{-1}\|,\|\mathbf{\Gamma}_s\|$ into \eqref{cor:trans}
, we obtain the explicit transient iteration complexity for some specific examples of Algorithm~\ref{D-SOBA-SUDA}, which are listed in Table~\ref{extrasient}. Note that all GT variants exhibit the same transient iteration complexity.

\begin{table}[!h]
    \renewcommand{\arraystretch}{1.5}
    \caption{Upper bounds of coefficients  for different heterogeneity-correction modes in Lemma~\ref{formaltransienttime},  where notation $\mathcal{O}$ is omitted for $\|\mathbf{O}\|$ and $\|\mathbf{O}^{-1}\|$. }
    \label{table_Suda_coeffient}
    \centering
     \resizebox*{\textwidth}{!}{
    \begin{tabular}{ccccccccc}
        \hline
          Mode& $\mathbf{A}$ &$\mathbf{B}$&$\mathbf{C}$&$\|\mathbf{O}\|$&$\|\mathbf{O}^{-1}\|$&$\|{\boldsymbol{\Lambda}}_a\|$ &$\|{\boldsymbol{\Lambda}}_b^{-1}\|$&$\|{\boldsymbol{\Gamma}}\|$ \\
        \hline
        ED & $\mathbf{W}$ & $(\mathbf{I}-\mathbf{W})^{\frac{1}{2}}$ & $\mathbf{W}$  &$1$&$\underline{\rho}^{-\frac{1}{2}}$&$\rho$&$({1-\rho})^{-\frac{1}{2}}$&$\sqrt{\rho}$\\
        \hline
         EXTRA & $\mathbf{I}$ & $(\mathbf{I}-\mathbf{W})^{\frac{1}{2}}$ & $\mathbf{W}$  &$1$&$\underline{\rho}^{-\frac{1}{2}}$&$1$&$({1-\rho})^{-\frac{1}{2}}$&$\sqrt{\rho}$\\
        \hline
       ATC-GT & $\mathbf{W}^2$ & $\mathbf{I}-\mathbf{W}$ & $\mathbf{W}^2$  &$1$&$1$&$\rho^2$&$({1-\rho})^{-1}$&$\frac{1+\rho}{2}$\\
        \hline
        Semi-ATC-GT & $\mathbf{W}$ & $\mathbf{I}-\mathbf{W}$ & $\mathbf{W}^2$  &$1$&$1$&$\rho$&$({1-\rho})^{-1}$&$\frac{1+\rho}{2}$\\
        \hline
        Non-ATC-GT & $\mathbf{I}$ & $\mathbf{I}-\mathbf{W}$ & $\mathbf{W}^2$  &$1$&$1$&$1$&$({1-\rho})^{-1}$&$\frac{1+\rho}{2}$\\
        \hline
    \end{tabular}}
\end{table}

\subsubsection{Theoretical gap between upper-level and lower-level}\label{Theoretical Gap}
Note that $\|{\boldsymbol{\Lambda}}_{sa}\|\le1$. We rewrite the upper bound of the transient iteration complexity in Lemma~\ref{formaltransienttime} as
\begin{equation}\label{apptranabb}
    \max\{n^3\delta_y,n^3\delta_z,n^2\delta_x,n\hat{\delta}_y,n\hat{\delta}_z,n\hat{\delta}_x\}
\end{equation}
where 
\begin{small}
\begin{equation}\label{delta_value}
\begin{aligned}
&\delta_y=\left(\frac{\|\mathbf{O}_y\|^2\|\mathbf{O}_y^{-1}\|^2}{1-\|\mathbf{\Gamma}_y\|}\right)^2\|{\boldsymbol{\Lambda}}_{ya}\|^2 ,\delta_z=\left(\frac{\|\mathbf{O}_z\|^2\|\mathbf{O}_z^{-1}\|^2}{1-\|\mathbf{\Gamma}_z\|}\right)^2\|{\boldsymbol{\Lambda}}_{za}\|^2 ,\delta_x=\left(\frac{\|\mathbf{O}_x\|\|\mathbf{O}_x^{-1}\|}{1-\|\mathbf{\Gamma}_z\|}\|{\boldsymbol{\Lambda}}_{za}\|\right)^2,
\\&\hat{\delta}_y=\left(\frac{\|\mathbf{O}_y\|\|\mathbf{O}_y^{-1}\|\|{\boldsymbol{\Lambda}}_{yb}^{-1}\|}{1-\|\mathbf{\Gamma}_y\|}\right)^{\frac{4}{3}},\hat{\delta}_z=\left(\frac{\|\mathbf{O}_z\|\|\mathbf{O}_z^{-1}\|\|{\boldsymbol{\Lambda}}_{zb}^{-1}\|}{1-\|\mathbf{\Gamma}_z\|}\right)^{\frac{4}{3}},
\\&\hat{\delta}_x=\left(\frac{\|\mathbf{O}_x\|^2\|\mathbf{O}_x^{-1}\|^2\|{\boldsymbol{\Lambda}}_{xb}^{-1}\|^2}{1-\|\mathbf{\Gamma}_x\|}\right)^{\frac{2}{3}}+\frac{\|\mathbf{O}_x\|\|\mathbf{O}_x^{-1}\|\|{\boldsymbol{\Lambda}}_{xb}^{-1}\|}{1-\|\mathbf{\Gamma}_x\|}.
\end{aligned}
\end{equation}
\end{small}
Suppose that we use the same communication matrices and heterogeneity-correction methods for updating $x,y,z$, \ie 
\begin{equation}
\begin{aligned}
&\|\mathbf{O}_x\|=\|\mathbf{O}_y\|=\|\mathbf{O}_z\|,\|\mathbf{O}_x^{-1}\|=\|\mathbf{O}_y^{-1}\|=\|\mathbf{O}_z^{-1}\|,\|\mathbf{\Gamma}_x\|=\|\mathbf{\Gamma}_y\|=\|\mathbf{\Gamma}_z\|,
\\&\|{\boldsymbol{\Lambda}}_{xa}\|=\|{\boldsymbol{\Lambda}}_{ya}\|=\|{\boldsymbol{\Lambda}}_{za}\|,
\|{\boldsymbol{\Lambda}}_{xb}^{-1}\|=\|{\boldsymbol{\Lambda}}_{yb}^{-1}\|=\|{\boldsymbol{\Lambda}}_{zb}^{-1}\|.
\end{aligned}
\end{equation}
Then we have
\begin{equation}\label{appdeltaxyz}
    \delta_x\lesssim\delta_y=\delta_z, \,\hat{\delta}_x\lesssim\hat{\delta}_y=\hat{\delta}_z,
\end{equation}

Now we fix the update strategies for $y,z$.
\eqref{appdeltaxyz} implies that we can appropriately increase 
$\delta_x,\hat{\delta}_x$ while keeping the transient iteration complexity \eqref{apptranabb} unchanged (at most scaled by a constant factor). For example, we can use a moderately sparser communication network for updating $x$ than $y,z$. We illustrate this point with three examples: \ours-ED, \ours-EXTRA and \ours-GT (variants), where $y,z$ share the same communication matrix $\mathbf{W}_y$.

\begin{itemize}[leftmargin = 2em]
    \item \ours-ED, \ours-EXTRA: From Table~\ref{table_Suda_coeffient}, we have
\begin{equation}
\begin{aligned}
    &\delta_x=\mathcal{O}\left((1-\rho(\mathbf{W}_x))^{-2}\right), \delta_y=\delta_z=\mathcal{O}\left((1-\rho(\mathbf{W}_y))^{-2}\right), \\&\hat{\delta}_x=\mathcal{O}\left((1-\rho(\mathbf{W}_x))^{-\frac{3}{2}}\right),\hat{\delta}_y=\hat{\delta}_z=\mathcal{O}\left((1-\rho(\mathbf{W}_y))^{-2}\right).
    \end{aligned}
\end{equation}
Substituting these values into \eqref{apptranabb}, we get the  transient iteration complexity is bounded by  
    \begin{equation}
        \max\left\{n^2(1-\rho(\mathbf{W}_x))^{-2},n^3(1-\rho(\mathbf{W}_y))^{-2}\right\}
    \end{equation}
    SPARKLE-ED will keep the transient iteration
complexity $n^3(1-\rho(\mathbf{W}_y))^{-2}$ (the dominated term) if 
    \begin{equation}\label{ed-w-const}
       (1-\rho(\mathbf{W}_x))^{-1}\lesssim\sqrt{n} (1-\rho(\mathbf{W}_y))^{-1}.
    \end{equation}

    \item \ours-GT variants: Results in Table~\ref{table_Suda_coeffient} imply that 
    \begin{equation}
\begin{aligned}
    &\delta_x=\mathcal{O}\left((1-\rho(\mathbf{W}_x))^{-2}\right), \delta_y=\delta_z=\mathcal{O}\left((1-\rho(\mathbf{W}_y))^{-2}\right), \\&\hat{\delta}_x=\mathcal{O}\left((1-\rho(\mathbf{W}_x))^{-2}\right),\hat{\delta}_y=\hat{\delta}_z=\mathcal{O}\left((1-\rho(\mathbf{W}_y))^{-\frac{8}{3}}\right).
    \end{aligned}
\end{equation}
Following the same argument as
before, we have the following upper bound of the transient iteration complexity of \ours-GT
\begin{equation}
        \max\left\{n^2(1-\rho(\mathbf{W}_x))^{-2},n^3(1-\rho(\mathbf{W}_y))^{-2},n(1-\rho(\mathbf{W}_y))^{-\frac{8}{3}}\right\} .
    \end{equation}
we get the constraints of the spectral gap $1-\rho(\mathbf{W}_x)$ that maintains the transient iteration
complexity $\max\left\{n^3(1-\rho(\mathbf{W}_y))^{-2},n(1-\rho(\mathbf{W}_y))^{-\frac{8}{3}}\right\}$:
  \begin{equation}\label{gt-w-const}
        (1-\rho(\mathbf{W}_x))^{-1}\lesssim \max\left\{\sqrt{n}(1-\rho(\mathbf{W}_y))^{-1}, n^{-1/2} (1-\rho(\mathbf{W}_y))^{-\frac{4}{3}}\right\}.
    \end{equation}

\end{itemize}

Denote the communication times per agent of $\mathbf{W}_x,\mathbf{W}_y$ by $c_x,c_y$ respectively. For example, we have $c_x=2$, $c_y=n-1$ when taking  Ring Graph for $x$ (\ie $[\mathbf{W}_x]_{ij}\neq 0$ iff $|i-j|\in\{0,1,n-1\}$ ), and Complete Graph for $y$ (\ie $\mathbf{W}_y=\frac{1}{n}\mathbf{1}_n\mathbf{1}_n^{\top}$).

Then for each agent, the communication cost per round is $\mathcal{O}(c_xp+c_yq)$.  If we take $a={c_x}/{c_y}$ to measure the relative sparsity of the two communication matrices, and consider $c_y=\mathcal{O}(1)$, then for each agent, the communication cost per round is $\mathcal{O}(ap+q)$. \eqref{ed-w-const} and  \eqref{gt-w-const} theoretically provide the range of the sparsity (connectivity) degree of $\mathbf{W}_x$ relative to $\mathbf{W}_y$. From \eqref{ed-w-const} and \eqref{gt-w-const}, we can set $a\ll1$, while maintaining the transient iteration complexity for \ours-GT, \ours-ED, \ours-EXTRA.

\subsubsection{The transient iteration complexities of some specific examples in SPARKLE}\label{deviation-tran}
Now we compute the transient iteration complexities of each SPARKLE-\textbf{L}-\textbf{U} algorithm, where $\textbf{L},\textbf{U}\in\{\text{GT (variants)},\text{ED},\text{EXTRA}\}$. For brevity, here we assume that $\mathbf{W}_x=\mathbf{W}_y=\mathbf{W}_z$, use the same heterogeneity-correction method to $y,z$,  and denote the spectral gap $1-\rho(\mathbf{W}_x)$ by $1-\rho$. 

Substituting the results in Table \ref{table_Suda_coeffient} into \eqref{apptranabb} and \eqref{delta_value}, we get 
\begin{equation}
\delta_x=\mathcal{O}\left(\frac{1}{(1-\rho)^2}\right),  \delta_y=\delta_z=\mathcal{O}\left(\frac{1}{(1-\rho)^2}\right)
\end{equation}
for any $\textbf{L},\textbf{U}\in \{\text{GT (variants)},\text{ED},\text{EXTRA}\}$, 
\begin{equation}
\hat{\delta}_x=\mathcal{O}\left(\frac{1}{(1-\rho)^2}\right),  \mathcal{O}\left(\frac{1}{(1-\rho)^{3/2}}\right),\mathcal{O}\left(\frac{1}{(1-\rho)^{3/2}}\right)
\end{equation}
for  $\textbf{U}= \{\text{GT (variants)},\text{ED},\text{EXTRA}\}$ respectively, and 
\begin{equation}
\hat{\delta}_y=\hat{\delta}_z=\mathcal{O}\left(\frac{1}{(1-\rho)^{8/3}}\right),\mathcal{O}\left(\frac{1}{(1-\rho)^2}\right),\mathcal{O}\left(\frac{1}{(1-\rho)^2}\right)
\end{equation}
for  $\textbf{L}=\{\text{GT (variants)},\text{ED},\text{EXTRA}\}$ respectively.

Combining the above results, we can directly obtain Table \ref{extrasient}, the transient iteration complexities of  SPARKLE with mixed heterogeneity-correction techniques in different levels.

\subsection{Convergence analysis in deterministic scenarios} \label{deterministic-proof}
The following lemma gives the convergence rate of Algorithm~\ref{D-SOBA-SUDA} without a moving average when there is no sample noise:

\begin{lemma}\label{deterministicrate}
   Suppose that  Assumptions~\ref{smooth}-~\ref{var} hold. If
    $\sigma^2=0$, 
   then there exist $\alpha,\beta,\gamma$ and $\theta=1$ such that
\begin{equation}
  \begin{aligned}
&\frac{1}{K+1}\sum_{k=0}^K\mathbb{E}\|\Phi(\bar{x}^k)\|^2
\\ \lesssim
& 
\left(\frac{\kappa^{16}\|\mathbf{O}_y\|^2\|\mathbf{O}_y^{-1}\|^2\|{\boldsymbol{\Lambda}}_{ya}\|^2\|{\boldsymbol{\Lambda}}_{yb}^{-1}\|^2\zeta^y_0}{1-\|\mathbf{\Gamma}_y\|}\right)^{\frac{1}{3}}\frac{1}{K}+\left(\frac{\kappa^{14}\|\mathbf{O}_z\|^2\|\mathbf{O}_z^{-1}\|^2\|{\boldsymbol{\Lambda}}_{za}\|^2\|{\boldsymbol{\Lambda}}_{zb}^{-1}\|^2\zeta^z_0}{1-\|\mathbf{\Gamma}_z\|}\right)^{\frac{1}{3}}\frac{1}{K}
\\&+\left(\frac{\kappa^8\|\mathbf{O}_x\|^2\|\mathbf{O}_x^{-1}\|^2\|{\boldsymbol{\Lambda}}_{xa}\|^2\|{\boldsymbol{\Lambda}}_{xb}^{-1}\|^2\zeta^x_0}{1-\|\mathbf{\Gamma}_x\|}\right)^{\frac{1}{3}}\frac{1}{K}+\widetilde{C}_\alpha \frac{1}{K}.
  \end{aligned}
\end{equation}
where $\widetilde{C}_\alpha$ is a series of overheads which is defined below.
   
   \begin{proof} 
Note that $\sigma^2=0$ implies that $L_1=\Theta
   (L^2)$ when $\alpha=\mathcal{O}(L_{\nabla \Phi}^{-1})$. Thus \eqref{f1} implies that: 
   \begin{equation}\label{deter_1}
  \begin{aligned}
&\frac{1}{K+1}\sum_{k=0}^K\mathbb{E}\|\Phi(\bar{x}^k)\|^2
\\ \lesssim& \frac{\Phi(\bar{x}_0)-\inf \Phi}{\alpha(K+1)}
 \\&+
L^2\left[\kappa^4
(\eta_1+\kappa^2L_{y^{\star}}^2\eta_2)\alpha^2+\left(\frac{\alpha^2L_{z^\star}^2}{\gamma^2\mu_g^2}+\frac{\kappa^4\alpha^2}{\beta^2\mu_g^2}L_{y^{\star}}^2\right)\right]\left(\frac{\Phi(\bar{x}_0)-\inf\Phi}{\alpha (K+1)}\right)
\\&+L^2\frac{\kappa^6\|\mathbf{O}_y\|^2\mathbb{E}\|\hat{\mathbf{e}}_y^{0}\|^2}{n(K+1)(1-\|\mathbf{\Gamma}_y\|)}+L^2\frac{\kappa^4\|\mathbf{O}_z\|^2\mathbb{E}\|\hat{\mathbf{e}}_z^{0}\|^2}{n(K+1)(1-\|\mathbf{\Gamma}_z\|)}+L^2\frac{\kappa^6\|\mathbf{O}_x\|^2\mathbb{E}\|\hat{\mathbf{e}}_x^{0}\|^2}{n(K+1)(1-\|\mathbf{\Gamma}_x\|)}
\\&+L^2\frac{\|z^1_{\star}\|^2}{\mu_g\gamma(K+1)}
+\frac{L^2\kappa^4}{K+1}\frac{1}{\beta\mu_g}\|\bar{y}_{0}-y^{\star}(\bar{x}^{0})\|^2.
  \end{aligned}
\end{equation}    
  
Then we aim to choose the stepsize $\alpha,\beta,\gamma$. Define:
\begin{equation}
  \begin{aligned}
\widetilde{C}_\alpha=&L_{\nabla \Phi}+\kappa^3\frac{\|\mathbf{O}_x\|\|\mathbf{O}_x^{-1}\|\|{\boldsymbol{\Lambda}}_{xa}\|L}{1-\|\mathbf{\Gamma}_x\|}
+\kappa^3L\left(\frac{\|\mathbf{O}_x\|\|\mathbf{O}_x^{-1}\|\|{\boldsymbol{\Lambda}}_{xa}\|\|{\boldsymbol{\Lambda}}_{xb}^{-1}\|}{1-\|\mathbf{\Gamma}_x\|}\right)^{\frac{1}{2}}
\\&+\kappa^4\frac{L_{g,1}^2}{\mu_g}+\kappa^4\frac{\|\mathbf{O}_y\|\|\mathbf{O}_y^{-1}\|\|{\boldsymbol{\Lambda}}_{ya}\|L_{g,1}}{1-\|\mathbf{\Gamma}_y\|}
+\kappa^4L_{g,1}\left(\frac{\kappa\|\mathbf{O}_y\|\|\mathbf{O}_y^{-1}\|\|{\boldsymbol{\Lambda}}_{ya}\|\|{\boldsymbol{\Lambda}}_{yb}^{-1}\|}{1-\|\mathbf{\Gamma}_y\|}\right)^{\frac{1}{2}}
\\&+\kappa^{6}L\frac{\|\mathbf{O}_z\|\|\mathbf{O}_z^{-1}\|\|{\boldsymbol{\Lambda}}_{za}\|}{1-\|\mathbf{\Gamma}_z\|}
+\kappa^{\frac{11}{2}}L\left(\frac{\|\mathbf{O}_z\|\|\mathbf{O}_z^{-1}\|\|{\boldsymbol{\Lambda}}_{za}\|\|{\boldsymbol{\Lambda}}_{zb}^{-1}\|}{1-\|\mathbf{\Gamma}_z\|}\right)^{\frac{1}{2}},
\\\widetilde{\alpha}_{yb,2}=&\left(\frac{1-\|\mathbf{\Gamma}_y\|}{\kappa^{13}\|\mathbf{O}_y\|^2\|\mathbf{O}_y^{-1}\|^2\|{\boldsymbol{\Lambda}}_{ya}\|^2\|{\boldsymbol{\Lambda}}_{yb}^{-1}\|^2\zeta^y_0}\right)^{\frac{1}{3}},\\ 
\widetilde{\alpha}_{zb,2}=&\left(\frac{1-\|\mathbf{\Gamma}_z\|}{\kappa^{11}\|\mathbf{O}_z\|^2\|\mathbf{O}_z^{-1}\|^2\|{\boldsymbol{\Lambda}}_{za}\|^2\|{\boldsymbol{\Lambda}}_{zb}^{-1}\|^2\zeta^z_0}\right)^{\frac{1}{3}},
\\\widetilde{\alpha}_{xb,2}=&\left(\frac{1-\|\mathbf{\Gamma}_x\|}{\kappa^{5}\|\mathbf{O}_x\|^2\|\mathbf{O}_x^{-1}\|^2\|{\boldsymbol{\Lambda}}_{xa}\|^2\|{\boldsymbol{\Lambda}}_{xb}^{-1}\|^2\zeta^x_0}\right)^{\frac{1}{3}}.
 \end{aligned}
\end{equation}
Then there exist 
\begin{equation}
\\\alpha=\Theta\left(\widetilde{C}_\alpha+\widetilde{\alpha}_{xb,2}^{-1}+\widetilde{\alpha}_{yb,2}^{-1}+\widetilde{\alpha}_{zb,2}^{-1}\right)^{-1},\beta=\Theta\left(\kappa^4\alpha\right) ,\gamma=\Theta\left(\kappa^4\alpha\right)
\end{equation}
such that \eqref{l15}, \eqref{betady}, \eqref{gammaz}, \eqref{l11}, \eqref{betay}, \eqref{gammaez}, \eqref{l12}, and \eqref{stepsize14} hold. Then all previous lemmas hold. 

Then from \eqref{deter_1} we have:
   \begin{equation}
  \begin{aligned}
&\frac{1}{K+1}\sum_{k=0}^K\mathbb{E}\|\Phi(\bar{x}^k)\|^2
\\ \lesssim& \frac{\kappa}{\alpha K}+\frac{\alpha^2\kappa^{14}\|\mathbf{O}_y\|^2\|\mathbf{O}_y^{-1}\|^2\|{\boldsymbol{\Lambda}}_{ya}\|^2\|{\boldsymbol{\Lambda}}_{yb}^{-1}\|^2\zeta^y_0}{K(1-\|\mathbf{\Gamma}_y\|)}
\\&+\frac{\alpha^2\kappa^{12}\|\mathbf{O}_z\|^2\|\mathbf{O}_z^{-1}\|^2\|{\boldsymbol{\Lambda}}_{za}\|^2\|{\boldsymbol{\Lambda}}_{zb}^{-1}\|^2\zeta^z_0}{K(1-\|\mathbf{\Gamma}_z\|)}+\frac{\alpha^2\kappa^6\|\mathbf{O}_x\|^2\|\mathbf{O}_x^{-1}\|^2\|{\boldsymbol{\Lambda}}_{xa}\|^2\|{\boldsymbol{\Lambda}}_{xb}^{-1}\|^2\zeta^x_0}{K(1-\|\mathbf{\Gamma}_x\|)}
\\ \lesssim& 
\left(\frac{\kappa^{16}\|\mathbf{O}_y\|^2\|\mathbf{O}_y^{-1}\|^2\|{\boldsymbol{\Lambda}}_{ya}\|^2\|{\boldsymbol{\Lambda}}_{yb}^{-1}\|^2\zeta^y_0}{1-\|\mathbf{\Gamma}_y\|}\right)^{\frac{1}{3}}\dfrac{1}{K}+\left(\frac{\kappa^{14}\|\mathbf{O}_z\|^2\|\mathbf{O}_z^{-1}\|^2\|{\boldsymbol{\Lambda}}_{za}\|^2\|{\boldsymbol{\Lambda}}_{zb}^{-1}\|^2\zeta^z_0}{1-\|\mathbf{\Gamma}_z\|}\right)^{\frac{1}{3}}\dfrac{1}{K}
\\&+\left(\frac{\kappa^8\|\mathbf{O}_x\|^2\|\mathbf{O}_x^{-1}\|^2\|{\boldsymbol{\Lambda}}_{xa}\|^2\|{\boldsymbol{\Lambda}}_{xb}^{-1}\|^2\zeta^x_0}{1-\|\mathbf{\Gamma}_x\|}\right)^{\frac{1}{3}}\dfrac{1}{K}+\widetilde{C}_\alpha \dfrac{1}{K}.
  \end{aligned}
\end{equation}

   \end{proof}
\end{lemma}

\subsection{Degenerating to single-level algorithms}\label{Degenerating}
We consider the bilevel problem with the following upper- and lower-level loss function on the $i$-th agent:
\[
F_i(x,y,\phi)= F_i(x,\phi),\quad G_i(x,y,\xi)\equiv \frac{\|y\|^2}{2}.
\]
Actually, this optimization problem with respect to $x$ is single-level, since we have  $\mathbf{z}^k\equiv0$, $\mathbf{y}^k\equiv0$, $u_{i}^k=\nabla_1 f_i(x_{i}^k,\xi_{i}^k)$ by induction. By taking $\theta=1$, we get the following single-level algorithm framework for decentralized stochastic single-level algorithm. As we discuss in previous sections, it can recover various heterogeneity-correction algorithms, including GT, EXTRA and ED, by selecting specific $\mathbf{A}_x,\mathbf{B}_x,\mathbf{C}_x$.

\begin{algorithm}[t]
  \caption{\ours:
degenerating to single-level decentralized stochastic algorithms}
  \label{D-SOBA-SUDA-deg}
  \begin{algorithmic}
  \REQUIRE{Initialize $\mathbf{x}^0=\mathbf{0}$, $\mathbf{d}_x^0=\mathbf{0}$, learning rate $\alpha_k$}.
  \FOR{$k=0,1,\cdots,K-1$}         
  \STATE $\mathbf{x}^{k+1}\hspace{0.2mm}=\mathbf{C}_x\mathbf{x}^k-\alpha_k\mathbf{A}_x \mathbf{u}^{k}-\mathbf{B}_x\mathbf{d}_x^k$,\hspace{1.7mm} $ \mathbf{d}_x^{k+1}=\mathbf{d}_x^k+\mathbf{B}_x\mathbf{x}^{k+1} ;$
  \ENDFOR
  \end{algorithmic}
\end{algorithm}

In this case, we have $z_k^{\star}\equiv0$, $y_k^{\star}\equiv0$.
Notice that $L_{y^{\star}}=0$, $L_{z^{\star}}=0$. It gives
\begin{equation}
\begin{aligned}
\eta_2&=\mathcal{O}\left(\beta^2\frac{\|\mathbf{O}_y\|^2\|\mathbf{O}_y^{-1}\|^2\|{\boldsymbol{\Lambda}}_{ya}\|^2\|{\boldsymbol{\Lambda}}_{yb}^{-1}\|^2}{(1-\|\mathbf{\Gamma}_y\|)^2}+\gamma^2\frac{\|\mathbf{O}_z\|^2\|\mathbf{O}_z^{-1}\|^2\|{\boldsymbol{\Lambda}}_{za}\|^2\|{\boldsymbol{\Lambda}}_{zb}^{-1}\|^2}{(1-\|\mathbf{\Gamma}_z\|)^2}\right),
\\\eta_1&=\mathcal{O}\left(\eta_2+\alpha^2\left(1+\frac{(1-\theta)^2}{\theta^2\|{\boldsymbol{\Lambda}}_{xb}^{-1}\|^2}\right)\frac{\|\mathbf{O}_x\|^2\|\mathbf{O}_x^{-1}\|^2\|{\boldsymbol{\Lambda}}_{xa}\|^2\|{\boldsymbol{\Lambda}}_{xb}^{-1}\|^2}{(1-\|\mathbf{\Gamma}_x\|)^2}\right).
\end{aligned}
\end{equation}

If we take
\[\alpha\lesssim \min\left\{1, \frac{1-\|\mathbf{\Gamma}_x\|}{\|\mathbf{O}_x\|\|\mathbf{O}_x^{-1}\|\|{\boldsymbol{\Lambda}}_{xa}\|}, \left(\frac{1-\|\mathbf{\Gamma}_x\|}{\|\mathbf{O}_x\|\|\mathbf{O}_x^{-1}\|\|{\boldsymbol{\Lambda}}_{xa}\|\|{\boldsymbol{\Lambda}}_{xb}^{-1}\|}\right)^{\frac{1}{2}}\right\}\]
and $\theta=1$, $\beta\to 0$, $\gamma\to 0$, then \eqref{l15}, \eqref{betady}, \eqref{gammaz}, \eqref{l11}, \eqref{betay}, \eqref{gammaez}, \eqref{l12}, and \eqref{stepsize14} hold. 
Thus all previous lemmas hold. 
Then \eqref{f1}  transforms into 
\begin{equation}
  \begin{aligned}
&\frac{1}{K+1}\sum_{k=0}^K\mathbb{E}\|\Phi(\bar{x}^k)\|^2
\\ \lesssim& \frac{f(\bar{x}_0)-\inf f}{\alpha(K+1)}+\frac{1}{n}\left(\theta(1-\theta)+\alpha\theta^2\right) \sigma_{f,1}^2
+\frac{(1-\theta)^2}{\theta(K+1)}\|\nabla f\left(\bar{x}^{0}\right)\|^2
 \\&+
\eta_1\alpha^2\left(\frac{f(\bar{x}_0)-\inf f}{\alpha (K+1)}+\frac{\theta}{n}\sigma_{f,1}^2\right)
+\frac{\|\mathbf{O}_x\|^2\mathbb{E}\|\hat{\mathbf{e}}_x^{0}\|^2}{n(K+1)(1-\|\mathbf{\Gamma}_x\|)}
\\&+\frac{\alpha^2\|\mathbf{O}_x\|^2\|\mathbf{O}_x^{-1}\|^2\|{\boldsymbol{\Lambda}}_{xa}\|^2}{n(1-\|\mathbf{\Gamma}_x\|)^2(K+1)}
\left[
\frac{1-\theta}{\theta}\sum_{i=1}^n\left\|\nabla f_i(\bar{x}^{0})\right\|^2\right]
\\&+\frac{1}{n}\left[\alpha^2\theta\left(\theta+\frac{1-\theta}{1-\|\mathbf{\Gamma}_x\|}\right)\frac{\|\mathbf{O}_x\|^2\|\mathbf{O}_x^{-1}\|^2\|{\boldsymbol{\Lambda}}_{xa}\|^2n}{1-\|\mathbf{\Gamma}_x\|}\right]\sigma_{f,1}^2.
  \end{aligned}
\end{equation}

It follows that
\begin{equation}\label{singleL}
  \begin{aligned}
    &\frac{1}{K+1}\sum_{k=0}^K\mathbb{E}\|\Phi(\bar{x}^k)\|^2
    \\\lesssim&\frac{f(\bar{x}_0)-\inf f}{\alpha(K+1)}+\frac{\alpha \sigma_{f,1}^2}{n}+
\left(
\alpha^4\frac{\|\mathbf{O}_x\|^2\|\mathbf{O}_x^{-1}\|^2\|{\boldsymbol{\Lambda}}_{xa}\|^2\|{\boldsymbol{\Lambda}}_{xb}^{-1}\|^2}{(1-\|\mathbf{\Gamma}_x\|)^2}\right)\left(\frac{f(\bar{x}_0)-\inf f}{\alpha (K+1)}+\frac{1}{n}\sigma_{f,1}^2\right)
\\&+\frac{\|\mathbf{O}_x\|^2\mathbb{E}\|\hat{\mathbf{e}}_x^{0}\|^2}{n(K+1)(1-\|\mathbf{\Gamma}_x\|)}+\alpha^2\frac{\|\mathbf{O}_x\|^2\|\mathbf{O}_x^{-1}\|^2\|{\boldsymbol{\Lambda}}_{xa}\|^2}{1-\|\mathbf{\Gamma}_x\|}\sigma_{f,1}^2
\\\lesssim&\frac{f(\bar{x}_0)-\inf f}{\alpha(K+1)}+\frac{\alpha \sigma_{f,1}^2}{n}
+\alpha^2\frac{\|\mathbf{O}_x\|^2\|\mathbf{O}_x^{-1}\|^2\|{\boldsymbol{\Lambda}}_{xa}\|^2}{1-\|\mathbf{\Gamma}_x\|}\sigma_{f,1}^2
\\&
+\frac{\alpha^2\|\mathbf{O}_x\|^2\|\mathbf{O}_x^{-1}\|^2\|\boldsymbol{\Lambda}_{xa}\|^2\|\boldsymbol{\Lambda}_{xb}^{-1}\|^2\zeta^x_0}{(K+1)(1-\|\mathbf{\Gamma}_x\|)}+\frac{\alpha^4\|\mathbf{O}\|_x^2\|\mathbf{O}_x^{-1}\|^2\|{\boldsymbol{\Lambda}}_{xa}\|^2\|{\boldsymbol{\Lambda}}_{xb}^{-1}\|^2\sigma_{f,1}^2}{n(1-\|\mathbf{\Gamma}_x\|)^2}.
  \end{aligned}
\end{equation}

Like \eqref{stepsizeconstants}, we take \begin{equation}
\begin{aligned}
    &C_0=1+ \frac{\|\mathbf{O}_x\|\|\mathbf{O}_x^{-1}\|\|{\boldsymbol{\Lambda}}_{xa}\|}{1-\|\mathbf{\Gamma}_x\|}+ \left(\frac{\|\mathbf{O}_x\|\|\mathbf{O}_x^{-1}\|\|{\boldsymbol{\Lambda}}_{xa}\|\|{\boldsymbol{\Lambda}}_{xb}^{-1}\|}{1-\|\mathbf{\Gamma}_x\|}\right)^{\frac{1}{2}},
    \\&\alpha_1=\sqrt{\frac{n}{K\sigma_{f,1}^2}},\quad \alpha_2=\left(\frac{1-\|\mathbf{\Gamma}_x\|}{K\|\mathbf{O}_x\|^2\|\mathbf{O}_x^{-1}\|^2\|{\boldsymbol{\Lambda}}_{xa}\|^2\sigma_{f,1}^2}\right)^{\frac{1}{3}},
    \\&\alpha_3=\left(\frac{1-\|\mathbf{\Gamma}_x\|}{\|\mathbf{O}_x\|^2\|\mathbf{O}_x^{-1}\|^2\|{\boldsymbol{\Lambda}}_{xa}\|^2\|{\boldsymbol{\Lambda}}_{xb}^{-1}\|^2\zeta_{0}^x}\right)^{\frac{1}{3}},
\\&\alpha_4=\left(\frac{n(1-\|\mathbf{\Gamma}_x\|)^2}{K\|\mathbf{O}\|_x^2\|\mathbf{O}_x^{-1}\|^2\|{\boldsymbol{\Lambda}}_{xa}\|^2\|{\boldsymbol{\Lambda}}_{xb}^{-1}\|^2\sigma_{f,1}^2}\right)^{\frac{1}{5}},
\\&\alpha=\Theta\left(C_0+\frac{1}{\alpha_1}
+\frac{1}{\alpha_2}+\frac{1}{\alpha_3}+\frac{1}{\alpha_4}\right)^{-1}.
  \end{aligned}  
\end{equation}
Substituting these values into \eqref{singleL}, we get

\begin{equation}
  \begin{aligned}
    &\frac{1}{K+1}\sum_{k=0}^K\mathbb{E}\|\Phi(\bar{x}^k)\|^2
\lesssim\frac{\sigma_{f,1}}{
\sqrt{nK}}
+\left(\frac{\|\mathbf{O}_x\|^2\|\mathbf{O}_x^{-1}\|^2\|{\boldsymbol{\Lambda}}_{xa}\|^2\sigma_{f,1}^2}{1-\|\mathbf{\Gamma}_x\|}\right)^{\frac{1}{3}}K^{-2/3}+\frac{C_0}{K}
\\&
+\left(\frac{\|\mathbf{O}_x\|^2\|\mathbf{O}_x^{-1}\|^2\|\boldsymbol{\Lambda}_{xa}\|^2\|\boldsymbol{\Lambda}_{xb}^{-1}\|^2\zeta^x_0}{(1-\|\mathbf{\Gamma}_x\|)}\right)^{\frac{1}{3}} \frac{1}{K}+\left(\frac{\|\mathbf{O}\|_x^2\|\mathbf{O}_x^{-1}\|^2\|{\boldsymbol{\Lambda}}_{xa}\|^2\|{\boldsymbol{\Lambda}}_{xb}^{-1}\|^2\sigma_{f,1}^2}{n(1-\|\mathbf{\Gamma}_x\|)^2}\right)^{\frac{1}{5}}K^{-4/5}.
  \end{aligned}
\end{equation}

Like Lemma \ref{formaltransienttime}, we get the transient iterating complexity for Algorithm \ref{D-SOBA-SUDA-deg} is  
\begin{equation}
    \left\{n^3 \left(\frac{\|\mathbf{O}_x\|^2\|\mathbf{O}_x^{-1}\|^2}{1-\|\mathbf{\Gamma}_x\|}\right)^2\|{\boldsymbol{\Lambda}}_{xa}\|^2 , n\left(\frac{\|\mathbf{O}_x\|\|\mathbf{O}_x^{-1}\|\|{\boldsymbol{\Lambda}}_{xb}^{-1}\|}{1-\|\mathbf{\Gamma}_x\|}\right)^{\frac{4}{3}}\|{\boldsymbol{\Lambda}}_{xa}\|,n\right\}.
\end{equation}
Substituting the value of relevant  norms in Table \ref{table_Suda_coeffient}, we get the transient iteration complexity for GT, EXTRA, ED are
\begin{equation}
    \mathcal{O}\left(\max\left\{\frac{n^3}{(1-\rho)^2},\frac{n}{(1-\rho)^{8/3}}\right\}\right),\,\mathcal{O}\left(\frac{n^3}{(1-\rho)^2}\right),\,\mathcal{O}\left(\frac{n^3}{(1-\rho)^2}\right)
\end{equation}
respectively, where $\rho:=\rho(\mathbf{W}_x)$. These upper bounds are the same as the state-of-the-art results shown in Table \ref{table:comparison}. It indicates that our analysis accurately captures the impacts of updates at each level on the convergence results.

\section{Experimental details}
\label{app:experiment}
In this section, we provide the details of our numerical experiments discussed in Section \ref{section:experiment}. We also provide addition experimental results which are not mentioned in the main text due to the space limitation. For all GT variants, we focus on one typical representative, ATC-GT, in our experiments, which we denote as GT for brevity. All experiments described in this section were run on an NVIDIA A100 server.
\subsection{Synthetic bilevel optimization}
\label{app:toyexperiment}
Here, we consider problem \eqref{DSBO_problem} whose upper- and lower level loss functions on the $i$-th agents ($1\leq i\leq N$) are denoted as:
\begin{subequations}
\label{exp:toymodel_detail}
\begin{align}
f_i(x,y)&=\mathbb{E}_{A_i,b_i}\left[\left\Vert A_iy-b_i \right\Vert^2\right], \\
g_i(x,y)&=\mathbb{E}_{A_i,b_i}\left[\left\Vert A_iy-x \right\Vert^2+C_r\left\Vert y\right\Vert^2\right],
\end{align}
\end{subequations}
where $x\in\mathbb{R}^D, y\in\mathbb{R}^K$ and $C_r$ denotes a fixed regularization parameter. For each agent $i$, we firstly generate the local solution $y_i^*, x_i^*$ as $y_i^*=y^*+\zeta_i$ and $x_i^*=A^*b^*+\xi_i$, where $x^*\sim\mathcal{N}(0,I_K)$ is a randomly generated vector, each element of $A^*$ is independently sampled from $\mathcal{N}(0,9)$.  The observation $(A_i,b_i)$ on agent $i$ is generated in a streaming manner by $A_i=A^*+\phi_i$, $b_i=x^*_i+\psi_i$, in which each element of $\phi_i\in\mathbb{R}^{K\times D}$ and $\psi_i\in\mathbb{R}^D$ are independently generated by $\mathcal{N}(0,\sigma_g^2)$. The terms $\xi_i\sim\mathcal{N}(0,\sigma_h^2I_K)$ and $\zeta_i\sim\mathcal{N}(0,\sigma_h^2I_D)$ control the heterogeneity of data distributions across different agents.

We set $D=20,K=10,\sigma_g=0.001, C_r=0.001$. Then we set $\sigma_h=0.5$ to represent severe heterogeneity across agents and $\sigma_h=0.1$ for mild heterogeneity. We run D-SOBA, \ours-GT, \ours-ED, and \ours-EXTRA over Ring, 2D-Torus~\cite{nedic2018network}, and fully connected networks with $N=64$ agents. The moving-average term $\theta=0.1$ and the step-size at the $t$-th iteration are $\alpha_t=\beta_t=\gamma_t=1/(500+0.01t)$. The batch size is 10. 
\begin{figure}[t]
\vspace{-8pt}
\centering
	\subfigure{
        \includegraphics[width=0.33\textwidth]{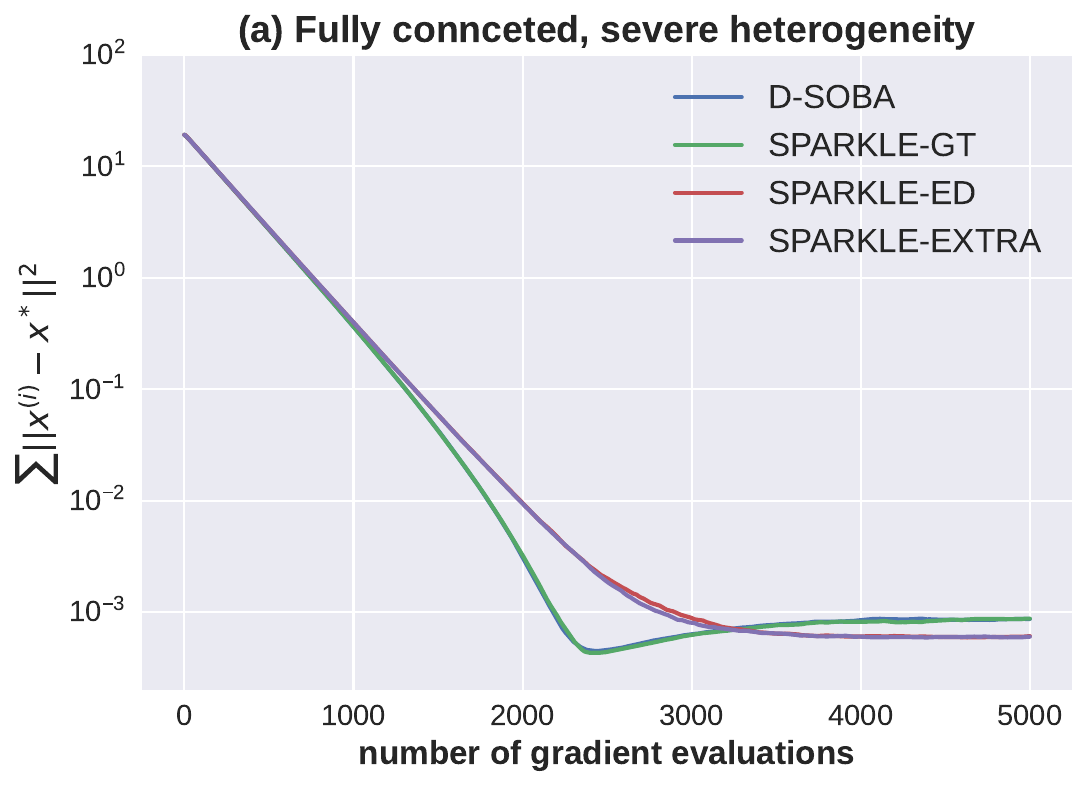}}
  \hspace{-10pt}
	\subfigure{
		\includegraphics[width=0.33\textwidth]{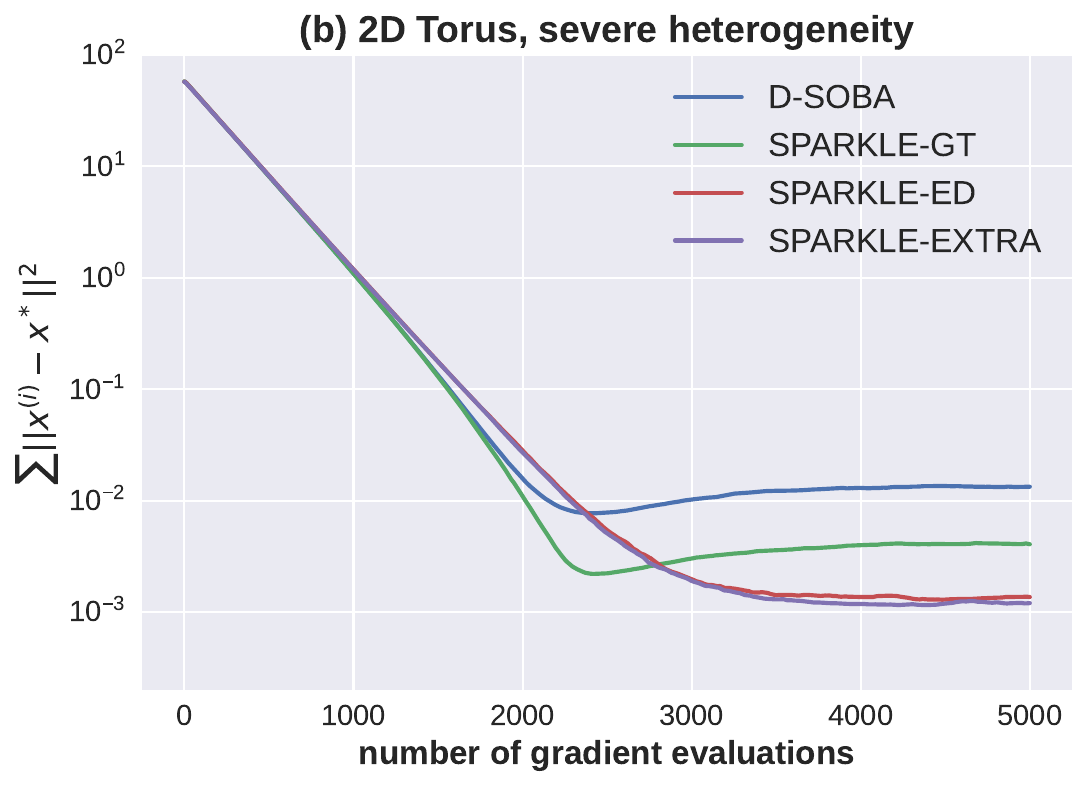}}
  \hspace{-10pt}
	\subfigure{
		\includegraphics[width=0.33\textwidth]{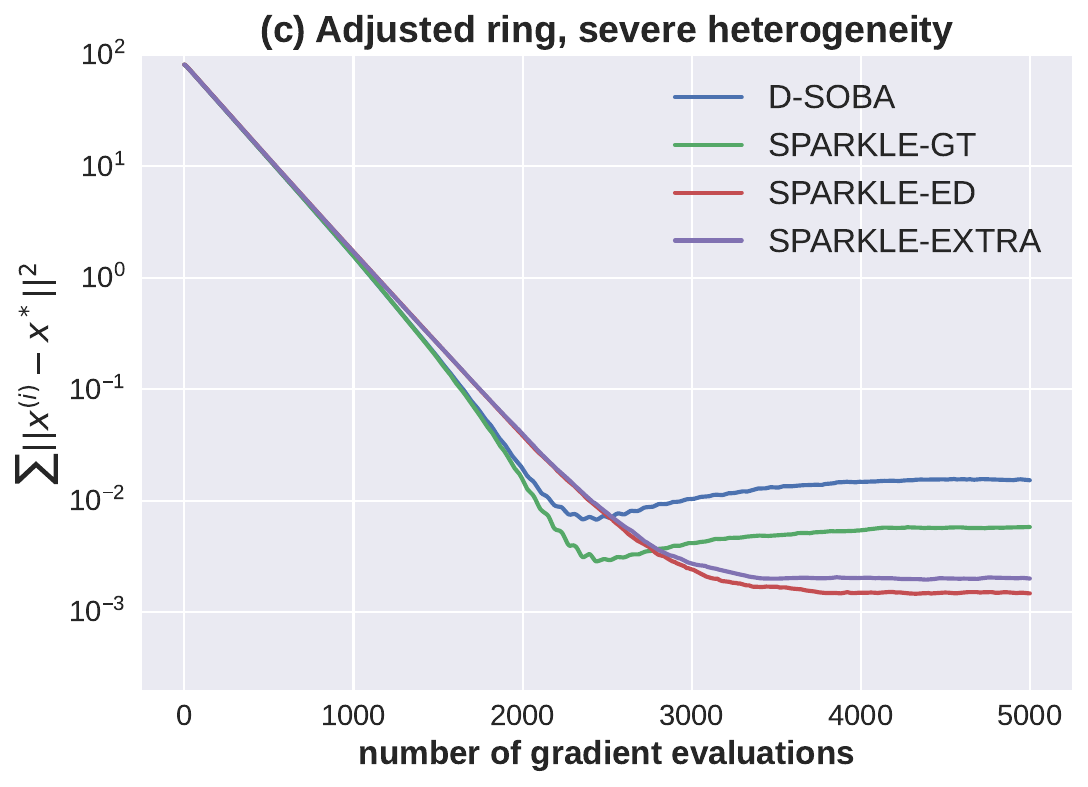}}
	\subfigure{
		\includegraphics[width=0.33\textwidth]{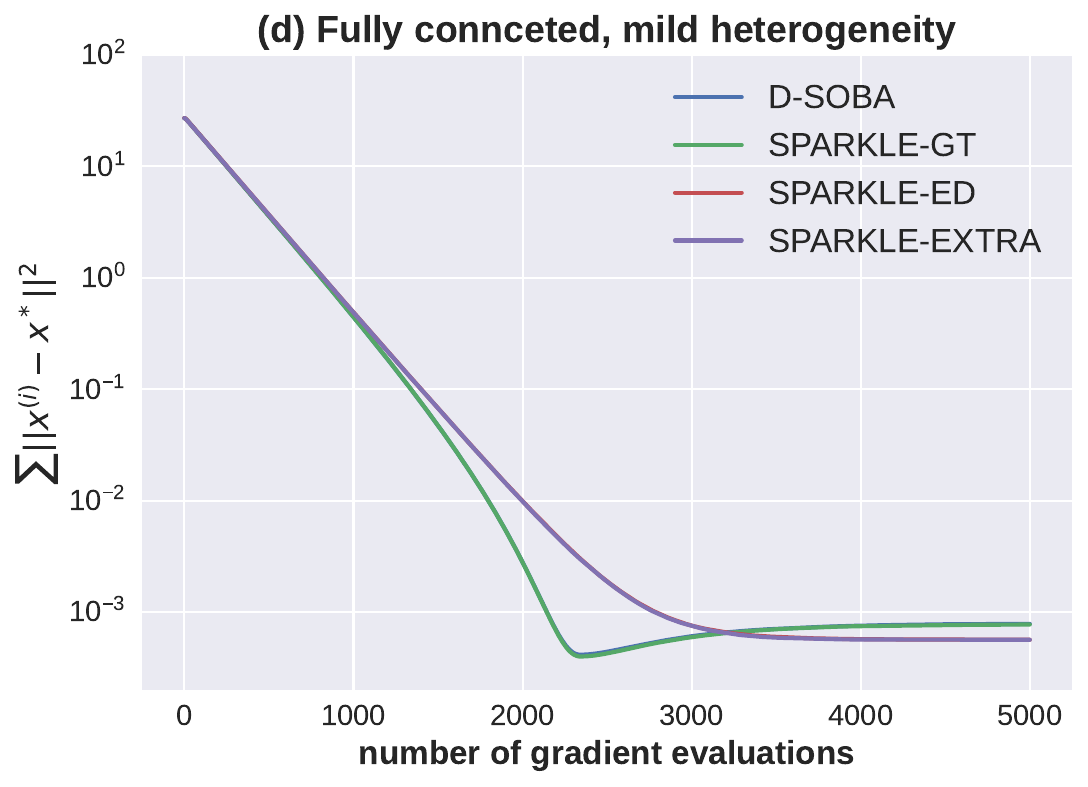}}
  \hspace{-10pt}
	\subfigure{
		\includegraphics[width=0.33\textwidth]{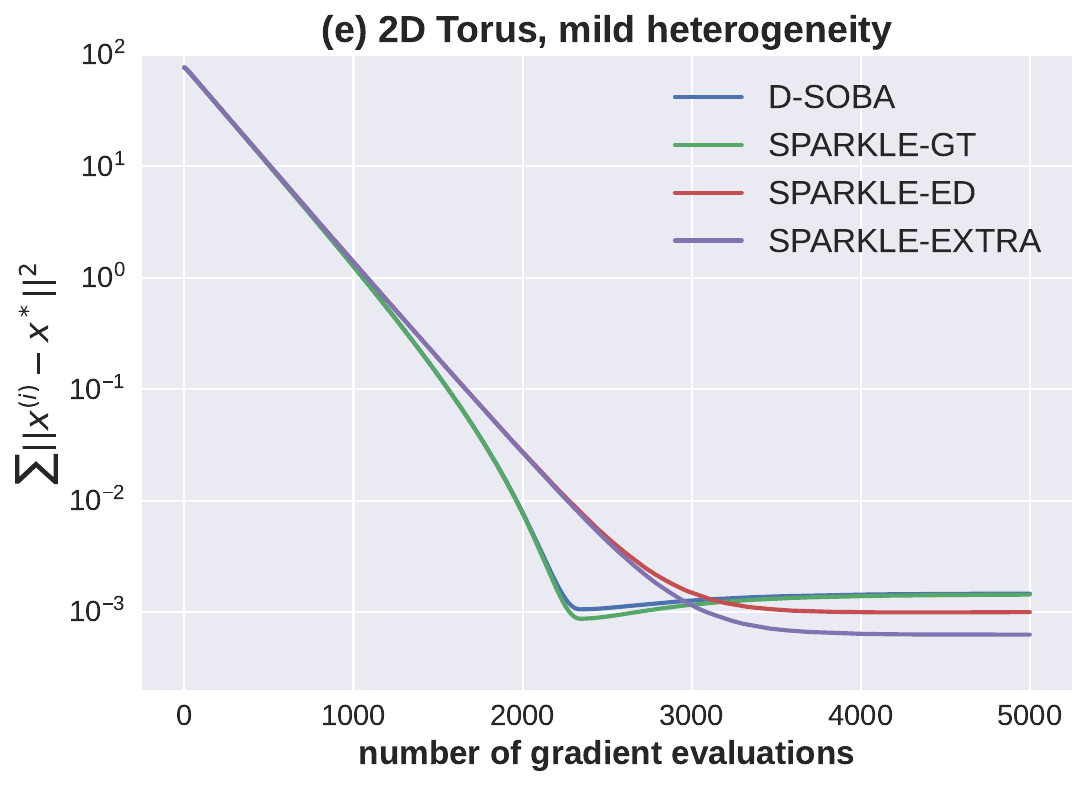}}
  \hspace{-10pt}
	\subfigure{
		\includegraphics[width=0.33\textwidth]{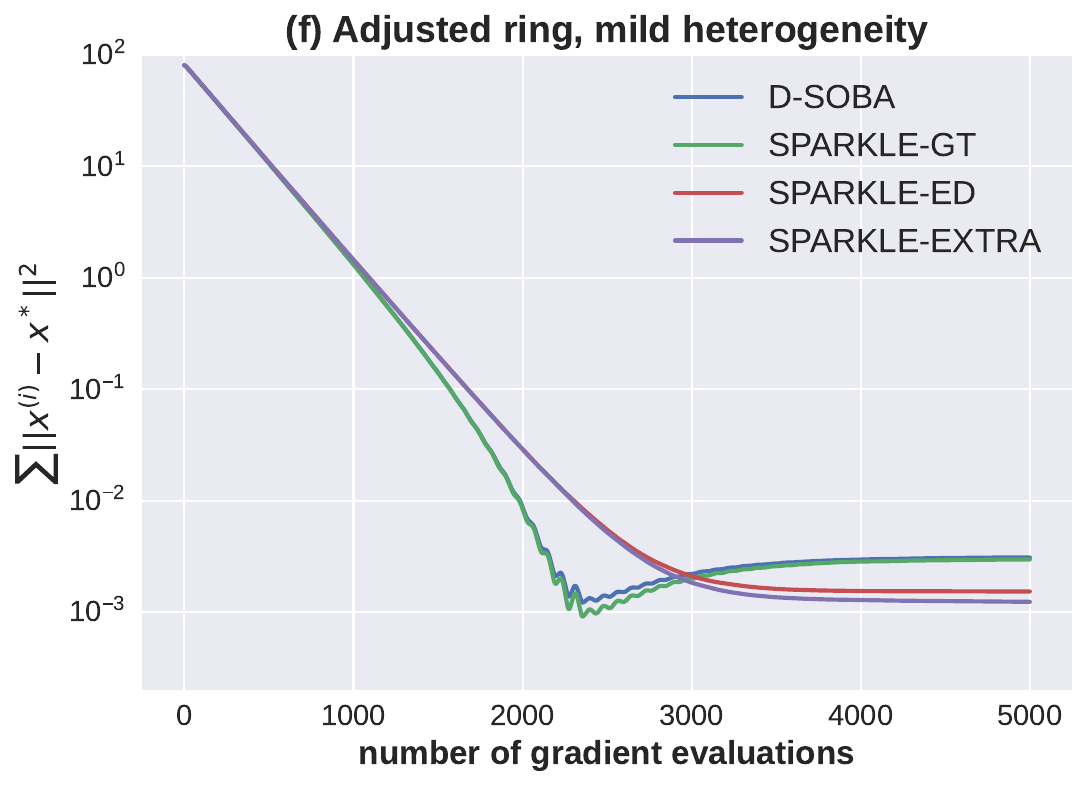}}
\caption{The estimation error of D-SOBA, \ours-GT, \ours-ED, and \ours-EXTRA under different networks and data heterogeneity.}
\vspace{-10pt}
\label{fig: FashionMINST_algs}
\end{figure}

Fig.~\ref{fig: FashionMINST_algs} illustrates the averaged estimation error $\sum_{i=1}^N\left\Vert x_i^{(t)}-x^*\right\Vert^2$ of the mentioned algorithms with different communication topology and data heterogeneity. It is observed that \ours with ED, EXTRA, GT achieve  better convergence performances with decentralized communication networks. Meanwhile, \ours-ED and \ours-EXTRA are more robust to data heterogeneity and the sparsity of network topology than \ours-GT. All the results are consistent with our theoretical results.

\subsection{Hyper-cleaning on FashionMNIST dataset}
\label{app:cleaning}
Here, we consider a data hyper-clean problem~\cite{shaban2019truncated} on FashionMNIST dataset~\cite{xiao2017fashion}. The FashionMNIST dataset consists of 60000 images for training and 10000 images for testing and we randomly split 50000 training images into a training set and the other 10000 images into a validation set.

The data hyper-cleaning problem aims to train a classifier from a corrupted dataset, in which the label of each training data is replaced by a random class number with a probability $p$ (\ie the corruption rate). It can be considered as a stochastic bilevel problem \eqref{DSBO_problem} whose upper- and lower-level loss functions on the $i$-th agents ($1\leq i\leq n$) are formulated as:
\begin{subequations}
\label{exp:hyper-clean_detail_app}
\begin{align}
f_i(x,y)&=\dfrac{1}{\left|\mathcal{D}_{val}^{(i)}\right|}\sum_{(\xi_e,\zeta_e)\in D_{val}^{(i)}}L(\phi(\xi_e;y),\zeta_e),\\
g_i(x,y)&=\dfrac{1}{\left|\mathcal{D}_{tr}^{(i)}\right|}\sum_{(\xi_e,\zeta_e)\in \mathcal{D}_{tr}^{(i)}}\sigma(x_e)L(\phi(\xi_e;y),\zeta_e)+C\left\Vert y\right\Vert^2,
\end{align}
\end{subequations}
where $\phi$ denotes a training model while $y$ denotes its parameters, $L$ denotes the cross-entropy loss function and $\sigma(x)=(1+e^{-x})^{-1}$ is the sigmoid function. $\mathcal{D}_{tr}^{(i)}$ and $\mathcal{D}_{val}^{(i)}$ denotes the training and validation set of the $i$-th agent, respectively. $C>0$ is a fixed regularization parameter.

\textbf{Data generation and experiment settings. }In this experiment, we let $\phi$ be a two-layer MLP network with a 300-dim hidden layer and ReLU activation while $y$ denotes its parameters. For $1\leq i\leq10$, we sample a probability distribution $\mathcal{P}_i$ randomly by Dirichlet distribution with parameters $\alpha=0.1$. The training and validation images with label $i$ are sent to different agents according the probability distribution $\mathcal{P}_i$. Then $\mathcal{D}_{tr}^{(i)}$ and $\mathcal{D}_{val}^{(i)}$ are generated sufficiently heterogeneous~\cite{lin2021quasi}. We set $C=0.001$. The batch size is set to 50.

\begin{figure}
\centering
	\subfigure{
		\includegraphics[width=0.48\textwidth]{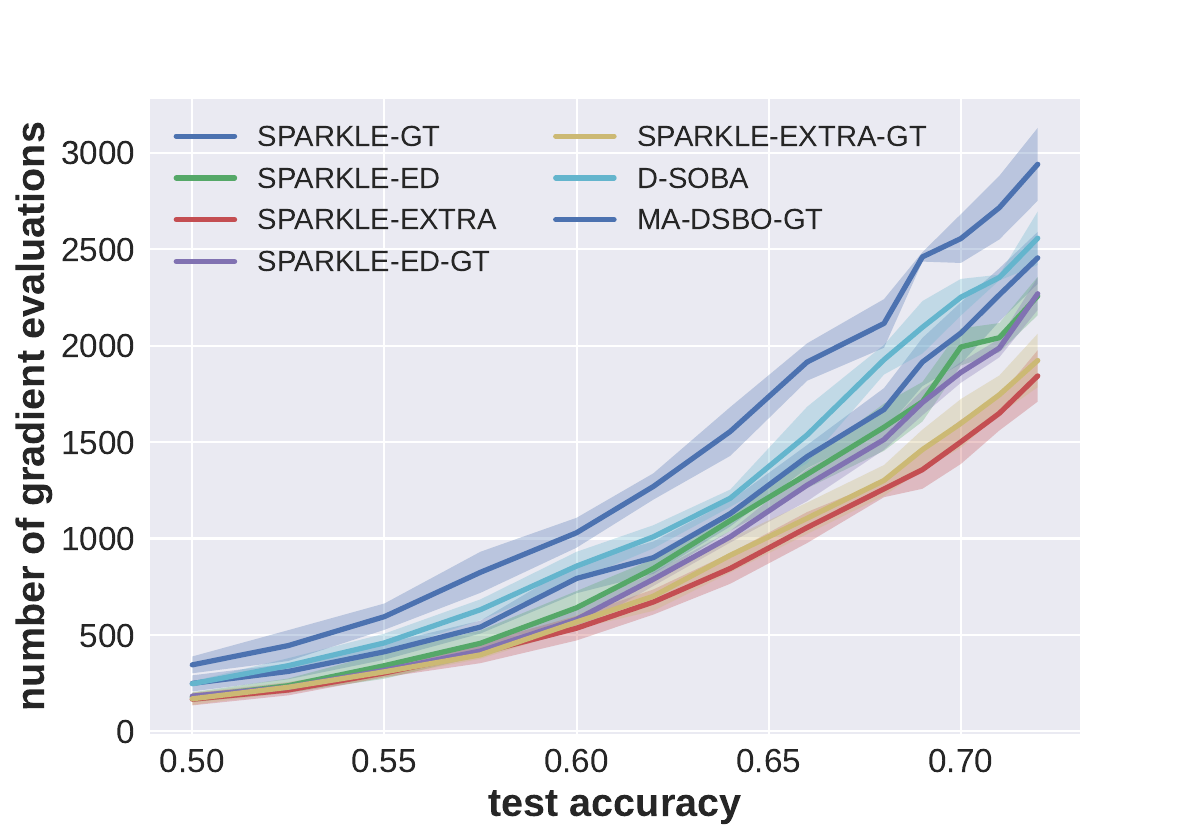}}
	\subfigure{
		\includegraphics[width=0.48\textwidth]{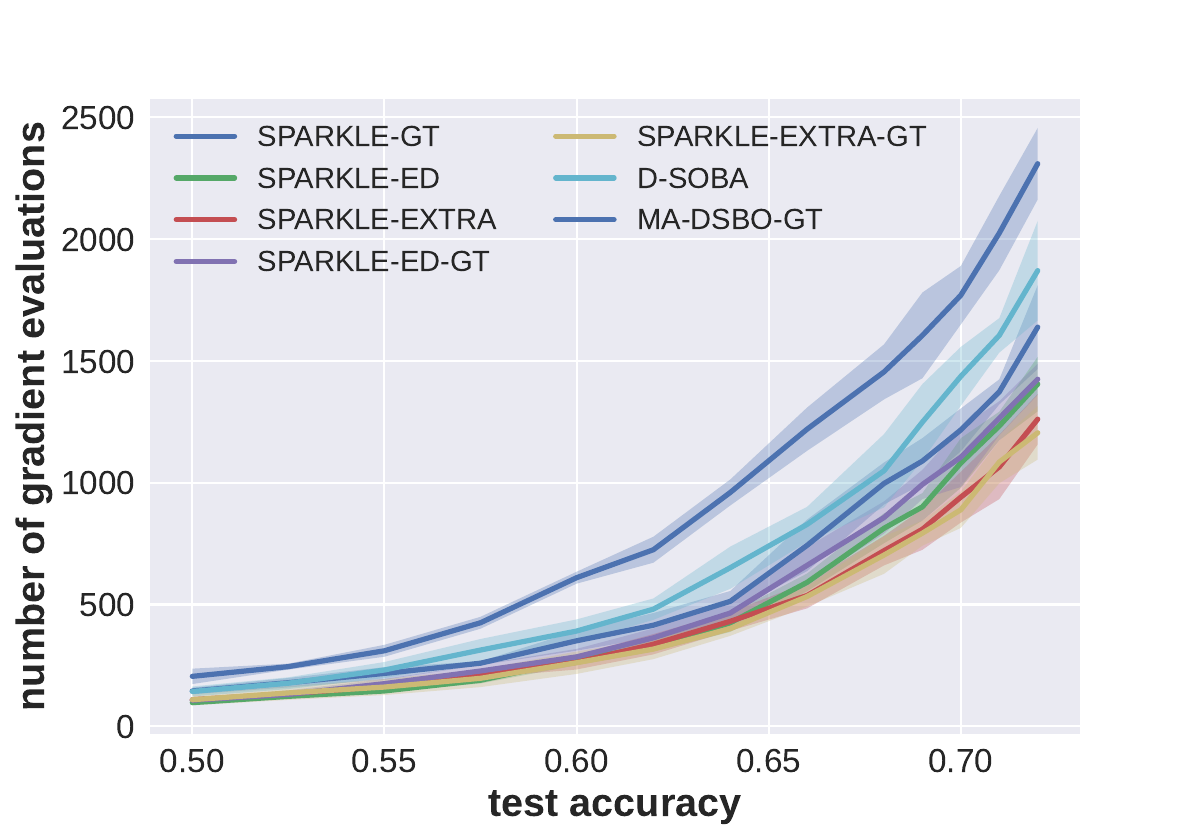}}
\caption{Hypergradient evaluation times for required test accuracy in hyper-cleaning problem. (Left: $p=0.2$; Right: $p=0.3$)}
\label{fig: evaluation times}
\end{figure}

\textbf{Convergence performances with different corruption rates. }We set the moving-average term $\theta_k=0.2$ and run D-SOBA \cite{kong2024decentralized}, MA-DSBO-GT \cite{chen2023decentralized}, MDBO \cite{gao2023convergence} \ours-GT, \ours-ED, \ours-EXTRA, \ours-ED-GT, and \ours-EXTRA-GT on an Adjusted Ring graph with $n=10$ agents and $p=0.1,0.2,0.3$ separately. The step-sizes for all the algorithms are set to $\alpha_k=\beta_k=\gamma_k=0.03$ and the term $\eta$ in MDBO is set to 0.5. The weight matrix of Adjust Ring $W=[w_{ij}]_{n\times n}$ satisfies:
\begin{align*}
w_{ij}=\begin{cases} a,\hspace{2.7em}\text{  if }j=i ,
\\ \dfrac{1-a}{2},\quad\text{if }(j-i)\%{n}=\pm1,
\\ 0,\hspace{2.7em}\text{  else}.
    \end{cases}
\end{align*} 
Moreover, we run \ours with ED in the lower level and auxiliary variable and gradient tracking in the upper level (i.e. \ours-ED-GT) as well as \ours with EXTRA in the lower level and auxiliary variable and gradient tracking in the upper level (i.e. \ours-EXTRA-GT) and compare their test accuracy with the other four algorithms.  

Figure~\ref{fig: FashionMINST_p_1}  shows that \ours-ED and \ours-EXTRA outperforms in different cases than \ours-GT. Meanwhile, \ours-EXTRA, \ours-EXTRA-GT  achieve similar test accuracy,  as do those for  \ours-ED and  \ours-ED-GT,  which matches our theoretical results in transient iteration analysis. Figure~\ref{fig: evaluation times} presents the times of gradient evaluation for different test accuracies of these algorithms at $p=0.2,0.3$, demonstrating similar results.

\textbf{Influence of network topology. }We set the corruption rate $p=0.3$, the step sizes $\alpha_k=\beta_k=\gamma_k=0.02$, and the moving-average term $\theta_k=0.2$. Then  we run SPARKLE-EXTRA and SPARKLE-EXTRA-GT on a network containing $n=10$ nodes with different topologies in the following two cases:
\begin{itemize}[leftmargin=1em]
    \item \textbf{Fixed upper, varied lower}: $x$ communicates through a five-peer graph; $y,z$ communicate through different adjusted rings with $\rho=0.647, 0.828, 0.924, 0.990$.
    \item \textbf{Fixed lower, varied upper}: $y,z$ communicate through a five-peer graph; $x$ communicates through different adjusted rings with $\rho=0.647, 0.828, 0.924, 0.990$.
\end{itemize}

The weight matrix of five-peer graph $W=[w_{ij}]_{n\times n}$ satisfies:
\begin{align*}
w_{ij}=\begin{cases} 0.2,\quad\text{if }(j-i)\%{n}=0,\pm1,\pm2,
\\ 0,\hspace{1.5em}\text{  else}.
    \end{cases}
\end{align*}

Figure~\ref{fig: FashionMINST_topo} shows the average test accuracy of both \ours-EXTRA and \ours-EXTRA-GT over 10 trials. It indicates that the test accuracy decays with  increasing spectral gap of topologies related to $y,z$ while the topology of $x$ is fixed during the whole iterations. However, such convergence gap becomes milder when the topologies of $y,z$ are fixed and that of $x$ varies. This phenomenon supports our theoretical findings, which suggest that the transient iteration complexity is more sensitive to the network topologies of $y,z$ than to that of $x$. 

\textbf{Influence of moving-average iteration on convergence. }
Moreover, for $\theta_t=0.05,0.2,0.3$, we run \ours-GT, \ours-ED, \ours-EXTRA, \ours-ED-GT, and \ours-EXTRA-GT on an Adjusted Ring graph with $n=10$ agents,  $\alpha_k=\beta_k=\gamma_k=0.03$ and $p=0.3$ for 3000 iterations. We obtain the average test accuracy of the last 40 iterations over 10 trials, and present the mean and standard deviation during the different trials in Table~\ref{table: FashionMINST_theta}. We can observe that most algorithms achieve the highest test accuracy when $\theta=0.2$, which may prove that a suitable $\theta$ can benefit the test accuracy in hyper-cleaning problems.

\begin{figure}[t]
\vspace{-10pt}
\centering
	\subfigure{
        \includegraphics[width=0.48\textwidth]{Figures//FashionMNIST/p_0.3_upper_extra.pdf}}
  \hspace{-18pt}
	\subfigure{
		\includegraphics[width=0.48\textwidth]{Figures//FashionMNIST/p_0.3_lower_extra.pdf}}
  \hspace{-18pt}
	\subfigure{
		\includegraphics[width=0.48\textwidth]{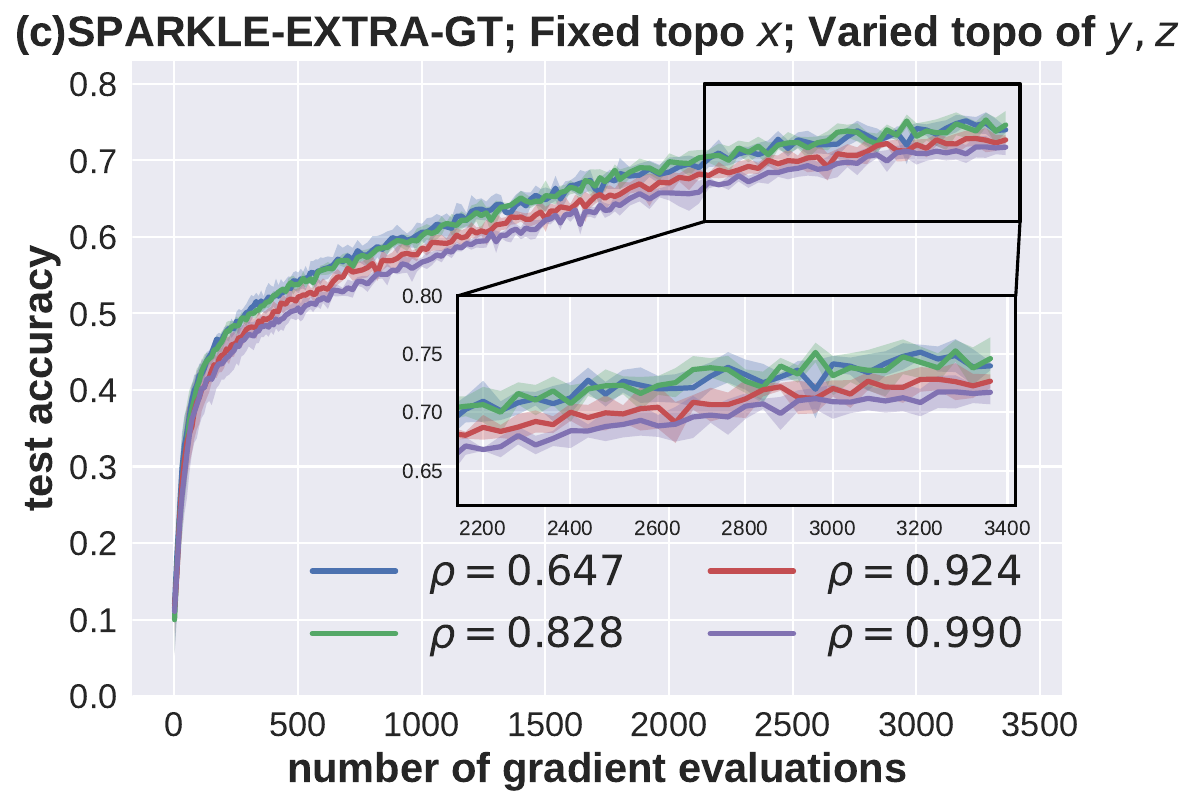}}
  \hspace{-18pt}
	\subfigure{
		\includegraphics[width=0.48\textwidth]{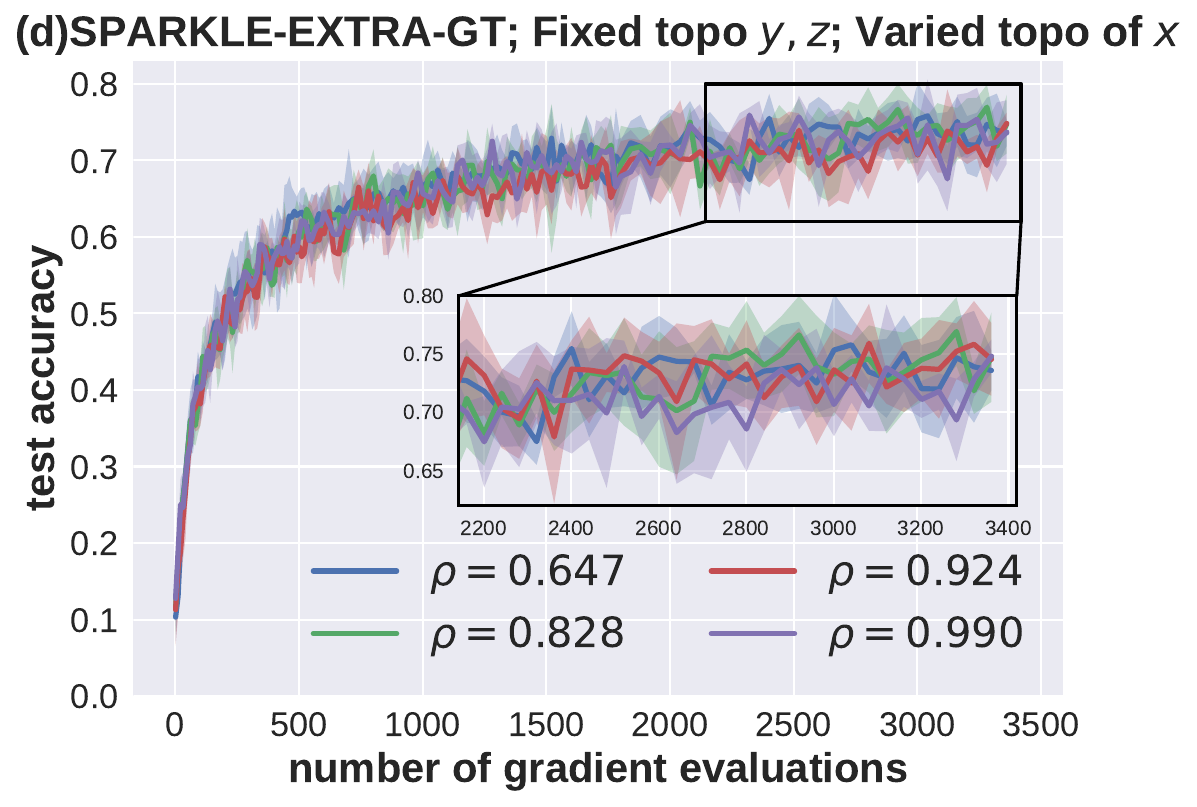}}
\caption{The average test accuracy of SPARKLE-EXTRA and SPARKLE-EXTRA-GT on hyper-cleaning with different communicating strategy of $x,y,z$.}
\label{fig: FashionMINST_topo}
\end{figure}
\begin{figure}
\centering
	\subfigure{
		\includegraphics[width=0.48\textwidth]{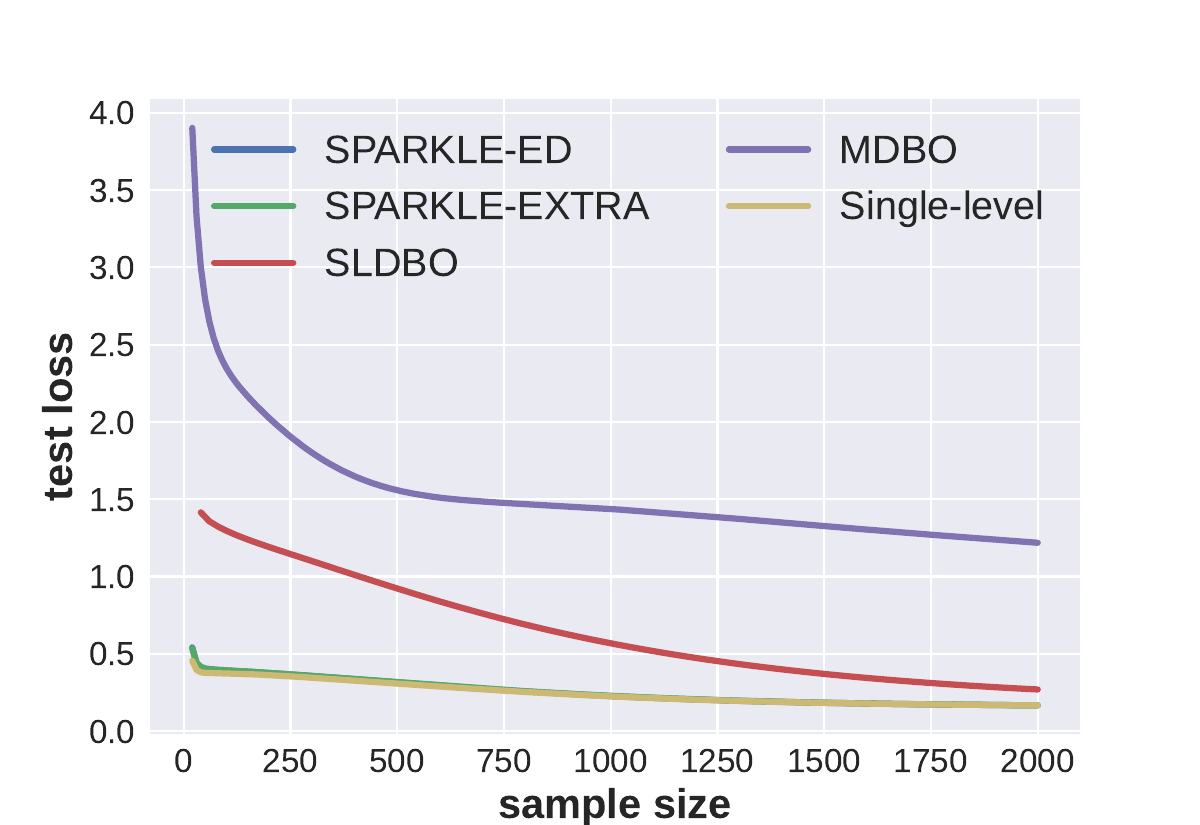}}
	\subfigure{
		\includegraphics[width=0.48\textwidth]{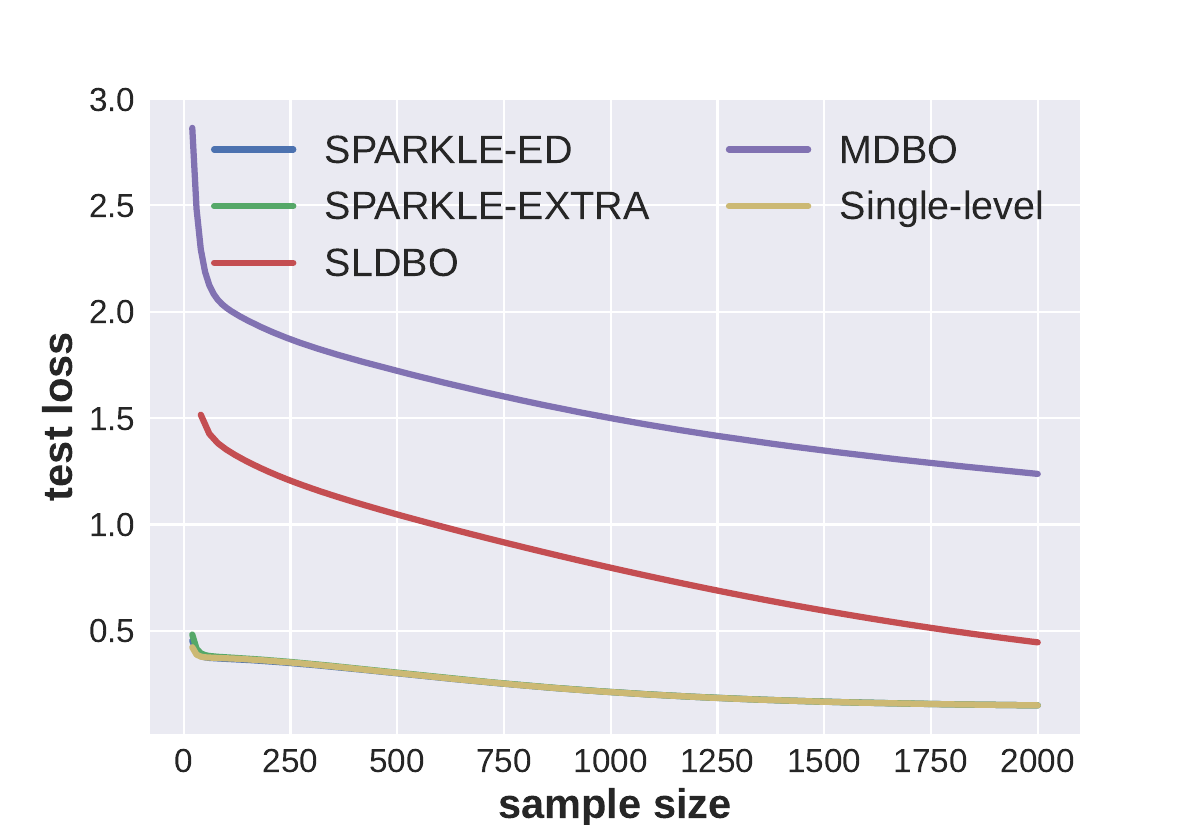}}
\caption{The test loss against samples generated by one agent of different algorithms in the policy evaluation. (Left: $n=20$, Right: $n=10$.)}
\label{fig: rein_test}
\end{figure}

\begin{table}
    \centering
    \caption{Mean and standard deviation of the average test accuracy of last 40 iterations during 10 trials with different moving-average terms}
    \begin{tabular}{ccccc}
    \hline
        \textbf{Algorithm} & $\theta=0.05$ & $\theta=0.2$ & $\theta=0.3$ \\ \hline
        \textbf{\ours-GT} & $\bf{0.7080\pm0.0215}$ & ${0.7045\pm0.0126}$ & $0.7064\pm0.0113$ \\ 
        \textbf{\ours-ED} & $0.7096\pm0.0074$  & $\bf{0.7113\pm0.0047}$ & $0.7110\pm0.0081$ \\ 
        \textbf{\ours-EXTRA} & $0.7190\pm0.0103$ & $\bf{0.7277\pm0.0090}$ & $0.7243\pm0.0028$ \\
        \textbf{\ours-ED-GT} & $0.7064\pm0.0063$ & $\bf{0.7178\pm0.0037}$ & $0.7162\pm0.0041$ \\ 
        \textbf{\ours-EXTRA-GT} & $0.7198\pm0.0051$ & $\bf{0.7262\pm0.0058}$ & $0.7247\pm0.0048$ \\  \hline
    \end{tabular}
    \label{table: FashionMINST_theta}
\end{table}

\begin{table}
\centering
\caption{The average training loss of the last 500 iterations for 10 independent trials in the distributed policy evaluation.}
\begin{tabular}{ccc}
\hline
   \textbf{Algorithm}             & $N=10$                     & $N=2$0 \\ \hline
\textbf{\ours-ED} & $0.2781\pm1.09\times10^{-3}$ & $0.3198\pm3.21\times10^{-3}$ \\
\textbf{\ours-EXTRA}   & $0.2743\pm0.88\times10^{-3}$ &   $0.3207\pm2.94\times10^{-3}$   \\
\textbf{MDBO}            & $1.0408\pm4.51\times10^{-3}$ &   $1.3293\pm8.38\times10^{-3}$   \\
\textbf{SLDBO}           & $0.4132\pm1.18\times10^{-3}$ &   $0.8374\pm2.47\times10^{-3}$   \\
\textbf{Single-level ED} & $0.2948\pm0.92\times10^{-3}$ &   $0.3164\pm3.12\times10^{-3}$   \\ \hline
\end{tabular}
\label{experiment:reinforce}
\end{table}

\subsection{Distributed policy evaluation in reinforcement learning} 
\label{app:rein}
Following the result of~\cite{yang2022decentralized}, we consider a multi-agent MDP problem in reinforcement learning on a distributed setting with $n$ agents. Denote $\mathcal{S}$ as the state space. Suppose that the value function in each state $s\in\mathcal{S}$ is a linear function $V(s)=\phi_s^\top x$, where $\phi_s\in\mathbb{R}^m$ is a feature and $x\in\mathbb{R}^m$ is a parameter. To obtain the optimal solution $x^*$, we consider the following Bellman minimization problem:
\begin{align}
    \min_{x\in\mathbb{R}^m}\quad F(x)=\dfrac{1}n\sum_{i=1}^n\left[ \dfrac{1}{2|\mathcal{S}|}\sum_{s\in\mathcal{S}}\left(\phi_s^\top x-\mathbb{E}_{s'}\left[r^i(s,s')+\gamma\phi_{s'}^\top x\middle|s\right]\right)^2\right]
\end{align}
where $r^i(s,s')$ denotes the reward incurred from transition $s$ to $s'$ on the $i$-th agent, $\gamma\in(0,1)$ denotes the discount factor. The expectation is taken over all random transitions from state $s$ to $s'$. It can be viewed as a bilevel optimization problem with the following upper- and lower-level loss:
\begin{subequations}
    \begin{align}
        f_i(x,y)&=\dfrac{1}{2|\mathcal{S}|}\sum_{s\in\mathcal{S}}(\phi_s^\top x-y_s)^2,\\
        g_i(x,y)&=\sum_{s\in\mathcal{S}}\left(y_s-\mathbb{E}_{s'}\left[r^i(s,s')+\gamma\phi_{s'}^\top x\middle|s\right]\right)^2,
    \end{align}
\end{subequations}
where $y=(y_1,\cdots,y_{|\mathcal{S}|})^\top\in\mathbb{R}^{|\mathcal{S}|}$. In our experiment, we set the number of states $|\mathcal{S}|=200$ and $m=10$. For each $s\in\mathcal{S}$, we generate its feature $\phi_s\sim U[0,1]^m$. The non-negative transition probabilities are generated randomly and standardized to satisfy $\sum_{s'\in\mathcal{S}}p_{s,s'}=1$. The mean reward $\bar{r}^i(s,s')$ are independently generated from the uniform distribution $U[0,1]$. In each iteration, the stochastic reward $r^i(s,s')\sim\mathcal{N}(\bar{r}^i(s,s'),0.02^2)$.

For $n=10,20$, we run \ours-ED and \ours-EXTRA as well as existing decentralized SBO algorithms MDBO~\cite{gao2023convergence} and SLDBO~\cite{dong2023single} (here  we use the stochastic gradient instead of deterministic gradient) over a Ring graph. For MDBO, the number of Hessian-inverse estimation iterations is set to $5$. The step sizes  are 0.03 for all methods. Figure~\ref{fig: rein_algs} illustrates the upper-level loss against samples generated by one agent for 10 independent trials. Table \ref{experiment:reinforce} shows the average training loss of the last 500 iterations for 10 independent trials of the four decentralized SBO algorithms as well as single-level ED~\cite{yuan2018exact1} (For bilevel algorithms, \emph{training loss} means the upper-level loss here). Both Figure~\ref{fig: rein_algs} and Table~\ref{experiment:reinforce} demonstrate that  \ours-ED and \ours-EXTRA converge faster than other methods.

Finally, we create  a fixed "test set" with 10000 sample generated from $\mathcal{S}$.  Figure \ref{fig: rein_test} shows the loss on the test set of SPARKLE-ED, SPARKLE-EXTRA, SLDBO, MDBO and single-level ED algorithm,  demonstrating the superior performance of SPARKLE compared to other decentralized SBO algorithms.

\subsection{Decentralized meta-learning}\label{meta_learning}
We consider a meta-learning problem as described in \citep{finn2017model}. There are  $R$ tasks $\{\mathcal{T}_s,s=1,\cdots,R\}$. Each task $\mathcal{T}_s$ has its own loss function $L(x,y_s,\xi)$, where  $\xi_s$ represents a stochastic sample drawn from the data distribution $\mathcal{D}_s$, $y_s$ denotes the task-specific parameters and $x$ denotes the  global parameters shared by all the tasks. In meta-learning problem, we aim to find the parameters $(x^*,y_1^*,\cdots,y_R^*)$ that minimizes the loss function across all $R$ tasks, \ie,
\begin{align}\label{meta,formulation}
    \min_{x,y_1,\cdots,y_R} l(x,y_1,\cdots,y_R)=\dfrac{1}{R}\sum_{s=1}^R\mathbb{E}_{\xi\sim\mathcal{D}_s}\left[L(x,y_s,\xi)\right].
\end{align}

\begin{figure}[t!]
\centering
	\subfigure{
		\includegraphics[width=0.48\textwidth]{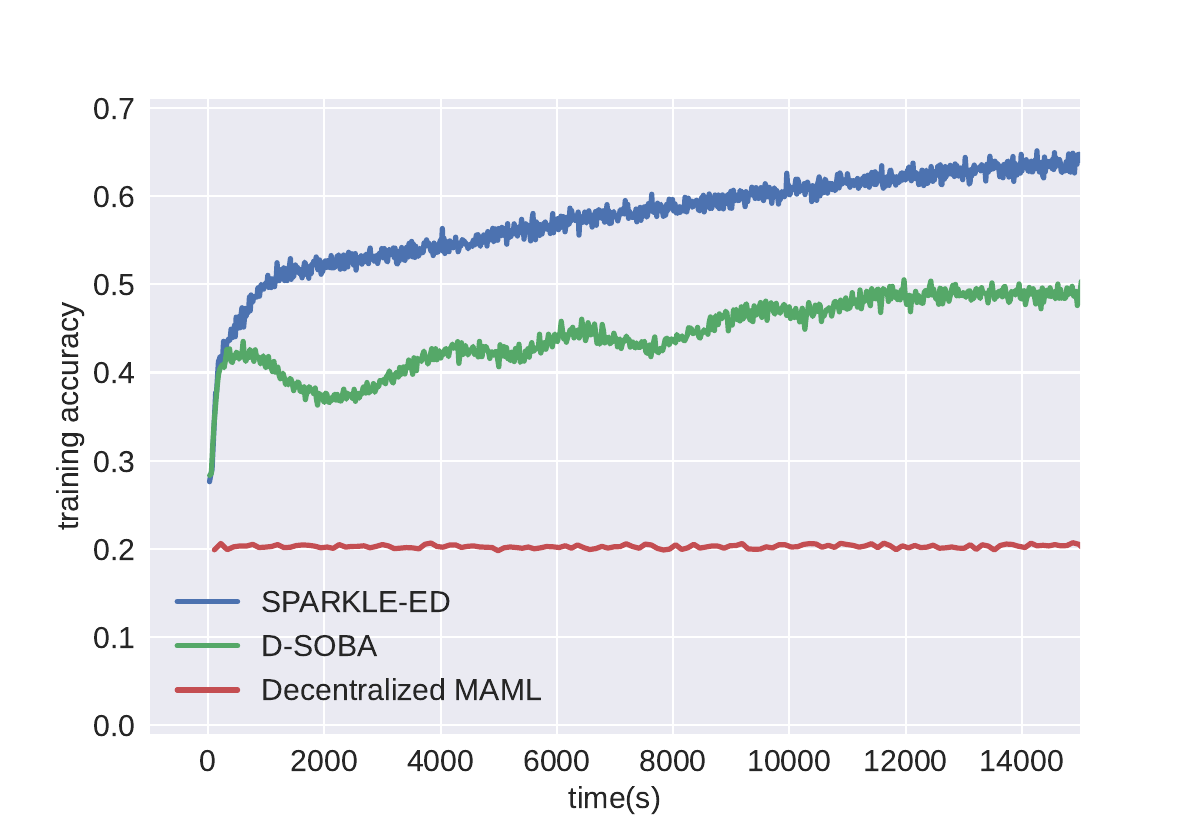}}
	\subfigure{
		\includegraphics[width=0.48\textwidth]{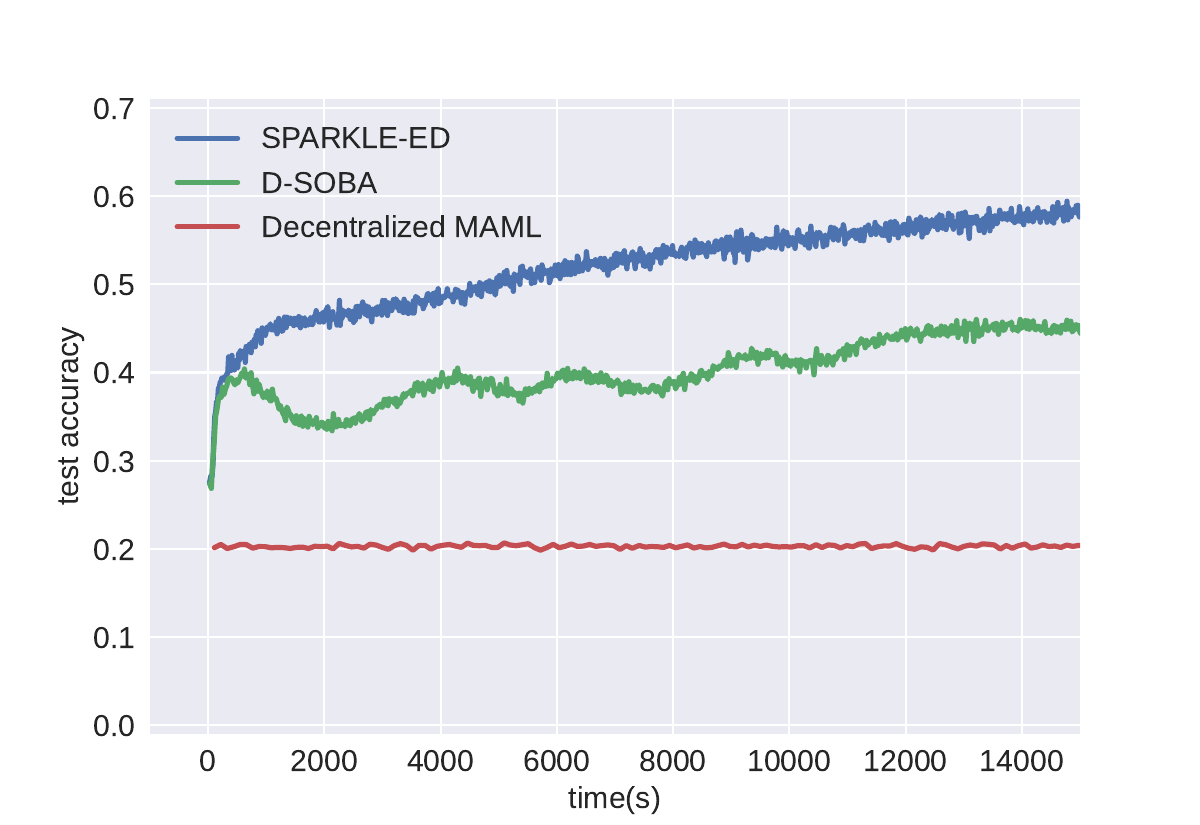}
        \label{fig: Fmeta}}
\caption{The accuracy on training and testing set of different algorithms for the meta-learning problem.}
\label{fig: FashionMINST_algs1}
\end{figure}

The problem \eqref{meta,formulation} can be formulated as a decentralized SBO problem with heterogeneous data distributions across $N$ nodes. For $i=1,2,\cdots,N$, let $\mathcal{D}^{\text{train}}_{s,i}$ and $\mathcal{D}^{\text{val}}_{s,i}$ denote the training and validation datasets for the $s$-th task $\mathcal{T}_s$ received by node $i$ respectively. We can then address the meta-learning problem by minimizing \eqref{DSBO_problem},  with the upper- and lower-level loss functions defined as:
\begin{equation}
    \begin{aligned}
        f_i(x,y)&=\dfrac{1}{R}\sum_{s=1}^R\mathbb{E}_{\xi\sim\mathcal{D}^{\text{val}}_{s,i}}\left[L(x,y_s,\xi)\right], \\
        g_i(x,y)&=\dfrac{1}{R}\sum_{s=1}^R\left[\mathbb{E}_{\xi\sim\mathcal{D}^{\text{train}}_{s,i}}\left[L(x,y_s,\xi)\right]+\mathcal{R}(y_s)\right],
    \end{aligned}
\end{equation}
where $L$ denotes the cross-entropy loss and $\mathcal{R}(y_s)=C_r\left\Vert y_s\right\Vert^2$ is a strongly convex regularization function.

In this experiment, we compare SPARKLE-ED with D-SOBA \citep{kong2024decentralized} and MAML \citep{finn2017model} in a decentralized communication setting over a 5-way 5-shot task across a network of $N=8$ nodes connected by Ring graph. The dataset used is miniImageNet \citep{vinyals2016matching}, derived from ImageNet \citep{russakovsky2015imagenet}, which comprises 100 classes, each containing 600 images of size $84\times84$. We set $R=2000$ and partition these classes into 64 for training, 16 for validation, and 20 for testing. For the training and validation classes, the data is split according to a Dirichlet distribution with parameter $\alpha=0.1$ \citep{lin2021quasi}. We utilize a four-layer CNN with four convolution blocks, where each block sequentially consists of a $3\times3$ convolution with 32 filters, batch normalizationm, ReLU activation, and $2\times2$ max pooling. The batch size is 32 , and $C_r=0.001$. The parameters of the last linear layer are designated as task-specific, while the other parameters are shared globally. For SPARKLE and D-SOBA, the step-sizes are $\beta=\gamma=0.1$ and $\alpha=0.01$. For MAML, the inner step-size is 0.1 and the outer step-size is 0.001, and the number of inner-loop steps as 3. For all algorithms, the task number is set to 32. And we only repeat the experiment only once due to the time limitation. Figure \ref{fig: FashionMINST_algs1} shows the average accuracy on the training dataset for all nodes, as well as 
 the test accuracy of the three algorithms. We  observe that SPARKLE-ED outperforms other algorithms, demonstrating the efficiency of SPARKLE in decentralized meta-learning problems.

\end{document}